\documentclass[11pt,reqno]{amsart}
\usepackage{amssymb,amsmath,amsfonts,amsthm,enumerate,stmaryrd}
\usepackage{mathrsfs}
\usepackage{graphicx}
\usepackage{caption}
\usepackage{subcaption}

\usepackage[left=.75 in, right=.75 in,top=.75 in, bottom=.75 in]{geometry}
\usepackage{nameref,hyperref,cleveref}
\numberwithin{equation}{section}

\providecommand{\abs}[1]{\left\vert#1\right\vert}
\providecommand{\norm}[1]{\left\Vert#1\right\Vert}

\providecommand{\Rn}[1]{\mathbb{R}^{#1}}

\providecommand{\csubset}{\subset\subset}

\providecommand{\br}[1]{\langle #1 \rangle}

\providecommand{\ns}[1]{\norm{#1}^2}
\providecommand{\bs}[1]{\left[#1\right]_\ell^2}

\providecommand{\ip}[1]{\left(#1 \right)}
 
\providecommand{\jump}[1]{\left\llbracket #1 \right\rrbracket }

\def\nab{\nabla}
\def\dt{\partial_t}
\def\hal{\frac{1}{2}}
\def\ep{\varepsilon}

\def\ls{\lesssim}
\def\p{\partial}

\def\sg{\mathbb{D}}

\def\da{\Delta_{\mathcal{A}}}
\def\naba{\nab_{\mathcal{A}}}
\def\diva{\diverge_{\mathcal{A}}}
\def\Sa{S_{\mathcal{A}}}

\def\h1{{_0}H^1((-\ell,\ell))}

\def\sdb{\mathcal{D}_{\shortparallel}}

\def\seb{\mathcal{E}_{\shortparallel}}

\providecommand{\abs}[1]{\left\vert#1\right\vert}
\providecommand{\norm}[1]{\left\Vert#1\right\Vert}

\providecommand{\ns}[1]{\norm{#1}^2}

\providecommand{\Rn}[1]{\mathbb{R}^{#1}}

\providecommand{\jump}[1]{\left\llbracket #1 \right\rrbracket }
\providecommand{\br}[1]{\langle #1 \rangle}
\providecommand{\pp}[1]{\left( \mspace{-2.0mu} \left( #1 \right) \mspace{-2.0mu} \right)}

\providecommand{\sdb}[1]{\bar{\mathcal{D}}_{#1}}
\providecommand{\seb}[1]{\bar{\mathcal{E}}_{#1}}

\def\hal{\frac{1}{2}}
\def\ep{\varepsilon}

\def\ls{\lesssim}

\def\jg{\jump{\gamma}}

\def\nab{\nabla}
\def\dt{\partial_t}
\def\p{\partial}

\def\da{\Delta_{\mathcal{A}}}
\def\naba{\nab_{\mathcal{A}}}
\def\diva{\diverge_{\mathcal{A}}}
\def\Sa{S_{\mathcal{A}}}
\def\sg{\mathbb{D}}
\def\sga{\mathbb{D}_{\mathcal{A}}}

\def\SH0{\mathcal{H}^0(\Omega)}

\def\oH{\mathring{H}}

\def\A{\mathcal{A}}
\def\B{\mathcal{B}}
\def\D{\mathcal{D}}
\def\E{\mathcal{E}}

\def\G{\mathcal{G}}
\def\h{\mathcal{H}}
\def\H{\mathcal{H}}
\def\J{\mathcal{J}}
\def\K{\mathcal{K}}
\def\L{\mathcal{L}}

\def\N{\mathcal{N}}
\def\P{\mathcal{P}}
\def\Q{\mathcal{Q}}
\def\R{\mathbb{R}}
\def\V{\mathcal{V}}

\def\af{\mathfrak{A}}

\def\sw{\mathscr{W}}

\def\swh{\hat{\mathscr{W}}}

\def\low{\alpha}
\def\linz{\kappa}
\def\epm{\ep_{\operatorname{max}}}
\def\omeq{\omega_{\operatorname{eq}}}
\def\theq{\theta_{\operatorname{eq}}}

\def\st{\;\vert\;}

\def\XXint#1#2#3{{\setbox0=\hbox{$#1{#2#3}{\int}$ }
\vcenter{\hbox{$#2#3$ }}\kern-.6\wd0}}

\DeclareMathOperator{\tr}{tr}

\DeclareMathOperator{\diverge}{div}

\DeclareMathOperator{\spn}{span}

\newtheorem{lem}{Lemma}[section]
\newtheorem{cor}[lem]{Corollary}
\newtheorem{prop}[lem]{Proposition}
\newtheorem{thm}[lem]{Theorem}
\newtheorem{remark}[lem]{Remark}

\newtheorem{dfn}[lem]{Definition}

\title[Stability of contact lines in fluids]{Stability of contact lines in fluids: 2D Navier-Stokes Flow}

\author{Yan Guo}
\address{
Division of Applied Mathematics\\
Brown University \\
182 George St., Providence, RI 02912, USA
}
\email[Y. Guo]{guoy@dam.brown.edu}
\thanks{Y. Guo was supported in part by NSF grant DMS-grant 1810868 and NSFC grant 10828103.}

\author{Ian Tice}
\address{
Department of Mathematical Sciences\\
Carnegie Mellon University\\
Pittsburgh, PA 15213, USA
}
\email[I. Tice]{iantice@andrew.cmu.edu}
\thanks{I. Tice was supported by an NSF CAREER Grant (DMS \#1653161). }



\begin{document}

\begin{abstract} 
In this paper we study the dynamics of an incompressible viscous fluid evolving in an open-top container in two dimensions.  The fluid mechanics are dictated by the Navier-Stokes equations.  
The upper boundary of the fluid is free and evolves within the container.  The fluid is acted upon by a uniform gravitational field, and capillary forces are accounted for along the free boundary.  The triple-phase interfaces where the fluid, air above the vessel, and solid vessel wall come in contact are called contact points, and the angles formed at the contact point are called contact angles.  The model that we consider integrates boundary conditions that allow for full motion of the contact points and angles.  Equilibrium configurations consist of quiescent fluid within a domain whose upper boundary is given as the graph of a function minimizing a gravity-capillary energy functional, subject to a fixed mass constraint.  The equilibrium contact angles can take on any values between $0$ and $\pi$ depending on the choice of capillary parameters.  The main thrust of the paper is the development of a scheme of a priori estimates that show that solutions emanating from data sufficiently close to the equilibrium exist globally in time and decay to equilibrium at an exponential rate.
\end{abstract}

\maketitle


\section{Introduction }\label{sec_intro}

\subsection{Equations of motion   }
The purpose of this paper is to study the dynamics of a viscous incompressible fluid occupying an open-top vessel in two dimensions.  The vessel is modeled as a bounded, connected, open subset $\mathcal{V} \subset \mathbb{R}^2$ obeying the following pair of assumptions.  First, we posit that the vessel's top is a rectangular channel by assuming that 
\begin{equation}
\V_{\operatorname{top}} :=  \V \cap \{y \in \Rn{2} \st y_2 \ge 0 \} = \{y \in \Rn{2} \st -\ell < y_1 < \ell, 0 \le y_2 < L \}
\end{equation}
for some given distances $\ell, L >0$.  Note that $L$ is the height of the channel, while $2\ell$ is its width.  The second assumption on the vessel is that its boundary, $\p \V \subset \R^2$, is $C^2$ away from the corner points $(\pm \ell, L)$.   We will use the notation 
\begin{equation}
\V_{\operatorname{btm}} :=  \V \cap \{y \in \Rn{2} \st y_2 \le 0 \} 
\end{equation}
to denote the bottom portion of the vessel, on which we place no further geometric restrictions.  We refer to Figure \ref{fig_1} for two examples of vessels of the type considered here.

\begin{figure}[t]
    \centering
    \begin{minipage}{0.45\textwidth}
        \centering
        \includegraphics[width=0.9\textwidth]{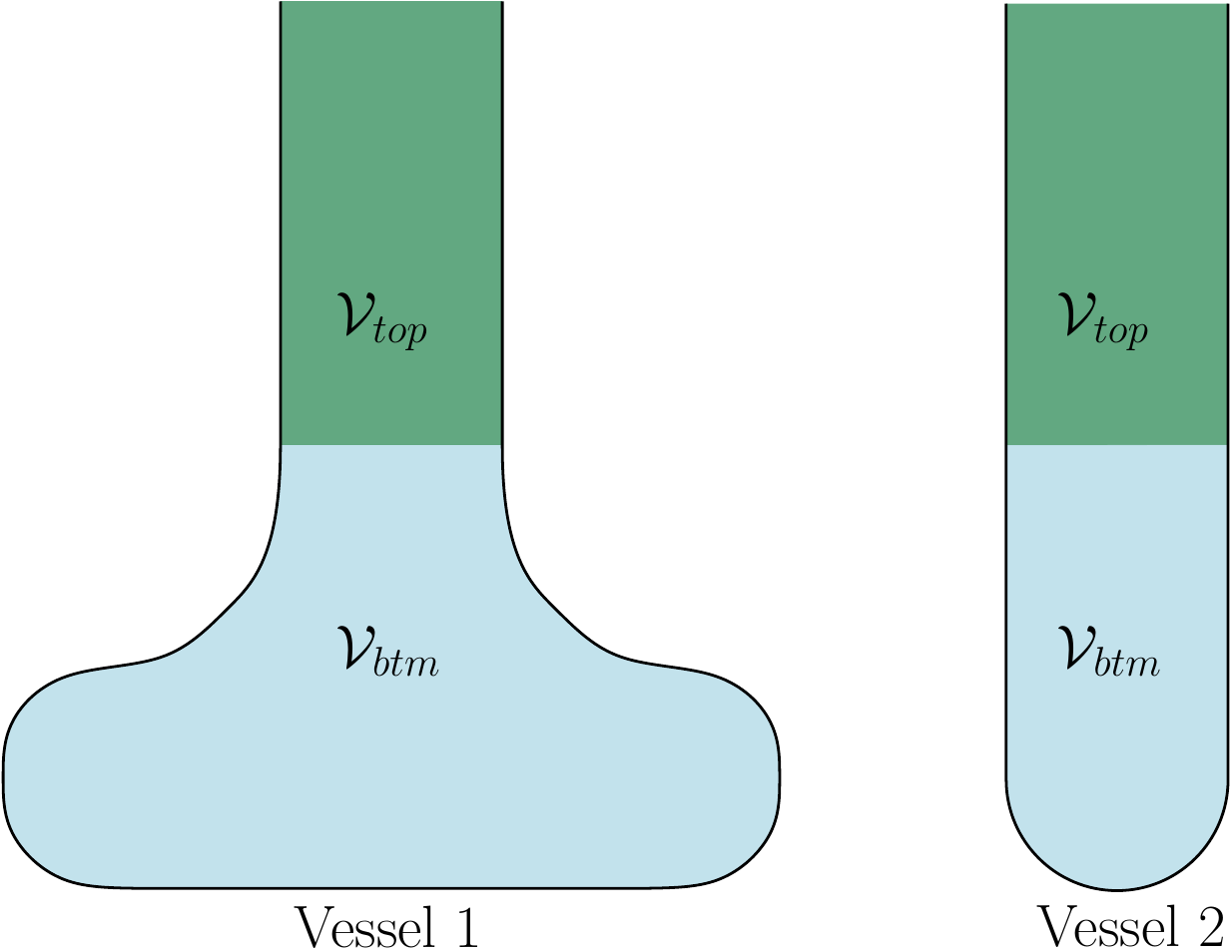} 
        \caption{Empty vessels}\label{fig_1}
    \end{minipage}\hfill
    \begin{minipage}{0.45\textwidth}
        \centering
        \includegraphics[width=0.9\textwidth]{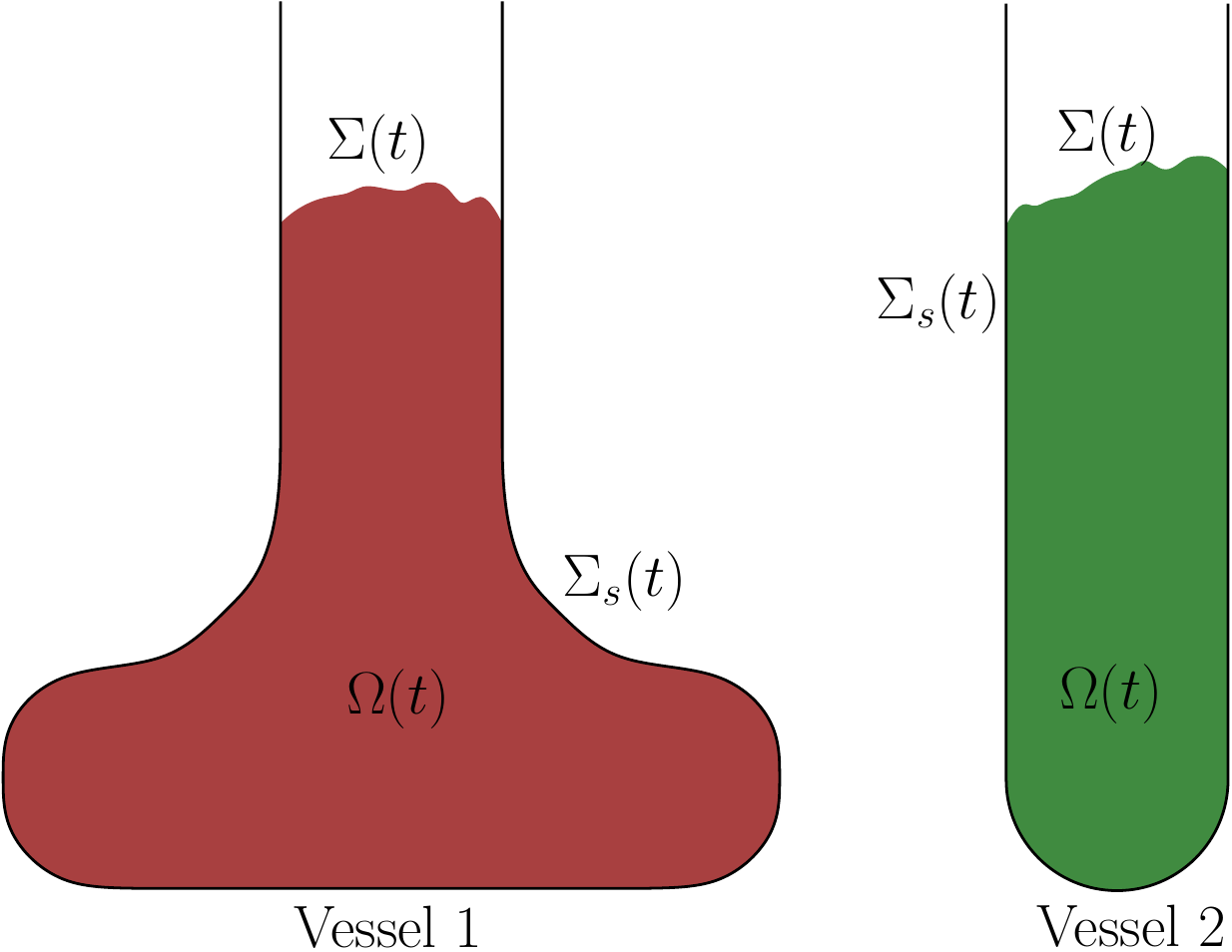} 
        \caption{Vessels with fluid}\label{fig_2}
    \end{minipage}
\end{figure}

The fluid is assumed to occupy the vessel in such a way that $\V_{\operatorname{btm}}$ is filled by the fluid, while $\V_{\operatorname{top}}$ is only partially filled, resulting in a free boundary where the fluid meets the air above the vessel.  For each time $t \ge 0$, this boundary is taken to be the graph of a function $\zeta(\cdot,t) : (-\ell,\ell)  \to (0,\infty)$ subject to the constraint that $\zeta(\pm \ell,t) \le L$.  The physical meaning of this constraint is that the fluid is assumed to not spill over the edges of the vessel.  Note, though, that we allow for the possibility that $\zeta(x,t) > L$ for some $x \in (-\ell,\ell)$ and $t \ge 0$, which corresponds to the fluid extending past the vessel's top away from the edges.  The points where the fluid, vessel, and air meet are $(\pm\ell, \zeta(\pm\ell,t))$ and are called the contact points.

In mathematical terms, we assume that the fluid occupies the time-dependent open set 
\begin{equation}\label{Omega_t_def}
\Omega(t) =  \V_{\operatorname{btm}} \cup  \{ y \in \R^{2} \st -\ell < y_1 < \ell, 0 < y_2 < \zeta(y_1,t)\}.
\end{equation}
We will write 
\begin{equation}
 \Sigma(t) = \{y \in \R^2 \st \abs{y_1} < \ell \text{ and } y_2 =   \zeta(y_1,t)\} \subset \p \Omega(t)
\end{equation}
for the moving fluid-vapor interface and 
\begin{equation}
 \Sigma_s(t) = \p \Omega(t) \backslash \Sigma(t)
\end{equation}
for the moving fluid-solid interface.  See Figure \ref{fig_2} for an example of two fluid domains in different types of vessels.  

The fluid's state is determined at each time by its velocity and pressure functions, $(u,P) : \Omega(t) \to \R^2 \times \R$, for which the associated viscous stress tensor is given by $S(P,u) : \Omega(t) \to \R^{2 \times 2}$ via 
\begin{equation}\label{stress_def}
 S(P,u) := PI - \mu \sg u,
\end{equation}
where $I$ is the $2 \times 2$ identity, $\mu >0$ is the fluid viscosity, and the symmetrized gradient is $\sg u = D u + (Du)^T$.  Extending the divergence operator to act on $S$ in the usual way, we have that $\diverge S(P,u) = \nab P - \mu \Delta u - \mu \nab \diverge u$.

In order to state the equations of motion, we first need to enumerate several terms that affect the dynamics.  The fluid is assumed to be of unit density and acted on by a uniform gravitational field pointing straight down with strength $g >0$.  Surface tension is accounted for, and we write $\sigma >0$ for the coefficient of tension  coefficient along the fluid-vapor interface, which is the graph of $\zeta(\cdot,t)$.  The parameter  $\beta > 0$ is the inverse slip length, which will appear in Navier's slip condition on the vessel side walls.  The energetic parameters $\gamma_{sv}, \gamma_{sf} \in \Rn{}$ measure the free-energy per unit length associated to the solid-vapor and solid-fluid interaction, respectively, and are the analogs of $\sigma$ for the other interfaces.  We define 
\begin{equation}\label{jg_def}
 \jg := \gamma_{sv} - \gamma_{sf},
\end{equation}
and we assume that $\jg$ and $\sigma$ satisfy the classical Young relation \cite{young}:
\begin{equation}\label{gamma_assume}
 \frac{\abs{\jg}}{\sigma} < 1.
\end{equation}
Finally, we define the contact point velocity response function $\sw: \R\to \R$ to be a $C^2$ increasing diffeomorphism such that $\sw(0) =0$.  

We can now state the equations of motion that govern the dynamics of the unknown triple $(u,P,\zeta)$ for $t >0$:
\begin{equation}\label{ns_euler}
\begin{cases}
\dt u + u\cdot \nab u +  \nab P - \mu \Delta u = 0 & \text{in }\Omega(t) \\ 
\diverge{u}=0 & \text{in }\Omega(t) \\ 
S(P,u) \nu = g \zeta \nu - \sigma \H(\zeta) \nu & \text{on } \Sigma(t) \\ 
(S(P,u)\nu - \beta u)\cdot \tau =0 &\text{on } \Sigma_s(t) \\
u \cdot \nu =0 &\text{on } \Sigma_s(t) \\
\partial_t \zeta = u_2 - u_1 \partial_{y_1}\zeta  &\text{on } \Sigma(t) \\ 
\sw(\dt \zeta(\pm \ell,t)) =  \jg \mp \sigma \frac{\p_1 \zeta}{\sqrt{1+\abs{\p_1 \zeta}^2}}(\pm \ell,t)
\end{cases}
\end{equation}
where $\nu$ the outward-pointing unit normal, $\tau$ is the associated unit tangent, and
\begin{equation}\label{H_def}
 \mathcal{H}(\zeta) := \p_1 \left( \frac{\p_1 \zeta}{\sqrt{1+\abs{\p_1 \zeta}^2}} \right)
\end{equation}
is the mean-curvature operator.  The first two equations in \eqref{ns_euler} are the incompressible Navier-Stokes equations for a fluid of unit density.  The third equation is the balance of stress on the free surface, which is also called the dynamic boundary condition.  Note that in principle the gravitational forcing term $-g e_2$ should appear as a bulk force in the first equation, but by shifting the pressure unknown via $P \mapsto P + g x_2$ we have shifted gravity to a surface term, as it is more convenient in this form.  The fourth and fifth equations in \eqref{ns_euler} constitute the Navier-slip condition; in contrast with the no-slip condition, the Navier condition allows for fluid slip along the fluid-solid interface, at the expense of generating a stress that acts against the motion.  The sixth equation in \eqref{ns_euler} is called the kinematic equation, as it tracks how the free surface function changes due to the fluid velocity.  The final equation in \eqref{ns_euler}, which is essential in our analysis and will be discussed more later in Section \ref{sec_res_disc}, is the contact point response equation.

The problem \eqref{ns_euler} is an evolution equation and must be augmented with two pieces of initial data:
\begin{enumerate}
 \item the initial free surface, $\zeta(\cdot,0): (-\ell,\ell) \to (0,\infty)$, which determines the initial fluid domain $\Omega(0)$,
 \item the initial fluid velocity $u_0 : \Omega(0) \to \R^2$, which satisfies $\diverge u_0 =0$ in $\Omega(0)$ and $u_0 \cdot \nu = 0$ on $\Sigma_s(0)$.
\end{enumerate}
As usual for the Navier-Stokes system, the initial pressure does not need to be specified.  The initial mass of the fluid is denoted by 
\begin{equation}
 M_0 := \abs{\Omega(0)} = \abs{\V_{\operatorname{btm}}} + M_{top}, \text{ where } M_{top} =  \int_{-\ell}^\ell \zeta(y_1,0) dy_1.
\end{equation}
The fluid's mass is conserved in time due to the combination of the kinematic boundary condition and the solenoidal condition for $u$ from \eqref{ns_euler}:
\begin{equation}\label{avg_prop}
\frac{d}{dt} \abs{\Omega(t)} = \frac{d}{dt}  \int_{-\ell}^\ell \zeta =  \int_{-\ell}^\ell \dt \zeta  = \int_{\Sigma(t)} u \cdot \nu = \int_{\Omega(t)} \diverge{u} = 0.
\end{equation}

\subsection{Equilibria    }

A steady state equilibrium solution to \eqref{ns_euler} corresponds to setting $u =0$, $P(y,t) = P_0 \in \R$, and $\zeta(y_1,t) = \zeta_0(y_1)$ with $\zeta_0$ and $P_0$ solving 
\begin{equation}\label{zeta0_eqn}
 \begin{cases}
 g \zeta_0 - \sigma \H(\zeta_0) = P_0 & \text{on } (-\ell,\ell) \\ 
 \sigma \frac{\p_1 \zeta_0}{\sqrt{1+\abs{\p_1 \zeta_0}^2}}(\pm \ell) = \pm \jg.
 \end{cases}
\end{equation}
By a slight abuse of notation, solutions to \eqref{zeta0_eqn} are called equilibrium capillary surfaces.  Note that the boundary condition specifies the cosine of the angle formed by the graph at the endpoints.  The constant pressure $P_0$ is not arbitrary; indeed, it is uniquely determined by specifying the mass in $\V_{\operatorname{top}}$ at equilibrium, i.e. prescribing 
\begin{equation}\label{zeta0_constraint}
 M_{top} =  \int_{-\ell}^\ell \zeta_0(y_1) dy_1.
\end{equation}
To see this, we use \eqref{zeta0_eqn} to compute 
\begin{equation}
 2\ell P_0 = \int_{-\ell}^\ell P_0 = \int_{-\ell}^\ell g \zeta_0 - \sigma \H(\zeta_0) = g M_{top} -\sigma \left. \frac{\p_1 \zeta_0}{\sqrt{1+\abs{\p_1 \zeta_0}^2}} \right\vert_{-\ell}^\ell 
= g M_{top} -2 \jg,
\end{equation}
which in turn implies that
\begin{equation}\label{p0_def}
 P_0 = \frac{g M_{top} -2 \jg}{2\ell}.
\end{equation}

The equations \eqref{zeta0_eqn} are the Euler-Lagrange equations associated to constrained minimizers of the energy functional $\mathscr{I} : W^{1,1}((-\ell,\ell)) \to \Rn{}$ defined via
\begin{equation}\label{zeta0_energy}
 \mathscr{I}(\zeta) = \int_{-\ell}^\ell \frac{g}{2} \abs{\zeta}^2 + \sigma \sqrt{1 + \abs{\zeta'}^2} - \jg(\zeta(\ell) + \zeta(-\ell)),
\end{equation}
subject to the mass constraint $M_{top} = \int_{-\ell}^\ell \zeta$.  In this framework the pressure $P_0$ is understood as the Lagrange multiplier associated to this constraint.  We now state an existence result for equilibrium capillary surfaces.  For a detailed proof we refer, for instance, to Appendix E of \cite{guo_tice_QS}, which is a one dimensional version of results found in the book of Finn \cite{finn}.

\begin{thm}\label{zeta0_wp}
There exists a constant $M_{min} \ge 0$ such that if $M_{top} > M_{min}$ then there exists a unique solution $\zeta_0 \in C^\infty([-\ell,\ell])$ to \eqref{zeta0_eqn} that satisfies \eqref{zeta0_constraint} with $P_0$ given by \eqref{p0_def}.  Moreover, $\zeta_0$ is even,  $\min_{[-\ell,\ell]} \zeta_0 >0$, and if $\mathscr{I}$ is given by \eqref{zeta0_energy}, then $\mathscr{I}(\zeta_0) \le \mathscr{I}(\psi)$ for all $\psi \in W^{1,1}((-\ell,\ell))$ such that $\int_{-\ell}^\ell \psi = M_{top}$.
\end{thm}

Throughout the rest of the paper we make the following two crucial assumptions on the parameters. 
\begin{enumerate}
 \item We assume that $M_{top} > M_{min}$ in order to have an equilibrium $\zeta_0$ as in Theorem \ref{zeta0_wp}. 
 \item We assume that the parameter $L >0$, the height of the rectangular channel $\V_{\operatorname{top}}$, satisfies the condition $\zeta_0(\pm\ell) < L$, which means the fluid is not just about to spill over the top of the vessel.
\end{enumerate}

\subsection{Previous work and origins of the model \eqref{ns_euler} }\label{sec_discussion_model}

The contact lines (or contact points in two dimension) that form at triple junctions between three distinct phases (fluid, solid, and vapor phases in the present paper) have been a subject of intense study since the pioneering work of Young \cite{young} in 1805.  For an exhaustive overview we refer to  de Gennes \cite{degennes}.  Here we will content ourselves with a terse review. 

The story began with the study of equilibrium configurations by Young \cite{young}, Laplace \cite{laplace}, and Gauss \cite{gauss}, who discovered the underlying variational principle for $\mathscr{I}$ described above and in Theorem \ref{zeta0_wp} (though, obviously, the theorem is restated in the modern language of Sobolev spaces).  A key byproduct of this work is that the angle formed between the solid wall and the fluid (through the vapor phase), which is known as the equilibrium contact angle $\theq$ (see Figure \ref{fig_3}), is related to the free energy parameters $\gamma_{sf}$, $\gamma_{sv}$, and $\sigma$ via Young's equation
\begin{equation}\label{young_relat}
 \cos(\theq) = \frac{\gamma_{sf} - \gamma_{sv}}{\sigma} = -\frac{\jg}{\sigma}.
\end{equation}
Note that this manifests in \eqref{zeta0_eqn} through the equations for $\p_1 \zeta_0$ at the endpoints.
 
\begin{figure}[h]
        \centering
        \includegraphics[width=0.45\textwidth]{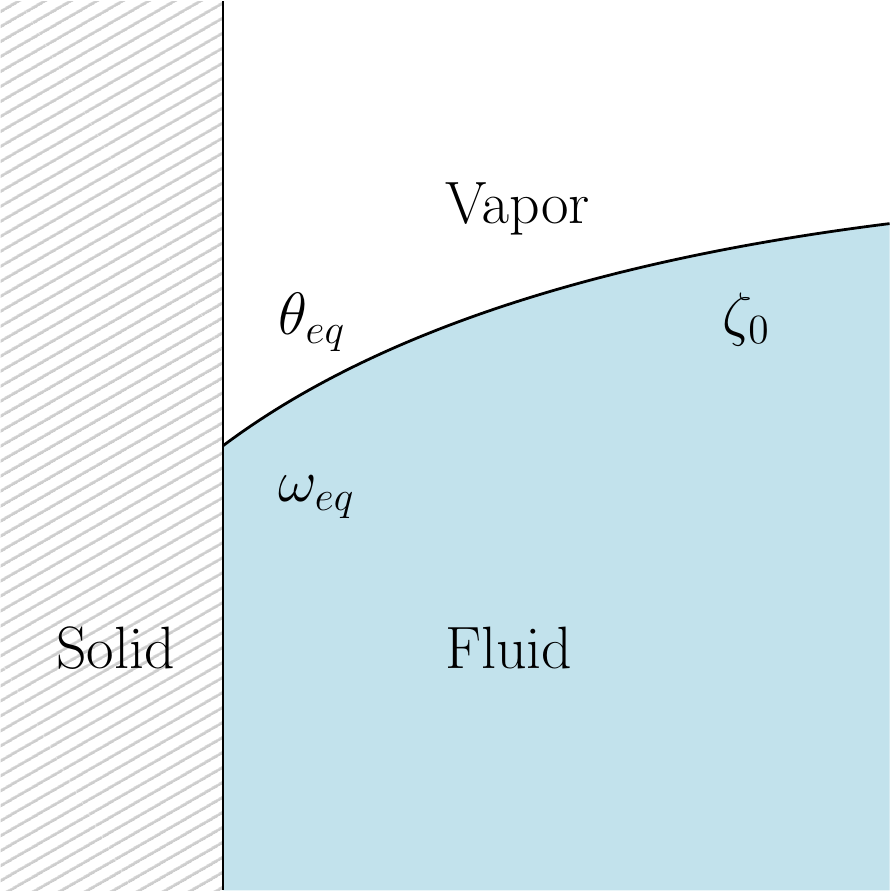} 
        \caption{Equilibrium contact angle}\label{fig_3}
\end{figure} 
 
The dynamic behavior of a contact line or point is significantly more delicate.  For instance, including a dynamic contact point in a fluid-solid-vapor model presents challenges to standard modeling assumptions made when working with viscous fluids.  Indeed, the free boundary kinematics (which may be rewritten as $\dt \zeta = u\cdot \nu\sqrt{1 + \abs{\nab \zeta}^2}$) and the typical no-slip boundary conditions for viscous fluids ($u=0$ at the fluid-solid interface) are incompatible: combining the two leads to the prediction that $\dt \zeta =0$ at the contact points, i.e. that the fluid is pinned at its initial position on the solid.  A moment's experimentation with an everyday coffee cup reveals this prediction to be nonsensical, and we are led to abandon the no-slip condition in favor of another boundary condition that allows for motion of the contact point.

The surveys of Dussan \cite{dussan} and Blake \cite{blake} provide a thorough discussion of the efforts of physicists and chemists in determining the dynamics of a contact point.  The general picture is that the dynamic contact angle, $\theta_{\operatorname{dyn}}$,  and the equilibrium angle, $\theta_{\operatorname{eq}}$, are related via 
 \begin{equation}\label{cl_motion}
  V_{cl} = F( \cos(\theta_{\operatorname{dyn}}) - \cos(\theq) ),
 \end{equation}
where $V_{cl}$ is the contact point velocity (along the solid) and $F$ is some increasing function such that $F(0)=0$.  The assumptions on $F$ enforce the experimentally observed fact that the slip of the contact line acts to restore the equilibrium angle (see Figure \ref{fig_4}).  Equations of the form \eqref{cl_motion}, but with different forms of $F$, have been derived in a number of ways.  Blake-Haynes \cite{blake_haynes} arrived at \eqref{cl_motion} through thermodynamic and molecular kinetics arguments.  Cox \cite{cox} used matched asymptotic analysis and hydrodynamic arguments.  Ren-E \cite{ren_e_deriv} derived \eqref{cl_motion} from thermodynamic principles applied to constitutive equations.  Ren-E \cite{ren_e} also performed molecular dynamics simulations and found an equation of the form \eqref{cl_motion}.  These simulations also indicated that the slip of the fluid along the solid obeys the well-known Navier-slip condition 
\begin{equation}\label{navier_slip}
u \cdot \nu =0 \text{ and }  S(P,u) \nu \cdot \tau = \beta u \cdot \tau
\end{equation}
for some parameter $\beta >0$.  The system \eqref{ns_euler} studied in the present paper synthesizes the Navier-slip boundary conditions \eqref{navier_slip} with the general form of the contact point equation \eqref{cl_motion}.  Indeed, the last equation in \eqref{ns_euler} may be rewritten as 
\begin{equation}
\sw(V_{cl}) =  \sw(\dt \zeta) =    \jg \mp \sigma \frac{\p_1 \zeta}{\sqrt{1+\abs{\p_1 \zeta}^2}}(\pm \ell,t)    =  \sigma (\cos(\theta_{\operatorname{dyn}}) - \cos(\theq) ),
\end{equation}
which is \eqref{cl_motion} with the convenient reformulation $\sw = \sigma F^{-1}$.

\begin{figure}
		\centering
		\begin{subfigure}{0.48\textwidth}
			\centering
			\includegraphics[width=0.9\textwidth]{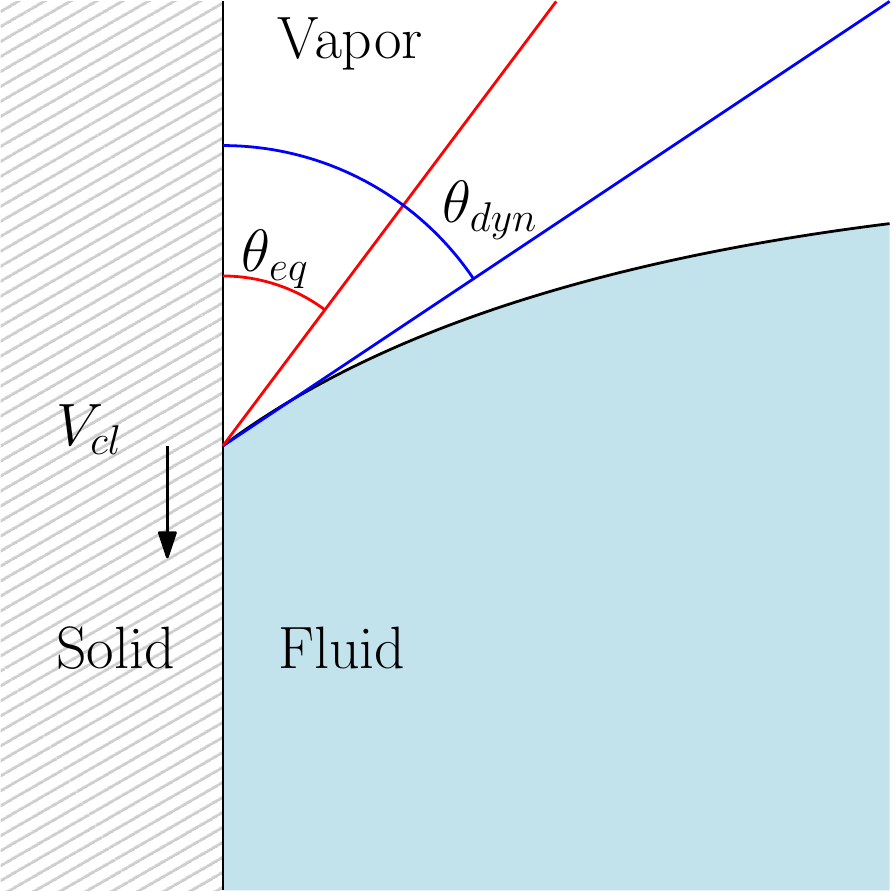}
			\caption{$\theta_{eq} < \theta_{dyn}$.} 
		\end{subfigure}
		\begin{subfigure}{0.48\textwidth}
			\centering
			\includegraphics[width=0.9\textwidth]{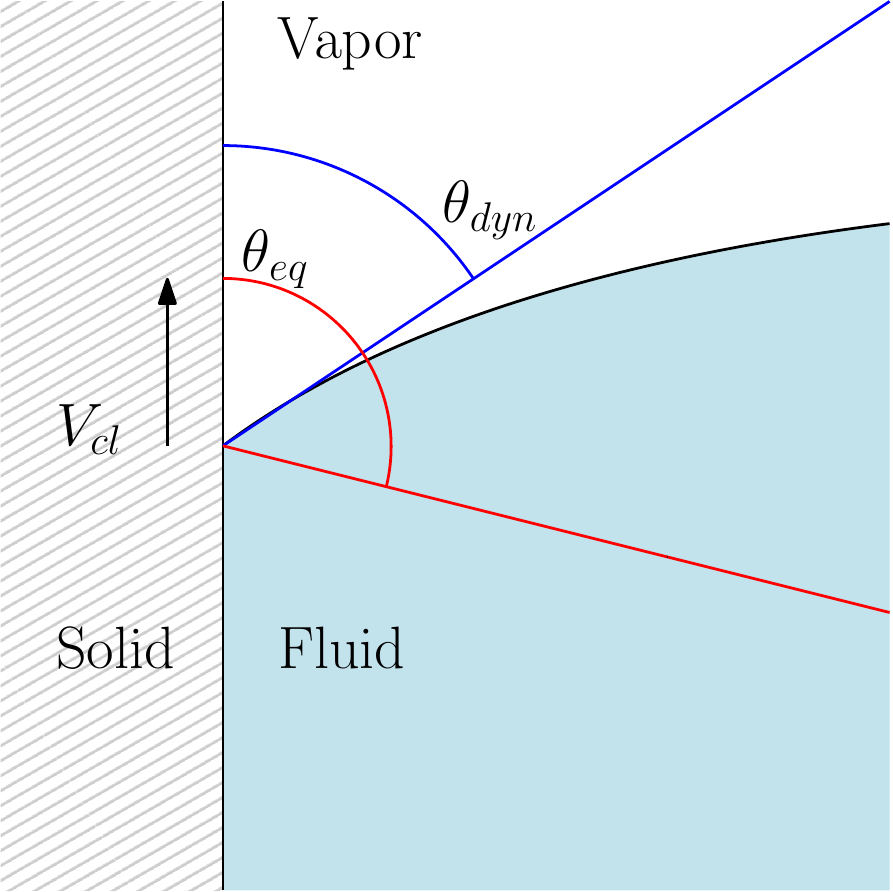}
			\caption{$\theta_{eq} > \theta_{dyn}$}
		\end{subfigure}
		\caption{
		The same dynamic fluid configuration, with the dynamic contact angle $\theta_{dyn}$ marked in blue,  but with different equilibrium contact angles $\theta_{eq}$, marked in red. In configuration (A) the condition $\theta_{eq} < \theta_{dyn}$ in \eqref{cl_motion} results in a downward pointing contact point velocity.  In configuration (B) the condition $\theta_{eq} > \theta_{dyn}$ in \eqref{cl_motion} results in an upward pointing contact point velocity.  In both cases, the resulting motion acts to return the dynamic angle to the equilibrium one.
		}
		\label{fig_4}
\end{figure}

Given the numerous derivations of \eqref{cl_motion}, we believe that its integration into the model \eqref{ns_euler} along with the Navier-slip condition yields a good general model for describing the dynamics of a viscous fluid with dynamic contact points and contact angles.  A  goal of this article is to provide further evidence for the validity of the model by proving that the equilibrium capillary surfaces are asymptotically stable, or more precisely, that sufficiently small perturbations of the equilibria give rise to global-in-time solutions that return to equilibrium exponentially fast as time diverges to infinity.  In recent previous work \cite{guo_tice_QS} we proved this in the much simpler case in which the Navier-Stokes equations in \eqref{ns_euler} were replaced by the Stokes equations, which yields a sort of quasi-static evolution.  The second author and Wu \cite{tice_wu} proved corresponding results for the Stokes droplet problem in which the vessel configuration is replaced with a droplet sitting atop a flat substrate, and with Zheng \cite{tice_zheng} established local existence results.  The Navier-Stokes problem presents numerous challenges relative to the Stokes problem, but we will delay a discussion of these to Section \ref{sec_res_disc}.

To the best of our knowledge, there are no other prior results in the literature related to models in which the full fluid mechanics are considered alongside dynamic contact points and contact angles.  However, there are results with a subset of these features.  Schweizer \cite{schweizer} studied a 2D Navier-Stokes problem with a fixed contact angle of $\pi/2$.  Bodea \cite{bodea} studied a similar problem with fixed $\pi/2$ contact angle in 3D channels with periodicity in one direction.  Kn\"upfer-Masmoudi \cite{knupfer_masmoudi} studied the dynamics of a 2D drop with fixed contact angle when the fluid is assumed to be governed by Darcy's law.  Related analysis of the fully stationary Navier-Stokes system with free, but unmoving boundary, was carried out in 2D by Solonnikov \cite{solonnikov} with contact angle fixed at $\pi$, by Jin \cite{jin} in 3D with angle $\pi/2$,  and by Socolowsky \cite{socolowsky} for 2D coating problems with fixed contact angles.  A simplified droplet model without fluid coupling was studied by Feldman-Kim \cite{feldman_kim}, who proved asymptotic stability using gradient flow techniques.  It is worth noting that much work has also been done on contact points in simplified thin-film models; we refer to the survey by Bertozzi \cite{bertozzi} for an overview.

We conclude this overview of the model with some stability heuristics.  Sufficiently regular solutions to \eqref{ns_euler} obey   the energy-dissipation equation
\begin{multline}\label{fund_en_evolve}
 \frac{d}{dt}\left( \int_{\Omega(t)} \hal \abs{u(\cdot,t)}^2 +   \mathscr{I}( \zeta(\cdot,t)) \right) \\
 + \int_{\Omega(t)} \frac{\mu}{2} \abs{\sg u(\cdot,t)}^2 + \int_{\Sigma_s(t)} \frac{\beta}{2} \abs{u(\cdot,t)  }^2 + \sum_{a=\pm 1} \dt \zeta(a \ell,t) \sw(\dt \zeta(a \ell,t))= 0,
\end{multline}
where $\mathscr{I}$ is the energy functional from \eqref{zeta0_energy}.  This identity may be derived in the usual way by dotting the first equation in \eqref{ns_euler} by $u$, integrating by parts over $\Omega(t)$, and employing the other equations.  The temporally differentiated term in parentheses is the physical energy, comprised of the fluid's kinetic energy (the first term) and the gravity-capillary potential energy (the second term).  The three remaining terms are the dissipation due to viscosity (the first term), slip along fluid-solid interface (the second), and slip along the contact point (the third).  Crucially, the assumptions on $\sw$ imply that $z \sw(z) >0$ for $z \neq 0$, which means the contact point dissipation term provides positive definite control of $\dt \zeta$ at the contact point.  Thus, the dissipation has a sign and serves to decrease the energy.  Since the equilibrium configuration $u=0$, $p=0$, $\zeta = \zeta_0$ is the unique global minimizer of the energy, \eqref{fund_en_evolve} formally suggests that global-in-time solutions will converge to the equilibrium as $t \to \infty$.  We will prove that this is indeed the case, provided that the initial data are sufficiently close to the equilibrium configuration, and we will show that such solutions must decay to equilibrium exponentially.

\subsection{Problem reformulation   }
 
In order to analyze the system \eqref{ns_euler} it is convenient to reformulate the problem in a fixed open set.  The stability heuristic given above suggests that for large time, the fluid domain should not differ much from the equilibrium domain, which suggests that we employ it as the fixed open set.  To this end we consider $\zeta_0 \in C^\infty([-\ell,\ell])$ from Theorem \ref{zeta0_wp} and define the equilibrium domain $\Omega \subset \R^2$ via 
\begin{equation}\label{Omega_def}
 \Omega := \V_{\operatorname{btm}} \cup  \{ x \in \Rn{2} \;\vert\; -\ell < x_1 < \ell \text{ and } 0 < x_2 < \zeta_0(x_1) \}.
\end{equation}
Note that $\p \Omega$ is $C^2$ away from the contact points $(\pm\ell,\zeta_0(\pm \ell))$, but that $\Omega$ has corner singularities there, so $\p \Omega$ is only Lipschitz globally.  Depending on the choice of the capillary parameters $\sigma$, $\gamma_{sv}$, and $\gamma_{sf}$, the angles formed at the contact points can take on any value between $0$ and $\pi$.  

We decompose the boundary $\p \Omega = \Sigma \sqcup \Sigma_s$, where
\begin{equation}\label{sigma_def}
 \Sigma := \{ x \in \Rn{2} \;\vert\; -\ell < x_1 < \ell \text{ and } x_2 = \zeta_0(x_1) \} \text{ and }
 \Sigma_s := \p \Omega \backslash \Sigma.
\end{equation}
The set $\Sigma$ is the equilibrium free surface, while $\Sigma_s$ denotes the ``sides'' of the equilibrium fluid configuration.  We will write $x \in \Omega$ as the spatial coordinate in the equilibrium domain.  We will write 
\begin{equation}\label{N0_def}
 \N_0 = (-\p_1 \zeta_0,1)
\end{equation}
for the non-unit normal vector field on $\Sigma$.

In our analysis we will assume that the free boundary is a small perturbation of the equilibrium interface by introducing the perturbation  $\eta:(-\ell,\ell) \times \R^{+} \to \R$ and positing that
\begin{equation}
 \zeta(x_1,t) = \zeta_0(x_1) + \eta(x_1,t).
\end{equation}
We will need to define an extension of $\eta$ that gains regularity.  To this end we first choose $E$ to be a bounded linear extension operator that maps $C^m((-\ell,\ell))$ to $C^m(\R)$ for all $0 \le m \le 5$ and  $W^{s,p}((-\ell,\ell))$ to $W^{s,p}(\R)$ for all $0 \le s \le 5$ and $1 \le p < \infty$ (such a map is readily constructed with the help of higher order reflections, Vandermonde matrices, and a cutoff function - see, for instance, Exercise 7.24 in \cite{leoni} for integer regularity, but non-integer regularity follows then by interpolating).  In turn, we define the extension of $\eta$ to be the function $\bar{\eta} : \{x \in \R^2 \st    x_2 \le  E\zeta_0(x_1)\} \times \R^+ \to \R$ given by
\begin{equation}\label{extension_def}
 \bar{\eta}(x,t) = \P E \eta(x_1,x_2 - E\zeta_0(x_1),t), 
\end{equation}
where $\mathcal{P}$ is the lower Poisson extension defined by \eqref{poisson_def}.  Note that although $\bar{\eta}(\cdot,t)$ is a priori defined in the unbounded set $\{x \in \R^2 \st    x_2 \le  E\zeta_0(x_1)\}$, in practice we will only ever use its restriction to the bounded set $\Omega \subset \{x \in \R^2 \st    x_2 \le  E\zeta_0(x_1)\}$.

Choose $\phi \in C^\infty(\R)$ such that $\phi(z) =0$ for $z \le \frac{1}{4} \min \zeta_0$ and $\phi(z) =z$ for $z \ge \hal \min \zeta_0$.  We combine $\phi$ and the extension $\bar{\eta}$ to define a map from the equilibrium domain to the moving domain $\Omega(t)$:
\begin{equation}\label{mapping_def}
 \Omega \ni x \mapsto   \left( x_1,x_2 +  \frac{\phi(x_2)}{\zeta_0(x_1)} \bar{\eta}(x,t) \right) := \Phi(x,t) = (y_1,y_2) \in \Omega(t).
\end{equation}
It is readily verified that the map $\Phi$ satisfies the following properties:
\begin{enumerate}
 \item $\Phi(x_1,\zeta_0(x_1),t) = (x_1, \zeta_0(x_1) + \eta(x_1,t)) = (x_1,\zeta(x_1,t))$, and hence $\Phi(\Sigma,t) = \Sigma(t)$,
 \item $\Phi(x,t) =x$ for $x \in \V_{\operatorname{btm}}$, i.e. the map is the identity in the bottom portion of the vessel and thus only distorts the upper rectangular channel $\V_{\operatorname{top}}$,
 \item $\Phi(\pm \ell, x_2,t) = (\pm \ell, x_2+ \phi(x_2)\bar{\eta}(\pm\ell ,x_2)/\zeta_0(\pm \ell))$, and hence $\Phi(\Sigma_s \cap \{x_1 = \pm \ell, x_2 \ge 0\},t) = \Sigma_s(t) \cap \{y_1 = \pm \ell, y_2 \ge 0\}$. 
\end{enumerate}
Moreover, if $\eta$ is sufficiently small (in an appropriate Sobolev space), then the mapping $\Phi$ will be a $C^1$ diffeomorphism of $\Omega$ onto $\Omega(t)$ that maps the components of $\p \Omega$ to the corresponding components of $\p \Omega(t)$.  

We will use $\Phi$ to reformulate \eqref{ns_euler} in $\Omega$, but first it is convenient to introduce some notation.  We write
\begin{equation}\label{A_def}
 \nab \Phi = 
\begin{pmatrix}
 1 & 0  \\
 A & J 
\end{pmatrix}
\text{ and }
 \mathcal{A} := (\nab \Phi^{-1})^T = 
\begin{pmatrix}
 1 & -A K \\
 0 & K
\end{pmatrix}
\end{equation}
for 
\begin{equation}\label{AJK_def}
 W = \frac{\phi}{\zeta_0}, \quad
 A = W \p_1 \bar{\eta} - \frac{W}{\zeta_0} \p_1 \zeta_0 \bar{\eta}, \quad 
 J = 1 + W \p_2 \bar{\eta} + \frac{\phi' \bar{\eta}}{\zeta_0}, \quad
 K = J^{-1}.
\end{equation}
Note that the Jacobian of our coordinate transformation is exactly $J = \det{\nab \Phi}$.

Provided that $\Phi$ is a diffeomorphism (which will always be satisfied in our analysis), we can reformulate \eqref{ns_euler} by using $\Phi$ to change coordinate systems.  This results in a PDE system that has the benefit of being posed in a fixed set but the downside of being significantly more nonlinear.  In the new system the PDE becomes 
\begin{equation}\label{ns_flattened}
\begin{cases}
\dt u -\dt \bar{\eta} W K \p_2 u + u \cdot \naba u  + \diva S_{\A}(P,u) =0 & \text{in }\Omega \\
\diva u = 0 &\text{in } \Omega \\
S_{\A}(P,u)\N = \left(g \zeta - \sigma \H(\zeta)    \right) \N &\text{on }\Sigma \\
\dt \eta = u \cdot \N &\text{on }\Sigma \\
(S_{\A}(P,u) \cdot \nu - \beta u)\cdot \tau = 0 &\text{on }\Sigma_s \\
u \cdot \nu =0 & \text{on }\Sigma_s \\
\sw(\dt \eta(\pm \ell,t)) = \jg \mp \sigma \frac{\p_1 \zeta}{\sqrt{1+\abs{\p_1 \zeta}^2}}(\pm \ell,t),
\end{cases}
\end{equation}
where $\zeta = \zeta_0 + \eta$ and 
\begin{equation}\label{N_def}
 \N = (-\p_1 \zeta,1) = \N_0 - (\p_1 \eta,0)
\end{equation}
is the non-unit normal to the moving free boundary.  Here we have written the differential operators $\naba$, $\diva$, and $\da$ with their actions given by $(\naba f)_i := \A_{ij} \p_j f$, $\diva X := \A_{ij}\p_j X_i$, and $\da f = \diva \naba f$ for appropriate $f$ and $X$.  The vector field $u\cdot \naba u$ has components $(u \cdot \naba u)_i := u_j \A_{jk} \p_k u_i$.  We also  write $\Sa(P,u)  = P I  - \mu \sg_{\A} u$ for the stress tensor, where $I$ the $2 \times 2$ identity and $(\sg_{\A} u)_{ij} = \A_{ik} \p_k u_j + \A_{jk} \p_k u_i$ is the symmetric $\A-$gradient.  Note that if we extend $\diva$ to act on symmetric tensors in the natural way, then $\diva \Sa(P,u) = \naba P - \mu \da u$ for vector fields satisfying $\diva u=0$.  

Now that we have reformulated our PDE system in a fixed domain, it is convenient to make a final modification by rewriting \eqref{ns_flattened} as a perturbation of the equilibrium configuration.  In other words, we posit that the solution has the special form  $u =0 + u$, $P = P_0 + p$, $\zeta = \zeta_0 + \eta$ for new unknowns $(u,p,\eta)$.  In order to record the perturbed equations, we first need to introduce some notation.

To begin, we use a Taylor expansion in $z$ to write
\begin{equation}
 \frac{y+z}{(1+\abs{y+z}^2)^{1/2}} = \frac{y}{(1+\abs{y}^2)^{1/2}} + \frac{z}{(1+\abs{y}^2)^{3/2}} +   \mathcal{R}(y,z),
\end{equation}
where $\mathcal{R} \in C^\infty(\R^{2})$ is given by
\begin{equation}\label{R_def}
 \mathcal{R}(y,z) =  \int_0^z 3 \frac{(s-z)(s+y)}{(1+ \abs{y+s}^2)^{5/2}}  ds.
\end{equation}
By construction,
\begin{equation}
 \frac{\p_1 \zeta}{(1+\abs{\p_1 \zeta}^2)^{1/2}} = \frac{\p_1 \zeta_0}{(1+\abs{\p_1 \zeta_0}^2)^{1/2}} + \frac{\p_1 \eta}{(1+\abs{\p_1 \zeta_0}^2)^{3/2}} +   \mathcal{R}(\p_1 \zeta_0,\p_1 \eta),
\end{equation}
which then allows us to use \eqref{zeta0_eqn} to compute
\begin{multline}\label{pert_comp_1}
 g \zeta - \sigma \H(\zeta) = \left( g \zeta_0 - \sigma \H(\zeta_0)\right) 
+ g \eta - \sigma \p_1 \left(\frac{\p_1 \eta}{(1+\abs{\p_1 \zeta_0}^2)^{3/2}}\right) 
-\sigma \p_1 \left(  \mathcal{R}(\p_1 \zeta_0,\p_1 \eta) \right) \\
= P_0  + g \eta - \sigma \p_1 \left(\frac{\p_1 \eta}{(1+\abs{\p_1 \zeta_0}^2)^{3/2}}\right) 
-\sigma \p_1 \left(  \mathcal{R}(\p_1 \zeta_0,\p_1 \eta) \right)
\end{multline}
and 
\begin{multline}
 \jg \mp \sigma \frac{\p_1 \zeta}{\sqrt{1+\abs{\p_1 \zeta}^2}}(\pm \ell,t) =
\jg \mp \frac{\sigma \p_1 \zeta_0}{(1+\abs{\p_1 \zeta_0}^2)^{1/2}}(\pm \ell) \mp \frac{\sigma \p_1 \eta}{(1+\abs{\p_1 \zeta_0}^2)^{3/2}}(\pm \ell,t) \\
\mp   \sigma \mathcal{R}(\p_1 \zeta_0,\p_1 \eta)(\pm \ell,t) 
=  \mp \frac{\sigma \p_1 \eta}{(1+\abs{\p_1 \zeta_0}^2)^{3/2}}(\pm \ell,t) \mp \sigma  \mathcal{R}(\p_1 \zeta_0,\p_1 \eta)(\pm \ell,t).
\end{multline}
Next, we compute 
\begin{multline}\label{pert_comp_2}
\diva \Sa(P,u) = \diva \Sa(p,u) \text{ in } \Omega, \; \Sa(P,u) \N = \Sa(p,u) \N + P_0 \N \text{ on }\Sigma, \\ 
\text{and } \Sa(P,u)\nu \cdot \tau = \Sa(p,u)\nu \cdot \tau \text{ on }\Sigma_s. 
\end{multline}
Finally,  we expand the velocity response function inverse $\sw \in C^2(\R)$.  Since $\sw$ is increasing, we may set 
\begin{equation}\label{linz_def}
 \linz = \sw'(0) >0. 
\end{equation}
We then define the perturbation $\swh \in C^2(\Rn{})$ as 
\begin{equation}\label{V_pert}
 \swh(z) = \frac{1}{\linz} \sw(z) - z.
\end{equation}

We now insert the expansions \eqref{pert_comp_1}--\eqref{pert_comp_2} and \eqref{V_pert} into \eqref{ns_flattened}.  This yields the following equivalent PDE system for the perturbed unknown s$(u,p,\eta)$:
\begin{equation}\label{ns_geometric}
\begin{cases}
\dt u -\dt \bar{\eta} W K \p_2 u + u \cdot \naba u  + \diva S_{\A}(p,u) =0 & \text{in }\Omega \\
\diva u = 0 &\text{in } \Omega \\
S_{\A}(p,u)\N = \left[g \eta - \sigma \p_1\left( \frac{\p_1 \eta}{(1+\abs{\p_1 \zeta_0}^2)^{3/2}}  + \mathcal{R}(\p_1\zeta_0,\p_1\eta) \right)   \right] \N &\text{on }\Sigma \\
\dt \eta = u \cdot \N &\text{on }\Sigma \\
(S_{\A}(p,u) \cdot \nu - \beta u)\cdot \tau = 0 &\text{on }\Sigma_s \\
u \cdot \nu =0 & \text{on }\Sigma_s \\
\linz \dt \eta(\pm \ell,t) + \linz \swh( \eta(\pm \ell,t) ) = \mp \sigma\left( \frac{\p_1 \eta}{(1+\abs{\p_1 \zeta_0}^2)^{3/2}}  + \mathcal{R}(\p_1\zeta_0,\p_1\eta)\right)(\pm \ell,t).
\end{cases}
\end{equation}
It is in this form that we will study the problem.  Throughout the paper we assume that $M_{top} > M_{min}$ is specified as in the discussion after Theorem \ref{zeta0_wp}.  For the data to \eqref{ns_geometric} we then have: 
\begin{enumerate}
 \item the initial free surface $\eta_0$, which we assume satisfies 
\begin{equation}\label{avg_zero_t=0}
 \int_{-\ell}^\ell \eta_0 =0
\end{equation}
so that the equilibrium mass, $M_{top}$, matches the initial mass, i.e.
\begin{equation}
 \int_{-\ell}^\ell (\zeta_0 + \eta_0) = M_{top},
\end{equation}
\item the initial velocity $u_0 : \Omega \to \R^2$, which we assume satisfies $\diverge_{\A_0} u_0 =0$ as well as the boundary conditions $u_0 \cdot \nu =0$ on $\Sigma_s$.
\end{enumerate}

\section{Main results and discussion }\label{sec_main}

\subsection{Energy and dissipation functionals and other notation}

In order to state our main result, we must first introduce some notation.  

\emph{Equilibrium angles and regularity parameters:} We begin by introducing the supplementary equilibrium contact angle 
\begin{equation}\label{omega_eq_def}
\omeq = \pi -\theq \in (0,\pi), 
\end{equation}
which is useful as it determines the angles created at the contact points in the fluid domain at equilibrium (see Figure \ref{fig_3}).  This angle, which can take on any value between $0$ and $\pi$ depending on the choice of the capillary parameters $\sigma$, $\gamma_{sv}$, and $\gamma_{sf}$, plays an important role in the elliptic regularity theory associated to $\Omega$, as it determines the possible regularity gain.  For the Stokes problem with boundary conditions related to those we use in \eqref{ns_geometric}, the regularity is related to the following parameter, computed by Orlt-S\"{a}ndig \cite{orlt_sandig}: 
\begin{equation}\label{epm_def}
 \epm =  \epm(\omeq) = \min\{1,-1+\pi/\omeq\} \in (0,1].
\end{equation}
For a parameter $0 < \ep \le \epm < 1$ we set 
\begin{equation}\label{qalpha_def}
 q_\ep = \frac{2}{2-\ep} \in (1,2).
\end{equation}
Note that $0 < \ep_- < \ep_+ < \epm(\omega)$ implies that 
\begin{equation}
 q_{\ep_-} < q_{\ep_+} < q_{\epm}.
\end{equation}
Then \cite{orlt_sandig} shows that the regularity available for the velocity in the associated Stokes problem cannot reach $W^{2,q_{\epm}}$.
 
With $\epm$ in hand, we fix three parameters $\low, \ep_-,\ep_+$ such that 
\begin{equation}\label{kappa_ep_def}
 0 <  \low < \ep_- < \ep_+ < \epm, \;  
\low < \min\left\{ \frac{1-\ep_+}{2}, \frac{\ep_+ - \ep_-}{2}  \right\}, 
  \text{ and } \ep_+ \le \frac{\ep_- + 1}{2}.
\end{equation}
For brevity we also write 
\begin{equation}\label{qpm_def}
 q_+ = q_{\ep_+} \text{ and }  q_- = q_{\ep_-} 
\end{equation}
in the notation established in \eqref{qalpha_def}.  We will crucially use these parameters to track regularity in this paper.

\emph{Norms:}  We write $W^{s,p}(\Gamma;\R^k)$ for $\Gamma \in \{\Omega,\Sigma,\Sigma_s\}$, $0 \le s \in \R$, $1 \le p \le \infty$, and $1 \le k \in \mathbb{N}$ for the usual Sobolev spaces of $\R^k-$valued functions on these sets. In particular, $W^{0,p}(\Gamma;\R^k) = L^p(\Gamma;\R^k)$. When $k=1$ we typically write $W^{s,p}(\Gamma) = W^{s,p}(\Gamma;\R)$. When $p=2$ we write $H^s(\Gamma;\R^k) = W^{s,2}(\Gamma;\R^k)$.  For the sake of brevity, we typically write our norms as $\norm{\cdot}_{W^{s,p}}$, suppressing the domain $\Gamma$ and the codomain $\R^k$.  We employ this notation whenever it is clear from the context what the set and codomain are; in situations where there is ambiguity (typically due to the evaluation of bulk-defined functions on $\Sigma$ or $\Sigma_s$ via trace operators) we will include the domain in the norm notation.  Next we define a useful pairing for the contact points that gives a contact point norm: we set 
\begin{equation}\label{bndry_pairing}
 [f,g]_\ell =  f(-\ell)g(-\ell) + f(\ell) g(\ell)    \text{ and } [f]_\ell = \sqrt{[f,f]_\ell}.
\end{equation}

\emph{Energy and dissipation functionals:} We define the following energy and dissipation functionals.  For $0 \le k \le 2$ we define the natural energy and dissipation via 
\begin{equation}\label{ED_natural}
\E_{\shortparallel,k} = \ns{\dt^k u}_{L^2} + \ns{\dt^k \eta}_{H^1} \text{ and } \D_{\shortparallel,k} = \ns{ \dt^k u}_{H^1} + \ns{\dt^k u}_{L^2(\Sigma_s)} + [\dt^{k+1} \eta]_\ell^2.
\end{equation}
We also set
\begin{equation}\label{ED_parallel}
 \seb = \sum_{k=0}^2 \E_{\shortparallel,k} \text{ and } \sdb=\sum_{k=0}^2 \D_{\shortparallel,k}.
\end{equation}
Then the full energy is
\begin{multline}\label{E_def}
 \E = \seb +   \ns{u}_{W^{2,q_+}} + \ns{\dt u}_{H^{1 + \ep_-/2}} + \ns{\dt^2 u}_{H^0} + \ns{p}_{W^{1,q_+}} + \ns{\dt p}_{L^2} \\
 + \ns{\eta}_{W^{3-1/q_+,q_+}} + \ns{\dt \eta}_{H^{3/2 + (\ep_- - \low)/2 }} + \ns{\dt^2 \eta}_{H^1},
\end{multline}
and the full dissipation is
\begin{multline}\label{D_def}
 \D = \sdb + \ns{u}_{W^{2,q_+}} + \ns{\dt u}_{W^{2,q_-}}   + \sum_{k=0}^2 [\dt^k u \cdot \N]_\ell^2 + \ns{p}_{W^{1,q_+}} + \ns{\dt p }_{W^{1,q_-}}  \\
 + \sum_{k=0}^2 \ns{\dt^k \eta}_{H^{3/2-\low}} + \ns{\eta}_{W^{3-1/q_+,q_+}} + \ns{\dt \eta}_{W^{3-1/q_-,q_-}}  + \ns{\dt^3 \eta}_{H^{1/2-\low}}   .
\end{multline}

\emph{Universal constants and Einstein summation:}   A generic constant $C>0$ will be called universal if it depends on $\Omega$, the dimension, or any of the parameters of the problem, but not on the solution or the initial data. In the usual manner, we allow the value of these constants to change from one estimate to the next. We employ the notation $a \ls b$ to mean that $a \le C b$ for a universal constant $C>0$, and we write $a \asymp b$ to mean that $a \ls b \ls a$.  From time to time we will use the Einstein convention of implicitly summing over repeated indices in vector and tensor expressions. 

\subsection{Main results and discussion  }\label{sec_res_disc}

Our main result is an a priori estimate for solutions to \eqref{ns_geometric} that shows that if solutions exist in a time horizon $[0,T)$ and have sufficiently small energy, then in fact the dissipation is integrable on $[0,T)$ and the energy decays exponentially.  Moreover, we have quantitative estimates in terms of the data.

\begin{thm}\label{main_apriori}
Let $\omeq \in (0,\pi)$ be given by \eqref{omega_eq_def}, $0 < \epm \le 1$ be given by \eqref{epm_def}, and suppose that $\low$, $\ep_-$, and $\ep_+$ satisfy \eqref{kappa_ep_def}.  Suppose that $\E$ and $\D$ are defined with these parameters via \eqref{E_def} and \eqref{D_def}, respectively.  Then there exists a universal constant $0 < \delta_0 <1$ such that if a solution to \eqref{ns_geometric} exists on the time horizon $[0,T)$ for $0 < T \le \infty$ and obeys the estimate 
\begin{equation}
 \sup_{0 \le t < T } \E(t) \le \delta_0, 
\end{equation}
then there exist universal constants $C,\lambda >0$ such that 
\begin{equation}
 \sup_{0 \le t <T} e^{\lambda t} \E(t) + \int_0^T \D(t) dt \le C \E(0).
\end{equation}
\end{thm}

The a priori estimates of this theorem may be coupled to a local existence theory that verifies the small energy condition is satisfied, provided the data are small enough and all necessary compatibility conditions are satisfied.  To keep the present paper of reasonable length, we will neglect to develop this local existence theory here.  Such a theory can be developed on the basis of the a priori estimates proved here in the same way that \cite{tice_zheng} develops the local theory for the Stokes version of \eqref{ns_geometric} based on the a priori estimates for the Stokes system that we developed in \cite{guo_tice_QS}.  Upon combining the local theory with our a priori estimates, we may deduce the existence of global-in-time decaying solutions.  This provides further evidence that the contact dynamics relation \eqref{cl_motion} together with the Navier-slip boundary conditions yield a good model of contact points in fluids.

\begin{thm}\label{main_gwp_dec}
Let $\omeq \in (0,\pi)$ be given by \eqref{omega_eq_def}, $0 < \epm \le 1$ be given by \eqref{epm_def}, and suppose that $\low$, $\ep_-$, and $\ep_+$ satisfy \eqref{kappa_ep_def}.  Suppose that $\E$ and $\D$ are defined with these parameters via \eqref{E_def} and \eqref{D_def}, respectively.    There exists a universal constant $0 < \delta_1 <1$ such that if $\E(0) \le \delta_1$, then there exists a unique global solution triple $(u,p,\eta)$ to \eqref{ns_geometric} on the time horizon $[0,\infty)$ such that 
\begin{equation}
 \sup_{0 \le t <\infty} e^{\lambda t} \E(t) + \int_0^\infty \D(t) dt \le C \E(0),
\end{equation}
where $C,\lambda >0$ are universal constants.
\end{thm}

In \cite{guo_tice_QS} we proved analogous results for the Stokes version of \eqref{ns_euler} (the terms $\dt u + u \cdot \nab u$ in the first equation are neglected), so it is prudent to begin the discussion of our current results by comparing and contrasting the Stokes and Navier-Stokes problems and the difficulties they present.  For both problems, an examination of the control provided by the basic energy-dissipation relation \eqref{fund_en_evolve} (the kinetic energy term in the energy is removed for the Stokes problem) reveals that neither the energy nor the dissipation provide enough control to close a scheme of a priori estimates.  As such, we are forced to analyze solutions in a higher regularity context, and it is here that it becomes clear that the geometry of the fluid domain is the central difficulty.  Indeed, the first issue it causes is that even after reformulation in a fixed domain as in \eqref{ns_geometric}, the only differential operators compatible with the domain are time derivatives.  We then need a strategy for bootstrapping from energy-dissipation control of the time derivatives to higher spatial regularity via elliptic estimates.  

It is at this point where we encounter the fundamental difficulty in analyzing the contact point problem.  Both the moving domain $\Omega(t)$ and the equilibrium domain $\Omega$ have corners at the contact points, and thus the boundary is at most globally Lipschitz.  In such domains it is well-known that the corners can harbor singularities in the solutions to elliptic equations.  For the Stokes problem in $\Omega$ with Navier-slip boundary conditions, the work of Orlt-S\"{a}ndig \cite{orlt_sandig} shows that the velocity can not even belong to $W^{2,q_{\epm}}$, where $\epm$ is determined by the equilibrium angle as in \eqref{epm_def}.  Consequently, regardless of how many temporal derivatives we gain basic control of, there is a fundamental barrier to the spatial regularity gain we can hope to achieve.

In closing a scheme of a priori estimates for \eqref{ns_geometric}, the mandate then becomes to make due with what's available and close with little spatial regularity.  In our work on the Stokes problem \cite{guo_tice_QS}, we do this by crucially exploiting a version of the normal trace estimate for the viscous stress.  This allows us to get a dissipative estimate for $\K \dt^2 \eta$ in $H^{-1/2}$, where $\K$ is the gravity-capillary operator associated to $\zeta_0$ (see \eqref{cap_op_def}).  With this in hand, we take advantage of the contact point boundary condition (the last equation in \eqref{ns_geometric}) in two essential ways.  First, this condition is responsible for providing a dissipative estimate of $[\dt^3 \eta]_\ell$.  Second, this condition serves as a boundary condition compatible with the elliptic operator $\K$, which couples with the aforementioned dissipative control to yield an $H^{3/2}$ estimate for $\dt^2 \eta$ in terms of the dissipation.  This estimate then serves as the starting point for a chain of elliptic estimates in weighted $L^2-$based Sobolev spaces that allow us to close for the Stokes problem.  Here the choice of $L^2-$based weighted spaces is convenient as it maintains consistency with the $L^2-$based estimates coming from the energy and dissipation.

For the Navier-Stokes problem considered in the present paper, the convective term $\dt u + u\cdot \nab u$ precludes the use of the normal trace estimate for the twice time-differentiated problem since neither the energy nor the dissipation provides control of $\dt^3 u$ in this case.  We are thus forced to seek another mechanism for obtaining a sufficiently high regularity estimate for $\dt^2 \eta$, which we need to kick start the chain of elliptic gains.  This is the central difficulty in dealing with the contact point Navier-Stokes system \eqref{ns_geometric}.

In place of the normal trace estimate, we instead employ a delicate argument using test functions in the weak formulation of the twice time-differentiated problem, together with the dissipative estimate of $[\dt^3 \eta]_\ell$.  This is delicate for two reasons.  First, we have very poor spatial regularity at that level of time derivative, so we must be careful with how the test function interacts with the solution.  Second, we aim to achieve estimates for the fractional regularity of $\dt^2 \eta$, but in the weak formulation we find $\dt^2 \eta$ interacting with the test function on $(-\ell,\ell)$ via an $H^1-$type inner product with the equilibrium free surface function $\zeta_0$ appearing as a weight (see \eqref{bndry_ips}).  The standard Fourier-analytic tricks that would try on a torus or full space don't work here due to the finite extent of $(-\ell,\ell)$ and the weight.  We are thus led to replace the standard Fourier tricks with the functional calculus associated to the gravity-capillary operator $\K$, which provides a scale of custom Sobolev spaces measuring fractional regularity in terms of the eigenfunctions of $\K$.  This allows us to build test functions that can produce higher fractional regularity estimates for $\dt^2 \eta$.  Unfortunately, despite major effort, we were unable to derive an exact $H^{3/2}$ estimate for $\dt^2 \eta$.  The obstacles are primarily due to the technical complications that arise from the criticality of $H^{1/2}$ in one dimension.

The principal technical achievement of this paper is the development of a scheme of a priori estimates that exchanges the full $H^{3/2}$ estimate used for the Stokes problem for a slightly weaker estimate in $H^{3/2 - \low}$, where $\low$ is given by \eqref{kappa_ep_def}.  Fortunately, this is just barely sufficient to kick start the elliptic gain and allow us to close.  In order to execute this, we have had to switch from weighted $L^2$-based Sobolev spaces to unweighted $L^q-$based spaces for values of $q$ just below the maximal value $q_{\epm}$.  This yields key technical advantages in dealing with several nonlinear terms.

\subsection{Technical overview and layout of paper  }

We now turn our attention to a brief technical overview of our methods, which we provide in a rough sketch form meant to highlight the main ideas while suppressing certain technical complications.  The starting point of our scheme of a priori estimates is a version of the energy dissipation relation \eqref{fund_en_evolve} for \eqref{ns_geometric}.  We need versions of this for the solution and its time derivatives up to order two.  These are recorded in Section \ref{sec_basic_tools}.  Upon differentiating \eqref{ns_geometric} we produce commutators, so we end up with an energy dissipation relation roughly of the form 
\begin{equation}\label{sum_1}
 \frac{d}{dt} \seb + \sdb = \mathscr{N},
\end{equation}
where $\seb$ and $\sdb$ are as in \eqref{ED_parallel} and $\mathscr{N}$ represents nonlinear interactions arising due to the commutators.  Section \ref{sec_basic_tools} also contains a number of other basic estimates.

To advance from the basic control provided by $\seb$ and $\sdb$ to higher spatial regularity estimates we need elliptic estimates for a Stokes problem related to \eqref{ns_geometric}.  We develop these in Section \ref{sec_elliptic} within the context of $L^q-$based spaces instead of the weighted $L^2-$based spaces we employed in \cite{guo_tice_QS}.  Here the main technical problem is associated to the upper bound on the regularity gain available due to the corner singularities in $\Omega$.  An interesting feature of our main result, Theorem \ref{A_stokes_stress_solve}, is that it treats the triple $(v,Q,\xi)$ as the elliptic unknown, but $\xi$ only appears on the boundary.  

With the elliptic estimates and \eqref{sum_1} in hand, we may identify most of the nonlinear terms that need to be estimated in order to close our scheme.  Due to the limited spatial regularity, these estimates are fairly delicate and require a good deal of care.  In particular, in dealing with $\mathscr{N}$ in \eqref{sum_1}, we need structured estimates of the form 
\begin{equation}\label{sum_3}
 \abs{\mathscr{N}} \le C \E^\theta \D \text{ for some } \theta >0,
\end{equation}
where $\E$ and $\D$ are the full energy and dissipation from \eqref{E_def} and \eqref{D_def}, 
in order to have any hope closing with \eqref{sum_1}.  In Section \ref{sec_nl_int_d} we record a host of nonlinear interaction estimates of this form.  In Section \ref{sec_nl_int_e} we record similar estimates but in terms of the energy functional instead of the dissipation.

Section \ref{sec_nl_elliptic} records estimates of the nonlinear terms that appear in applying the elliptic estimates from Theorem \ref{A_stokes_stress_solve}.  An interesting feature of these is that the upper bound of regularity identified by Orlt-S\"{a}ndig \cite{orlt_sandig} yields an open interval $(0,q_{\epm})$ of possible integrability exponents.  We take advantage of this by using two different exponents $0 < q_- < q_+ < q_{\epm}$ as in \eqref{kappa_ep_def}, with $q_+$ associated to the non-differentiated problem and $q_-$ associated to the once-differentiated problem.  The parameter $q_-$ can be made arbitrarily close to $q_{\epm}$, but the tiny increase we get in advancing to $q_+$ plays an essential role in Proposition \ref{ne2_f2}, which highlights how delicate the nonlinear estimates are.

As mentioned above, the key to starting the elliptic gains is an estimate of $\dt^2 \eta$ in $H^{3/2 -\low}$.  To achieve this we employ the functional calculus associated to the gravity-capillary operator $\K$, defined by \eqref{cap_op_def}.  We develop this in Section \ref{sec_fnal_calc}.  This calculus provides us with the ability to make sense of fractional powers of $\K$, which is essential in our test function method for deriving the needed estimate.  It also provides us with a scale of custom Sobolev spaces defined in terms of the eigenfunctions of $\K$, which we characterize in terms of standard Sobolev spaces in Theorem \ref{s_embed} when the regularity parameter satisfies $0 \le s \le 2$.  A serious technical complication in our test function / functional calculus method is that we would like to exploit an equivalence of the form 
\begin{equation}\label{sum_4}
 (\dt^2 \eta, (\K)^{\hal - \low} \dt^2 \eta)_{1,\Sigma} \asymp \ns{\dt^2 \eta}_{H^{3/2 - \low}}
\end{equation}
where $(\cdot,\cdot)_{1,\Sigma}$ is as in \eqref{bndry_ips} and $\low$ satisfies \eqref{kappa_ep_def}, but we cannot guarantee that $(\K)^{\hal - \low} \dt^2 \eta \in H^1$ in our functional framework.  To get around this, for $j \in \mathbb{N}$ and $s \ge 0$ we introduce the operators $D_j^s$ in Section \ref{sec_djr}.  These are approximations of the fractional differential operators $D^s := (\K)^{s/2}$ formed by projecting onto the first $j$ eigenfunctions of $\K$ in the spectral representation of $D^{s}$.  The eigenfunctions are smooth up to the boundary, so they work nicely when replaced on the left side of \eqref{sum_4}.  We then aim to recover the desired control by working with these operators and sending $j \to \infty$.

In Section \ref{sec_enhancements} we carry out the details of our test function / functional calculus method to derive the estimate for $\dt^2 \eta$ in $H^{3/2-\low}$.  Along the way we also use similar methods to derive a couple other useful estimates for $\eta,$ $\dt \eta$, and $\dt p$.  These all serve as enhancements to the basic energy-dissipation estimate \eqref{sum_1} since they are given in similar form.

In Section \ref{sec_aprioris} we complete the proof of Theorem \ref{main_apriori}.  We combine an integrated form of \eqref{sum_1} with the enhancement estimates to form the core estimates in energy dissipation form.  These are then coupled to the elliptic estimates to gain spatial regularity.  We then employ our array of nonlinear estimates to derive an estimate of the form 
\begin{equation}
 \E(t)   + \int_{s}^t \D \ls   \E(s)
\end{equation}
for all $0 \le s \le t < T$, and from this we complete the proof with a version of Gronwall's inequality, Proposition \ref{gronwall_variant}.

Appendix \ref{sec_nonlinear_records} records the lengthy forms of various nonlinearities and commutators.  Appendix \ref{sec_analysis_tools} contains a number of useful tools from analysis that are used throughout the paper, including product and composition estimates, estimates for the Poisson extension, and the Bogovskii operator.

\section{Basic tools}\label{sec_basic_tools}

In this section we record a number of basic identities and estimates associated to the problem \eqref{ns_geometric}.
 
\subsection{Energy-dissipation relation  }

Upon applying temporal derivatives to \eqref{ns_geometric} and keeping track of the essential transport terms, we arrive at the following general linearization:
\begin{equation}\label{linear_geometric}
\begin{cases}
\dt v -\dt \bar{\eta} \frac{\phi}{\zeta_0} K \p_2 v + u \cdot \naba v  + \diva S_{\A}(q,v) =F^1  & \text{in }\Omega \\
\diva v = F^2 &\text{in } \Omega \\
S_{\A}(q,v)\N = \left[g \xi - \sigma \p_1\left( \frac{\p_1 \xi}{(1+\abs{\p_1 \zeta_0}^2)^{3/2}}  + F^3 \right)   \right] \N  + F^4 &\text{on }\Sigma \\
\dt \xi - v \cdot \N = F^6 &\text{on }\Sigma \\
(S_{\A}(q,v) \cdot \nu - \beta v)\cdot \tau = F^5 &\text{on }\Sigma_s \\
v \cdot \nu =0 & \text{on }\Sigma_s \\
\linz \dt \xi(\pm \ell,t) = \mp \sigma\left( \frac{\p_1 \xi}{(1+\abs{\p_1 \zeta_0}^2)^{3/2}}  + F^3\right)(\pm \ell,t) - \linz F^7.
\end{cases}
\end{equation}
We will mostly be interested in this problem for $v = \dt^k u$, $\xi = \dt^k \eta$ and $q = \dt^k p$, in which case the forcing terms have the special form given in Appendix \ref{sec_nonlinear_records}.

We now aim to record the weak formulation of \eqref{linear_geometric}.  First, we will need to introduce some useful bilinear forms.   Suppose that $\eta$ is given and that $\A$, $J$, $K$, and $\N$  are determined as in \eqref{AJK_def} and \eqref{N_def}.   We  define 
\begin{equation}\label{bulk_ips}
 \pp{u,v} := \int_\Omega \frac{\mu}{2} \sg_{\A} u : \sg_{\A} v J + \int_{\Sigma_s} \beta (u\cdot \tau) (v\cdot \tau) J
\text{ and } (u,v)_0 := \int_{\Omega} u \cdot v J.
\end{equation}
With these in hand we can formulate an integral version of \eqref{linear_geometric}.

\begin{lem}\label{geometric_evolution}
Suppose that $u,p,\eta$ are given and satisfy \eqref{ns_geometric}.  Further suppose that $(v,q,\xi)$ are sufficiently regular and solve \eqref{linear_geometric}.  Then for sufficiently regular test functions $w$ satisfying $w \cdot \nu =0$ on $\Sigma_s$ we have that 
\begin{multline}\label{ge_01}
\br{\dt v,Jw} + (-\dt \bar{\eta} \frac{\phi}{\zeta_0} K \p_2 v + u \cdot \naba v,Jw)_0 +
  \pp{v,w} - (q,\diva w)_0   \\
 = \int_\Omega F^1   \cdot w J   - \int_{\Sigma_s} J (w\cdot \tau)F^5 
-   \int_{-\ell}^\ell g \xi (w \cdot \N) - \sigma \p_1 \left( \frac{\p_1 \xi }{(1+\abs{\p_1 \zeta_0}^2)^{3/2}} +F^3\right)w\cdot  \N + F^4 \cdot w
\end{multline}
and
\begin{multline}\label{ge_00}
\br{\dt v,Jw} + (-\dt \bar{\eta} \frac{\phi}{\zeta_0} K \p_2 v + u \cdot \naba v,Jw)_0 +
  \pp{v,w} - (q,\diva w)_0 + (\xi,w\cdot \N)_{1,\Sigma} + \linz [\dt \xi,w\cdot \N]_\ell  \\
 = \int_\Omega F^1   \cdot w J   - \int_{\Sigma_s} J (w\cdot \tau)F^5 
- \int_{-\ell}^\ell \sigma  F^3   \p_1 (w \cdot \N) + F^4 \cdot w  
  - \linz [w\cdot \N,  F^7]_\ell,
\end{multline}
where $[\cdot,\cdot]_\ell$ is defined in \eqref{bndry_pairing} and $\linz >0$ is as in \eqref{linz_def}.
\end{lem}
\begin{proof}
Upon taking the dot product of the first equation in \eqref{linear_geometric} with $J w$ and integrating over $\Omega$, we arrive at the identity 
\begin{equation}\label{ge_1}
 \br{\dt v,Jw} + (-\dt \bar{\eta} \frac{\phi}{\zeta_0} K \p_2 v + u \cdot \naba v,Jw)_0 + I = II
\end{equation}
where we have written
\begin{equation}
I := \int_\Omega  \diva \Sa(q,v)  \cdot w J \text{ and }   II := \int_\Omega F^1 \cdot w J.
\end{equation}
In expanding the term $I$ we will employ a pair of identities that are readily verified through elementary computations, using the definitions of $J$, $\A$, and $\N$ from \eqref{AJK_def} and \eqref{N_def}: first,  
\begin{equation}\label{ge_3}
\p_k(J \A_{jk}) =0 \text{ for each }j; 
\end{equation}
and, second, 
\begin{equation}\label{ge_4}
 J \A \nu = 
\begin{cases}
 J \nu &\text{on }\Sigma_s \\
 \N/\sqrt{1 +\abs{\p_1 \zeta_0}^2} &\text{on }\Sigma.
\end{cases}
\end{equation}

From \eqref{ge_3} and an integration by parts, we can write 
\begin{equation}\label{ge_7}
 I = \int_\Omega \p_k (J \A_{jk} \Sa(q,v)_{ij}) w_i = \int_\Omega -J \A_{jk} \p_k w_i \Sa(q,v)_{ij} + \int_{\p \Omega} (J \A \nu) \cdot (\Sa(q,v) w) 
:= I_1 + I_2.
\end{equation}
The term $I_1$ is readily rewritten using the definition of $\Sa(q,v)$ (given just below \eqref{N_def}):
\begin{equation}\label{ge_20}
 I_1 = \int_\Omega \frac{\mu}{2} \sga v: \sga w J - q \diva{w} J.
\end{equation}
To handle $I_2$ we use the first equation in \eqref{ge_4} to see that
\begin{multline}
 \int_{\Sigma_s} (J \A \nu) \cdot (\Sa(q,v) w) = \int_{\Sigma_s} J \nu \cdot (\Sa(q,v) w) = 
\int_{\Sigma_s} J w \cdot (\Sa(q,v) \nu) \\
= \int_{\Sigma_s} J \left(\beta (v \cdot \tau) (w \cdot \tau) + w\cdot \tau F^5\right),
\end{multline}
and the second equality in \eqref{ge_4} to see that
\begin{multline}
 \int_{\Sigma} (J \A \nu) \cdot (\Sa(q,v) w) \\ = 
\int_{-\ell}^\ell (\Sa(q,v) \N)\cdot w = \int_{-\ell}^\ell g \xi (w \cdot \N) - \sigma \p_1 \left( \frac{\p_1 \xi }{(1+\abs{\p_1 \zeta_0}^2)^{3/2}} +F^3\right)w\cdot  \N + F^4 \cdot w.
\end{multline}
Since $\p \Omega = \Sigma_s \cup \Sigma$, we then have that 
\begin{multline}\label{ge_21}
I_2 = \int_{\Sigma_s} J \left(\beta (v \cdot \tau) (w \cdot \tau) + w\cdot \tau F^5\right) \\
+\int_{-\ell}^\ell g \xi (w \cdot \N) - \sigma \p_1 \left( \frac{\p_1 \xi }{(1+\abs{\p_1 \zeta_0}^2)^{3/2}} +F^3\right)w\cdot  \N + F^4 \cdot w.
\end{multline}
Upon combining \eqref{ge_20} and \eqref{ge_21} with \eqref{ge_1} and recalling the definition of $\pp{\cdot,\cdot}$ from \eqref{bulk_ips}, we deduce that \eqref{ge_01} holds.

It remains to show that \eqref{ge_01} can be rewritten as \eqref{ge_00}.  To this end, we integrate by parts and use the equations in \eqref{linear_geometric} to rewrite 
\begin{multline}\label{ge_8}
 \int_{-\ell}^\ell - \sigma \p_1 \left( \frac{\p_1 \xi }{(1+\abs{\p_1 \zeta_0}^2)^{3/2}} +F^3 \right)w\cdot \N  \\
 =
\int_{-\ell}^\ell \left( \frac{\p_1 \xi }{(1+\abs{\p_1 \zeta_0}^2)^{3/2}} +F^3 \right)\p_1 (w\cdot \N)
- \sigma \left. \left( \frac{\p_1 \xi }{(1+\abs{\p_1 \zeta_0}^2)^{3/2}} +F^3 \right) (w\cdot \N) \right\vert_{-\ell}^\ell 
\end{multline}
with 
\begin{multline}\label{ge_9}
- \sigma \left.  \left( \frac{\p_1 \xi }{(1+\abs{\p_1 \zeta_0}^2)^{3/2}} +F^3 \right)  (w\cdot \N) \right\vert_{-\ell}^\ell =
 \sum_{a=\pm 1} \left(\linz \dt\xi(a\ell) + \linz F^7(a\ell)  \right)(w\cdot \N(a\ell) ) \\
 = 
 \sum_{a=\pm 1} \linz (v \cdot \N(a\ell))(w\cdot \N(a\ell) ) + \linz (w\cdot\N(a\ell)) ( F^6(a\ell) + F^7(a\ell) ).
\end{multline}
Combining \eqref{ge_8} and \eqref{ge_9} with \eqref{ge_01} and rearranging then yields \eqref{ge_00}.
\end{proof}

The most natural use of Lemma \ref{geometric_evolution} occurs with $w=v$, but we will record a slight variant of this.  This results in the following fundamental energy-dissipation identity.

\begin{thm}\label{linear_energy}
Suppose that $\zeta=\zeta_0 + \eta$ is given and  $\A$ and $\N$ are determined in terms of $\zeta$ as in \eqref{AJK_def} and \eqref{N_def}.  Suppose that $(v,q,\xi)$ satisfy \eqref{linear_geometric} and that $\omega(\cdot,t) \in H^1_0(\Omega;\R^2)$ is sufficiently regular for the following expression to be well-defined.  Then
\begin{multline} \label{linear_energy_0}
 \frac{d}{dt} \left( \int_{\Omega} J \frac{\abs{v}^2}{2} +    \int_{-\ell}^\ell \frac{g}{2} \abs{\xi}^2 + \frac{\sigma}{2} \frac{\abs{\p_1 \xi}^2}{(1+\abs{\p_1 \zeta_0}^2)^{3/2}} - \int_\Omega J v\cdot \omega \right) \\
 + \frac{\mu}{2} \int_\Omega \abs{\sga v}^2 J 
+\int_{\Sigma_s} \beta J \abs{v \cdot \tau}^2 + \linz [\dt \xi,\dt \xi]_\ell  
= \int_\Omega F^1  \cdot v J + q J(F^2-\diva \omega)  - \int_{\Sigma_s}  J (v \cdot \tau)F^5 \\
-  \int_{-\ell}^\ell \sigma  F^3  \p_1(v \cdot \N) + F^4 \cdot v  -  g \xi F^6 -  \sigma \frac{\p_1 \xi \p_1 F^6}{(1+\abs{\p_1 \zeta_0}^2)^{3/2}}
  - \linz [v\cdot \N,  F^7]_\ell + \linz [\dt \xi,F^6] \\
- \int_\Omega v \cdot \dt(J \omega)  +   (-\dt \bar{\eta} \frac{\phi}{\zeta_0} K \p_2 v + u \cdot \naba v,J \omega)_0 + \pp{v,\omega}  - \int_\Omega F^1 \omega J.
\end{multline}
\end{thm} 
\begin{proof}
We use $v-\omega$ as a test function in Lemma \ref{geometric_evolution} to see that 
\begin{multline}\label{linear_energy_1} 
\br{\dt v,Jv} + (-\dt \bar{\eta} \frac{\phi}{\zeta_0} K \p_2 v + u \cdot \naba v,Jv)_0 +
  \pp{v,v} - (q,\diva v)_0 + (\xi,v\cdot \N)_{1,\Sigma} + \linz [\dt \xi,v\cdot \N]_\ell  \\
 = \int_\Omega F^1   \cdot v J   - \int_{\Sigma_s} J (v\cdot \tau)F^5 
- \int_{-\ell}^\ell \sigma  F^3   \p_1 (v \cdot \N) + F^4 \cdot v  
  - \linz [v\cdot \N, F^7]_\ell \\
+ \br{\dt v, J \omega} +   (-\dt \bar{\eta} \frac{\phi}{\zeta_0} K \p_2 v + u \cdot \naba v,J \omega)_0 + \pp{v,\omega} - (q,\diva \omega)_0 - \int_\Omega F^1 \omega J.
\end{multline}

First note that 
\begin{equation}\label{linear_energy_2}
 \pp{v,v} = \int_\Omega \abs{\sga v}^2 J 
+\int_{\Sigma_s} \beta J \abs{v \cdot \tau}^2.
\end{equation}
Next, we expand 
\begin{equation}
 \br{\dt v,Jv} = \frac{d}{dt}\int_{\Omega} \frac{\abs{v}^2}{2} - \int_{\Omega} \dt J \frac{\abs{v}^2}{2}.
\end{equation}
Using the identities \eqref{ge_3} and \eqref{ge_4},  we may integrate by parts to compute 
\begin{multline}
\int_{\Omega} \left( -\dt \bar{\eta} \frac{\phi}{\zeta_0} K \p_2 v + u \cdot \naba v \right)\cdot Jv  
= \int_{\Omega} \frac{\abs{v}^2}{2} \left(-J \diva{u} + \p_2 \left(\frac{\dt \bar{\eta} \phi}{\zeta_0} \right)   \right) \\
+ \int_{\Sigma} \frac{\abs{v}^2}{2 \sqrt{1+ \abs{\p_1 \zeta_0}^2}} \left(-\dt \eta + u \cdot \N \right)
+ \int_{\Sigma_s} J u \cdot \nu \frac{\abs{v}^2}{2}.
\end{multline}
Since $\dt \eta = u\cdot \N$ on $\Sigma$, $u\cdot \nu =0$ on $\Sigma_s$< and $\diva u=0$, we arrive at the equality
\begin{equation}
\int_{\Omega} \left( -\dt \bar{\eta} \frac{\phi}{\zeta_0} K \p_2 v + u \cdot \naba v \right)\cdot Jv =  \int_{\Omega} \frac{\abs{v}^2}{2}  \p_2 \left(\frac{\dt \bar{\eta} \phi}{\zeta_0}   \right).
\end{equation}
We then compute 
\begin{equation}
 J = 1 + \p_2\left( \frac{\bar{\eta} \phi}{\zeta_0} \right) \Rightarrow \p_2 \left(\frac{\dt \bar{\eta} \phi}{\zeta_0}   \right) = \dt J,
\end{equation}
which shows that
\begin{equation}\label{linear_energy_3}
 \br{\dt v,Jv} + (-\dt \bar{\eta} \frac{\phi}{\zeta_0} K \p_2 v + u \cdot \naba v,Jv)_0 = \frac{d}{dt}\int_{\Omega} \frac{\abs{v}^2}{2}.
\end{equation}

On the other hand, we may use \eqref{linear_geometric} to compute 
\begin{multline}\label{linear_energy_4}
(\xi,v\cdot \N)_{1,\Sigma} =  (\xi,\dt \xi - F^6)_{1,\Sigma} \\
 =  \dt \left(   \int_{-\ell}^\ell \frac{g}{2} \abs{\xi}^2 + \frac{\sigma}{2} \frac{\abs{\p_1 \xi}^2}{(1+\abs{\p_1 \zeta_0}^2)^{3/2}} \right) - \int_{-\ell}^\ell  g \xi F^6 +  \sigma \frac{\p_1 \xi \p_1 F^6}{(1+\abs{\p_1 \zeta_0}^2)^{3/2}},
\end{multline}
\begin{equation}\label{linear_energy_5}
 (q,\diva v)_0 - (q,\diva \omega)_0 = \int_\Omega q J(F^2 - \diva \omega),
\end{equation}
and 
\begin{equation}\label{linear_energy_6}
 [\dt \xi, v \cdot \N]_\ell =  [\dt \xi,\dt \xi]_\ell -  [\dt \xi, F^6]_\ell.
\end{equation}
Then \eqref{linear_energy_0} follows by plugging \eqref{linear_energy_2}, \eqref{linear_energy_3}, \eqref{linear_energy_4}, \eqref{linear_energy_5}, and \eqref{linear_energy_6} into \eqref{linear_energy_1} and noting that 
\begin{equation}
  \br{\dt v,J\omega} = \frac{d}{dt}\int_{\Omega} Jv\cdot \omega - \int_{\Omega} v \cdot \dt(J\omega).
\end{equation}

\end{proof}

Next we record an application of this to \eqref{ns_geometric}.   

\begin{cor}\label{basic_energy}
Suppose that $(u,p,\eta)$ solve \eqref{ns_geometric}, and consider the function $\Q$ given by \eqref{Q_def}.  Then
\begin{multline}\label{be_0}
 \frac{d}{dt} \left(\int_\Omega \hal J \abs{u}^2 +  \int_{-\ell}^\ell \frac{g}{2} \abs{\eta}^2 + \frac{\sigma}{2} \frac{\abs{\p_1 \eta}^2}{(1+\abs{\p_1 \zeta_0}^2)^{3/2}} +  \int_{-\ell}^\ell \sigma \Q(\p_1 \zeta_0,\p_1 \eta)\right)  \\
 + \frac{\mu}{2} \int_\Omega \abs{\sga u}^2 J 
+\int_{\Sigma_s} \beta J \abs{u \cdot \tau}^2 + \linz [\dt \eta]_\ell^2
 =- \linz [u\cdot \N,\swh(\dt \eta)]_\ell .
\end{multline}
\end{cor}
\begin{proof}
From \eqref{ns_geometric} we have that $v =u$, $q =p$, and $\xi = \eta$ solve \eqref{linear_geometric} with $F^i=0$ for $i\neq 3,7$ and $F^3 =   \mathcal{R}(\p_1 \zeta_0,\p_1 \eta)$, $F^7 = \swh(\dt \eta)$.  The identity  \eqref{be_0} then follows by applying Theorem \ref{linear_energy} with $\omega =0$ and noting that in this case
\begin{equation}
 -\int_{-\ell}^\ell \sigma F^3 \p_1 (v\cdot \N) = -\int_{-\ell}^\ell \sigma \dt \sigma \mathcal{Q}(\p_1 \zeta,\p_1 \eta) = -\frac{d}{dt} \int_{-\ell}^\ell \sigma \mathcal{Q}(\p_1 \zeta_0,\p_1 \eta).
\end{equation}

\end{proof}

Next we record the consequences of conservation of mass.

\begin{prop}\label{avg_zero_prop}
If $(u,p,\eta)$ solve \eqref{ns_geometric}, then 
\begin{equation}
 \int_{-\ell}^\ell \dt^j \eta =0 
\end{equation}
for $0 \le j \le 3$.
\end{prop}
\begin{proof}
Integrating the condition $J \diva u =0$ against $J$ over $\Omega$ and using \eqref{ge_3} and \eqref{ge_4} together with the divergence theorem shows  that 
\begin{equation}
\frac{d}{dt} \int_{-\ell}^\ell \eta = \int_{-\ell}^\ell \dt \eta = \int_\Omega J \diva u =0.
\end{equation}
The result for $1\le j \le 3$ follows immediately from this, and for $j=0$ it follows from the assumption \eqref{avg_zero_t=0}. 
\end{proof}

\subsection{Coefficient bounds   }

The smallness of the perturbation $\eta$ will play an essential role in most of the arguments in the paper, from guaranteeing that $\Phi$ is a diffeomorphism to enabling certain nonlinear estimates.  The following lemma records this smallness is a quantitative way.

\begin{lem}\label{eta_small}
Let $q_+$ be as in \eqref{qpm_def}.  There exists a universal $0 < \gamma < 1$ such that if $\norm{\eta}_{W^{3-1/q_+,q_+}} \le \gamma$, then the following hold for $\A$ defined by \eqref{A_def}, $A, J, K$  defined by \eqref{AJK_def},  and $\N$ and $\N_0$  defined by \eqref{N_def} and \eqref{N0_def}, respectively.
\begin{enumerate}
 \item We have the estimates 
\begin{equation}\label{es_01}
 \max\{\norm{J-1}_{L^\infty}, \norm{K-1}_{L^\infty}, \norm{A}_{L^\infty}, \norm{\N- \N_0}_{L^\infty}  \}    \le \hal \text{ and } 
   \norm{\mathcal{A}}_{L^\infty} \ls 1.
\end{equation}

 \item For every $u \in H^1(\Omega; \R^2)$ such that $u\cdot \nu =0$ on $\Sigma_s$ we have that 
\begin{equation}
 \frac{\mu}{4} \int_\Omega \abs{\sg u}^2 + \frac{\beta}{2} \int_{\Sigma_s} \abs{u}^2 \le  \frac{\mu}{2} \int_\Omega \abs{\sga u}^2 J + \beta \int_{\Sigma_s} J\abs{u\cdot \tau}^2
 \le  \mu  \int_\Omega \abs{\sg u}^2 + 2\beta  \int_{\Sigma_s} \abs{u}^2
\end{equation}

 \item The map $\Phi$ defined by \eqref{mapping_def} is a diffeomorphism.
 
\end{enumerate}

\end{lem}
\begin{proof}
The first and third items follow from standard product estimates, Proposition \ref{poisson_prop}, and the Sobolev embeddings.  The second item is a simple modification of Proposition 4.3 in \cite{guo_tice_per}.

\end{proof}

\subsection{$M$ as a multiplier}

It will be useful to define the following matrix in terms of $\eta$:
\begin{equation}\label{M_def}
 M = K \nab \Phi = (J \A^T)^{-1},
\end{equation}
where $\A$ is as in \eqref{A_def} and $J$ and $K$ are as in \eqref{AJK_def}.  We will view this matrix as inducing a linear map via multiplication.  Our first result records the boundedness properties of this map.

\begin{prop}\label{M_multiplier}
Let $M$ be given by \eqref{M_def} and suppose that  $1 \le q \le 2/(1-\ep_+)$.  Then we have the inclusions $M \in \L( W^{1,q}(\Omega;\R^2))$ and $M,\dt M \in \L(L^{q}(\Omega;\R^2))$ as well as the estimates 
\begin{equation}
 \norm{M \zeta}_{W^{1,q}} \ls (1+\sqrt{\E}) \norm{\zeta}_{W^{1,q}} \text{ and } \norm{M \zeta}_{L^{q}} + \norm{\dt M \zeta}_{L^{q}} \ls (1+\sqrt{\E}) \norm{\zeta}_{L^{q}}.
\end{equation}
\end{prop}
\begin{proof}
First note that 
\begin{equation}
 \norm{M \zeta}_{W^{1,q}} \ls \norm{\abs{M}\abs{\zeta}}_{L^q} + \norm{\abs{M} \abs{\nab \zeta}}_{L^q} + \norm{\abs{\nab M} \zeta}_{L^q} \ls \norm{M}_{L^\infty} \norm{\zeta}_{W^{1,q}} +\norm{\abs{\nab M} \zeta}_{L^q}.
\end{equation}
It's easy to see that 
\begin{equation}
 \norm{M}_{L^\infty}  \ls 1 + \norm{\bar{\eta}}_{W^{1,\infty}} \ls 1+ \sqrt{\E},
\end{equation}
which handles the first term on the right.  For the second we need to use H\"{o}lder's inequality, and we must break to cases.

In the first case we assume that $q$ is subcritical, i.e. $1 \le q < 2$.    Then 
\begin{equation}
 \frac{1- \ep_+}{2} + \frac{1}{q^\ast} = \frac{1-\ep_+}{2} + \frac{1}{q} - \frac{1}{2}  < \frac{1}{q},
\end{equation}
so we can bound 
\begin{equation}
 \norm{\abs{\nab M} \zeta}_{L^q} \ls \norm{\nab M}_{L^{2/(1-\ep_+)}} \norm{\zeta}_{L^{q^\ast}} \ls \norm{\bar{\eta}}_{L^{2/(1-\ep_+)}}\norm{\zeta}_{W^{1,q}} \ls \sqrt{\E} \norm{\zeta}_{W^{1,q}}.
\end{equation}
In the second case we assume criticality, i.e. $q=2$.  Then by the critical Sobolev embedding, 
\begin{equation}
 \norm{\abs{\nab M} \zeta}_{L^2} \ls \norm{\nab M}_{L^{2/(1-\ep_+)}} \norm{\zeta}_{L^{2/\ep_+}} \ls \norm{\bar{\eta}}_{L^{2/(1-\ep_+)}}\norm{\zeta}_{W^{1,2}} \ls \sqrt{\E} \norm{\zeta}_{W^{1,q}}.
\end{equation}
In the third case we assume supercriticality, i.e. $2 < q \le 2/(1-\ep_+)$.  Then 
\begin{equation}
 \norm{\abs{\nab M} \zeta}_{L^q} \ls \norm{\nab M}_{L^{2/(1-\ep_+)}} \norm{\zeta}_{L^{\infty}} \ls \norm{\bar{\eta}}_{L^{2/(1-\ep_+)}}\norm{\zeta}_{W^{1,q}} \ls \sqrt{\E} \norm{\zeta}_{W^{1,q}}.
\end{equation}

Thus, in any case we have 
\begin{equation}
 \norm{\abs{\nab M} \zeta}_{L^q} \ls \sqrt{\E} \norm{\zeta}_{W^{1,q}}, 
\end{equation}
and the first estimate follows.  To prove the second estimate we simply note that by Theorem \ref{catalog_energy}
\begin{equation}
 \norm{ M}_{L^\infty} + \norm{\dt M}_{L^\infty} \ls 1 + \norm{\bar{\eta}}_{W^{1,\infty}}+ \norm{\dt \bar{\eta}}_{W^{1,\infty}} \ls 1 + \sqrt{\E}.
\end{equation}
\end{proof}

The matrix $M$ plays an important role in switching from the operator $\diverge$ to $\diva$.  We record this information in the following.

\begin{prop}\label{M_properties}
Let $M$ be given by \eqref{M_def} and $1 \le q \le 2/(1-\ep_+)$.  Then the following hold for $u \in W^{1,q}(\Omega;\R^2)$.
\begin{enumerate}
 \item $\diverge u =p$ if and only if $\diva(Mu)=K p$.
 \item $u \cdot \nu =0$ on $\Sigma_s$ if and only if $(Mu) \cdot \nu =0$ on $\Sigma_s$.
 \item $u \cdot \N_0 = (Mu)\cdot \N$ on $\Sigma$.
\end{enumerate}
\end{prop}
\begin{proof}
We compute
\begin{equation}
 \diverge(M^{-1} v) = \p_j( J\A_{ij}v_i) = J \A_{ij} \p_j v_i = J \diva{v}.
\end{equation}
Hence, upon setting $Mu = v$ we see that
\begin{equation}
 \diverge{u} = p \text{ if and only if } \diva(Mu) = Kp.
\end{equation}
This proves the first item.  For the second note that 
\begin{equation}
 K\nab \Phi^T \nu = K\nu \text{ on } \{ x \in \p \Omega \;\vert \;  x_1 = \pm \ell, x_2 \ge 0  \} \text{ and } K\nab \Phi^T \nu = \nu \text{ on } \{x \in \p \Omega \; \vert \; x_2 < 0 \},
\end{equation}
so on $\Sigma_s$ we have that
\begin{equation}
 M u \cdot \nu =0 \Leftrightarrow  u \cdot (K \nab \Phi^T \nu) =0 \Leftrightarrow u \cdot \nu =0.
\end{equation}
Finally, for the third item we compute on $\Sigma$:
\begin{equation}
 J \A \N_0 = \N \Rightarrow \N_0 = K (\A)^{-1} \N = K \nab \Phi^T \N,
\end{equation}
which implies that
\begin{equation}
 u \cdot \N_0 = u \cdot K \nab \Phi^T \N = K \nab \Phi u \cdot \N = M u \cdot \N.
\end{equation}

\end{proof}

\subsection{Various bounds   }

In subsequent parts of the paper we will need to repeatedly employ various $L^q$ estimates for $u$, $p$, $\eta$ and their derivatives in terms of either $\sqrt{\E}$ or $\sqrt{\D}$, defined respectively in \eqref{E_def} and \eqref{D_def}.  Thus, we now turn our attention to recording a precise catalog of such estimates, which are available due to the control provided by $\E$ and $\D$ and various auxiliary estimates.  In order to efficiently record this catalog, we will use tables of the following form.  
\begin{displaymath}
\def\arraystretch{1.2}
\begin{array}{|c | c | c | c | c |}
\hline
\text{Function} = f  & f & \nab f & \nab^2 f & \nab^3 f \\ 
\hline
\varphi & \infty & a & b & c \\ \hline
\tr \varphi  & \infty- & e &  &  \\ \hline
\end{array}
\end{displaymath}
The top row indicates that the first column labels the function under consideration, and the subsequent columns give the $q$ for which the number of derivatives indicated in top row  belong to $L^q$.  In this notation $q = \infty$ indicates $L^\infty$, while $\infty-$ indicates inclusion in $L^q$ for all $1 \le q < \infty$ (with bounds that diverge as $q \to \infty$ as in the critical Sobolev inequality), and an empty cell indicates no estimate available.  The set on which the $L^q$ norm is evaluated is always understood to be the ``natural'' set on which the function is defined: $\Omega$ for $u$, $p$, $\bar{\eta}$, and $(-\ell,\ell)$ for $\eta$.  The notation $\tr$ indicates that the function under consideration is the trace onto either $\Sigma$ or $\Sigma_s$.  For example, if we state that the above sample table records estimates in terms of $\sqrt{\E}$, and $\varphi$ is defined in $\Omega$, then this indicates that 
\begin{equation}
\norm{\varphi}_{L^\infty(\Omega)} + \norm{\nab \varphi}_{L^a(\Omega)} + \norm{\nab^2 \varphi}_{L^b(\Omega)} + \norm{\nab^3 \varphi}_{L^c(\Omega)} \ls \sqrt{\E},
\end{equation}
\begin{equation}
 \norm{\tr \varphi}_{L^q(\Sigma)} +  \norm{\tr \varphi}_{L^q(\Sigma_s)} \le C_q \sqrt{\E} \text{ for all } 1 \le q < \infty,
\end{equation}
where $C_q \to \infty$ as $q \to \infty$, and 
\begin{equation}
\norm{\tr \nab \varphi}_{L^e(\Sigma)}  + \norm{\tr \nab \varphi}_{L^e(\Sigma_s)} \ls \sqrt{\E}.
\end{equation}

With this notation established, we now turn to recording the catalogs.  We begin with the estimates in terms of the energy.

\begin{thm}\label{catalog_energy}
The following three tables record the $L^q$ bounds for $u$, $p$, $\eta$ and their derivatives in terms of the energy $\sqrt{\E}$, as defined in \eqref{E_def}.
\begin{displaymath}
\def\arraystretch{1.2}
\begin{array}{|c | c | c | c | c |}
\hline
\text{Function} = f  & f & \nab f & \nab^2 f & \nab^3 f \\ 
\hline
u & \infty & 2/(1-\ep_+) & 2/(2-\ep_+) &  \\ \hline
\dt u & \infty & 4/(2-\ep_-) &  &  \\ \hline
\dt^2 u & 2 &  & &  \\ \hline
\tr u & \infty & 1/(1-\ep_+) & &  \\ \hline
\tr \dt u & \infty &  & &  \\ \hline
\end{array}
\end{displaymath}
\begin{displaymath}
\def\arraystretch{1.2}
\begin{array}{|c | c | c | c | c |}
\hline
\text{Function} = f  & f & \nab f & \nab^2 f & \nab^3 f \\ 
\hline
p & 2/(1-\ep_+) & 2/(2-\ep_+) & &  \\ \hline
\dt p & 2  & & &  \\ \hline
\tr p & 1/(1-\ep_+) & & &  \\ \hline
\end{array}
\end{displaymath}
\begin{displaymath}
\def\arraystretch{1.2}
\begin{array}{|c | c | c | c | c |}
\hline
\text{Function} = f  & f & \nab f & \nab^2 f & \nab^3 f \\ 
\hline
\eta & \infty & \infty & 1/(1-\ep_+) &  \\ \hline
\dt \eta & \infty & \infty &  &  \\ \hline
\dt^2 \eta & \infty & 2  & &  \\ \hline
\bar{\eta} & \infty & \infty & 2/(1-\ep_+) & 2/(2-\ep_+) \\ \hline
\dt \bar{\eta} & \infty & \infty & 4/(2- (\ep_--\low) ) &   \\ \hline
\dt^2 \bar{\eta} & \infty  &  4 &   &  \\ \hline
\tr \bar{\eta} & \infty & \infty & 1/(1-\ep_+) &  \\ \hline
\tr \dt \bar{\eta} & \infty & \infty &  &  \\ \hline
\tr \dt^2 \bar{\eta} & \infty & 2  & &  \\ \hline
\end{array}
\end{displaymath}

\end{thm}
\begin{proof}
The estimates for $u$, $p$, $\eta$ and their derivatives follow directly from the standard Sobolev embeddings and trace theorems, together with the definition of $\E$.   The estimates for $\bar{\eta}$ at its derivatives follow similarly, except that we also employ Proposition \ref{poisson_prop} to account for the regularity gains arising from the appearance of the Poisson extensions $\mathcal{P}$ in the definition of $\bar{\eta}$.
\end{proof}

Next we record the catalog of estimates in terms of the dissipation.

\begin{thm}\label{catalog_dissipation}
The following three tables record the $L^q$ bounds for $u$, $p$, $\eta$ and their derivatives in terms of the dissipation $\sqrt{\D}$, as defined in \eqref{D_def}.
\begin{displaymath}
\def\arraystretch{1.2}
\begin{array}{|c | c | c | c | c |}
\hline
\text{Function} = f  & f & \nab f & \nab^2 f & \nab^3 f \\ 
\hline
u & \infty & 2/(1-\ep_+) & 2/(2-\ep_+) &  \\ \hline
\dt u & \infty & 2/(1-\ep_-) & 2/(2-\ep_-) &  \\ \hline
\dt^2 u & \infty- & 2 & &  \\ \hline
\tr u & \infty & 1/(1-\ep_+) & &  \\ \hline
\tr \dt u & \infty & 1/(1-\ep_-) & &  \\ \hline
\tr \dt^2 u & \infty- & & &  \\ \hline
\end{array}
\end{displaymath}
\begin{displaymath}
\def\arraystretch{1.2}
\begin{array}{|c | c | c | c | c |}
\hline
\text{Function} = f  & f & \nab f & \nab^2 f & \nab^3 f \\ 
\hline
p & 2/(1-\ep_+) & 2/(2-\ep_+) & &  \\ \hline
\dt p & 2/(1-\ep_-) & 2/(2-\ep_-)& &  \\ \hline
\tr p & 1/(1-\ep_+) & & &  \\ \hline
\tr \dt p & 1/(1-\ep_-) & & &  \\ \hline
\end{array}
\end{displaymath}
\begin{displaymath}
\def\arraystretch{1.2}
\begin{array}{|c | c | c | c | c |}
\hline
\text{Function} = f  & f & \nab f & \nab^2 f & \nab^3 f \\ 
\hline
\eta & \infty & \infty & 1/(1-\ep_+) &  \\ \hline
\dt \eta & \infty & \infty & 1/(1-\ep_-) &  \\ \hline
\dt^2 \eta & \infty & 1/\low & &  \\ \hline
\dt^3 \eta & 1/\low & & &  \\ \hline
\bar{\eta} & \infty & \infty & 2/(1-\ep_+) & 2/(2-\ep_+) \\ \hline
\dt \bar{\eta} & \infty & \infty & 2/(1-\ep_-) & 2/(2-\ep_-)  \\ \hline
\dt^2 \bar{\eta} & \infty  &  2/\low & 2/(1+\low)&  \\ \hline
\dt^3 \bar{\eta} & 2/\low & 2/(1+2\low) & &  \\ \hline
\tr \bar{\eta} & \infty & \infty & 1/(1-\ep_+) &  \\ \hline
\tr \dt \bar{\eta} & \infty & \infty & 1/(1-\ep_-) &  \\ \hline
\tr \dt^2 \bar{\eta} & \infty & 1/\low & &  \\ \hline
\tr \dt^3 \bar{\eta} & 1/\low & & &  \\ \hline
\end{array}
\end{displaymath}

\end{thm}
\begin{proof}
The estimates for $u$, $p$, $\eta$ and their derivatives follow directly from the standard Sobolev embeddings and trace theorems, together with the definition of $\D$.   The estimates for $\bar{\eta}$ at its derivatives follow similarly, except that we also employ Proposition \ref{poisson_prop} to account for the regularity gains arising from the appearance of the Poisson extensions $\mathcal{P}$ in the definition of $\bar{\eta}$.
\end{proof}

\section{Elliptic theory for Stokes problems}\label{sec_elliptic}

In this section we record some elliptic theory for the Stokes problem.  We begin with analysis of a model problem in $2D$ cones and then build to a theory in the domain $\Omega$ given by \eqref{Omega_def}.  The material here roughly mirrors the material in Section 5 of \cite{guo_tice_QS}, except that here we work in $L^q-$based spaces without weights rather than $L^2-$based weighted spaces.

\subsection{Analysis in cones }

Given an opening angle $\omega \in (0,\pi)$, we define the infinite $2D$ cone 
\begin{equation}\label{cone1_def}
 K_\omega = \{x \in \Rn{2} \st  r>0 \text{ and } \theta \in (-\pi/2,-\pi/2 + \omega)   \},
\end{equation}
where $r$ and $\theta$ are the usual polar coordinates in $\Rn{2}$ with the set $\{\theta =-\pi/2\}$ chosen to coincide with the negative $x_2$ axis.   We define two parts of $\p K_\omega$ via
\begin{equation}
 \Gamma_- = \{ x \in \Rn{2} \st r>0 \text{ and } \theta =-\pi/2  \} \text{ and } \Gamma_+ = \{x \in \Rn{2} \st r>0 \text{ and } \theta =-\pi/2 + \omega  \}.
\end{equation}

Next we introduce a special matrix-valued function.  Suppose that $\af: K_\omega \to \Rn{2\times 2}$ is a map satisfying the following four properties.  First, $\af$ is smooth on $K_\omega$ and  $\af$ extends to a smooth function on $\bar{K}_\omega \backslash \{0\}$ and a continuous function on $\bar{K}_\omega$.  Second, $\af$ satisfies the following for all $a,b \in \mathbb{N}$:
\begin{equation}\label{frak_A_assump}
\begin{split}
& \lim_{r\to 0} \sup_{\theta \in [-\pi/2,-\pi/2 + \omega]} \abs{  (r \p_r)^a \p_\theta^b [ \af(r,\theta) \af^T(r,\theta) - I  ]    } =0  \\
& \lim_{r\to 0} \sup_{\theta \in [-\pi/2,-\pi/2 + \omega]} \abs{  (r \p_r)^a \p_\theta^b [ \af_{ij}(r,\theta)\p_j \af_{ik}(r,\theta)  ]    } =0 \text{ for  }k \in \{1,2\} \\
& \lim_{r\to 0} \sup_{\theta \in [-\pi/2,-\pi/2 + \omega]} \abs{  (r \p_r)^a \p_\theta^b [ \af(r,\theta)  - I  ]    } =0  \\
& \lim_{r\to 0}   (r \p_r)^a  [ \af(r,\theta_0)\nu   - \nu  ]     =0    \text{ for } \theta_0 =-\pi/2,-\pi/2 + \omega \\
& \lim_{r\to 0}  (r \p_r)^a  \left[ \left(\af\nu \otimes \af^T (\af \nu)^\bot + (\af\nu)^\bot \otimes \af^T (\af\nu)\right)(r,\theta_0)   - I  \right]     =0    \text{ for } \theta_0 =-\pi/2,-\pi/2 + \omega \\
\end{split}
\end{equation}
where again $(r,\theta)$ denote polar coordinates and $(z_1,z_2)^\bot = (z_2,-z_1)$. Third, the matrix $\af \af^T$ is uniformly elliptic on $K_\omega$.  Fourth,  $\det \af =1$ and $\p_j( \af_{ij}) = 0 \text{ for }i=1,2.$

The matrix $\af$ serves to determine the coefficients in a variant of the Stokes problem in the cone $K_\omega$.  This problem, which we call the $\af-$Stokes problem, reads:
\begin{equation}\label{af_stokes_cone}
\begin{cases}
 \diverge_\af S_\af(Q,v) = G^1 &\text{in } K_\omega \\
 \diverge_\af v = G^2 &\text{in } K_\omega \\
 v \cdot \af \nu = G^3_\pm &\text{on } \Gamma_\pm \\
 \mu \sg_\af v \af \nu \cdot (\af \nu)^\bot = G^4_\pm &\text{on } \Gamma_\pm,
\end{cases}
\end{equation}
where here the operators $\diverge_\af$ and $S_\af$ are defined in the same way as $\diva$ and $S_\A$.  When $\af = I_{2 \times 2}$ all of the above assumptions are trivially verified, and we arrive at the usual Stokes problem
\begin{equation}\label{stokes_cone}
\begin{cases}
 \diverge S(Q,v) = G^1 &\text{in } K_\omega \\
 \diverge v = G^2 &\text{in } K_\omega \\
 v \cdot  \nu = G^3_\pm &\text{on } \Gamma_\pm \\
 \mu \sg v  \nu \cdot  \tau = G^4_\pm &\text{on } \Gamma_\pm.
\end{cases}
\end{equation}
The purpose of the  assumptions in \eqref{frak_A_assump} is to guarantee that the problems \eqref{af_stokes_cone} and \eqref{stokes_cone} have the same elliptic regularity properties.

Next we introduce a parameter depending on the cone's opening angle that determines how much regularity is gained in these Stokes problems.  Given  $\omega \in (0,\pi)$ we define
\begin{equation}\label{crit_wt}
 \epm(\omega) = \min\{1,-1+\pi/\omega\} \in (0,1]. 
\end{equation}
We can now state the elliptic regularity for these Stokes problems.

\begin{thm}\label{cone1_solve}
Let $\omega \in (0,\pi)$ and $\epm(\omega)$ be as in \eqref{crit_wt}. Let $0 <\delta < \epm(\omega)$ and set 
\begin{equation}
 q_\delta = \frac{2}{2-\delta} \in (1,2).
\end{equation}
Suppose that $\af$ satisfies the four properties stated above and  that the data $G^1,G^2, G^3_\pm,G^4_\pm$ for the problem \eqref{stokes_cone} satisfy 
\begin{equation}\label{cone1_solve_01}
G^1 \in L^{q_\delta}(K_\omega), G^2 \in W^{1,q_\delta}(K_\omega), G^3_\pm \in W^{2-1/q_\delta,q_\delta}(\Gamma_\pm), G^4_\pm \in W^{1-1/q_\delta,q_\delta}(\Gamma_\pm)
\end{equation}
as well as the compatibility condition
\begin{equation}\label{cone1_solve_cc}
 \int_{K_\omega} G^2  = \int_{\Gamma_+} G^3_+  + \int_{\Gamma_-} G^3_- .
\end{equation}
Suppose that $(v,Q) \in H^1(K_\omega) \times H^0(K_\omega)$ satisfy  $\diverge_\af v = G^2$, $v\cdot \af \nu = G^3_\pm$ on $\Gamma_\pm$, and  
\begin{equation}
 \int_{K_\omega} \frac{\mu}{2} \sg_\af v : \sg_\af w  - Q \diverge_\af w   = \int_{K_\omega} G^1  \cdot w      + \int_{\Gamma_+} \G^4_+ w\cdot \frac{(\af \nu)^\bot}{\abs{\af \nu}} + \int_{\Gamma_-} \G^4_-  w\cdot \frac{(\af \nu)^\bot}{\abs{\af \nu}} 
\end{equation}
for all $w \in \{ w \in H^1(K_\omega) \st w\cdot (\af \nu) =0 \text{ on } \Gamma_\pm\}$.   Finally, suppose that $v,Q$ and all of the data $G^i$ are supported in $\bar{K}_\omega \cap B[0,1]$.  Then $v \in W^{2,q_\delta}(K_\omega) \cap H^{1+\delta}(K_\omega)$, $Q \in W^{1,q_\delta}(K_\omega) \cap H^\delta(K_\omega)$, and  
\begin{multline}\label{cone1_solve_02}
 \norm{v}_{W^{2,q_\delta}} + \norm{v}_{H^{1+\delta}} + \norm{ Q}_{W^{1,q_\delta}}  + \norm{Q}_{H^\delta}    \\
 \ls \norm{G^1}_{L^{q_\delta} } +  \norm{G^2}_{W^{1,q_\delta}} +  \norm{G^3_-}_{W^{2-1/q_\delta,q_\delta} } +  \norm{G^3_+}_{W^{2-1/q_\delta,q_\delta}} 
 +  \norm{G^4_-}_{W^{1-1/q_\delta,q_\delta}} +  \norm{G^4_+}_{W^{1-1/q_\delta,q_\delta}}.
\end{multline}
\end{thm}
\begin{proof}
In the case $\af = I$ the result is proved in Corollary 4.2 of \cite{orlt_sandig} when $G^3_\pm =0$ and $G^4_\pm =0$ and in Theorem 3.6 of \cite{orlt} in the general case.  The choice of $q_\delta$ is determined by the eigenvalues of an operator pencil associated to \eqref{stokes_cone}, which may be found in the ``G-G eigenvalue computations'' of \cite{orlt_sandig} (with $\chi_1 = \chi_2 = \pi/2$).  These results show that these eigenvalues for the Stokes problem \eqref{stokes_cone} in the cone $K_\omega$  are $\pm 1 + n \pi/\omega$ for $n \in \mathbb{Z}$, which leads to the constraint $q < 2/(2-\epm(\omega))$ in $W^{2,q} \times W^{1,q}$ estimates.  However, Theorem 8.2.1 of \cite{kmr_1}, together with the assumptions on $\af$, guarantee that the operator pencils that determine the regularity of \eqref{af_stokes_cone} coincide with those of \eqref{stokes_cone}, so the estimates of \cite{orlt_sandig} and \cite{orlt} remain valid for the $\af-$Stokes problem. 
\end{proof}

\subsection{The Stokes problem in $\Omega$}

We now study the following Stokes problem in $\Omega$, as defined by \eqref{Omega_def}:
\begin{equation}\label{stokes_omega}
\begin{cases}
\diverge S(Q,v)  = G^1 &\text{in } \Omega \\
 \diverge v = G^2 &\text{in } \Omega \\
 v \cdot \nu = G^3_+ &\text{on } \Sigma \\
 \mu \sg v \nu \cdot \tau = G^4_+ &\text{on } \Sigma\\
 v \cdot \nu = G^3_- &\text{on } \Sigma_s \\
 \mu \sg v \nu \cdot \tau = G^4_- &\text{on } \Sigma_s.
\end{cases}
\end{equation}
Consider $0 < \delta < \epm$ (defined by \eqref{epm_def} in terms of $\omeq$ from \eqref{omega_eq_def}) and $q_\delta = 2/(2-\delta) \in (1,2)$.  We will study this problem with data belonging to the space $\mathfrak{X}_\delta$, which we define as the space of  $6-$tuples
\begin{multline}
 (G^1,G^2,G^3_+,G^3_-,G^4_+,G^4_-) \\
 \in L^{q_\delta}(\Omega) \times W^{1,q_\delta}(\Omega) \times W^{2-1/q_\delta,q_\delta}(\Sigma  )\times W^{2-1/q_\delta,q_\delta}(\Sigma_s  ) \times  W^{1-1/q_\delta,q_\delta}(\Sigma) \times W^{1-1/q_\delta,q_\delta}(\Sigma_s)             
\end{multline}
such that 
\begin{equation}
 \int_{\Omega} G^2 = \int_{\Sigma} G^3_+ + \int_{\Sigma_s} G^3_-.
\end{equation}
We endow this space with the obvious norm 
\begin{multline}
 \norm{ (G^1,G^2,G^3_+,G^3_-,G^4_+,G^4_-)}_{\mathfrak{X}_\delta} = \norm{G^1}_{L^{q_\delta}} + \norm{G^2}_{W^{1,q_\delta}} + \norm{G^3_+}_{W^{2-1/q_\delta,q_\delta}}  + \norm{G^3_-}_{W^{2-1/q_\delta,q_\delta}} \\
 + \norm{G^4_+}_{W^{1-1/q_\delta,q_\delta}}  + \norm{G^4_-}_{W^{1-1/q_\delta,q_\delta}}
\end{multline}

We have the following weak existence result, which works without constraint on $\delta \in (0,1)$.  

\begin{thm}\label{stokes_om_weak}
Assume that $(G^1,G^2,G^3_+,G^3_-,G^4_+,G^4_-)  \in \mathfrak{X}_\delta$ for any $0 < \delta < 1$.   Then there exist a unique pair $(v,Q) \in H^1(\Omega) \times \oH^0(\Omega)$ that is a weak solution to \eqref{stokes_omega} in the sense that  $\diverge v = G^2$, $v\cdot \nu = G^3$ on $\p \Omega$, and  
\begin{equation}\label{stokes_om_weak_form}
 \int_{\Omega} \frac{\mu}{2} \sg v : \sg w - Q \diverge w = \int_{\Omega} G^1  \cdot w   + \int_{\Sigma} G^4_+ (w\cdot \tau) + \int_{\Sigma_s} G^4_- (w \cdot \tau) 
\end{equation}
for all $w \in \{ w \in H^1(\Omega) \st w\cdot \nu =0 \text{ on } \p \Omega\}$.  Moreover, 
\begin{equation}\label{stokes_weak_0}
 \norm{v}_{H^1} + \norm{Q}_{L^2} \ls  \norm{ (G^1,G^2,G^3_+,G^3_-,G^4_+,G^4_-)}_{\mathfrak{X}_\delta}.
\end{equation}
\end{thm}
\begin{proof}
The argument is standard and doesn't use the higher-regularity structure of $\mathfrak{X}_\delta$.  See, for instance, Theorem 5.3 in \cite{guo_tice_QS}.
\end{proof}

For second-order regularity we do need the constraints on $\delta$ in order to use Theorem \ref{cone1_solve}.

\begin{thm}\label{stokes_om_reg}
Let $\epm\in (0,1]$ be given by \eqref{epm_def}, and $0 < \delta < \epm$.  Let $(G^1,G^2,G^3_+,G^3_-,G^4_+,G^4_-)   \in \mathfrak{X}_{\delta}$, and let $(v,Q) \in H^1(\Omega) \times \oH^0(\Omega)$ be the weak solution to \eqref{stokes_omega} constructed in Theorem \ref{stokes_om_weak}.   Then $v \in W^{2,q_\delta}(\Omega) \cap H^{1+\delta}(\Omega)$, $Q \in W^{1,q_\delta}(\Omega)\cap \oH^\delta(\Omega)$, and 
\begin{equation}\label{stokes_om_reg_0}
 \norm{v}_{W^{2,q_\delta}} + \norm{v}_{H^{1+\delta}}  + \norm{Q}_{W^{1,q_\delta}} + \norm{Q}_{H^{\delta}}\ls   \norm{ (G^1,G^2,G^3_+,G^3_-,G^4_+,G^4_-)}_{\mathfrak{X}_\delta} .
\end{equation}
\end{thm}
\begin{proof}
The argument used in Theorem 5.5 of \cite{guo_tice_QS} works in the present case as well, except that we use the estimates of Theorem \ref{cone1_solve} in place of the estimates from Theorem 5.1 in \cite{guo_tice_QS}.
\end{proof}

In what follows it will be useful to rephrase Theorem \ref{stokes_om_reg} as follows.  For $0 < \delta < \epm$ we define the operator 
\begin{equation}\label{stokes_om_iso_def1}
 T_\delta : \left( W^{2,q_\delta}(\Omega) \cap H^{1+\delta}(\Omega) \right) \times \left( W^{1,q_\delta}_\delta(\Omega) \cap \oH^\delta(\Omega) \right) \to \mathfrak{X}_\delta
\end{equation}
via
\begin{equation}\label{stokes_om_iso_def2}
 T_\delta(v,Q) = (\diverge S(Q,v), \diverge v, v\cdot n \vert_{\Sigma},v\cdot n \vert_{\Sigma_s}, \mu \sg v n \cdot \tau  \vert_{\Sigma}, \mu \sg v n \cdot \tau  \vert_{\Sigma_s}).  
\end{equation}
We may then deduce the following from Theorems \ref{stokes_om_weak} and \ref{stokes_om_reg}.

\begin{cor}\label{stokes_om_iso}
Let $\epm$ be as in \eqref{epm_def}.   If  $0 <\delta < \epm$,  then the operator $T_{\delta}$ defined by \eqref{stokes_om_iso_def1} and \eqref{stokes_om_iso_def2} is an isomorphism.
\end{cor}

\subsection{The $\A$-Stokes problem in $\Omega$}

Next we consider a version of the Stokes problem with coefficients that depend on a given function  $\eta \in W^{3-1/q_\delta,q_\delta}$ with $0 < \delta < \epm$.  The function $\eta$ determines the coefficients  $\A,$ $J$, and $\N$ via \eqref{AJK_def} and \eqref{N_def}, and we study the system
\begin{equation}\label{A_stokes}
\begin{cases}
\diva S_\A(Q,v) = G^1 & \text{in }\Omega \\
J \diva v = G^2 & \text{in } \Omega \\
v\cdot \N/ \abs{\N_0} = G^3_+ &\text{on } \Sigma \\
\mu \sg_\A v \N \cdot \mathcal{T} / \abs{\N_0}^2 = G^4_+ &\text{on } \Sigma \\
v\cdot J\nu = G^3_- &\text{on } \Sigma_s \\
\mu \sg_\A v \nu \cdot \tau = G^4_- &\text{on } \Sigma_s.
\end{cases}
\end{equation}
Note here that $\N = \N_0 - \p_1 \eta e_1$ for $\N_0$, given by \eqref{N0_def}, the outward normal vector on $\Sigma$ and $\mathcal{T} = \mathcal{T}_0 + \p_1 \eta e_2$ for $\mathcal{T}_0 = e_1 + \p_1 \zeta_0 e_2$ the associated  tangent vector.  

We begin our analysis of this problem by introducing the operator 
\begin{equation}\label{A_stokes_om_iso_def1}
 T_\delta[\eta] :  \left( W^{2,q_\delta}(\Omega) \cap H^{1+\delta}(\Omega) \right) \times \left( W^{1,q_\delta}_\delta(\Omega) \cap \oH^\delta(\Omega) \right) \to \mathfrak{X}_\delta
\end{equation}
given by 
\begin{equation}\label{A_stokes_om_iso_def2}
 T_\delta[\eta](v,Q) = (\diva S_\A(Q,v), J \diva v, v\cdot \N/\abs{\N_0} \vert_{\Sigma}, v\cdot J\nu \vert_{\Sigma_s}, \mu \sga v \N \cdot \mathcal{T} / \abs{\N_0}^2  \vert_{\Sigma}, \mu \sga v \nu \cdot \tau  \vert_{\Sigma_s}).  
\end{equation}

The map $T_\delta[\eta]$, which encodes the solvability of \eqref{A_stokes}, is an isomorphism under a smallness assumption on $\eta$.

\begin{thm} \label{A_stokes_om_iso}
Let $\epm$ be as in \eqref{epm_def}.  Let $0 < \delta < \epm$ and $q_\delta = 2/(2-\delta)$.  There exists a $\gamma >0$ such that if $\norm{\eta}_{W^{3-1/q_\delta,q_\delta}} < \gamma$, then  the operator $T_{\delta}[\eta]$ defined by \eqref{A_stokes_om_iso_def1} and \eqref{A_stokes_om_iso_def2} is well-defined and yields a bounded isomorphism.
\end{thm}
\begin{proof}
We divide the proof into steps.

\emph{Step 1 - Setup: }  First note that 
\begin{equation}
 \int_\Omega J \diva v = \int_{\p \Omega} J \A \nu \cdot v = \int_{\Sigma}   \frac{\N}{\abs{\N_0}} \cdot v+ \int_{\Sigma_s} J \nu \cdot v,
\end{equation}
which establishes the compatibility between the second and third terms needed for $T_\delta[\eta]$ to map into $\mathfrak{X}_\delta$.  

Now assume that $\gamma < 1$ is as small as in Lemma \ref{eta_small}.  We write $T_\delta[\eta](v,q) = T_\delta(v,Q) - \G(v,Q)$ where $T_\delta$ is defined by \eqref{stokes_om_iso} and $\G$ denotes the linear map with components
\begin{equation}
\begin{split}
\G^1(v,Q) &=  \diverge_{I-\A} S_{\A}(Q,v) - \diverge \mu \sg_{I-\A}(v)  \\
\G^2(v) &=  \diverge_{I-\A} v  + (1-J) \diva v \\
\G^3_+(v)  &= (1+(\p_1\zeta_0)^2)^{-1/2}[\p_1 \eta v_1  ]    \\
\G^4_+(v) &= (1+(\p_1\zeta_0)^2)^{-1} [ \mu \sg_{I-\A} v \N_0\cdot \mathcal{T}_0 -\mu \p_1 \eta (\sg_\A v \N_0 \cdot e_2 - \sg_\A v e_1 \cdot \mathcal{T}_0  ) +\mu (\p_1 \eta)^2 \sg_\A v e_1 \cdot e_2    ] \\
\G^3_- & = (1-J) v\cdot \nu  \\
\G^4_-(v) & =  \mu \sg_{I-\A} v \nu \cdot \tau.
\end{split}
\end{equation} 
Since both $T_\delta[\eta]$ and $T_\delta$ enforce the compatibility between the second and third terms, $\G$ does as well.  Then the equation $T_\delta[\eta](v,Q) = G := (G^1,G^2,G^3_+,G^3_-,G^4_+,G^4_-)$ is equivalent to 
\begin{equation}\label{A_stokes_om_iso_1}
 T_\delta(v,Q) = G + \G(v,Q).
\end{equation}

\emph{Step 2 - $\G$ boundedness: }  We now claim that 
\begin{equation}\label{A_stokes_om_iso_0}
 \norm{\G(v,Q) }_{\mathfrak{X}_\delta} \ls \norm{\eta}_{W^{3-1/q_\delta,q_\delta}}  \left( \norm{v}_{W^{2,q_\delta}} + \norm{Q}_{W^{1,q_\delta}}  \right).
\end{equation}
We proceed term by term.

\textbf{$\G^1$ estimate:}  We need to bound $\G^1(v,Q)$ in $L^{q_\delta}(\Omega)$.  We estimate the first term via 
\begin{multline}
\norm{\diverge_{I-A} S_{\A}(Q,v)  }_{L^{q_\delta}}  \ls \norm{\nab \bar{\eta}}_{L^\infty} \left(  \norm{ \nab Q }_{L^{q_\delta}} + \norm{ \nab^2 v }_{L^{q_\delta}} \right) \\
+ \norm{\nab \bar{\eta}}_{L^\infty} \norm{\nab^2 \bar{\eta}}_{L^{2/(1-\delta)}} \left(\norm{Q}_{L^{2/(1-\delta)}} + \norm{\nab v}_{L^{2/(1-\delta)}}   \right) \ls \norm{\eta}_{W^{3-1/q_\delta,q_\delta}} \left(\norm{v}_{W^{2,q_\delta}} + \norm{Q}_{W^{1,q_\delta}} \right).
\end{multline}
Similarly, we estimate the second term as
\begin{equation}
\norm{\diverge \mu \sg_{I-\A}(v)}_{L^{q_\delta}} \ls \norm{\nab \bar{\eta}}_{L^\infty} \norm{D^2 v}_{L^{q_\delta}} + \norm{\nab^2 \bar{\eta}}_{L^{2/(1-\delta)}}  \norm{\nab v}_{L^{2/(1-\delta)}}  
\ls \norm{\eta}_{W^{3-1/q_\delta,q_\delta}}  \norm{v}_{W^{2,q_\delta}}.
\end{equation}
Combining these two, we deduce that 
\begin{equation}
\norm{ \G^1(v,Q) }_{L^{q_\delta}} \ls  \norm{\eta}_{W^{3-1/q_\delta,q_\delta}} \left(\norm{v}_{W^{2,q_\delta}} + \norm{Q}_{W^{1,q_\delta}} \right).
\end{equation}

\textbf{$\G^2$ estimate: }  We need to bound $\G^2(v)$ in $W^{1,q_\delta}(\Omega)$.  For the first term
\begin{equation}
\norm{ \diverge_{I-\A} v  }_{W^{1,q_\delta}} \ls  \norm{\nab \bar{\eta}}_{L^\infty} \norm{v}_{W^{2,q_\delta}} +\norm{\nab^2 \bar{\eta}}_{L^{2/(1-\delta)}}  \norm{\nab v}_{L^{2/(1-\delta)}}  \ls \norm{\eta}_{W^{3-1/q_\delta,q_\delta}}  \norm{v}_{W^{2,q_\delta}}.
\end{equation}
Similarly, for the second term we bound 
\begin{multline}
\norm{ (1-J) \diva v   }_{W^{1,q_\delta}} \ls  \norm{\nab \bar{\eta}}_{L^\infty} \norm{v}_{W^{2,q_\delta}} + (1+\norm{\nab \bar{\eta}}_{L^\infty}) \norm{\nab^2 \bar{\eta}}_{L^{2/(1-\delta)}}  \norm{\nab v}_{L^{2/(1-\delta)}} \\
\ls \norm{\eta}_{W^{3-1/q_\delta,q_\delta}}  \norm{v}_{W^{2,q_\delta}}.
\end{multline}
Combining these, we deduce that 
\begin{equation}
\norm{ \G^2(v) }_{W^{1,q_\delta}} \ls  \norm{\eta}_{W^{3-1/q_\delta,q_\delta}}  \norm{v}_{W^{2,q_\delta}}
\end{equation}

\textbf{$\G^3_+$ estimate:}  We need to bound $\G^3_+(v)$ in $W^{2-1/q_\delta,q_\delta}(\Sigma)$.  For this we use the trace characterization of boundary norms and the fact that $W^{2,q_\delta}(\Omega)$ is an algebra to estimate:
\begin{multline}
 \norm{\G^3_+(v) }_{W^{2-1/q_\delta,q_\delta}(\Sigma)} \ls  \norm{ \p_1 \bar{\eta} v_1   }_{W^{2,q_\delta}(\Omega)} \ls \norm{\p_1 \bar{\eta} v_1   }_{W^{2,q_\delta}(\Omega)} \\
 \ls \norm{\p_1 \bar{\eta}   }_{W^{2,q_\delta}(\Omega)} \norm{v   }_{W^{2,q_\delta}(\Omega)} 
 \ls \norm{\bar{\eta}   }_{W^{3,q_\delta}(\Omega)} \norm{v   }_{W^{2,q_\delta}(\Omega)}
 \ls \norm{\eta   }_{W^{3-1/q_\delta,q_\delta}(\Omega)} \norm{v   }_{W^{2,q_\delta}(\Omega)}.
\end{multline}

\textbf{$\G^3_-$ estimate:} We need to bound $\G^3_+(v)$ in $W^{2-1/q_\delta,q_\delta}(\Sigma)$.  Since $\nu$ is determined by $\Sigma_s$, which is $C^2$, we can argue as with $\G^3_+$ to estimate 
\begin{equation}
\norm{\G^3_-(v)}_{W^{2-1/q_\delta,q_\delta}(\Sigma)} \ls \norm{(1-J) v}_{W^{2,q_\delta}(\Sigma)} \ls \norm{\bar{\eta}   }_{W^{3,q_\delta}(\Omega)} \norm{v   }_{W^{2,q_\delta}(\Omega)}
 \ls \norm{\eta   }_{W^{3-1/q_\delta,q_\delta}(\Omega)} \norm{v   }_{W^{2,q_\delta}(\Omega)}.
\end{equation}

\textbf{$\G^4_+$ estimate:} We need to bound $\G^4(v)$ in in $W^{1-1/q_\delta,q_\delta}(\Sigma)$.  Recall that  $\N_0 = -\p_1 \zeta_0 e_1 + e_2$ and $\mathcal{T}_0 = e_1 + \p_1 \zeta_0 e_2$ are smooth, so we can bound 
\begin{equation}
 \norm{\G^4_+(v)}_{W^{1-1/q_\delta,q_\delta}(\Sigma)} \ls  \norm{\sg_{I-\A} v}_{W^{1-1/q_\delta,q_\delta}(\Sigma)} + \norm{\p_1 \eta \sg_\A v }_{W^{1-1/q_\delta,q_\delta}(\Sigma)} + 
 \norm{(\p_1\eta)^2 \sg_\A v}_{W^{1-1/q_\delta,q_\delta}(\Sigma)}.
\end{equation}
We then use the trace characterization again to bound 
\begin{multline}
  \norm{\sg_{I-\A} v}_{W^{1-1/q_\delta,q_\delta}(\Sigma)}+   \norm{\p_1 \eta \sg_\A v}_{W^{1-1/q_\delta,q_\delta}(\Sigma)}  \ls  \norm{\sg_{I-\A} v}_{W^{1,q_\delta}(\Omega)} +  \norm{\p_1 \bar{\eta} \sg_\A v}_{W^{1,q_\delta}(\Omega)} \\
  \ls   \norm{\nab \bar{\eta} \nab v}_{L^{q_\delta}(\Omega)}
+   \norm{\nab^2 \bar{\eta} \nab v}_{L^{q_\delta}(\Omega)}
+  \norm{\nab \bar{\eta} \nab^2 v}_{L^{q_\delta}(\Omega)}  \\
\ls \norm{\nab \bar{\eta}}_{L^\infty} \norm{v}_{W^{2,q_\delta}} + \norm{\nab^2 \bar{\eta}}_{L^{2/(1-\delta)}} \norm{\nab v}_{L^{2/(1-\delta)}}  
\ls   \norm{\eta   }_{W^{3-1/q_\delta,q_\delta}(\Omega)} \norm{v   }_{W^{2,q_\delta}(\Omega)}.
\end{multline}
Similarly, 
\begin{multline}
 \norm{(\p_1\eta)^2 \sg_\A v}_{W^{1-1/q_\delta,q_\delta}(\Sigma)} \ls \norm{(\p_1\bar{\eta})^2 \sg_\A v}_{W^{1,q_\delta}(\Omega)} \\
 \ls    \norm{\nab \bar{\eta}}_{L^\infty}^2 \norm{v}_{W^{2,q_\delta}} + \norm{\nab \bar{\eta}}_{L^\infty} \norm{\nab^2 \bar{\eta}}_{L^{2/(1-\delta)}} \norm{\nab v}_{L^{2/(1-\delta)}}  
\ls  \norm{\eta   }_{W^{3-1/q_\delta,q_\delta}(\Omega)} \norm{v   }_{W^{2,q_\delta}(\Omega)}. 
\end{multline}
Combining these shows that 
\begin{equation}
  \norm{\G^4_+(v)}_{W^{1-1/q_\delta,q_\delta}(\Sigma)} \ls  \norm{\eta   }_{W^{3-1/q_\delta,q_\delta}(\Omega)} \norm{v   }_{W^{2,q_\delta}(\Omega)}. 
\end{equation}

\textbf{$\G^4_-$ estimate:}  We need to bound $\G^4_-(v)$ in $W^{1-1/q_\delta,q_\delta}(\Sigma_s)$.  Since $\nu$ and $\tau$ are determined by $\Sigma_s$ and are thus   $C^2$ we can estimate in exactly the same way as above:
\begin{multline}
  \norm{\G^4_-(v)}_{W^{1-1/q_\delta,q_\delta}(\Sigma_s)}  =   \norm{\sg_{I-\A} v}_{W^{1-1/q_\delta,q_\delta}(\Sigma_s)}   \ls  \norm{\sg_{I-\A} v}_{W^{1,q_\delta}(\Omega)}  \\
  \ls   \norm{\nab \bar{\eta} \nab v}_{L^{q_\delta}(\Omega)}
+   \norm{\nab^2 \bar{\eta} \nab v}_{L^{q_\delta}(\Omega)}
+  \norm{\nab \bar{\eta} \nab^2 v}_{L^{q_\delta}(\Omega)}  \\
\ls \norm{\nab \bar{\eta}}_{L^\infty} \norm{v}_{W^{2,q_\delta}} + \norm{\nab^2 \bar{\eta}}_{L^{2/(1-\delta)}} \norm{\nab v}_{L^{2/(1-\delta)}}  
\ls   \norm{\eta   }_{W^{3-1/q_\delta,q_\delta}(\Omega)} \norm{v   }_{W^{2,q_\delta}(\Omega)}.
\end{multline}

\textbf{Synthesis:}  Combining the above estimates shows that the bound \eqref{A_stokes_om_iso_0} holds.

\emph{Step 3 - Isomorphism: }   The map $T_{\delta}$ is an isomorphism, so \eqref{A_stokes_om_iso_1} is equivalent to the fixed point problem
\begin{equation}\label{A_stokes_om_iso_3}
 (v,Q) = T_{\delta}^{-1}(G+ \G(v,Q)) =: \Psi(v,Q)
\end{equation}
for $\Psi$ a map from $Z := \left( W^{2,q_\delta}(\Omega) \cap H^{1+\delta}(\Omega) \right) \times \left( W^{1,q_\delta}_\delta(\Omega) \cap \oH^\delta(\Omega) \right)$ to itself.  From \eqref{A_stokes_om_iso_0} we have that 
\begin{equation}
 \norm{\Psi(v_1,Q_1) - \Psi(v_2,Q_2)}_Z \le C \norm{\eta}_{W^{3-1/q_\delta,q_\delta}} \norm{T^{-1}_{\delta}}_{\text{op}}  \norm{(v_1,Q_1) - (v_2,Q_2) }_{Z}.   
\end{equation}
Hence, if $\gamma$ is sufficiently small, then $\Psi$ is a contraction and thus there exists a unique $(v,Q)$ solving  \eqref{A_stokes_om_iso_1} for every $G$.  In turn, this means that $T_\delta[\eta]$ is an isomorphism with this choice of $\gamma$. 
\end{proof}

\subsection{The $\A$-Stokes problem in $\Omega$ with $\beta \neq 0$}

As the next step we modify the boundary conditions in \eqref{A_stokes} to include the Navier-slip friction term on the vessel walls. The new system is:
\begin{equation}\label{A_stokes_beta}
\begin{cases}
\diva S_\A(Q,v) = G^1 & \text{in }\Omega \\
J \diva v = G^2 & \text{in } \Omega \\
v\cdot \N/ \abs{\N_0} = G^3_+ &\text{on } \Sigma \\
\mu \sg_\A v \N \cdot \mathcal{T} / \abs{\N_0}^2 = G^4_+ &\text{on } \Sigma \\
v\cdot J\nu = G^3_- &\text{on } \Sigma_s \\
\mu \sg_\A v \nu \cdot \tau + \beta v\cdot \tau = G^4_- &\text{on } \Sigma_s,
\end{cases}
\end{equation}
where $\beta>0$ is the Navier-slip friction coefficient.

We have the following existence result.

\begin{thm}\label{A_stokes_beta_solve}
Let $\epm$ be as in \eqref{epm_def}. Let $0 < \delta < \epm$, $q_\delta = 2/(2-\delta)$, and suppose that  $\ns{\eta}_{W^{3-1/q_\delta,q_\delta} } < \gamma$, where $\gamma$ is as in Theorem \ref{A_stokes_om_iso}.  If  $(G^1,G^2,G^3_+,G^3_-,G^4_+,G^4_-)   \in \mathfrak{X}_{\delta}$, then there exists a unique 
\begin{equation}
 (v,Q) \in \left( W^{2,q_\delta}(\Omega) \cap H^{1+\delta}(\Omega) \right) \times \left( W^{1,q_\delta}_\delta(\Omega) \cap \oH^\delta(\Omega) \right)
\end{equation}
solving \eqref{A_stokes_beta}.  Moreover, the solution obeys the estimate
\begin{equation}\label{A_stokes_beta_solve_0}
 \norm{v}_{W^{2,q_\delta}} + \norm{v}_{H^{1+\delta}}  + \norm{Q}_{W^{1,q_\delta}} + \norm{Q}_{H^{\delta}}\ls   \norm{ (G^1,G^2,G^3_+,G^3_-,G^4_+,G^4_-)}_{\mathfrak{X}_\delta} .
\end{equation}
\end{thm}
\begin{proof}

Define the operator $R: \left( W^{2,q_\delta}(\Omega) \cap H^{1+\delta}(\Omega) \right) \times \left( W^{1,q_\delta}_\delta(\Omega) \cap \oH^\delta(\Omega) \right) \to \mathfrak{X}_\delta$ via 
\begin{equation}
 R(v,q) = (0,0,0,0,0,\beta v\cdot \nu\vert_{\Sigma_s}),
\end{equation}
which is bounded and well-defined since $v\cdot \nu \in W^{2-1/q_\delta,q_\delta}(\Sigma_s)$.  Standard Sobolev theory shows that the embedding $W^{2-1/q_\delta,1_\delta}(\Sigma_s) \hookrightarrow W^{1-1/q_\delta,q_\delta}(\Sigma_s)$ is compact, so $R$ is a compact operator.  Theorem \ref{A_stokes_om_iso} tells us that the operator $T_{\delta}[\eta]$ is an isomorphism, so the compactness of $R$ implies that  $T_{\delta}[\eta] + R$ is a Fredholm operator.  We claim that this map is injective. Once this is proved,  the Fredholm alternative implies that the map is also surjective and hence is an isomorphism.  

To prove the claim we assume $(T_{\delta}[\eta] + R)(v,Q) =0$, i.e. \eqref{A_stokes_beta} holds with all the $G^i$ terms vanishing.  We multiply the first equation in \eqref{A_stokes_beta} by $J v$ and integrate by parts, arguing as in Lemma \ref{geometric_evolution}, to arrive at the identity 
\begin{equation}
\int_\Omega \frac{\mu}{2} \abs{\sg_\A v}^2 J + \int_{\Sigma_s} \beta\abs{v\cdot \tau}^2 J =0.
\end{equation}
Thus $v=0$, but then $0 = \nab_\A Q = \A \nab Q=0$, which implies, since $\A$ is invertible (via Lemma \ref{eta_small}), that $Q$ is constant.  Since $Q \in \oH^{\delta}$ we then have that $Q =0$.  This proves the claim.

\end{proof}

\subsection{The $\A$-Stokes problem in $\Omega$ with a boundary equations for $\xi$}
 
We finally have the tools needed to address the desired problem, which synthesizes the $\A-$Stokes system in $\Omega$ with boundary conditions on $\Sigma$ involving a new unknown $\xi$: 
\begin{equation}\label{A_stokes_stress}
\begin{cases}
\diva S_\A(Q,v) = G^1 & \text{in }\Omega \\
J \diva v = G^2 & \text{in } \Omega \\
v\cdot \N/ \abs{\N_0} = G^3_+ &\text{on } \Sigma \\
S_\A(Q,v) \N =  \left[ g\xi  -\sigma \p_1\left( \frac{\p_1 \xi}{(1+\abs{\p_1 \zeta_0}^2)^{3/2}} + G^6 \right)  \right] \N   + G^4_+ \mathcal{T} + G^5\N  &\text{on } \Sigma \\
v\cdot J \nu = G^3_- &\text{on } \Sigma_s \\
 (\Sa(Q,v)\nu - \beta v)\cdot \tau = G^4_- &\text{on } \Sigma_s \\
 \mp \sigma \frac{\p_1 \xi}{(1+\abs{\p_1 \zeta_0}^2)^{3/2}} (\pm \ell) = G^7_\pm. 
\end{cases}
\end{equation}

We have the following existence result for  \eqref{A_stokes_stress}.

\begin{thm}\label{A_stokes_stress_solve}
Let $\epm$ be as in \eqref{epm_def}.  Let $0 < \delta < \epm$, $q_\delta = 2/(2-\delta)$, and suppose that  $\ns{\eta}_{W^{3-1/q_\delta,q_\delta} } < \gamma$, where $\gamma$ is as in Theorem \ref{A_stokes_om_iso}.  If  $(G^1,G^2,G^3_+,G^3_-,G^4_+,G^4_-)   \in \mathfrak{X}_{\delta}$, and $G^5,\p_1 G^6 \in W^{1-1/q_\delta,q_\delta}(\Sigma)$, and $G^7_\pm \in \R$,
then there exists a unique 
\begin{equation}
 (v,Q,\xi) \in \left( W^{2,q_\delta}(\Omega) \cap H^{1+\delta}(\Omega) \right) \times \left( W^{1,q_\delta}_\delta(\Omega) \cap \oH^\delta(\Omega) \right) \times W^{3-1/q_\delta,q_\delta}(\Sigma)
\end{equation}
solving \eqref{A_stokes_stress}.  Moreover, the solution obeys the estimate
\begin{multline}\label{A_stokes_stress_0}
 \norm{v}_{W^{2,q_\delta}} + \norm{v}_{H^{1+\delta}}  + \norm{Q}_{W^{1,q_\delta}} + \norm{Q}_{H^{\delta}}  + \norm{\xi}_{W^{3-1/q_\delta,q_\delta}}\\
 \ls   \norm{ (G^1,G^2,G^3_+,G^3_-,G^4_+,G^4_-)}_{\mathfrak{X}_\delta} 
 + \ns{G^5}_{W^{1-1/q_\delta,q_\delta}} + \ns{\p_1 G^6}_{W^{1-1/q_\delta,q_\delta}} + [G^7]_\ell^2,
\end{multline}
where we recall that $[\cdot,\cdot]_\ell$ is defined in \eqref{bndry_pairing}.
\end{thm}
\begin{proof}
First note that since $\abs{\N} = \abs{\mathcal{T}}$, 
\begin{equation}
 S_\A(Q,v) \N =  \left[ g\xi  -\sigma \p_1\left( \frac{\p_1 \xi}{(1+\abs{\p_1 \zeta_0}^2)^{3/2}} + G^6 \right)  \right] \N   + G^4_+ \mathcal{T} + G^5\N
\end{equation}
is equivalent to 
\begin{equation}
S_\A(Q,v) \N \cdot \frac{\N}{\abs{\N}^2}  =  \left[ g\xi  -\sigma \p_1\left( \frac{\p_1 \xi}{(1+\abs{\p_1 \zeta_0}^2)^{3/2}} + G^6 \right)  \right]     +G^5 
\end{equation}
and 
\begin{equation}
S_\A(Q,v) \N \cdot \frac{\mathcal{T}}{\abs{\N_0}^2} =  G^4_+ \frac{\abs{\N}^2}{\abs{\N_0}^2} \in.
\end{equation}

Note that the same sort of argument used in the proof of Theorem \ref{A_stokes_om_iso} shows that 
\begin{equation}
 \norm{G^4_+ \frac{\abs{\N}^2}{\abs{\N_0}^2} }_{W^{1-q_\delta,q_\delta}} \ls  \norm{G^4_+  }_{W^{1-q_\delta,q_\delta}}
\end{equation}
since $\norm{\eta}_{W^{3-1/q_\delta,q_\delta}} \le 1$.  We may then use Theorem \ref{A_stokes_beta_solve} to produce the pair $(v,Q)$ solving
\begin{equation} 
\begin{cases}
\diva S_\A(Q,v) = G^1 & \text{in }\Omega \\
J \diva v = G^2 & \text{in } \Omega \\
v\cdot \N/ \abs{\N_0} = G^3_+ &\text{on } \Sigma \\
\mu \sg_\A v \N \cdot \mathcal{T} / \abs{\N_0}^2 = G^4_+ \frac{\abs{\N}^2}{\abs{\N_0}^2} &\text{on } \Sigma \\
v\cdot J\nu = G^3_- &\text{on } \Sigma_s \\
\mu \sg_\A v \nu \cdot \tau + \beta v\cdot \tau = -G^4_- &\text{on } \Sigma_s,
\end{cases}
\end{equation}
and obeying the estimates \eqref{A_stokes_beta_solve_0}.  With this $(v,Q)$ in hand we then have a solution to \eqref{A_stokes_stress} as soon as we find $\xi$ solving 
\begin{equation}\label{A_stokes_stress_solve_1}
g\xi  -\sigma \p_1\left( \frac{\p_1 \xi}{(1+\abs{\p_1 \zeta_0}^2)^{3/2}} \right) = S_\A(Q,v) \N\cdot \frac{\N}{\abs{\N}^2}     +\sigma \p_1 G^6   -   G^5 
\end{equation}
on $\Sigma$ subject to the boundary conditions
\begin{equation}\label{A_stokes_stress_solve_2}
  \mp \sigma \left(\frac{\p_1 \xi}{(1+\abs{\p_1 \zeta_0}^2)^{3/2}} + F^3 \right)(\pm \ell) = G^7_\pm.
\end{equation}
The estimate \eqref{A_stokes_beta_solve_0} guarantees that $S_\A(q,v) \N\cdot \frac{\N}{\abs{\N}^2} \in W^{1-1/q_\delta,q_\delta}(\Sigma)$, the usual elliptic theory provides a unique $\xi \in W^{3-1/q_\delta,q_\delta}(\Sigma)$ satisfying \eqref{A_stokes_stress_solve_1} and \eqref{A_stokes_stress_solve_2} and obeying the estimate 
\begin{equation}\label{A_stokes_stress_solve_3}
\norm{\xi}_{W^{1-1/q_\delta,q_\delta}} \ls \norm{S_\A(Q,v) \N\cdot \frac{\N}{\abs{\N}^2}}_{W^{1-1/q_\delta,q_\delta}} + \norm{\p_1 G^6}_{W^{1-1/q_\delta,q_\delta}} + \norm{G^5}_{_{W^{1-1/q_\delta,q_\delta}}} + [G^7]_\ell^2 .
\end{equation}
Then \eqref{A_stokes_stress_0} follow by combining \eqref{A_stokes_beta_solve_0} and \eqref{A_stokes_stress_solve_3}.
\end{proof}

\section{Nonlinear estimates I: interaction terms, dissipative form}\label{sec_nl_int_d}

In this section we begin our study of the estimates available for the nonlinearities that appear in the system \eqref{ns_geometric} and its derivatives.  Here we focus on the interaction terms as they appear in Theorem \ref{linear_energy} and on deriving estimates in terms of the dissipation functional.  In order to avoid tedious restatements of the same hypothesis, we assume throughout this section that a solution to \eqref{ns_geometric} exists on the time horizon $(0,T)$ for $0 < T \le \infty$ and obeys the small-energy estimate 
\begin{equation}
 \sup_{0\le t < T} \E(t) \le \gamma^2 < 1,
\end{equation}
where $\gamma \in (0,1)$ is as in Lemma \ref{eta_small}.  In particular, this means that the estimates of Lemma \ref{eta_small} are available for use, and we will use them often without explicit reference.

\subsection{General interaction functional estimates}

We begin by studying the terms involving $F^1$, $F^4$, and $F^5$ in Theorem \ref{linear_energy}.  The structure of these is not particularly delicate, so we can derive general dual estimates in which the particular form of the test function is irrelevant.

We begin by studying $F^1$.

\begin{prop}\label{nid_f1}
Suppose that $F^1$ is as defined in either \eqref{dt1_f1} or \eqref{dt2_f1}.  Then 
\begin{equation}\label{nid_f1_0}
\abs{\int_\Omega J w \cdot F^1} \ls \norm{w}_{H^1} \left(\sqrt{\E} + \E \right) \sqrt{\D}.
\end{equation}
\end{prop}
\begin{proof}
We will present the proof only in the more involved case that $F^1$ is defined by \eqref{dt2_f1}, which corresponds to two temporal derivatives.  The case \eqref{dt1_f1}, which corresponds to one temporal derivative, follows from a simpler and easier argument.  There are fifteen terms appearing in \eqref{dt2_f1}, and we will deal with them one at a time, proving that each can be estimated in the stated form.  For the sake of brevity, throughout the proof we will repeatedly make use of four essential tools without explicitly referring to them: H\"{o}lder's inequality,  the standard Sobolev embeddings for $w \in H^1(\Omega)$, the fact that $\E \le 1$, and the catalogs of $L^q$ estimates given in Theorems \ref{catalog_energy} and \ref{catalog_dissipation}.  For the latter we will always use the following ordering convention: the ordering in expressions of the form
\begin{equation}
 a b c \ls A \sqrt{\E} \sqrt{\D} \text{ and } a b' c' \ls A \sqrt{\D} \E
\end{equation}
implies that we bound $a \ls A$, use Theorem \ref{catalog_energy} to bound $b \ls \sqrt{\E}$ and $c' \ls \E$, and use Theorem \ref{catalog_dissipation} to estimate $c \ls \sqrt{\D}$ and $b' \ls \sqrt{\D}$. In other words, the order of appearance of $\E$ and $\D$ on the right side corresponds to the order on the left and indicates which of Theorems \ref{catalog_energy} and \ref{catalog_dissipation} is being used implicitly.

\textbf{Term: $- 2\diverge_{\dt \A} S_\A(\dt p,\dt u)$. } We first bound 
\begin{multline}
 \abs{\int_\Omega J w \cdot (- 2\diverge_{\dt \A} S_\A(\dt p,\dt u)) } \ls \int_\Omega \abs{w} \abs{ \dt \A} (\abs{\nab \dt p} + \abs{\nab^2 \dt u} ) \\
 + \int_\Omega \abs{w} \abs{\dt \A} \abs{\nab \A} (\abs{\dt p} + \abs{\nab \dt u}) =: I + II.
\end{multline}
For $I$ we then bound 
\begin{equation}
I \ls \norm{w}_{L^{2/\ep_-}} \norm{  \dt \bar{\eta}}_{W^{1,\infty}} \left( \norm{\nab \dt p}_{L^{2/(2-\ep-)}} +\norm{\nab^2 \dt u}_{L^{2/(2-\ep-)}}  \right) \ls \norm{w}_{H^1} \sqrt{\E} \sqrt{\D},
\end{equation}
and for $II$ we bound 
\begin{equation}
 II \ls \norm{w}_{L^{2/(\ep_- + \ep_+)}} \norm{ \dt \bar{\eta}}_{W^{1,\infty}} \norm{ \bar{\eta}}_{W^{2,2/(1-\ep_+)}} \left( \norm{\dt p}_{L^{2/(1-\ep_-)}} + \norm{\nab \dt u}_{L^{2/(1-\ep_-)}} \right) 
 \ls \norm{w}_{H^1} \E \sqrt{\D}.
\end{equation}
Combining these shows that this term can be estimated as stated.

\textbf{Term: $2\mu \diva \sg_{\dt \A} \dt u$. }  We first bound 
\begin{equation}
 \abs{\int_\Omega J w \cdot (2\mu \diva \sg_{\dt \A} \dt u )} \ls \int_\Omega \abs{w} \abs{\nab \dt \A} \abs{\nab \dt u} + \int_\Omega \abs{w}\abs{ \dt \A} \abs{\nab^2 \dt u} =: I + II.
\end{equation}
We then bound 
\begin{equation}
 I \ls \norm{w}_{L^{4/(3\ep_-)}} \left( \norm{\bar{\eta}}_{W^{2,2/(1-\ep_-)}} +  \norm{ \dt \bar{\eta}}_{W^{2,2/(1-\ep_-)}} \right) \norm{\nab \dt u}_{L^{4/(2-\ep-)}} \ls \norm{w}_{H^1} \sqrt{\D} \sqrt{\E}
\end{equation}
and 
\begin{equation}
 II \ls \norm{w}_{L^{2/\ep_-}} \norm{ \dt \bar{\eta}}_{W^{1,\infty}} \norm{\nab^2 \dt u}_{L^{2/(2-\ep_-)}} \ls \norm{w}_{H^1} \sqrt{\E} \sqrt{\D}.
\end{equation}
Combining these shows that this term can be estimated as stated.

\textbf{Term: $- \diverge_{\dt^2 \A} S_\A(p,u)$. }  We first bound 
\begin{equation}
 \abs{\int_\Omega J w \cdot (- \diverge_{\dt^2 \A} S_\A(p,u))  } \ls \int_\Omega \abs{w} \abs{\dt^2 \A} (\abs{\nab p} + \abs{\nab^2 u}) + \int_\Omega \abs{w} \abs{\dt^2 \A} \abs{\nab \A} \abs{\nab u} =: I + II.
\end{equation}
Then we estimate 
\begin{equation}
 I \ls \norm{w}_{L^{2/(\ep_+ - \low)}}\left(\norm{\dt \bar{\eta}}_{W^{1,2/\low}} +  \norm{\dt^2 \bar{\eta}}_{W^{1,2/\low}} \right) \left(\norm{\nab p}_{L^{2/(2-\ep_+)}} + \norm{\nab^2 u}_{L^{2/(2-\ep_+)}} \right) \ls \norm{w}_{H^1} \sqrt{\D} \sqrt{\E}
\end{equation}
and 
\begin{equation}
II \ls \norm{w}_{L^{2/(2\ep_+ - \low)}} \left(\norm{\dt \bar{\eta}}_{W^{1,2/\low}} +  \norm{\dt^2 \bar{\eta}}_{W^{1,2/\low}} \right) \norm{ \bar{\eta}}_{W^{2,2/(1-\ep_+)}} \norm{\nab u}_{L^{2/(1-\ep_+)}} \ls \norm{w}_{H^1} \sqrt{\D} \E.
\end{equation}
Combining these shows that this term can be estimated as stated.

\textbf{Term: $2 \mu \diverge_{\dt \A} \sg_{\dt \A} u$. } We first estimate 
\begin{equation}
 \abs{\int_\Omega J w \cdot (2 \mu \diverge_{\dt \A} \sg_{\dt \A} u) } \ls \int_\Omega \abs{w} \abs{ \dt \A}^2 \abs{\nab^2 u} + \int_\Omega \abs{w} \abs{\dt \A} \abs{\nab  \dt \A} \abs{\nab u} =: I + II.
\end{equation}
We then bound 
\begin{equation}
 I \ls \norm{w}_{L^{2/\ep_+}} \ns{\dt \bar{\eta}}_{W^{1,\infty}} \norm{\nab^2 u}_{L^{2/(2-\ep_+)}} \ls \norm{w}_{H^1} \E \sqrt{\D}
\end{equation}
and 
\begin{equation}
 II \ls  \norm{w}_{L^{2/(\ep_- + \ep_+)}} \norm{ \dt \bar{\eta}}_{W^{1,\infty}} 
\left( \norm{\bar{\eta}}_{W^{2,2/(1-\ep_-)}} +  \norm{ \dt \bar{\eta}}_{W^{2,2/(1-\ep_-)}} \right) 
   \norm{\nab u}_{L^{2/(1-\ep_+)}} \ls \norm{w}_{H^1} \sqrt{\E} \sqrt{\D} \sqrt{\E}.
\end{equation}
Combining these shows that this term can be estimated as stated.

\textbf{Term: $\mu \diva \sg_{\dt^2 \A} u$. } We initially estimate 
\begin{equation}
 \abs{\int_\Omega J w \cdot (\mu \diva \sg_{\dt^2 \A} u)} \ls \int_\Omega \abs{w} \abs{ \dt^2 \A} \abs{\nab^2 u} + \int_\Omega \abs{w} \abs{\nab \dt^2 \A} \abs{\nab u}  =: I + II.
\end{equation}
Then we bound 
\begin{equation}
 I \ls \norm{w}_{L^{2/(\ep_+ - \low)}}\left(\norm{ \dt \bar{\eta}}_{W^{1,2/\low}} +  \norm{ \dt^2 \bar{\eta}}_{W^{1,2/\low}} \right) \norm{\nab^2 u}_{L^{2/(2-\ep_+)}} \ls \norm{w}_{H^1} \sqrt{\D} \sqrt{\E}
\end{equation}
and 
\begin{equation}
 II \ls \norm{w}_{L^{2/(\ep_+ - \low)}} \left( \norm{ \dt \bar{\eta}}_{W^{2,2/(1+\low)}}+\norm{ \dt^2 \bar{\eta}}_{W^{2,2/(1+\low)}} \right) \norm{\nab u}_{L^{2/(1-\ep_+)}} \ls \norm{w}_{H^1} \sqrt{\D} \sqrt{\E}.
\end{equation}
Combining these shows that this term can be estimated as stated.

\textbf{Term: $- 2 u \cdot \nab_{\dt \A} \dt u$. } For this term we bound 
\begin{multline}
\abs{\int_\Omega Jw \cdot (2 u \cdot \nab_{\dt \A} \dt u)} \ls \int_{\Omega} \abs{w} \abs{u} \abs{ \dt \A} \abs{\nab \dt u} \ls \norm{w}_{L^{2}} \norm{u}_{L^\infty} \norm{\dt \bar{\eta}}_{W^{1,\infty}} \norm{\nab \dt u}_{L^{2}} \\
\ls \norm{w}_{H^1} \sqrt{\E} \sqrt{\D} \sqrt{\E}.
\end{multline}

\textbf{Term: $- 2 \dt u \cdot \naba \dt u$. }  We bound 
\begin{equation}
\abs{\int_{\Omega} J w \cdot(2 \dt u \cdot \naba \dt u)} \ls \int_{\Omega} \abs{w} \abs{\dt u} \abs{\nab \dt u} \ls \norm{w}_{L^{2}} \norm{\dt u}_{L^\infty} \norm{\nab \dt u}_{L^{2}} 
\ls \norm{w}_{H^1} \sqrt{\E} \sqrt{\D}. 
\end{equation}

\textbf{Term: $- 2 \dt u \cdot \nab_{\dt \A} u$. }  We bound 
\begin{multline}
 \abs{\int_{\Omega} Jw \cdot(2 \dt u \cdot \nab_{\dt \A} u)} \ls \int_{\Omega} \abs{w} \abs{\dt u} \abs{\dt \A} \abs{\nab u} \ls \norm{w}_{L^{2}} \norm{\dt u}_{L^\infty} \norm{\dt \bar{\eta}}_{W^{1,\infty}} \norm{\nab u}_{L^{2}} \\
 \ls \norm{w}_{H^1} \E \sqrt{\D}.
\end{multline}

\textbf{Term: $- u \cdot \nab_{\dt^2 \A} u$. } We estimate 
\begin{multline}
 \abs{\int_{\Omega} Jw \cdot(u \cdot \nab_{\dt^2 \A} u)}  \ls \int_{\Omega} \abs{w} \abs{u} \abs{ \dt^2 \A} \abs{\nab u} \ls \norm{w}_{L^{2/(1 - \low)}} \norm{u}_{L^\infty} 
 \left(\norm{\dt \bar{\eta}}_{W^{1,2/\low}} +  \norm{\dt^2 \bar{\eta}}_{W^{1,2/\low}} \right) 
\norm{\nab u}_{L^{2}} \\
 \ls \norm{w}_{H^1} \sqrt{\E} \sqrt{\D} \sqrt{\E}.
\end{multline}

\textbf{Term: $- \dt^2 u \cdot \naba u$. } We estimate 
\begin{equation}
  \abs{\int_{\Omega} Jw \cdot(\dt^2 u \cdot \naba u)}  \ls \int_{\Omega} \abs{w} \abs{\dt^2 u} \abs{\nab u} \ls \norm{w}_{L^{2/\ep_+}} \norm{\dt^2 u}_{L^2} \norm{\nab u}_{L^{2/(1-\ep_+)}} \ls \norm{w}_{H^1} \sqrt{\D} \sqrt{\E}.
\end{equation}

\textbf{Term: $2 \dt \bar{\eta} \frac{\phi}{\zeta_0} \dt K \p_2 \dt u$. }  For this term we bound 
\begin{multline}
 \abs{\int_{\Omega} J w \cdot(2 \dt \bar{\eta} \frac{\phi}{\zeta_0} \dt K \p_2 \dt u) }  \ls \int_{\Omega} \abs{w} \abs{\dt \bar{\eta}} \abs{\dt K}\abs{\nab \dt u} \ls \norm{w}_{L^2} \norm{\dt \bar{\eta}}_{L^\infty} \norm{ \dt \bar{\eta}}_{W^{1,\infty}} \norm{\nab \dt u}_{L^2} \\
\ls \norm{w}_{H^1} \E \sqrt{\D}. 
\end{multline}

\textbf{Term: $2 \dt^2 \bar{\eta} \frac{\phi}{\zeta_0} K \p_2 \dt u$. } We estimate 
\begin{multline}
\abs{\int_{\Omega} J w \cdot(2 \dt^2 \bar{\eta} \frac{\phi}{\zeta_0} K \p_2 \dt u) } \ls \int_{\Omega} \abs{w} \abs{\dt^2 \bar{\eta}} \abs{\nab \dt u} \ls \norm{w}_{L^2} \norm{\dt^2 \bar{\eta}}_{L^\infty} \norm{\nab \dt u}_{L^2} \ls \norm{w}_{H^1} \sqrt{\D} \sqrt{\E}.
\end{multline}

\textbf{Term: $2 \dt^2 \bar{\eta} \frac{\phi}{\zeta_0} \dt K  \p_2 u$. }  We bound 
\begin{multline}
 \abs{\int_{\Omega} Jw \cdot(2 \dt^2 \bar{\eta} \frac{\phi}{\zeta_0} \dt K  \p_2 u)} \ls \int_{\Omega} \abs{w} \abs{\dt^2 \bar{\eta}} \abs{\dt K} \abs{\nab u} \ls  \norm{w}_{L^2} \norm{\dt^2 \bar{\eta}}_{L^\infty} \norm{\dt \bar{\eta}}_{W^{1,\infty}} \norm{\nab u}_{L^2} \\
 \ls \norm{w}_{H^1} \sqrt{\D} \E.
\end{multline}

\textbf{Term: $\dt^3 \bar{\eta} \frac{\phi}{\zeta_0} K \p_2 u$. } We estimate 
\begin{equation}
\abs{\int_{\Omega} J w \cdot(\dt^3 \bar{\eta} \frac{\phi}{\zeta_0} K \p_2 u)}  \ls \int_{\Omega} \abs{w} \abs{\dt^3 \bar{\eta}} \abs{\nab u} \ls \norm{w}_{L^{2/\ep_+}} \norm{\dt^3 \bar{\eta}}_{L^2} \norm{\nab u}_{L^{2/(1-\ep_+)}} \ls \norm{w}_{H^1} \sqrt{\D} \sqrt{\E}.
\end{equation}

\textbf{Term: $\dt \bar{\eta} \frac{\phi}{\zeta_0} \dt^2 K\p_2 u$. } For the final term we bound 
\begin{multline}
 \abs{\int_\Omega J w \cdot (\dt \bar{\eta} \frac{\phi}{\zeta_0} \dt^2 K\p_2 u)} \ls \int_{\Omega}  \abs{w} \abs{\dt \bar{\eta}} \abs{\dt^2 K} \abs{\nab u} \\
 \ls \norm{w}_{L^{2/(1-\low)}} \norm{\dt \bar{\eta}}_{L^\infty} 
  \left(\norm{\dt \bar{\eta}}_{W^{1,2/\low}} +  \norm{\dt^2 \bar{\eta}}_{W^{1,2/\low}} \right)
 \norm{\nab u}_{L^2}  
 \ls \norm{w}_{H^1} \sqrt{\E} \sqrt{\D} \sqrt{\E}.
\end{multline}
\end{proof}

Next we study the $F^4$ nonlinearity.

\begin{prop}\label{nid_f4}
Suppose that $F^4$ is defined as in either \eqref{dt1_f4} or \eqref{dt2_f4}.  Then we 
\begin{equation}\label{nid_f4_0}
\abs{ \int_{-\ell}^\ell w \cdot  F^4  } \ls \norm{w}_{H^1} (\sqrt{\E} + \E) \sqrt{\D}
\end{equation}
for all $w \in H^1(\Omega)$.
\end{prop}
\begin{proof}
We will present the proof only in the more involved case that $F^4$ is defined by \eqref{dt2_f4}, which corresponds to two temporal derivatives.  The case \eqref{dt1_f4}, which corresponds to one temporal derivative, follows from a simpler and easier argument.  There are eleven terms appearing in \eqref{dt2_f4}, and we will deal with them mostly one at a time, with just a bit of grouping.  We will prove that each can be estimated in the stated form.  For the sake of brevity, throughout the proof we will repeatedly make use of five essential tools without explicitly referring to them: H\"{o}lder's inequality, standard trace estimates for $H^1(\Omega)$, the standard Sobolev embeddings for $H^1(\Omega)$ and $H^{1/2}((-\ell,\ell))$, the fact that $\E \le 1$, and the catalogs of $L^q$ estimates given in Theorems \ref{catalog_energy} and \ref{catalog_dissipation}.  For the latter we will again use the ordering convention described at the start of the proof of Proposition \ref{nid_f1}.

\textbf{Term: $ 2\mu \sg_{\dt \A} \dt u \N$.}  We bound 
\begin{multline}
\abs{\int_{-\ell}^\ell 2\mu w \cdot (\sg_{\dt \A} \dt u)(\N) }
\ls \int_{-\ell}^\ell \abs{w} \abs{\dt \A} \abs{\nab \dt u} 
 \ls  \norm{w}_{L^{1/\ep_-}(\Sigma)}   \norm{\dt \eta}_{W^{1,\infty}} \norm{\nab \dt u}_{L^{1/(1-\ep_-)}(\Sigma)}   \\
\ls \norm{w}_{H^1}   \sqrt{\E} \sqrt{\D} . 
\end{multline}

\textbf{Term: $ \mu \sg_{\dt^2 \A} u \N$.} We estimate 
\begin{multline}
\abs{ \int_{-\ell}^\ell w \cdot \mu \sg_{\dt^2 \A} u \N } 
\ls \int_{-\ell}^\ell \abs{w} \abs{\dt^2 \A} \abs{\nab u}   \\
\ls \norm{w}_{L^{1/(\ep_+ - \low)}(\Sigma)}  \left(\norm{\dt \eta}_{W^{1,1/\low}} + \norm{\dt^2 \eta}_{W^{1,1/\low}} \right)  \norm{\nab u}_{L^{1/(1-\ep_+)}(\Sigma)} 
\ls \norm{w}_{H^1}   \sqrt{\D}\sqrt{\E}.
\end{multline}

\textbf{Term: $\mu \sg_{\dt \A} u \dt \N$.}  We bound 
\begin{multline}
\abs{ \int_{-\ell}^\ell w \cdot \mu \sg_{\dt \A} u \dt \N  }
\ls \int_{-\ell}^\ell \abs{w} \abs{\dt \A} \abs{\nab u} \abs{\p_1 \dt \eta} \\
 \ls  \norm{w}_{L^{1/\ep_+}(\Sigma)}   \norm{\dt \eta}_{W^{1,\infty}} \norm{\nab u}_{L^{1/(1-\ep_+)}(\Sigma)}  \norm{\dt \p_1 \eta}_{L^\infty} 
\ls \norm{w}_{H^1}  \E \sqrt{\D}. 
\end{multline}

\textbf{Term: $\left[  2g \dt \eta  - 2\sigma \p_1 \left(\frac{\p_1 \dt \eta}{(1+\abs{\p_1 \zeta_0}^2)^{3/2}} \right)  \right]  \dt \N $.} We estimate
\begin{multline}
\abs{ \int_{-\ell}^\ell w \cdot \left[  2g \dt \eta  - 2\sigma \p_1 \left(\frac{\p_1 \dt \eta}{(1+\abs{\p_1 \zeta_0}^2)^{3/2}}  \right)  \right]  \dt \N  }
\ls \int_{-\ell}^\ell \abs{w}(\abs{\dt \eta} + \abs{\p_1 \dt \eta} + \abs{\p_1^2 \dt \eta} ) \abs{\p_1 \dt \eta} \\
\ls \norm{w}_{L^{1/\ep_-}(\Sigma)} \norm{\dt \eta}_{W^{2,1/(1-\ep_-)}}  \norm{\p_1 \dt \eta}_{L^\infty} 
\ls \norm{w}_{H^1} \sqrt{\D} \sqrt{\E} .
\end{multline}

\textbf{Term: $\p_1 \dt[ \mathcal{R}(\p_1 \zeta_0,\p_1 \eta)]    \dt \N$.} For this term we initially expand
\begin{multline}
 \p_1 \dt[ \mathcal{R}(\p_1 \zeta_0,\p_1 \eta)]  = \p_1 [\p_z \mathcal{R}(\p_1 \zeta_0,\p_1 \eta) \p_1 \dt \eta]  =  \frac{\p_z \mathcal{R}(\p_1 \zeta_0,\p_1 \eta)}{\p_1 \eta} \p_1 \eta \p_1^2 \dt \eta  + \p_z^2 \mathcal{R}(\p_1 \zeta_0,\p_1 \eta) \p_1^2 \eta \p_1 \dt \eta \\
 + \frac{ \p_z \p_y \mathcal{R}(\p_1 \zeta_0,\p_1 \eta)}{\p_1 \eta} \p_1 \eta \p_1^2 \zeta_0 \p_1 \dt \eta.
\end{multline}
This and Proposition \ref{R_prop} then allow us to bound 
\begin{multline}
\abs{ \int_{-\ell}^\ell w \cdot  \p_1 [\dt[ \mathcal{R}(\p_1 \zeta_0,\p_1 \eta)] ] \dt \N } 
\ls \int_{-\ell}^\ell \abs{w} \abs{\p_1 \dt \eta} \abs{\p_1 \eta}\abs{\p_1^2 \dt \eta} 
+ \int_{-\ell}^\ell \abs{w} \abs{\p_1 \dt \eta} \abs{\p_1^2 \eta} \abs{\p_1 \dt \eta} \\
+\int_{-\ell}^\ell \abs{w} \abs{\p_1 \dt \eta} \abs{\p_1 \eta} \abs{\p_1 \dt \eta} =: I + II + III.
\end{multline}
We then bound 
\begin{equation}
I \ls \norm{w}_{L^{1/\ep_-}(\Sigma)} \norm{\p_1 \dt \eta}_{L^\infty} \norm{\p_1 \eta}_{L^\infty} \norm{\p_1^2 \dt \eta}_{L^{1/(1-\ep_-)}} \ls \norm{w}_{H^1} \E \sqrt{\D},
\end{equation}
\begin{equation}
 II \ls   \norm{w}_{L^{1/\ep_+}(\Sigma)} \ns{\p_1 \dt \eta}_{L^\infty} \norm{\p_1^2   \eta}_{L^{1/(1-\ep_+)}} \ls \norm{w}_{H^1} \E \sqrt{\D},
\end{equation}
and 
\begin{equation}
 III \ls \norm{w}_{L^{2}(\Sigma)} \ns{\p_1 \dt \eta}_{L^\infty} \norm{\p_1   \eta}_{L^{2}} \ls \norm{w}_{H^1} \E \sqrt{\D}.
\end{equation}
Combining these then shows that this term can be estimated as stated.

\textbf{Term: $ -2 S_{\A}(\dt p,\dt u)  \dt \N$.} We estimate 
\begin{multline}
\abs{ \int_{-\ell}^\ell -2 w \cdot  S_{\A}(\dt p,\dt u)  \dt \N }
\ls \int_{-\ell}^\ell \abs{w} (\abs{\dt p} + \abs{\nab \dt u}) \abs{\p_1 \dt \eta}
\\
 \ls \norm{w}_{L^{1/\ep_-}(\Sigma)}   \left(\norm{\dt p}_{L^{1/(1-\ep_-)}(\Sigma)} + \norm{\nab \dt u}_{L^{1/(1-\ep_-)}(\Sigma)} \right) \norm{\dt \p_1 \eta}_{L^\infty} 
\ls \norm{w}_{H^1} \sqrt{\D} \sqrt{\E}.
\end{multline}

\textbf{Term: $\left[ g\eta  - \sigma \p_1 \left(\frac{\p_1 \eta}{(1+\abs{\p_1 \zeta_0}^2)^{3/2}}   \right)  \right]  \dt^2 \N$.} We bound
\begin{multline}
\abs{ \int_{-\ell}^\ell w \cdot \left[ g\eta  - \sigma \p_1 \left(\frac{\p_1 \eta}{(1+\abs{\p_1 \zeta_0}^2)^{3/2}}   \right)  \right]  \dt^2 \N  }
\ls \int_{-\ell}^\ell \abs{w} (\abs{\eta} + \abs{\p_1 \eta} + \abs{\p_1^2 \eta} ) \abs{\p_1 \dt^2 \eta} \\
\ls \norm{w}_{L^{1/(\ep_+ - \low)}(\Sigma)} \norm{\eta}_{W^{2,1/(1-\ep_+)}} \norm{\p_1 \dt^2 \eta}_{L^{1/\low}}
 \ls \norm{w}_{H^1} \sqrt{\E} \sqrt{\D}.
\end{multline}

\textbf{Term: $ \p_1[  \mathcal{R}(\p_1 \zeta_0,\p_1 \eta) ]     \dt^2 \N$.}  To handle this term we expand
\begin{equation}
 \p_1 [ \mathcal{R}(\p_1 \zeta_0,\p_1 \eta)]    =  \frac{\p_z \mathcal{R}(\p_1 \zeta_0,\p_1 \eta)}{\p_1 \eta} \p_1 \eta \p_1^2 \eta  
 + \frac{ \p_y \mathcal{R}(\p_1 \zeta_0,\p_1 \eta)}{(\p_1 \eta)^2}\p_1^2 \zeta_0 (\p_1 \eta)^2 .
\end{equation}
This and Proposition \ref{R_prop} then allow us to bound 
\begin{multline}
\abs{ \int_{-\ell}^\ell w \cdot \p_1[  \mathcal{R}(\p_1 \zeta_0,\p_1 \eta) ]     \dt^2 \N  }
\ls \int_{-\ell}^\ell \abs{w} \abs{\p_1 \dt^2 \eta} (\abs{\p_1 \eta} \abs{\p_1^2 \eta} + \abs{\p_1 \eta}^2  )  \\
\ls \norm{w}_{L^{1/(\ep_+-\low)}(\Sigma)} \norm{\p_1 \dt^2 \eta}_{L^{1/\low}} \left(\norm{\p_1 \eta}_{L^\infty} \norm{\p_1^2 \eta}_{L^{1/(1-\ep_+)}} +  \norm{\p_1 \eta}_{L^\infty} \norm{\p_1 \eta}_{L^{1/(1-\ep_+)}}   \right)
\ls \norm{w}_{H^1} \sqrt{\D} \E.
\end{multline}

\textbf{Term: $- S_{\A}(p,u)   \dt^2 \N$.}  For the final term we  estimate 
\begin{multline}
\abs{ - \int_{-\ell}^\ell w \cdot S_{\A}(p,u)   \dt^2 \N }
\ls \int_{-\ell}^\ell \abs{w} (\abs{p} + \abs{\nab u}) \abs{\p_1 \dt^2 \eta} \\
\ls \norm{w}_{L^{1/(\ep_+ - \low)}(\Sigma)} \left(\norm{p}_{L^{1/(1-\ep_+)}(\Sigma)} + \norm{\nab u}_{L^{1/(1-\ep_+)}(\Sigma)}   \right) \norm{\p_1 \dt^2 \eta}_{L^{1/\low}}
\ls \norm{w}_{H^1} \sqrt{\E} \sqrt{\D}.
\end{multline}
 
\end{proof}

Finally, we study the $F^5$ nonlinearity.

\begin{prop}\label{nid_f5}
Suppose that $F^5$ is given by either \eqref{dt1_f5} or \eqref{dt2_f5}.  Then
\begin{equation}\label{nid_f5_0}
 \abs{ \int_{\Sigma_s}  J(w \cdot \tau)F^5 } \ls \norm{w}_{H^1} \sqrt{\E}\sqrt{\D} 
\end{equation}
for every $w \in H^1(\Omega)$.
\end{prop}
\begin{proof}
As in the proof of Propositions \ref{nid_f1} and \ref{nid_f4}, we will only prove the result in the harder case of two temporal derivatives, which occurs when $F^5$ is given by \eqref{dt2_f5}.  Then $F^5$ consists only of two terms.  

Using the bounds in Theorems \ref{catalog_energy} and \ref{catalog_dissipation} (once more with the ordering convention described at the start of the proof of Proposition \ref{nid_f1}) together with the Sobolev embeddings and trace estimates, we estimate the first term in $F^5$ via 
\begin{multline}
 \abs{\int_{\Sigma_s} J (w\cdot \tau) (2 \mu \sg_{\dt \A} \dt u \nu \cdot \tau) } \ls \int_{\Sigma_s} \abs{w} \abs{\dt \A} \abs{\nab \dt u} \\
 \ls \norm{w}_{L^{1/\ep_-}(\Sigma_s)} \norm{\dt \bar{\eta}}_{W^{1,\infty}(\Sigma_s)}  \norm{\nab \dt u}_{L^{1/(1-\ep_-)}(\Sigma_s)} \ls \norm{w}_{H^1} \sqrt{\E} \sqrt{\D}.
\end{multline}
Similarly, we bound the second term via 
\begin{multline}
 \abs{\int_{\Sigma_s} J (w\cdot \tau) (\mu \sg_{\dt^2 \A} u \nu \cdot \tau) } 
 \ls \int_{\Sigma_s} \abs{w} \abs{\dt^2 \A} \abs{\nab u} \\
 \ls \norm{w}_{L^{1/(\ep_+ - \low)}(\Sigma_s)}  \left( \norm{\dt \bar{\eta}}_{W^{1,1/\low}(\Sigma_s)}  + \norm{\dt^2 \bar{\eta}}_{W^{1,1/\low}(\Sigma_s)} \right) \norm{\nab u}_{L^{1/(1-\ep_+)}(\Sigma_s)} \ls \norm{w}_{H^1} \sqrt{\D} \sqrt{\E}.
\end{multline}
These bounds can then be combined to conclude that the stated estimate holds. 
\end{proof}

We synthesize the results of Propositions \ref{nid_f1}, \ref{nid_f4}, and \ref{nid_f5} into the following result.

\begin{thm}\label{nid_v_est}
Consider the functional $H^1(\Omega) \ni w\mapsto \br{\mathcal{F},w} \in \mathbb{R}$ defined by 
\begin{equation}
\br{\mathcal{F},w} = \int_\Omega F^1 \cdot w J 
-  \int_{-\ell}^\ell  F^4 \cdot w 
- \int_{\Sigma_s}  J (w \cdot \tau)F^5, 
\end{equation} 
where $F^1,F^4,F^5$ are defined either by \eqref{dt1_f1}, \eqref{dt1_f4}, and \eqref{dt1_f5} or else by \eqref{dt2_f1}, \eqref{dt2_f4}, and \eqref{dt2_f5}.  Then 
\begin{equation}
\abs{\br{\mathcal{F},w}} \ls  \norm{w}_{H^1}  ( \sqrt{\E} + \E)\sqrt{\D}
\end{equation}
for all $w \in H^1(\Omega)$.
\end{thm}
\begin{proof}
This follows immediately from Propositions \ref{nid_f1}, \ref{nid_f4}, and \ref{nid_f5}. 
\end{proof}

\subsection{General interaction functional estimates II: pressure term}

We next turn our attention to the term $F^2$ appearing in Theorem \ref{linear_energy}.  We again derive a general dual estimate.

\begin{thm}\label{nid_p_est}
Suppose that $F^2$ is given by either \eqref{dt1_f2} or \eqref{dt2_f2}.  Then 
\begin{equation}\label{nid_p_est_0}
 \abs{ \int_\Omega J \psi F^2} \ls \norm{\psi}_{L^2} \sqrt{\D} \sqrt{\E}
\end{equation}
for every $\psi \in L^2(\Omega)$.
\end{thm}
\begin{proof}
Again, we will only prove the result in the harder case of two temporal derivatives, which occurs when $F^2$ is given by \eqref{dt2_f2}.  In this case $F^2$ only consists of two terms.  

From the bounds in Theorems \ref{catalog_energy} and \ref{catalog_dissipation} (again using the ordering convention described at the start of the proof of Proposition \ref{nid_f1}) together with the H\"{o}lder's inequality and the fact that $\Omega$ has finite measure, we bound the first term via 
\begin{multline}
\abs{\int_{\Omega} J \psi \diverge_{\dt^2 \A} u   } \ls \int_{\Omega} \abs{\psi} \abs{\dt^2 \A} \abs{\nab u} \\
\ls \norm{\psi}_{L^2}   \left(\norm{\dt \bar{\eta}}_{W^{1,2/\low}} +  \norm{\dt^2 \bar{\eta}}_{W^{1,2/\low}} \right) \norm{\nab u}_{L^{2/(1-\ep_+)}} \norm{1}_{L^{2/(\ep_+ - \low)}}
\ls \norm{\psi}_{L^2} \sqrt{\D} \sqrt{\E}.
\end{multline}
For the second term we argue similarly to estimate 
\begin{equation}
 \abs{\int_{\Omega} J \psi 2\diverge_{\dt \A}\dt u} \ls \int_{\Omega} \abs{\psi} \abs{\dt \A}\abs{\nab \dt u} \ls \norm{\psi}_{L^2} \norm{\dt \bar{\eta}}_{W^{1,\infty}} \norm{\nab \dt u}_{L^2} \ls \norm{\psi}_{L^2} \sqrt{\E} \sqrt{\D}.
\end{equation}
Upon combining these, we arrive at the stated bound.
\end{proof}

\subsection{Special interaction estimates I: velocity terms}

The $F^3$ nonlinear interaction term in Theorem \ref{linear_energy} requires greater care than we have used above.  Indeed, we will not derive general dual estimates, but will instead derive estimates that take careful advantage of the structure of the test function.  When two time derivatives are applied, the $F^3$ nonlinearity from \eqref{dt2_f3} has the form
\begin{equation}
 F^3 = \dt^2 [ \mathcal{R}(\p_1 \zeta_0,\p_1 \eta)] = \p_z \mathcal{R}(\p_1 \zeta_0,\p_1 \eta)  \p_1 \dt^2 \eta  + \p_z^2 \mathcal{R}(\p_1 \zeta_0,\p_1 \eta) (\p_1 \dt \eta)^2,
\end{equation}
where $\mathcal{R}$ is as in \eqref{R_def}.  For the purposes of estimating $F^3$ we will write
\begin{equation}
 \dt^2 u \cdot \N = \dt^3 \eta - F^6.  
\end{equation}
We may then decompose the relevant interaction term in Theorem \ref{linear_energy} as
\begin{multline}\label{nid_f3_decomp}
-  \int_{-\ell}^\ell \sigma  F^3   \p_1 (\dt^2 u \cdot \N)  = 
-  \int_{-\ell}^\ell \sigma \p_z \mathcal{R} (\p_1 \zeta_0,\p_1 \eta)  \p_1 \dt^2 \eta  \p_1 \dt^3 \eta
-  \int_{-\ell}^\ell \sigma \p_z \mathcal{R}(\p_1 \zeta_0,\p_1 \eta)  \p_1 \dt^2 \eta  \p_1 F^6 \\
-  \int_{-\ell}^\ell \sigma \p_z^2 \mathcal{R}(\p_1 \zeta_0,\p_1 \eta) (\p_1 \dt \eta)^2 \p_1 \dt^3 \eta 
-  \int_{-\ell}^\ell \sigma \p_z^2 \mathcal{R}(\p_1 \zeta_0,\p_1 \eta) (\p_1 \dt \eta)^2 \p_1 F^6 \\
=: I + II + III + IV.
\end{multline}

We will handle each of these separately, starting with $I$.

\begin{prop}\label{nid_f3_I}
Let $I$ be as in \eqref{nid_f3_decomp}.  Then we have the estimates 
\begin{equation}
 \abs{I +  \frac{d}{dt} \int_{-\ell}^\ell \sigma \p_z \mathcal{R}(\p_1 \zeta_0,\p_1 \eta)  \frac{\abs{\p_1 \dt^2 \eta}^2}{2} } \ls \sqrt{\E} \D
\end{equation}
and 
\begin{equation}
   \abs{ \int_{-\ell}^\ell \sigma \p_z \mathcal{R}(\p_1 \zeta_0,\p_1 \eta)  \frac{\abs{\p_1 \dt^2 \eta}^2}{2} } \ls \sqrt{\E} \seb.
\end{equation}
\end{prop}
\begin{proof}
The essential feature of $I$ is the appearance of the total time derivative $\p_1 \dt^2 \eta \p_1  \dt^3 \eta$, which allows us to write 
\begin{multline}
 I = -  \int_{-\ell}^\ell \sigma \p_z \mathcal{R} (\p_1 \zeta_0,\p_1 \eta) \dt \frac{\abs{\p_1 \dt^2 \eta}^2}{2} = -  \frac{d}{dt} \int_{-\ell}^\ell \sigma \p_z \mathcal{R}(\p_1 \zeta_0,\p_1 \eta)  \frac{\abs{\p_1 \dt^2 \eta}^2}{2} \\
 + \int_{-\ell}^\ell \sigma \p_z^2 \mathcal{R}(\p_1 \zeta_0,\p_1 \eta) \p_1 \dt \eta  \frac{\abs{\p_1 \dt^2 \eta}^2}{2}.
\end{multline}
Using Theorems \ref{catalog_energy} and \ref{catalog_dissipation} (with the ordering convention described in the proof of Proposition \ref{nid_f1}) in conjunction with Proposition \ref{R_prop}, we then estimate 
\begin{equation}
\abs{\int_{-\ell}^\ell \sigma \p_z^2 \mathcal{R}(\p_1 \zeta_0,\p_1 \eta) \p_1 \dt \eta  \frac{\abs{\p_1 \dt^2 \eta}^2}{2}} \ls \int_{-\ell}^\ell \abs{\p_1 \dt \eta} \abs{\p_1 \dt^2 \eta}^2 \ls \norm{\p_1 \dt \eta}_{L^\infty} \ns{\p_1 \dt^2 \eta}_{L^{2}} \ls \sqrt{\E} \D
\end{equation}
and 
\begin{equation}
   \abs{ \int_{-\ell}^\ell \sigma \p_z \mathcal{R}(\p_1 \zeta_0,\p_1 \eta)  \frac{\abs{\p_1 \dt^2 \eta}^2}{2} } \ls \int_{-\ell}^\ell \abs{\p_1 \eta} \abs{\p_1 \dt^2 \eta}^2 \ls \norm{\p_1 \eta}_{L^\infty} \ns{\p_1 \dt^2 \eta}_{L^{2}} \ls \sqrt{\E} \seb.
\end{equation}
These are the stated bounds.

\end{proof}

Next we deal with the term $II$.

\begin{prop}\label{nid_f3_II}
 Let $II$ be as given in \eqref{nid_f3_decomp}.  Then
\begin{equation}
 \abs{II} \ls \E \D.
\end{equation}
\end{prop}
\begin{proof}
We begin by writing $F^6 = -2 \dt u_1 \p_1 \dt \eta - u_1 \p_1 \dt^2 \eta$ in order to expand
\begin{multline}
II =   \int_{-\ell}^\ell \sigma \frac{\p_z \mathcal{R}(\p_1 \zeta_0,\p_1 \eta)}{\p_1 \eta} \p_1 \eta  \p_1 \dt^2 \eta  2 \p_1  \dt u_1 \p_1 \dt \eta 
+  \int_{-\ell}^\ell \sigma \frac{\p_z \mathcal{R}(\p_1 \zeta_0,\p_1 \eta)}{\p_1 \eta} \p_1 \eta  \p_1 \dt^2 \eta  2   \dt u_1 \p_1^2 \dt \eta  \\
+ \int_{-\ell}^\ell \sigma \frac{\p_z \mathcal{R}(\p_1 \zeta_0,\p_1 \eta)}{\p_1 \eta} \p_1 \eta  \p_1 \dt^2 \eta  \p_1 u_1 \p_1 \dt^2 \eta 
+ \int_{-\ell}^\ell \sigma \p_z \mathcal{R}(\p_1 \zeta_0,\p_1 \eta)   \p_1 \dt^2 \eta  u_1 \p_1^2 \dt^2 \eta 
\\=: II_1 + II_2 + II_3 + II_4.
\end{multline}
To estimate these terms we will use Theorems \ref{catalog_energy} and \ref{catalog_dissipation} (with the ordering convention described in the proof of Proposition \ref{nid_f1}), Proposition \ref{R_prop}, the ordering $0 <2 \low < \ep_- < \ep_+$ assumed in \eqref{kappa_ep_def}, and H\"{o}lder's inequality.  This yields the bounds 
\begin{multline}
\abs{II_1} \ls \int_{-\ell}^\ell \abs{\p_1 \eta}  \abs{\p_1 \dt^2 \eta} \abs{ \p_1  \dt u} \abs{ \p_1 \dt \eta} 
\ls \norm{\p_1 \eta}_{L^\infty} \norm{\p_1 \dt^2 \eta}_{L^{1/\low}} \norm{\p_1 \dt u}_{L^{1/(1-\ep_-)}(\Sigma)} \norm{\p_1 \dt \eta}_{L^\infty} \\
\ls \sqrt{\E} \D \sqrt{\E},
\end{multline}
\begin{multline}
 \abs{II_2} \ls \int_{-\ell}^\ell \abs{\p_1 \eta} \abs{\p_1 \dt^2 \eta} \abs{\dt u}\abs{\p_1^2 \dt \eta} \ls \norm{\p_1 \eta}_{L^\infty} \norm{\p_1 \dt^2 \eta}_{L^{1/\low}} \norm{\dt u}_{L^\infty(\Sigma)} \norm{\p_1^2 \dt \eta}_{L^{1/(1-\ep_-)}} \\
\ls \sqrt{\E} \sqrt{\D} \sqrt{\E} \sqrt{\D}, 
\end{multline}
and
\begin{equation}
\abs{III_3} \ls \int_{-\ell}^\ell \abs{\p_1 \eta} \abs{\p_1 \dt^2 \eta}^2 \abs{\p_1 u}
\ls \norm{\p_1 \eta}_{L^\infty} \ns{\p_1 \dt^2 \eta}_{L^{1/\low}} \norm{\p_1 u}_{L^{1/(1-\ep_+)}(\Sigma)} \ls \sqrt{\E} \D \sqrt{\E}.
\end{equation}
For $II_4$ we note that $\p_1 \dt^2 \eta \p_1^2 \dt^2 \eta$ is a total derivative, so we can integrate by parts and use the fact that $u_1 =0$ at the endpoints to see that 
\begin{multline} 
 \abs{II_4} = \abs{\int_{-\ell}^\ell \sigma \p_z \mathcal{R}(\p_1 \zeta_0,\p_1 \eta)  u_1  \p_1 \frac{\abs{ \p_1 \dt^2 \eta}^2}{2} } = 
 \abs{- \int_{-\ell}^\ell \sigma \p_1 \left[ \p_z \mathcal{R}(\p_1 \zeta_0,\p_1 \eta)  u_1 \right]  \frac{\abs{ \p_1 \dt^2 \eta}^2}{2} }
 \\
= \abs{ \int_{-\ell}^\ell  \sigma \left[  \frac{\p_z \mathcal{R}(\p_1 \zeta_0,\p_1 \eta)}{\p_1 \eta} \p_1 \eta \p_1  u_1 +  \p_z^2 \mathcal{R}(\p_1 \zeta_0,\p_1 \eta) \p_1^2 \eta u_1 + \frac{\p_y \p_z \mathcal{R}(\p_1 \zeta_0,\p_1 \eta)\p_1^2 \zeta_0}{\p_1\eta} \p_1 \eta   u_1\right]  \frac{\abs{ \p_1 \dt^2 \eta}^2}{2} }.
\end{multline}
We may then use the same tools listed above to estimate 
\begin{multline}
\abs{II_4} \ls \int_{-\ell}^\ell \abs{\p_1 \dt^2 \eta}^2 \left( \abs{\p_1 \eta}\abs{\p_1 u} + \abs{\p_1^2 \eta} \abs{u} + \abs{\p_1 \eta}\abs{u} \right)  \\
\ls  \ns{\p_1 \dt^2 \eta}_{L^{1/\low}} \left( \norm{\p_1 \eta}_{L^\infty} \norm{\p_1 u}_{L^{1/(1-\ep_+)(\Sigma)}} + \norm{\p_1^2 \eta}_{L^{1/(1-\ep_+)}} \norm{u}_{L^\infty(\Sigma)}  + \norm{\p_1 \eta}_{L^\infty} \norm{u}_{L^\infty(\Sigma)} \right) \\
\ls \D \E.
\end{multline}
Combining these bounds then yields the stated estimate.
\end{proof}

The term $III$ is next.

\begin{prop}\label{nid_f3_III}
Let $III$ be as given in \eqref{nid_f3_decomp}.  Then we have the bounds
\begin{equation}
 \abs{III + \frac{d}{dt} \int_{-\ell}^\ell \sigma \p_z^2 \mathcal{R}(\p_1 \zeta_0,\p_1 \eta) (\p_1 \dt \eta)^2 \p_1 \dt^2 \eta}
\ls (\sqrt{\E} + \E)\D
\end{equation}
and 
\begin{equation}
  \abs{\int_{-\ell}^\ell \sigma \p_z^2 \mathcal{R}(\p_1 \zeta_0,\p_1 \eta) (\p_1 \dt \eta)^2 \p_1 \dt^2 \eta } \ls  \sqrt{\E}\seb.
\end{equation}
\end{prop}
\begin{proof}
To handle the term $III$ we begin by pulling a time derivative out of the integral:
\begin{multline}
III = -  \int_{-\ell}^\ell \sigma \p_z^2 \mathcal{R}(\p_1 \zeta_0,\p_1 \eta) (\p_1 \dt \eta)^2 \p_1 \dt^3 \eta  =  -  \frac{d}{dt} \int_{-\ell}^\ell \sigma \p_z^2 \mathcal{R}(\p_1 \zeta_0,\p_1 \eta) (\p_1 \dt \eta)^2 \p_1 \dt^2 \eta \\
 + \int_{-\ell}^\ell \sigma \p_z^3 \mathcal{R}(\p_1 \zeta_0,\p_1 \eta) (\p_1 \dt \eta)^3 \p_1 \dt^2 \eta 
 + \int_{-\ell}^\ell \sigma \p_z^2 \mathcal{R}(\p_1 \zeta_0,\p_1 \eta) 2 \p_1 \dt \eta  \abs{\p_1 \dt^2 \eta}^2.
\end{multline} 
We then employ Theorems \ref{catalog_energy} and \ref{catalog_dissipation} (with the ordering convention described in the proof of Proposition \ref{nid_f1}) and  Proposition \ref{R_prop} to bound 
\begin{multline}
\abs{\int_{-\ell}^\ell \sigma \p_z^3 \mathcal{R}(\p_1 \zeta_0,\p_1 \eta) (\p_1 \dt \eta)^3 \p_1 \dt^2 \eta} \ls \int_{-\ell}^\ell \abs{\p_1 \dt \eta}^3 \abs{\p_1 \dt^2\eta} \ls \ns{\p_1 \dt\eta}_{L^\infty} \norm{\p_1 \dt\eta}_{L^\infty} \norm{\p_1 \dt^2 \eta}_{L^{1/\low}} \\
\ls \E \D
\end{multline}
and
\begin{equation}
\abs{\int_{-\ell}^\ell \sigma \p_z^2 \mathcal{R}(\p_1 \zeta_0,\p_1 \eta) 2 \p_1 \dt \eta  \abs{\p_1 \dt^2 \eta}^2} \ls \int_{-\ell}^\ell \abs{\p_1 \dt \eta} \abs{\p_1 \dt^2 \eta}^2 \ls \norm{\p_1 \dt \eta}_{L^\infty} \ns{ \p_1 \dt^2 \eta}_{L^{1/\low}} \ls \sqrt{\E} \D.
\end{equation}
Upon combining these, we arrive at the first stated estimate. 

To derive the second estimate we first note that standard Sobolev embeddings and interpolation show that if $\psi \in H^{3/2}((-\ell,\ell))$ then 
\begin{equation}
 \norm{\p_1 \psi}_{L^4} \ls \norm{\p_1 \psi}_{H^{1/4}} \ls \norm{\psi}_{H^{5/4}} \ls \norm{\psi}_{H^1}^{1/2} \norm{\psi}_{H^{3/2}}^{1/2}.
\end{equation}
Applying this with $\psi = \dt \eta$ and again using Theorem \ref{catalog_energy}  and  Proposition \ref{R_prop} with the definitions of $\E$ (see \eqref{E_def}) and $\seb$ (see \eqref{ED_parallel}), we bound
\begin{multline}
 \abs{\int_{-\ell}^\ell \sigma \p_z^2 \mathcal{R}(\p_1 \zeta_0,\p_1 \eta) (\p_1 \dt \eta)^2 \p_1 \dt^2 \eta } \ls \int_{-\ell}^\ell \abs{\p_1 \dt \eta}^2 \abs{\p_1 \dt^2 \eta} \ls \ns{\p_1 \dt \eta}_{L^4} \norm{\p_1 \dt^2 \eta}_{L^2}  \\
\ls \norm{\dt \eta}_{H^1} \norm{\dt \eta}_{H^{3/2}}\norm{\dt^2 \eta}_{H^1} \ls \sqrt{\seb} \sqrt{\E} \sqrt{\seb},
\end{multline}
which is the second stated estimate.
\end{proof}

Finally, we handle the term $IV$.

\begin{prop}\label{nid_f3_IV}
Let $IV$ be as given in \eqref{nid_f3_decomp}.  Then
\begin{equation} 
 \abs{IV} \ls (\E + \E^{3/2})\D.
\end{equation}
\end{prop}
\begin{proof}
We first write $F^6= -2 \dt u_1 \p_1 \dt \eta - u_1 \p_1 \dt^2 \eta$ in order to split
\begin{multline}
IV =   \int_{-\ell}^\ell \sigma \p_z^2 \mathcal{R}(\p_1 \zeta_0,\p_1 \eta) (\p_1 \dt \eta)^2    \p_1 (2 \dt u_1 \p_1 \dt \eta) \\
+  \int_{-\ell}^\ell \sigma \p_z^2 \mathcal{R}(\p_1 \zeta_0,\p_1 \eta) (\p_1 \dt \eta)^2  \p_1(u_1 \p_1 \dt^2 \eta) =: IV_1 + IV_2.
\end{multline}

Then Theorems \ref{catalog_energy} and \ref{catalog_dissipation} (with the ordering convention described in the proof of Proposition \ref{nid_f1}) and  Proposition \ref{R_prop} allow us to estimate 
\begin{multline}
\abs{IV_1} \ls \int_{-\ell}^\ell \abs{\p_1 \dt \eta}^2 \left( \abs{\p_1 \dt u} \abs{\p_1 \dt \eta} + \abs{\dt u} \abs{\p_1^2 \dt \eta}    \right) \\
\ls \ns{\p_1 \dt \eta}_{L^\infty} \left( \norm{\p_1 \dt u}_{L^{1/(1-\ep_-)}(\Sigma)} \norm{\p_1 \dt \eta}_{L^\infty} + \norm{\dt u}_{L^\infty(\Sigma)} \norm{\p_1^2 \dt \eta}_{L^{1/(1-\ep_-)}}  \right) \ls \E \D.
\end{multline}

To handle $IV_2$ we further expand 
\begin{multline}
IV_2 = \int_{-\ell}^\ell \sigma \p_z^2 \mathcal{R}(\p_1 \zeta_0,\p_1 \eta) (\p_1 \dt \eta)^2  \p_1 u_1 \p_1 \dt^2 \eta 
+ \int_{-\ell}^\ell \sigma \p_z^2 \mathcal{R}(\p_1 \zeta_0,\p_1 \eta) (\p_1 \dt \eta)^2  u_1 \p_1^2 \dt^2 \eta  =: IV_{3} + IV_4.
\end{multline}
The term $IV_3$ can be estimated as $IV_1$ was, recalling from \eqref{kappa_ep_def} that $\low < \ep_- < \ep_+$:
\begin{equation}
 \abs{IV_3} \ls \int_{-\ell}^\ell \abs{\p_1 \dt \eta}^2  \abs{\p_1 u } \abs{\p_1 \dt^2 \eta} \ls \ns{\p_1 \dt\eta}_{L^\infty} \norm{\p_1 u}_{L^{1/(1-\ep_+)(\Sigma)}} \norm{\p_1 \dt^2 \eta}_{L^{1/\low}} \ls \E \D.
\end{equation}
On the other hand, for $IV_4$ we need to integrate by parts again, using the fact that $u_1$ vanishes at the endpoints:
\begin{multline}
  IV_4 = -\int_{-\ell}^\ell \sigma \left[ \p_z^2 \mathcal{R}(\p_1 \zeta_0,\p_1 \eta) (\p_1 \dt \eta)^2  \p_1 u_1 + 2 \p_z^2 \mathcal{R}(\p_1 \zeta_0,\p_1 \eta) \p_1 \dt \eta \p_1^2 \dt \eta u_1 \right] \p_1 \dt^2 \eta  \\
-\int_{-\ell}^\ell \sigma \left[ \p_z^3 \mathcal{R}(\p_1 \zeta_0,\p_1 \eta) \p_1^2 \eta  + \p_y \p_z^2 \mathcal{R}(\p_1 \zeta_0,\p_1 \eta)  \p_1^2 \zeta_0 \right] (\p_1 \dt \eta)^2  u_1   \p_1 \dt^2 \eta =: IV_5 + IV_6.
\end{multline}
These terms can then be estimated as above:
\begin{multline}
\abs{IV_5} \ls \int_{-\ell}^\ell \left( \abs{\p_1 \dt \eta}^2 \abs{\p_1 u} + \abs{\p_1 \dt \eta} \abs{\p_1^2 \dt \eta} \abs{u} \right) \abs{\p_1 \dt^2 \eta} \\
\ls 
\left(\ns{\p_1 \dt \eta}_{L^\infty} \norm{\p_1 u}_{L^{1/(1-\ep_+)}(\Sigma)} + \norm{\p_1 \dt \eta}_{L^\infty} \norm{\p_1^2 \dt \eta}_{L^{1/(1-\ep_-)}} \norm{u}_{L^\infty(\Sigma)}  \right) \norm{\p_1 \dt^2 \eta}_{L^{1/\low}} \\
\ls \left(\E \sqrt{\D} + \sqrt{\E} \sqrt{\D} \sqrt{\E}   \right) \sqrt{\D}
\end{multline}
and 
\begin{multline}
 \abs{IV_6} \ls \int_{-\ell}^\ell (\abs{\p_1^2 \eta} +1) \abs{\p_1 \dt \eta}^2 \abs{u} \abs{\p_1 \dt^2 \eta} \\
\le \left(\norm{\p_1^2 \eta}_{L^{1/(1-\ep_+)}}  +1 \right) \ns{\p_1 \dt\eta}_{L^\infty} \norm{u}_{L^\infty(\Sigma)} \norm{\p_1 \dt^2 \eta}_{L^{1/\low}} \ls \left(\sqrt{\E} +1 \right) \E \D .
\end{multline}
The stated estimate then follows by combining all of these.
\end{proof}

Now that we have controlled $I$--$IV$ in \eqref{nid_f3_decomp} we can record a unified estimate.

\begin{thm}\label{nid_f3_dt2}
Let $F^3$ be given by \eqref{dt2_f3}.  Then we have the bounds
\begin{multline}
\abs{ -  \int_{-\ell}^\ell \sigma  F^3   \p_1 (\dt^2 u \cdot \N) +  \frac{d}{dt} \int_{-\ell}^\ell \left[ \sigma \p_z \mathcal{R}(\p_1 \zeta_0,\p_1 \eta)  \frac{\abs{\p_1 \dt^2 \eta}^2}{2} + \sigma \p_z^2 \mathcal{R}(\p_1 \zeta_0,\p_1 \eta) (\p_1 \dt \eta)^2 \p_1 \dt^2 \eta \right] } 
\\
\ls (\sqrt{\E} + \E + \E^{3/2})\D
\end{multline}
and
\begin{equation}
 \abs{\int_{-\ell}^\ell \left[ \sigma \p_z \mathcal{R}(\p_1 \zeta_0,\p_1 \eta)  \frac{\abs{\p_1 \dt^2 \eta}^2}{2} + \sigma \p_z^2 \mathcal{R}(\p_1 \zeta_0,\p_1 \eta) (\p_1 \dt \eta)^2 \p_1 \dt^2 \eta \right]} \ls \sqrt{\E} \seb.
\end{equation}
\end{thm}
\begin{proof}
The results follows from combining \eqref{nid_f3_decomp} with Propositions \ref{nid_f3_I}, \ref{nid_f3_II}, \ref{nid_f3_III}, and  \ref{nid_f3_IV}.
\end{proof}

A similar and  simpler result holds for $F^3$ when only one time derivative is applied as in \eqref{dt1_f3}.   We will record it without proof.

\begin{thm}\label{nid_f3_dt}
Let $F^3$ be given by \eqref{dt1_f3}.  Then 
\begin{equation}
\abs{ -  \int_{-\ell}^\ell \sigma  F^3   \p_1 (\dt^2 u \cdot \N)  } 
\ls (\sqrt{\E} + \E)\D.
\end{equation}
\end{thm}

\subsection{Special interaction estimates II: free surface terms}

The term involving $F^6$ and $F^7$ in Theorem \ref{linear_energy} also require a delicate treatment.  We record these now, starting with $F^6$.

\begin{thm}\label{nid_f6}
We have the estimate
\begin{equation}
 \abs{   \int_{-\ell}^\ell   g \dt^2 \eta F^6 +  \sigma \frac{\p_1 \dt^2 \eta \p_1 F^6}{(1+\abs{\p_1 \zeta_0}^2)^{3/2}}} \ls \sqrt{\E} \D
\end{equation}
when $F^6$ is given by \eqref{dt2_f6}, and we have the estimate
\begin{equation}
 \abs{   \int_{-\ell}^\ell   g \dt \eta F^6 +  \sigma \frac{\p_1 \dt \eta \p_1 F^6}{(1+\abs{\p_1 \zeta_0}^2)^{3/2}}} \ls \sqrt{\E} \D
\end{equation}
when $F^6$ is given by \eqref{dt1_f6}.
\end{thm}
\begin{proof}
We begin by using the definition of $F^6$ in \eqref{dt2_f6} to split
\begin{multline}
   \int_{-\ell}^\ell   g \dt^2 \eta F^6 +  \sigma \frac{\p_1 \dt^2 \eta \p_1 F^6}{(1+\abs{\p_1 \zeta_0}^2)^{3/2}} 
=    \int_{-\ell}^\ell   g \dt^2 \eta (-2 \dt u_1 \p_1 \dt \eta)  + g \dt^2 \eta (- u_1 \p_1 \dt^2 \eta) \\
   + \int_{-\ell}^\ell \sigma \frac{\p_1 \dt^2 \eta \p_1 (-2 \dt u_1 \p_1 \dt \eta)}{(1+\abs{\p_1 \zeta_0}^2)^{3/2}}    +  \int_{-\ell}^\ell\sigma \frac{\p_1 \dt^2 \eta \p_1 (- u_1 \p_1 \dt^2 \eta)}{(1+\abs{\p_1 \zeta_0}^2)^{3/2}}
   =: I + II +III.
\end{multline}
We will estimate these three terms using Theorems \ref{catalog_energy} and \ref{catalog_dissipation} (with the ordering convention described in the proof of Proposition \ref{nid_f1}) and H\"older's inequality.  For $I$ we directly estimate 
\begin{multline}
\abs{I} \ls \int_{-\ell}^\ell \abs{\dt^2 \eta} \left( \abs{\dt u} \abs{\p_1 \dt \eta} + \abs{u} \abs{\p_1 \dt^2 \eta} \right) \ls \norm{\dt^2 \eta}_{L^2} \left(\norm{\dt u}_{L^\infty(\Sigma)} \norm{\p_1 \dt \eta}_{L^2} + \norm{u}_{L^\infty(\Sigma)} \norm{\p_1 \dt^2 \eta}_{L^2}   \right) \\
\ls \sqrt{\D} \left(\sqrt{\E} \sqrt{\D} + \sqrt{\E} \sqrt{\D}\right).
\end{multline}
Similarly, for $II$ we apply the product rule and estimate (recalling \eqref{kappa_ep_def})
\begin{multline}
\abs{II} \ls \int_{-\ell}^\ell \abs{\p_1 \dt^2 \eta} \left( \abs{\p_1 \dt u} \abs{\p_1 \dt \eta} + \abs{\dt u} \abs{\p_1^2 \dt \eta}   \right)  \\
\ls \norm{\p_1 \dt^2 \eta}_{L^{1/\low}} \left( \norm{\p_1 \dt u}_{L^{1/(1-\ep_-)}(\Sigma)} \norm{\p_1 \dt \eta}_{L^\infty}  + \norm{\dt u}_{L^\infty(\Sigma)} \norm{\p_1^2 \dt \eta}_{L^{1/(1-\ep_-)}}  \right) \\
\ls \sqrt{\D}\left( \sqrt{\D} \sqrt{\E} + \sqrt{\E} \sqrt{\D}\right).
\end{multline}
On the other hand, for $III$ we expand with the product rule and then integrate by parts and exploit the vanishing of $u_1$ at the endpoints:
\begin{multline}
III = -\int_{-\ell}^\ell\sigma \frac{\p_1 \dt^2 \eta \p_1  u_1 \p_1 \dt^2 \eta}{(1+\abs{\p_1 \zeta_0}^2)^{3/2}} + 
\int_{-\ell}^\ell   \sigma \p_1 \left(\frac{u_1  }{(1+\abs{\p_1 \zeta_0}^2)^{3/2}} \right)  \frac{\abs{ \p_1 \dt^2 \eta }^2}{2} \\
= \frac{\sigma}{2} \int_{-\ell}^\ell \abs{\p_1 \dt^2 \eta}^2\left( u_1  \p_1 \left(\frac{1 }{(1+\abs{\p_1 \zeta_0}^2)^{3/2}} \right)  - \frac{\p_1 u_1  }{(1+\abs{\p_1 \zeta_0}^2)^{3/2}}     \right).
\end{multline}
In this form we can estimate with the same tools as above, crucially using that $2 \low < \ep_+$, to see that 
\begin{equation}
\abs{III} \ls \int_{-\ell}^\ell \abs{\p_1 \dt^2 \eta}^2 \left( \abs{u} + \abs{\p_1 u} \right) \ls \ns{\p_1 \dt^2 \eta}_{L^{1/\low}} \norm{ u}_{W^{1/(1-\ep_+)}(\Sigma)}  \ls \D \sqrt{\E}.
\end{equation}
Combining these then provides the stated bound. 
\end{proof}

Next we record the $F^7$ bound.

\begin{thm}\label{nid_f7}
We have the estimate
\begin{equation} 
\abs{ [\dt^2 u\cdot \N,F^7]_\ell} \ls \sqrt{\E} \D
\end{equation}
when $F^7$ is given by \eqref{dt2_f7}, and we have the estimate
\begin{equation} 
\abs{ [\dt u \cdot \N,F^7]_\ell} \ls \sqrt{\E} \D
\end{equation}
when $F^7$ is given by \eqref{dt1_f7}.
\end{thm}
\begin{proof}
Once more we only record the proof in the harder case when $F^7$ is given by \eqref{dt2_f7}.  We begin by estimating 
\begin{equation}
 \abs{F^7} \ls \abs{\swh'(\dt \eta)} \abs{\dt^3 \eta} + \abs{\swh''(\dt \eta)} \abs{\dt^2 \eta}^2.
\end{equation}
Since $\norm{\dt \eta}_{L^\infty} \ls \sqrt{\E} \ls 1$, we can bound 
\begin{equation}
 \abs{\swh'(z)} = \frac{1}{\low}\abs{ \int_0^z \sw''(r)dr}\ls \abs{z} \text{ for } z \in [-\norm{\dt \eta}_{L^\infty} ,\norm{\dt \eta}_{L^\infty} ].
\end{equation}
From this, basic trace theory, and the bound $\sum_{k=1}^3\max_{\pm \ell} \abs{\dt^k \eta} \ls \sqrt{\D}$ we then estimate
\begin{equation}
 \max_{\pm \ell} \abs{F_7} \ls \max_{\pm \ell} \left( \abs{\dt \eta} \abs{\dt^3 \eta} + \abs{\dt^2 \eta}^2  \right) \ls \sqrt{\D} \left( \norm{\dt \eta}_{H^1} + \norm{\dt^2 \eta}_{H^1}  \right)  \ls \sqrt{\D} \sqrt{\E}.
\end{equation}
From this and the fact that  $[\dt^2 u \cdot \N]_\ell = [\dt^3 \eta ]_\ell \ls \sqrt{\D}$, we deduce that
\begin{equation}
 \abs{  [\dt^2 u\cdot \N,F^7]_\ell } \ls \sqrt{\E} \sqrt{\D} [\dt^2 u \cdot \N]_\ell \ls \sqrt{\E} \D,
\end{equation}
which is the stated estimate.
\end{proof}

We conclude with two more estimates involving the free surface function.  The first is for a term involving the function $\Q$ from \eqref{Q_def} that appears in Corollary \ref{basic_energy}.

\begin{thm}\label{nid_Q}
 Let $\Q$ be the smooth function defined by \eqref{Q_def}.  Then 
\begin{equation}
 \abs{\int_{-\ell}^\ell \sigma \Q(\p_1 \zeta_0, \p_1 \eta)  } \ls  \sqrt{\E} \ns{\eta}_{H^1}
\end{equation}
\end{thm}
\begin{proof}
According to Proposition \ref{R_prop} and Theorem \ref{catalog_energy}, we have that 
\begin{equation}
 \abs{\int_{-\ell}^\ell \sigma \Q(\p_1 \zeta_0, \p_1 \eta)  } \ls \int_{-\ell}^\ell \abs{\p_1 \eta}^3 \ls \norm{\p_1 \eta}_{L^\infty} \ns{\eta}_{H^1} \ls \sqrt{\E} \ns{\eta}_{H^1} .
\end{equation}
This is the stated estimate.
\end{proof} 
 
Our final estimate involves the term $\swh$, as defined in \eqref{V_pert}.

\begin{thm}\label{nid_W}
We have that 
\begin{equation} 
\abs{  [u\cdot \N,\swh(\dt \eta)]_\ell } \ls \norm{\dt \eta}_{H^1} \bs{u\cdot \N}.
\end{equation}
\end{thm}
\begin{proof}
The definition of $\swh \in C^2$ in \eqref{V_pert} shows that $\abs{\swh(z)} \ls z^2$ for $\abs{z} \ls 1$.  Since $\dt \eta = u\cdot \N$ at $\pm \ell$ we can use standard trace theory to deduce the stated bound:  
\begin{equation}
 \abs{  [u\cdot \N,\swh(\dt \eta)]_\ell } \ls \sum_{a=\pm 1} \low \abs{u\cdot \N(a\ell,t)}^2 \abs{\dt \eta(a \ell,t)}  \ls \norm{\dt \eta}_{H^1} [u\cdot \N]_\ell^2.  
\end{equation}

\end{proof}

\section{Nonlinear estimates II: interaction terms, energetic form}\label{sec_nl_int_e}

In this section we continue our study of the nonlinear interaction terms appearing in  Theorem \ref{linear_energy}.  However, the focus now is on estimates in terms of the energy functional.  Once more, in order to avoid tedious restatements of the same hypothesis, we assume throughout this section that a solution to \eqref{ns_geometric} exists on the time horizon $(0,T)$ for $0 < T \le \infty$ and obeys the small-energy estimate 
\begin{equation}
 \sup_{0\le t < T} \E(t) \le \gamma^2 < 1,
\end{equation}
where $\gamma \in (0,1)$ is as in Lemma \ref{eta_small}.  Again, this means that the estimates of Lemma \ref{eta_small} are available for use, and we will use them often without explicit reference.

\subsection{General interaction functional estimates }

We begin by deriving general dual estimates for the terms involving $F^1$, $F^4$, and $F^5$ in terms of the energy functional.  First we consider $F^1$.

\begin{prop}\label{nie_f1}
Suppose that $F^1$ is as defined by \eqref{dt1_f1}.  Then 
\begin{equation}
\abs{\int_\Omega J w \cdot F^1} \ls \norm{w}_{H^1} \left(\E + \E^{3/2} \right) .
\end{equation}
\end{prop}
\begin{proof}
The terms $F^1$, as defined by \eqref{dt1_f1}, contains six separate terms.  We will estimate each of these, employing H\"{o}lder's inequality and Theorem \ref{catalog_energy} repeatedly and without explicit reference. 

\textbf{Term: $-\diverge_{\dt \A} S_\A(p,u)$. } We first bound
\begin{equation}
\abs{\int_\Omega J w \cdot (-\diverge_{\dt \A} S_\A(p,u))  }  \ls \int_\Omega \abs{w}\abs{\dt \A} \left(\abs{\nab p} +\abs{\nab^2 u} \right) +\int_{\Omega} \abs{w} \abs{\dt \A} \abs{\nab \A}\left(\abs{p} + \abs{\nab u} \right) =: I + II.
\end{equation}
We then estimate 
\begin{equation}
I \ls \norm{w}_{L^{2/\ep_+}} \norm{\dt \bar{\eta}}_{W^{1,\infty}} \left( \norm{\nab p}_{L^{2/(2-\ep_+)}} + \norm{\nab^2 u}_{L^{2/(2-\ep_+)}} \right) \ls \norm{w}_{H^1}  \E,
\end{equation}
and 
\begin{equation}
 II \ls \norm{w}_{L^{1/\ep_+}} \norm{\dt \bar{\eta}}_{W^{1,\infty}} \norm{\bar{\eta}}_{W^{2,2/(1-\ep_+)}}   \left( \norm{p}_{L^{2/(1-\ep_+)}} + \norm{\nab u}_{L^{2/(1-\ep_+)}} \right) \ls \norm{w}_{H^1} \E^{3/2}.
\end{equation}
The combination of these estimates shows this term can be estimated as stated.

\textbf{Term: $\mu \diva \sg_{\dt \A} u$. }  We first bound
\begin{equation}
 \abs{\int_\Omega J w \cdot (\mu \diva \sg_{\dt \A} u)  } \ls \int_\Omega \abs{w} \abs{\dt \A} \abs{\nab^2 u}  + \int_\Omega \abs{w} \abs{\nab \dt \A} \abs{\nab u} =: I + II.
\end{equation}
For $I$ we bound
\begin{equation}
 I \ls \norm{w}_{L^{2/\ep_+}} \norm{\dt \bar{\eta}}_{W^{1,\infty}} \norm{\nab^2 u}_{L^{2/(2-\ep_+)}} \ls \norm{w}_{H^1} \E.
\end{equation}
For $II$ we use \eqref{kappa_ep_def} to see that $0 < \ep_- - \low < 1$ and $0 < 2\ep_+ + \ep_- - \low < 2$, so 
\begin{equation}
 II \ls \norm{w}_{L^{4/(2\ep_+ + \ep_- - \low)}}  \left( \norm{\bar{\eta}}_{W^{2,4/(2-(\ep_--\low))}} +  \norm{ \dt \bar{\eta}}_{W^{2,4/(2-(\ep_--\low))}} \right) \norm{\nab u}_{L^{2/(1-\ep_+)}} 
 \ls \norm{w}_{H^1} \E.
\end{equation}
Combining these then shows that this term can be estimated as stated.

\textbf{Term: $- u \cdot \nab_{\dt \A} u$. } We bound 
\begin{equation}
\abs{\int_\Omega Jw \cdot (- u \cdot \nab_{\dt \A} u) } \ls  \int_{\Omega} \abs{w} \abs{u} \abs{\dt \A} \abs{\nab u} 
\ls \norm{w}_{L^{2}} \norm{u}_{L^\infty} \norm{\dt \bar{\eta}}_{W^{1,\infty}} \norm{\nab u}_{L^{2}}
\ls \norm{w}_{H^1} \E^{3/2}.
\end{equation}

\textbf{Term: $- \dt u \cdot \naba   u$. }  We estimate 
\begin{multline}
 \abs{\int_\Omega Jw \cdot (- \dt u \cdot \naba   u)} \ls \int_{\Omega} \abs{w}\abs{\dt u}\abs{\nab u} \\
 \ls \norm{w}_{L^{4/(2\ep_+ + \ep_-)}} \norm{\dt u}_{L^{4/(2-\ep_-)}} \norm{\nab u}_{L^{2/(1-\ep_+)}}
\ls \norm{w}_{H^1} \E.
\end{multline}

\textbf{Term: $\dt^2 \bar{\eta} \frac{\phi}{\zeta_0} K \p_2 u$. }  We bound 
\begin{equation}
\abs{\int_\Omega J w \cdot (\dt^2 \bar{\eta} \frac{\phi}{\zeta_0} K \p_2 u) } \ls \int_\Omega \abs{w} \abs{\dt^2 \bar{\eta}} \abs{\nab u} \ls \norm{w}_{L^{4}} \norm{\dt^2 \bar{\eta}}_{L^4} \norm{\nab u}_{L^{2}} \ls \norm{w}_{H^1} \E.
\end{equation}

\textbf{Term: $\dt \bar{\eta} \frac{\phi}{\zeta_0} \dt K \p_2 u$. }  We bound 
\begin{equation}
\abs{\int_\Omega J w \cdot (\dt \bar{\eta} \frac{\phi}{\zeta_0} \dt K \p_2 u)} \ls \int_\Omega \abs{w} \abs{\dt \bar{\eta}} \abs{\dt K} \abs{\nab u} \ls \norm{w}_{L^{2}} \norm{\dt \bar{\eta}}_{L^\infty}\norm{\dt \bar{\eta}}_{W^{1,\infty}} \norm{\nab u}_{L^2} \ls \norm{w}_{H^1} \E^{3/2}.
\end{equation}

\end{proof}

Our next result concerns energetic estimates for $F^4$.

\begin{prop}\label{nie_f4}
Suppose that $F^4$ is defined by \eqref{dt1_f4}.  Then we 
\begin{equation} 
\abs{ \int_{-\ell}^\ell w \cdot  F^4  } \ls \norm{w}_{H^1} (\E + \E^{3/2})
\end{equation}
for all $w \in H^1(\Omega)$.
\end{prop}
\begin{proof}
The terms $F^4$, as defined by \eqref{dt1_f4}, contains four separate terms.  We will estimate each of these, employing H\"{o}lder's inequality, trace estimates, and Theorem \ref{catalog_energy} repeatedly and without explicit reference.

\textbf{Term: $\mu \sg_{\dt \A} u \N $. } We bound 
\begin{equation}
\abs{ \int_{-\ell}^\ell w \cdot (\mu \sg_{\dt \A} u \N)  } \ls \int_{-\ell}^\ell \abs{w} \abs{\dt \nab \eta} \abs{\nab u} \ls \norm{w}_{L^{1/\ep_+}(\Sigma)} \norm{\dt \eta}_{W^{1,\infty}} \norm{\nab u}_{L^{1/(1-\ep_+)}(\Sigma)} \ls \norm{w}_{H^1} \E.
\end{equation}

\textbf{Term: $\left(  g\eta  - \sigma \p_1 \left(\frac{\p_1 \eta}{(1+\abs{\p_1 \zeta_0}^2)^{3/2}}\right)\right) \dt\N$. }  We estimate 
\begin{multline}
\abs{\int_{-\ell}^\ell w \cdot \left(  g\eta  - \sigma \p_1 \left(\frac{\p_1 \eta}{(1+\abs{\p_1 \zeta_0}^2)^{3/2}}\right)\right) \dt\N}  \ls \int_{-\ell}^\ell \abs{w} \left(\abs{\eta} + \abs{\p_1 \eta} + \abs{\p_1^2 \eta} \right) \abs{ \dt \p_1 \eta} \\
\ls \norm{w}_{L^{1/\ep_+}(\Sigma)} \norm{\eta}_{W^{2,1/(1-\ep_+)}} \norm{\dt \eta}_{W^{1,\infty}} \ls \norm{w}_{H^1} \E.
\end{multline}

\textbf{Term: $-\sigma \p_1 [\mathcal{R}(\p_1 \zeta_0,\p_1 \eta)] \dt \N$. }  We first expand 
\begin{equation}
 \p_1 [\mathcal{R}(\p_1 \zeta_0,\p_1 \eta)]  = \p_y \mathcal{R}(\p_1 \zeta_0, \p_1 \eta) \p_1^2 \zeta_0 + \p_z\mathcal{R}(\p_1 \zeta_0,\p_1 \eta) \p_1^2 \eta
\end{equation}
and then use Proposition \ref{R_prop} in order to estimate 
\begin{equation}
\abs{\int_{-\ell}^\ell w \cdot(-\sigma \p_1 [\mathcal{R}(\p_1 \zeta_0,\p_1 \eta)] \dt \N) }  \ls
\int_{-\ell}^\ell \abs{w} \abs{\p_1 \eta}^2 \abs{\p_1 \dt \eta} + \int_{-\ell}^\ell \abs{w}  \abs{\p_1 \eta} \abs{\p_1^2 \eta} \abs{\p_1 \dt \eta} =:I +II.
\end{equation}
Then we bound 
\begin{equation}
 I \ls \norm{w}_{L^{2}(\Sigma)} \ns{\p_1 \eta}_{L^\infty} \norm{\p_1 \dt \eta}_{L^2} \ls \norm{w}_{H^1} \E
\end{equation}
and 
\begin{equation}
 II \ls \norm{w}_{L^{1/\ep_+}(\Sigma)} \norm{\p_1 \eta}_{L^\infty} \norm{\p_1^2 \eta}_{L^{1/(1-\ep_+)}} \norm{\p_1 \dt \eta}_{L^\infty}  \ls \norm{w}_{H^1} \E^{3/2}.
\end{equation}
Upon combining these, we find that this term can be estimated as stated.

\textbf{Term: $ - S_{\A}(p,u) \dt \N$. }  We bound 
\begin{multline}
\abs{\int_{-\ell}^\ell w \cdot ( - S_{\A}(p,u) \dt \N)} \ls \int_{-\ell}^\ell \abs{w} \left(\abs{p} + \abs{\nab u} \right)\abs{\dt \p_1 \eta} \\
\ls \norm{w}_{L^{1/\ep_+}(\Sigma)} \left(\norm{p}_{L^{1/(1-\ep_+)}(\Sigma) }+ \norm{\nab u}_{L^{1/(1-\ep_+)}(\Sigma) } \right) \norm{\p_1 \dt \eta}_{L^\infty} \ls \norm{w}_{H^1} \E.
\end{multline}
\end{proof}

Next we study the term $F^5$.

\begin{prop}\label{nie_f5}
Suppose that $F^5$ is given by  \eqref{dt1_f5}.  Then
\begin{equation} 
 \abs{ \int_{\Sigma_s}  J(w \cdot \tau)F^5 } \ls \norm{w}_{H^1} \E
\end{equation}
for every $w \in H^1(\Omega)$.
\end{prop}
\begin{proof}
Using trace estimates and the Sobolev embeddings together with Theorem \ref{catalog_energy}, we bound 
\begin{multline} 
 \abs{ \int_{\Sigma_s}  J(w \cdot \tau) ( \mu \sg_{\dt \A} u \nu \cdot \tau )} \ls 
\int_{\Sigma_s} \abs{w} \abs{\dt \A} \abs{\nab u}  \ls \norm{w}_{L^{1/\ep_+}(\Sigma_s)} \norm{\dt \bar{\eta}}_{W^{1,\infty}} \norm{\nab u}_{L^{1/(1-\ep_+)}(\Sigma_s)} \\
\ls \norm{w}_{H^1} \E
\end{multline}
This is the stated bound.
\end{proof}

We combine the above estimates into the following theorem, which is the analog of Theorem \ref{nid_v_est}.

\begin{thm}\label{nie_v_est}
Consider the functional $H^1(\Omega) \ni w\mapsto \br{\mathcal{F},w} \in \mathbb{R}$ defined by 
\begin{equation}
\br{\mathcal{F},w} = \int_\Omega F^1 \cdot w J 
-  \int_{-\ell}^\ell  F^4 \cdot w 
- \int_{\Sigma_s}  J (w \cdot \tau)F^5, 
\end{equation} 
where $F^1,F^4,F^5$ are defined via \eqref{dt1_f1}, \eqref{dt1_f4}, and \eqref{dt1_f5}, respectively.  Then 
\begin{equation}
\abs{\br{\mathcal{F},w}} \ls  \norm{w}_{H^1}  (\E+ \E^{3/2}) 
\end{equation}
for all $w \in H^1(\Omega)$.
\end{thm}
\begin{proof}
This follows immediately from Propositions \ref{nie_f1}, \ref{nie_f4}, and \ref{nie_f5}. 
\end{proof}

\subsection{General interaction functional with free surface terms }

Next we turn our attention to a general estimate involving the free surface and $F^3$.  

\begin{thm}\label{nie_ST}
Suppose that $F^3$ is given by \eqref{dt1_f3}.  Then we have the estimate
\begin{equation}
\abs{\int_{-\ell}^\ell g \dt \eta (w \cdot \N) - \sigma \p_1 \left( \frac{\p_1 \dt \eta }{(1+\abs{\p_1 \zeta_0}^2)^{3/2}} +F^3\right)w\cdot  \N } \ls \left(1+ \sqrt{\E}  \right) \norm{\dt \eta}_{H^{3/2 + (\ep_- - \low)/2}} \norm{w}_{H^1}
\end{equation}
for every $w \in H^1(\Omega)$.
\end{thm}
\begin{proof}
The first term is easy to deal with:
\begin{equation}
 \abs{\int_{-\ell}^\ell g \dt \eta w \cdot \N} \ls \int_{-\ell}^\ell \abs{\dt \eta}\abs{w} \ls \norm{\dt \eta}_{L^2} \norm{w}_{L^2(\Sigma)} \ls \norm{\dt \eta}_{L^2} \norm{w}_{H^1(\Omega)}.
\end{equation}
The second and third terms require more work.

Let $s = 1 - (\ep_- - \low) \in (0,1)$, which requires that 
\begin{equation}\label{nie_ST_7}
 2 - \frac{s}{2} = \frac{3}{2} + \frac{\ep_- - \low}{2}.
\end{equation}
Using this and Proposition \ref{frac_IBP_prop}  we may estimate 
\begin{equation}
 \abs{\int_{-\ell}^\ell  \sigma \p_1 \left( \frac{\p_1 \dt \eta }{(1+\abs{\p_1 \zeta_0}^2)^{3/2}}  + F^3\right)w\cdot  \N } \ls \left(\norm{ \frac{\p_1 \dt \eta }{(1+\abs{\p_1 \zeta_0}^2)^{3/2}}}_{H^{1-s/2}} + \norm{F^3}_{H^{1-s/2}} \right) \norm{w \cdot \N}_{H^{s/2}}. 
\end{equation}
Since $\zeta_0$ is smooth, Proposition \ref{supercrit_prod} shows that 
\begin{equation}\label{nie_ST_1}
\norm{ \frac{\p_1 \dt \eta }{(1+\abs{\p_1 \zeta_0}^2)^{3/2}}}_{H^{1-s/2}} \ls \norm{ \p_1 \dt \eta }_{H^{1-s/2}} \ls \norm{\dt \eta }_{H^{2-s/2}} = \norm{\dt \eta}_{H^{3/2 + (\ep_- - \low)/2}}.
\end{equation}

To handle the term
\begin{equation}
 F^3 = \dt [\mathcal{R}(\p_1 \zeta_0,\p_1 \eta)] = \p_z \mathcal{R}(\p_1 \zeta_0,\p_1 \eta) \p_1 \dt \eta
\end{equation}
 we first use the fact that $H^{1-s/2}((-\ell,\ell))$ is an algebra to bound 
\begin{equation}
 \norm{F^3}_{H^{1-s/2}} \ls  \norm{\p_z \mathcal{R}(\p_1 \zeta_0,\p_1 \eta)}_{H^{1-s/2}} \norm{\p_1 \dt \eta}_{H^{1-s/2}}
\end{equation}
and then we use Proposition \ref{frac_comp} with $f(x,z) = \p_z \mathcal{R}(\p_1 \zeta_0(x),z)$ to estimate 
\begin{equation}
  \norm{\p_z \mathcal{R}(\p_1 \zeta_0,\p_1 \eta)}_{H^{1-s/2}} \ls \norm{\p_1 \eta}_{H^{1-s/2}},
\end{equation}
which yields (again using \eqref{nie_ST_7})
\begin{equation}\label{nie_ST_2}
\norm{F^3}_{H^{1-s/2}} \ls \norm{\eta}_{H^{3/2 + (\ep_- - \low)/2}} \norm{\dt \eta}_{H^{3/2 + (\ep_- - \low)/2}}.
\end{equation}
However,  
\begin{equation}
 \frac{3}{2} + \frac{\ep_- - \low}{2} \le \frac{3}{2} + \ep_+
\end{equation}
and 
\begin{equation}
 \frac{1}{2} = \frac{2-\ep_+}{2} - \frac{1}{1} \left( 2 + \frac{\ep_+}{2} - \frac{3}{2} - \ep_+ \right) = \frac{1}{q_+} - \frac{1}{1}\left(3 - \frac{1}{q_+} - \frac{3}{2}-\ep_+ \right),
\end{equation}
so we have the embedding 
\begin{equation}\label{nie_ST_3}
 W^{3-1/q_+,q_+}((-\ell,\ell)) \hookrightarrow H^{3/2 + \ep_+}((-\ell,\ell)).
\end{equation}
Then \eqref{nie_ST_2} and \eqref{nie_ST_3} tell us that 
\begin{equation}\label{nie_ST_4}
\norm{F^3}_{H^{1-s/2}} \ls \sqrt{\E} \norm{\dt \eta}_{H^{3/2 + (\ep_- - \low)/2}}. 
\end{equation}

Next we use Proposition \ref{supercrit_prod} (with $1/2 + \ep_+ > \max\{1/2, s/2\}$), the usual trace estimate, the embedding \eqref{nie_ST_3}, and the bound $\E \le 1$ to see that 
\begin{equation}
 \norm{w \cdot \N}_{H^{s/2}} \ls \norm{\N}_{H^{1/2 + \ep_+}} \norm{w}_{H^{s/2}} \ls \left(1 + \norm{\eta}_{H^{3/2 + \ep_+}}  \right)\norm{w}_{H^{1/2}(\Sigma)} \ls \left( 1+ \sqrt{\E}\right) \norm{w}_{H^1(\Omega)} \ls \norm{w}_{H^1}.
\end{equation}
Combining this with \eqref{nie_ST_1} and \eqref{nie_ST_4},  we conclude that 
\begin{equation}
 \abs{\int_{-\ell}^\ell  \sigma \p_1 \left( \frac{\p_1 \dt \eta }{(1+\abs{\p_1 \zeta_0}^2)^{3/2}}  + F^3\right)w\cdot  \N } \ls  (1+\sqrt{\E} ) \norm{\dt \eta}_{H^{3/2 + (\ep_- - \low)/2}} \norm{w}_{H^1},
\end{equation}
which completes the proof.

\end{proof}

\section{Nonlinear estimates III: elliptic estimate terms}\label{sec_nl_elliptic}

In this section we complete our study of the nonlinear terms coming from \eqref{ns_geometric} by turning our attention to elliptic estimates.  More precisely, we study the terms appearing in applications of Theorem \ref{A_stokes_stress_solve}.  As in the previous two sections, we assume throughout this section that a solution to \eqref{ns_geometric} exists on the time horizon $(0,T)$ for $0 < T \le \infty$ and obeys the small-energy estimate 
\begin{equation}
 \sup_{0\le t < T} \E(t) \le \gamma^2 < 1,
\end{equation}
where $\gamma \in (0,1)$ is as in Lemma \ref{eta_small}.  This means that the estimates of Lemma \ref{eta_small} are available for use, and we will use them often without explicit reference.

\subsection{No time derivatives}

We begin with the elliptic estimates we will need for the problem \eqref{ns_geometric}, i.e. when no temporal derivatives are applied.  When we compare \eqref{ns_geometric} and \eqref{A_stokes_stress} we get
\begin{equation}\label{ne0_Gform}
\begin{split}
 G^1 & =  -\dt u +\dt \bar{\eta} \frac{\phi}{\zeta_0} K \p_2  u - u \cdot \naba  u, \; 
 G^2 = 0,  \\ 
 G^3_+ &= \dt \eta /\abs{\N_0}, \; 
 G^3_- = 0, \; 
 G^4_+ =0, \; 
 G^4_- = 0  \\
 G^5 &= 0, \; 
 G^6 = \mathcal{R}(\p_1 \zeta_0,\p_1 \eta), \;
 G^7 = \low \dt  \eta \pm \mathcal{R}(\p_1 \zeta_0,\p_1 \eta).
\end{split}
\end{equation}
This dictates the form of the estimates we need. 
 
We begin with the bounds for $G^1$  in \eqref{ne0_Gform}.

\begin{prop}\label{ne0_g1}
We have the bound 
\begin{equation}
 \norm{-\dt u +\dt \bar{\eta} \frac{\phi}{\zeta_0} K \p_2  u - u \cdot \naba  u}_{L^{q_+}} \ls \norm{\dt u}_{L^2} + \sqrt{\E}\left(  \norm{u}_{L^2} + \norm{\dt \eta}_{H^1}\right) 
\end{equation}
\end{prop}
\begin{proof}
The bound $\norm{\dt u}_{L^{q_+}} \ls \norm{\dt u}_{L^2}$ follows from the fact that $q_+ < 2$.  Using H\"{o}lder's inequality and Theorem \ref{catalog_energy} we then bound
\begin{equation}
 \norm{ u \cdot \naba  u}_{L^{q_+}} \ls   \norm{ \abs{u} \abs{\nab u}}_{L^{q_+}} \ls \norm{\nab u}_{L^{2/(1-\ep_+)}} \norm{u}_{L^2} \ls \sqrt{\E} \norm{u}_{L^2},
\end{equation}
and (also using Proposition \ref{poisson_prop})
\begin{equation}
\norm{\dt \bar{\eta} \frac{\phi}{\zeta_0} K \p_2  u}_{L^{q_+}} \ls \norm{\abs{\dt \bar{\eta}} \abs{\nab   u}}_{L^{q_+}} \ls  \norm{\nab u}_{L^{2/(1-\ep_+)}} \norm{\dt \bar{\eta}}_{L^{2}} \ls \sqrt{\E} \norm{\dt \eta}_{H^1}.
\end{equation}
The result follows by combining these.
\end{proof}

We continue with the bounds for $G^3$  in \eqref{ne0_Gform}.

\begin{prop}\label{ne0_g3}
Let $\N_0$ be given by \eqref{N0_def}.  Then we have the estimate 
\begin{equation}
 \norm{\dt \eta /\abs{\N_0}}_{W^{2-1/q_+,q_+}} \ls \norm{\dt \eta}_{H^{3/2-\low}} .
\end{equation}
\end{prop}
\begin{proof}
First note that
\begin{equation}
 \frac{1}{\abs{\N_0}} = \frac{1}{\sqrt{1 + \abs{ \p_1 \zeta_0}^2}}
\end{equation}
is smooth, and we may thus bound 
\begin{equation}
 \norm{\dt \eta /\abs{\N_0}}_{W^{2-1/q_+,q_+}} \ls  \norm{\dt \eta }_{W^{2-1/q_+,q_+}}. 
\end{equation}

Next note that \eqref{kappa_ep_def} implies that $2\low + \ep_+ < 1$, so 
\begin{equation}
 2 - \frac{1}{q_+} = 2 - \frac{2 - \ep_+}{2} = 1 + \frac{\ep_+}{2}  \le \frac{3}{2} - \low
\end{equation}
and 
\begin{equation}
\frac{1}{q_+} = \frac{2-\ep_+}{2} \ge \frac{2\low + \ep_+}{2} = \hal - \frac{1}{1} \left(\frac{3}{2} - \low -1 - \frac{\ep_+}{2}  \right).
\end{equation}
These parameter bounds and the Sobolev embeddings show that
\begin{equation}
H^{3/2-\low}((-\ell,\ell)) \hookrightarrow W^{2-1/q_+, 2/(2\low + \ep_+)}((-\ell,\ell))  \hookrightarrow  W^{2-1/q_+,q_+}((-\ell,\ell)).
\end{equation}
This allows us to estimate 
\begin{equation}
 \norm{\dt^2 \eta}_{W^{2-1/q_+,q_+}} \ls \norm{\dt^2 \eta}_{H^{3/2 - \low}},
\end{equation}
and the result follows by combining these bounds.

\end{proof}

Our next result records the bounds for $G^6$  in \eqref{ne0_Gform}.

\begin{prop}\label{ne0_g6}
Let $\mathcal{R}$ be as defined in \eqref{R_def}.  Then we have the estimate
\begin{equation}
 \norm{\p_1 [\mathcal{R}(\p_1 \zeta_0,\p_1 \eta)]}_{W^{1-1/q_+,q_+}} \ls \sqrt{\E} \norm{\eta}_{W^{3 - 1/q_+,q_+}}
\end{equation}
\end{prop}
\begin{proof}
We compute 
\begin{equation}
\p_1 [\mathcal{R}(\p_1 \zeta_0,\p_1 \eta)] = \p_y \mathcal{R}(\p_1 \zeta_0,\p_1 \eta) \p_1^2 \zeta_0 + \p_z \mathcal{R}(\p_1 \zeta_0,\p_1 \eta) \p_1^2 \eta. 
\end{equation}
We then use this with the product estimate from \eqref{ne1_g6_1} to bound 
\begin{multline}
 \norm{\p_1 [\mathcal{R}(\p_1 \zeta_0,\p_1 \eta)]}_{W^{1-1/q_+,q_+}} 
\\ 
 \ls \norm{\p_y \mathcal{R}(\p_1 \zeta_0,\p_1 \eta) }_{W^{1,q_+}} \norm{\p_1^2 \zeta_0}_{W^{1-1/q_+,q_+}}  
 + \norm{\p_z \mathcal{R}(\p_1 \zeta_0,\p_1 \eta)}_{W^{1,q_+}} \norm{\p_1^2 \eta}_{W^{1-1/q_+,q_+}} \\
\ls 
\norm{\p_y \mathcal{R}(\p_1 \zeta_0,\p_1 \eta) }_{W^{1,q_+}}   
 + \norm{\p_z \mathcal{R}(\p_1 \zeta_0,\p_1 \eta)}_{W^{1,q_+}} \norm{\eta}_{W^{3-1/q_+,q_+}}.
\end{multline}
On the other hand, Proposition \ref{R_prop} and Theorem \ref{catalog_energy} allow us to bound 
\begin{equation}
\norm{\p_y \mathcal{R}(\p_1 \zeta_0,\p_1 \eta) }_{W^{1,q_+}} \ls \norm{\p_1 \eta}_{L^\infty} \left(\norm{\p_1 \eta}_{L^{q_+}}   + \norm{\p_1^2 \eta}_{L^{q_+}}\right) \ls \sqrt{\E} \norm{\eta}_{W^{2,q_+}}
\end{equation}
and 
\begin{equation}
\norm{\p_z \mathcal{R}(\p_1 \zeta_0,\p_1 \eta)}_{W^{1,q_+}} \ls \norm{\p_1 \eta}_{L^{q_+}} +  \norm{\p_1^2 \eta}_{L^{q_+}} \ls \sqrt{\E}.
\end{equation}
The result follows by combining these and recalling that $3 - 1/q_+ > 2$.
\end{proof}

Finally, we record the bounds for $G^7$  in \eqref{ne0_Gform}.

\begin{prop}\label{ne0_g7}
Let $\mathcal{R}$ be as in \eqref{R_def}.  Then we have the estimate 
\begin{equation}
 [\mathcal{R}(\p_1 \zeta_0,\p_1 \eta)]_\ell \ls \sqrt{\E} \norm{\eta}_{W^{2,q_+}}
\end{equation}
\end{prop}
\begin{proof}
From trace theory,  Theorem \ref{catalog_energy}, and Proposition \ref{R_prop} we may estimate 
\begin{equation}
 [\mathcal{R}(\p_1 \zeta_0,\p_1 \eta)]_\ell \ls \norm{\mathcal{R}(\p_1 \zeta_0,\p_1 \eta)}_{W^{1,q_+}} \ls  \norm{\p_1 \eta}_{L^\infty} \left(\norm{\p_1 \eta}_{L^{q_+}}   + \norm{\p_1^2 \eta}_{L^{q_+}}\right) \ls \sqrt{\E} \norm{\eta}_{W^{2,q_+}}.
\end{equation}
This is the stated estimate.

\end{proof}

\subsection{One time derivative}

We now turn our attention to the elliptic estimates for the once time differentiated problem.  In order to apply  Theorem \ref{A_stokes_stress_solve}, we are led to consider the following $G^i$ terms  for $F^1$--$F^7$ given by \eqref{dt1_f1}--\eqref{dt1_f7}:  
\begin{equation}\label{ne1_Gform}
\begin{split}
 G^1 & = F^1  -\dt^2 u +\dt \bar{\eta} \frac{\phi}{\zeta_0} K \p_2 \dt u - u \cdot \naba \dt u, \; 
 G^2 = J F^2,  \\ 
 G^3_+ &= (\dt^2 \eta - F^6)/\abs{\N_0}, \; 
 G^3_- = 0, \; 
 G^4_+ = F^4 \cdot \frac{\mathcal{T}}{\abs{\mathcal{T}}^2}, \; 
 G^4_- = F^5  \\
 G^5 &= F^4 \cdot \frac{\mathcal{N}}{\abs{\N}^2}, \; 
 G^6 = F^3, \;
 G^7 = \linz \dt^2 \eta + \linz F^7.
\end{split}
\end{equation}

We begin by estimating the $G^1$ term from \eqref{ne1_Gform}.

\begin{prop}\label{ne1_g1}
Let $F^1$ be given by \eqref{dt1_f1}. We have the estimate 
\begin{equation}
 \norm{F^1    +\dt \bar{\eta} \frac{\phi}{\zeta_0} K \p_2 \dt u - u \cdot \naba \dt u}_{L^{q_-}} \ls (\sqrt{\E} + \E) \sqrt{\D}.
\end{equation}
\end{prop}
\begin{proof}

We will estimate term by term using H\"{o}lder's inequality and the bounds from Theorems \ref{catalog_energy} and \ref{catalog_dissipation}, once more using the ordering scheme used in the proof of Proposition \ref{nid_f1}.  Combining the estimates of each term then yields the stated estimate.  Recall from \eqref{kappa_ep_def} that $0 < 2\low < \ep_- < \ep_+$, which in particular means that 
\begin{equation}\label{ne1_g1_est}
q_- = \frac{2}{2-\ep_-} < \frac{2}{2-\ep_+} = q_+. 
\end{equation}

\textbf{Term: $\diverge_{\dt \A} S_\A(p,u)$.}  First note that 
\begin{equation} \frac{1 - \ep_+}{2} + \frac{1-\ep_+}{2} + \frac{1}{\infty} \le \frac{2-\ep_-}{2} = \frac{1}{q_-}.
\end{equation}
Using this and \eqref{ne1_g1_est} we can then bound
\begin{multline}
 \norm{\diverge_{\dt \A} S_\A(p,u)}_{L^{q_-}} \ls  \norm{\abs{\dt \A} \abs{\nab \A}  ( \abs{p} + \abs{\nab u})  }_{L^{q_-}} +  \norm{\abs{\dt \A} (\abs{\nab p} + \abs{\nab^2 u})  }_{L^{q_-}} \\
 \ls \norm{\dt \bar{\eta}}_{W^{1,\infty}} \norm{\bar{\eta}}_{W^{2,2/(1-\ep_+)}} \left( \norm{p}_{L^{2/(1-\ep_+)}} + \norm{\nab u}_{L^{2/(1-\ep_+)}} \right) 
+ \norm{\dt \bar{\eta}}_{W^{1,\infty}} \left( \norm{\nab p}_{L^{q_+}} + \norm{\nab^2 u}_{L^{q_+}}  \right) \\
\ls \E \sqrt{\D} + \sqrt{\E} \sqrt{\D}.
\end{multline}

\textbf{Term: $\mu \diva \sg_{\dt \A} u$.}  For this term note that
\begin{equation}
 \frac{1- \ep_-}{2} + \frac{1-\ep_+}{2} \le \frac{2-\ep_-}{2} = \frac{1}{q_-}.
\end{equation}
This and \eqref{ne1_g1_est} allow us to bound
\begin{multline}
\norm{ \mu \diva \sg_{\dt \A} u}_{L^{q_-}} \ls \norm{\abs{\nab \dt \A} \abs{\nab u} }_{L^{q_-}} +  \norm{\abs{\dt \A} \abs{\nab^2 u} }_{L^{q_-}} \\ 
\ls 
\left(\norm{\bar{\eta}}_{W^{2,2/(1-\ep_-)}} + \norm{\dt \bar{\eta}}_{W^{2,2/(1-\ep_-)}} \right) \norm{\nab u}_{L^{2/(1-\ep_+)}} 
+ \norm{\dt \bar{\eta}}_{W^{1,\infty}} \norm{\nab^2 u}_{L^{q_+}} 
\ls \sqrt{\D} \sqrt{\E} + \sqrt{\E} \sqrt{\D}.
\end{multline}

\textbf{Term: $u \cdot \nab_{\dt \A} u$. }  We simply use \eqref{ne1_g1_est} to bound
\begin{equation}
 \norm{u \cdot \nab_{\dt \A} u}_{L^{q_-}} \ls  \norm{\abs{u}\abs{\dt \A}  \abs{\nab u}}_{L^{q_-}}
\ls \norm{u}_{L^\infty} \norm{\dt \bar{\eta}}_{W^{1,\infty}} \norm{\nab u}_{L^{q_+}} \ls \E \sqrt{\D}.
\end{equation}

\textbf{Term: $\dt u \cdot \naba u$. } Again we use \eqref{ne1_g1_est} to bound
\begin{equation}
\norm{\dt u \cdot \naba u}_{L^{q_-}} \ls \norm{\abs{\dt u} \abs{\nab u}}_{L^{q_-}} \ls \norm{\dt u}_{L^\infty} \norm{\nab u}_{L^{q_+}} \ls \sqrt{\D} \sqrt{\E}.
\end{equation}

\textbf{Term: $\dt^2 \bar{\eta} \frac{\phi}{\zeta_0} K \p_2 u$. }   Again we use \eqref{ne1_g1_est} to bound
\begin{equation}
 \norm{\dt^2 \bar{\eta} \frac{\phi}{\zeta_0} K \p_2 u}_{L^{q_-}} \ls  \norm{ \abs{\dt^2 \bar{\eta}} \abs{\nab u}}_{L^{q_-}} \ls \norm{\dt^2 \bar{\eta}}_{L^\infty} \norm{\nab u}_{L^{q_+}} \ls \sqrt{\D} \sqrt{\E}.
\end{equation}

\textbf{Term: $\dt \bar{\eta} \frac{\phi}{\zeta_0} \dt K \p_2 u$. } Once more \eqref{ne1_g1_est} let's us bound
\begin{equation}
 \norm{\dt \bar{\eta} \frac{\phi}{\zeta_0} \dt K \p_2 u }_{L^{q_-}} \ls   \norm{\abs{\dt \bar{\eta}} \abs{\dt K} \abs{\nab u} }_{L^{q_-}} \ls \norm{\dt \bar{\eta}}_{L^\infty} \norm{\dt \bar{\eta}}_{W^{1,\infty}} \norm{\nab u}_{L^{q_+}} \ls \E \sqrt{\D}.
\end{equation}

\textbf{Term: $\dt \bar{\eta} \frac{\phi}{\zeta_0} K \p_2 \dt u$. } Since 
\begin{equation}\label{ne1_g1_est2}
 \frac{1-\ep_-}{2} \le \frac{2 - \ep_-}{2} = \frac{1}{q_-}
\end{equation}
we can bound
\begin{equation}
 \norm{\dt \bar{\eta} \frac{\phi}{\zeta_0} K \p_2 \dt u}_{L^{q_-}} \ls  \norm{\abs{\dt \bar{\eta}} \abs{\nab \dt u}}_{L^{q_-}} \ls \norm{\dt \bar{\eta}}_{L^\infty} \norm{\nab \dt u}_{L^{2/(1-\ep_-)}} \ls \sqrt{\E} \sqrt{\D}.
\end{equation}

\textbf{Term: $u \cdot \naba \dt u$. } For this term we use \eqref{ne1_g1_est2} again to bound
\begin{equation}
 \norm{u \cdot \naba \dt u}_{L^{q_-}} \ls  \norm{\abs{u} \abs{\nab \dt u}}_{L^{q_-}} \ls \norm{u}_{L^\infty} \norm{\nab \dt u}_{L^{2/(1-\ep_-)}} \ls \sqrt{\E} \sqrt{\D}.
\end{equation}

\end{proof}

Next we estimate the $G^2$ term from \eqref{ne1_Gform}.

\begin{prop}\label{ne1_g2}
Let $F^2$ be given by \eqref{dt1_f2}.  Then we have that 
\begin{equation}
 \norm{J F^2}_{W^{1,q_-}} \ls (\sqrt{\E}+ \E) \sqrt{\D}  .
\end{equation}
\end{prop}
\begin{proof}
We begin by noting that $JF^2 = -J\diverge_{\dt \A} u$, so 
\begin{equation}
\norm{J F^2}_{W^{1,q_-}} \ls \norm{J\diverge_{\dt \A} u}_{L^{q_-}} + \norm{\nab(J\diverge_{\dt \A} u)}_{L^{q_-}}.
\end{equation}
We will estimate each of these terms with H\"{o}lder's inequality and the bounds from Theorems \ref{catalog_energy} and \ref{catalog_dissipation}, again using the ordering scheme used in the proof of Proposition \ref{nid_f1}.  For the first we use the fact that $q_- < q_+$ to bound 
\begin{equation}
\norm{J\diverge_{\dt \A} u}_{L^{q_-}} \ls \norm{\abs{\dt \A} \abs{\nab u}}_{L^{q_-}} \ls \norm{\dt \bar{\eta}}_{W^{1,\infty}} \norm{\nab u}_{L^{q_+}} \ls \sqrt{\E} \sqrt{\D}.
\end{equation}
For the second term we expand with the product rule and note that \eqref{kappa_ep_def} implies
\begin{equation}\label{ne1_g2_est}
 \frac{1- \ep_+}{2} +  \frac{1- \ep_+}{2} \le  \frac{2- \ep_-}{2} = \frac{1}{q_-} \text{ and }  \frac{1- \ep_+}{2} +  \frac{1- \ep_-}{2} \le  \frac{2- \ep_-}{2} = \frac{1}{q_-}
\end{equation}
which allows us to bound 
\begin{multline}
\norm{\nab(J\diverge_{\dt \A} u)}_{L^{q_-}} \ls \norm{\abs{\nab J} \abs{\dt \A} \abs{\nab u} }_{L^{q_-}} + \norm{ \abs{\nab \dt \A} \abs{\nab u} }_{L^{q_-}} + \norm{\abs{\dt \A} \abs{\nab^2 u} }_{L^{q_-}} \\
\ls \norm{\bar{\eta}}_{W^{2,2/(1-\ep_+)}} \norm{\dt \bar{\eta}}_{W^{1,\infty}} \norm{\nab u}_{L^{2/(1-\ep_+)}} 
+ \left(\norm{\bar{\eta}}_{W^{2,2/(1-\ep_-)}} + \norm{\dt \bar{\eta}}_{W^{2,2/(1-\ep_-)}} \right) \norm{\nab u}_{L^{2/(1-\ep_+)}}  \\
+ \norm{\dt \bar{\eta}}_{W^{1,\infty}} \norm{\nab^2 u}_{L^{q_+}} \ls \E \sqrt{\D} + \sqrt{\D} \sqrt{\E} + \sqrt{\E} \sqrt{\D} \ls(\sqrt{\E} + \E)\sqrt{\D}.
\end{multline}
Combining these bounds then yields the stated estimate.
\end{proof}

The next result records the bounds for the $G^3$ term from \eqref{ne1_Gform}.

\begin{prop}\label{ne1_g3}
Let $F^6$ be given by \eqref{dt1_f6}.  We have the bound 
\begin{equation}
 \norm{(\dt^2 \eta - F^6)/\abs{\N_0}}_{W^{2-1/q_-,q_-}} \ls \norm{\dt^2 \eta}_{H^{3/2-\low}} + \sqrt{\E} \sqrt{\D}.
\end{equation}
\end{prop}
\begin{proof}
First recall that $N = -\p_1 \zeta_0 e_1 + e_2$, so 
\begin{equation}
 \frac{1}{\abs{\N_0}} = \frac{1}{\sqrt{1 + \abs{ \p_1 \zeta_0}^2}}
\end{equation}
is smooth, and we may thus bound 
\begin{equation}
 \norm{(\dt^2 \eta - F^6)/\abs{\N_0}}_{W^{2-1/q_-,q_-}} \ls  \norm{\dt^2 \eta - F^6}_{W^{2-1/q_-,q_-}}. 
\end{equation}

Next note that \eqref{kappa_ep_def} implies that $2\low + \ep_- < 1$, so 
\begin{equation}
 2 - \frac{1}{q_-} = 2 - \frac{2 - \ep_-}{2} = 1 + \frac{\ep_-}{2}  \le \frac{3}{2} - \low
\end{equation}
and 
\begin{equation}
\frac{1}{q_-} = \frac{2-\ep_-}{2} \ge \frac{2\low + \ep_-}{2} = \hal - \frac{1}{1} \left(\frac{3}{2} - \low -1 - \frac{\ep_-}{2}  \right),
\end{equation}
and so these parameter bounds and the Sobolev embeddings show that
\begin{equation}
H^{3/2-\low}((-\ell,\ell)) \hookrightarrow W^{2-1/q_-, 2/(2\low + \ep_-)}((-\ell,\ell))  \hookrightarrow  W^{2-1/q_-,q_-}((-\ell,\ell)).
\end{equation}
This allows us to estimate 
\begin{equation}
 \norm{\dt^2 \eta}_{W^{2-1/q_-,q_-}} \ls \norm{\dt^2 \eta}_{H^{3/2 - \low}}.
\end{equation}
For the $F^6= -u_1 \p_1 \dt \eta$ term we the fact that $W^{2-1/q_-,q_-}((-\ell,\ell))$ is an algebra in conjunction with trace theory and the definitions of $\E$ and $\D$ in \eqref{E_def} and \eqref{D_def}, respectively, to bound 
\begin{equation}
\norm{F^6}_{W^{2-1/q_-,q_-}} \ls \norm{u}_{W^{2-1/q_-,q_-}(\Sigma)} \norm{\p_1 \dt \eta}_{W^{2-1/q_-,q_-}}  \ls \norm{u}_{W^{2,q_-}} \norm{\dt \eta}_{W^{3-1/q_-,q_-}} \ls \sqrt{\E} \sqrt{\D}. 
\end{equation}
Combining these then yields the stated bound.
\end{proof}

Next we bound the the $G^4$ and $G^5$ terms from \eqref{ne1_Gform}.

\begin{prop}\label{ne1_g4_g5}
Let $F^4$ and $F^5$ be given by \eqref{dt1_f4} and \eqref{dt1_f5}, respectively.  Then we have the bound 
\begin{equation}
 \norm{F^4 \cdot \frac{\mathcal{T}}{\abs{\mathcal{T}}^2}}_{W^{1-1/q_-,q_-}} +  \norm{F^4 \cdot \frac{\mathcal{N}}{\abs{\mathcal{N}}^2}}_{W^{1-1/q_-,q_-}} +  \norm{F^5 }_{W^{1-1/q_-,q_-}} \ls (\sqrt{\E} + \E) \sqrt{\D}.
\end{equation}
\end{prop}
\begin{proof}
Recall that $\E$ and $\D$ are as defined in \eqref{E_def} and \eqref{D_def}. We begin with the $F^5 = \mu \sg_{\dt \A} u \nu \cdot \tau$ term.  We use trace theory, the product rule,  Theorems \ref{catalog_energy} and \ref{catalog_dissipation}, and \eqref{ne1_g2_est} to bound
\begin{multline}\label{ne1_g4_1}
  \norm{F^5 }_{W^{1-1/q_-,q_-}} \ls  \norm{\sg_{\dt \A} u }_{W^{1,q_-}(\Omega)} \ls \norm{\abs{\dt \A} \abs{\nab u}}_{L^{q_-}} + \norm{\abs{\nab \dt \A} \abs{\nab u}}_{L^{q_-}} + \norm{\abs{\dt \A} \abs{\nab^2 u}}_{L^{q_-}} \\
 \ls  \norm{\dt \bar{\eta}}_{W^{1,\infty}} \norm{\nab u}_{L^{q_+}}  
+ \left(\norm{\bar{\eta}}_{W^{2,2/(1-\ep_-)}} + \norm{\dt \bar{\eta}}_{W^{2,2/(1-\ep_-)}} \right) \norm{\nab u}_{L^{2/(1-\ep_+)}}  
+ \norm{\dt \bar{\eta}}_{W^{1,\infty}} \norm{\nab^2 u}_{L^{q_+}}  \\
\ls \sqrt{\E} \sqrt{\D} + \sqrt{\D} \sqrt{\E} + \sqrt{\E} \sqrt{\D}.
\end{multline}

We next turn our attention to the $F^4$ term.  First note that since
\begin{equation}
 \frac{\mathcal{T}}{\abs{\mathcal{T}}^2} = \frac{(1,\p_1 \zeta_0 + \p_1 \eta)}{1 + \abs{\p_1 \zeta_0 + \p_1 \eta}^2}, \N = (-\p_1 \zeta_0 - \p_1 \eta,1),\text{ and } \frac{\mathcal{N}}{\abs{\N}^2} = \frac{(-\p_1 \zeta_0 - \p_1 \eta,1)}{1 + \abs{\p_1 \zeta_0 + \p_1 \eta}^2}
\end{equation}
we have that 
\begin{equation}\label{ne1_g4_2}
 \norm{ \frac{\mathcal{T}}{\abs{\mathcal{T}}^2}}_{W^{1,q_-}} +  \norm{ \frac{\mathcal{N}}{\abs{\mathcal{N}}^2}}_{W^{1,q_-}} + \norm{\N}_{W^{1,q_-}} \ls 1+ \sqrt{\E}  \ls 1.
\end{equation}
This allows us to employ Theorem \ref{supercrit_prod} with $1 > \max\{1/q_-,1-1/q_1\}$ to estimate 
\begin{multline}
  \norm{F^4 \cdot \frac{\mathcal{T}}{\abs{\mathcal{T}}^2}}_{W^{1-1/q_-,q_-}} +   \norm{F^4 \cdot \frac{\mathcal{N}}{\abs{\mathcal{N}}^2}}_{W^{1-1/q_-,q_-}} 
  \ls  \norm{F^4 }_{W^{1-1/q_-,q_-}} \left( \norm{ \frac{\mathcal{T}}{\abs{\mathcal{T}}^2}}_{W^{1,q_-}} +  \norm{ \frac{\mathcal{N}}{\abs{\mathcal{N}}^2}}_{W^{1,q_-}}  \right) \\
  \ls \norm{F^4 }_{W^{1-1/q_-,q_-}}.  
\end{multline}
We will then estimate the $F^4$ norm on the right term by term to arrive at the stated estimate.  

\textbf{Term: $\mu \sg_{\dt \A} u \N$.}  For this term we first use Theorem \ref{supercrit_prod} and trace theory to bound 
\begin{equation}
 \norm{\mu \sg_{\dt \A} u \N}_{W^{1-1/q_-,q_-}} \ls \norm{\sg_{\dt \A} u}_{W^{1-1/q_-,q_-}(\Sigma)} \norm{\N}_{W^{1,q_-}} \ls \norm{\sg_{\dt \A} u}_{W^{1,q_-}(\Omega)} \norm{\N}_{W^{1,q_-}}.
\end{equation}
Using this, \eqref{ne1_g4_2}, and the estimate from \eqref{ne1_g4_1}, we deduce that 
\begin{equation}
 \norm{\mu \sg_{\dt \A} u \N}_{W^{1-1/q_-,q_-}} \ls \sqrt{\E} \sqrt{\D}. 
\end{equation}

\textbf{Term: $g\eta  \dt \N$.}  For this term we use Theorem \ref{supercrit_prod} to bound 
\begin{equation}
 \norm{g\eta  \dt \N}_{W^{1-1/q_-,q_-}} \ls \norm{\eta}_{W^{1,q_-}} \norm{\p_1 \dt \eta}_{W^{1-1/q_-,q_-}} \ls \norm{\eta}_{W^{1,q_-}} \norm{\dt \eta}_{W^{2-1/q_-,q_-}} \ls \sqrt{\E} \sqrt{\D}. 
\end{equation}

\textbf{Term: $- \sigma \p_1 \left(\frac{\p_1 \eta}{(1+\abs{\p_1 \zeta_0}^2)^{3/2}}  \right) \dt \N $.} We begin be expanding with the product rule and using the fact that $\zeta_0$ is smooth to bound 
\begin{equation}
 \norm{- \sigma \p_1 \left(\frac{\p_1 \eta}{(1+\abs{\p_1 \zeta_0}^2)^{3/2}}  \right) \dt \N}_{W^{1-1/q_-,q_-}} \ls \norm{\p_1 \eta \p_1 \dt \eta}_{W^{1-1/q_-,q_-}} + \norm{\p_1^2 \eta \p_1 \dt \eta}_{W^{1-1/q_-,q_-}}.
\end{equation}
Then Theorem \ref{supercrit_prod} and the fact that $2 < 3 - 1/q_+$ and $q_- < q_+$ allows us to bound 
\begin{equation}
  \norm{\p_1 \eta \p_1 \dt \eta}_{W^{1-1/q_-,q_-}}  \ls   \norm{\p_1 \eta }_{W^{1,q_-}}   \norm{ \p_1 \dt \eta}_{W^{1-1/q_-,q_-}} \ls \norm{\eta }_{W^{2,q_-}}   \norm{\dt \eta}_{W^{2-1/q_-,q_-}} \ls \sqrt{\E} \sqrt{\D}.
\end{equation}
Similarly, Theorem \ref{supercrit_prod} and the bounds $q_- < q_+$ and $3-1/q_- < 3 - 1/q_+$ imply that
\begin{multline}\label{ne1_g4_3}
 \norm{\p_1^2 \eta \p_1 \dt \eta}_{W^{1-1/q_-,q_-}} \ls \norm{\p_1^2 \eta }_{W^{1-1/q_-,q_-}} \norm{  \p_1 \dt \eta}_{W^{1 ,q_-}} \ls \norm{\eta }_{W^{3-1/q_-,q_-}} \norm{   \dt \eta}_{W^{2 ,q_-}} \\
 \ls \norm{\eta }_{W^{3-1/q_+,q_+}} \norm{   \dt \eta}_{W^{2 ,q_-}} \ls \sqrt{\E} \sqrt{\D}.
\end{multline}
Combining these then shows that 
\begin{equation}
 \norm{- \sigma \p_1 \left(\frac{\p_1 \eta}{(1+\abs{\p_1 \zeta_0}^2)^{3/2}}  \right) \dt \N}_{W^{1-1/q_-,q_-}} \ls  \sqrt{\E} \sqrt{\D}. 
\end{equation}

\textbf{Term: $ - \sigma \p_1 \left(  \mathcal{R}(\p_1 \zeta_0,\p_1 \eta) \right) \dt \N $.}  For this term we first expand 
\begin{equation}
\p_1 \left(  \mathcal{R}(\p_1 \zeta_0,\p_1 \eta) \right) = \p_y \mathcal{R}(\p_1 \zeta_0,\p_1 \eta) \p_1^2 \zeta_0 + \p_z \mathcal{R}(\p_1 \zeta_0,\p_1 \eta) \p_1^2 \eta
\end{equation}
and then use Theorem \ref{supercrit_prod} to bound 
\begin{multline}
 \norm{\sigma \p_1 \left(  \mathcal{R}(\p_1 \zeta_0,\p_1 \eta) \right) \dt \N }_{W^{1-1/q_-,q_-}} \ls  \norm{\p_y \mathcal{R}(\p_1 \zeta_0,\p_1 \eta) \p_1^2 \zeta_0}_{W^{1,q_-}} \norm{\p_1 \dt \eta}_{W^{1-1/q_-,q_-}} 
 \\ + \norm{\p_z \mathcal{R}(\p_1 \zeta_0,\p_1 \eta)}_{W^{1,q_-}} \norm{\p_1^2 \eta \p_1 \dt \eta}_{W^{1-1/q_-,q_-}}.
\end{multline}
The fact that $W^{1,q_-}((-\ell,\ell))$ is an algebra, Proposition \ref{R_prop}, and Theorem \ref{catalog_energy} then show that 
\begin{multline}
 \norm{\p_y \mathcal{R}(\p_1 \zeta_0,\p_1 \eta) \p_1^2 \zeta_0}_{W^{1,q_-}} \ls  \norm{\p_y \mathcal{R}(\p_1 \zeta_0,\p_1 \eta) }_{W^{1,q_-}} \ls
  \norm{\abs{\p_1 \eta}^2 }_{L^{q_-}} + \norm{\p_1^2 \eta}_{L^{q-}}  \\
  \ls   \norm{\p_1 \eta}_{L^\infty}^2  + \norm{\eta}_{W^{3-1/q_+,q+}} \ls \sqrt{\E} 
\end{multline}
and 
\begin{equation}
\norm{\p_z \mathcal{R}(\p_1 \zeta_0,\p_1 \eta)}_{W^{1,q_-}} \ls \norm{\p_1 \eta}_{L^{q_-}} +    \norm{\p_1^2 \eta}_{L^{q-}}  \ls   \norm{\p_1 \eta}_{L^\infty}^2  + \norm{\eta}_{W^{3-1/q_+,q+}} \ls \sqrt{\E}.
\end{equation}
Combining these with \eqref{ne1_g4_3} then shows that 
\begin{equation}
  \norm{\sigma \p_1 \left(  \mathcal{R}(\p_1 \zeta_0,\p_1 \eta) \right) \dt \N }_{W^{1-1/q_-,q_-}}
 \ls \sqrt{\E}\left( \norm{\dt \eta}_{W^{2-1/q_-,q_-}}  + \sqrt{\E} \sqrt{\D} \right) \ls  \left(\sqrt{\E} + \E\right) \sqrt{\D}.
\end{equation}

\textbf{Term: $- S_{\A}(p,u)\dt \N $.}  For this term we use Theorem \ref{supercrit_prod}, trace theory, the bound 
\begin{equation}
 \frac{1-\ep_+}{2} + \frac{1-\ep_+}{2} = \frac{2- 2\ep_+}{2} < \frac{2 - \ep_-}{2} = \frac{1}{q_-},
\end{equation}
and H\"older's inequality to estimate
\begin{multline}
 \norm{S_{\A}(p,u)\dt \N }_{W^{1-1/q_-,q_-}} \le  \norm{S_{\A}(p,u) }_{W^{1-1/q_-,q_-}(\Sigma)} \norm{\p_1 \dt \eta}_{W^{1,q_-}} \ls \norm{S_{\A}(p,u) }_{W^{1,q_-}(\Omega)} \norm{\dt \eta}_{W^{2,q_-}} \\
\ls \left(\norm{p}_{W^{1,q_-}} +  \norm{u}_{W^{2,q_-}} + \norm{\abs{\nab \A} \abs{\nab u}}_{L^{q_-}}\right) \norm{\dt \eta}_{W^{2,q_-}}  \\
\ls \left(\norm{p}_{W^{1,q_+}} +  \norm{u}_{W^{2,q_+}} + \norm{\bar{\eta}}_{W^{2,2/(1-\ep_+)}} \norm{\nab u}_{L^{2/(1-\ep_+)}}  \right) \norm{\dt \eta}_{W^{2,q_-}} \ls (\sqrt{\E}+\E) \sqrt{\D}.
\end{multline}

\end{proof}

The next result records the estimates for the $G^6$ term from \eqref{ne1_Gform}.

\begin{prop}\label{ne1_g6}
Let $F^3$ be as in \eqref{dt1_f3}.  Then we have the estimate 
\begin{equation}
 \norm{\p_1 F^3 }_{W^{1-1/q_-,q_-}} \ls  \left( \sqrt{\E} + \E\right) \sqrt{\D}.
\end{equation}
\end{prop}
\begin{proof} 
We begin by expanding
\begin{multline}
\p_1 F^3 = \p_1 \dt [  \mathcal{R}(\p_1 \zeta_0,\p_1 \eta)] = 
 \p_z \p_y \mathcal{R}(\p_1 \zeta_0,\p_1 \eta) \p_1^2 \zeta_0 \p_1 \dt \eta \\
+ \p_z^2  \mathcal{R}(\p_1 \zeta_0,\p_1 \eta) \p_1^2 \eta \p_1 \dt \eta
+ \p_z  \mathcal{R}(\p_1 \zeta_0,\p_1 \eta) \p_1^2 \dt \eta
:= I + II +III.
\end{multline}
Since $1 > \max\{1/q_-,1-1/q_1\}$ we can use Theorem \ref{supercrit_prod} to bound 
\begin{equation}\label{ne1_g6_1}
 \norm{\varphi \psi}_{W^{1-1/q_-,q_-}}\ls \norm{\varphi}_{W^{1,q_-}} \norm{\psi}_{W^{1-1/q_-,q_-}},
\end{equation}
and we will this to handle each of $I$, $II$, and $III$.

We begin with $I$ by using \eqref{ne1_g6_1} twice together with the fact that $q_- < q_+$ to bound 
\begin{multline}
 \norm{I}_{W^{1-1/q_-,q_-}} \ls \norm{\p_z \p_y \mathcal{R}(\p_1 \zeta_0,\p_1 \eta)}_{W^{1,q_-}}  \norm{\p_1^2 \zeta_0 \p_1 \dt \eta}_{W^{1-1/q_-,q_-}}  \\
 \ls  \norm{\p_z \p_y \mathcal{R}(\p_1 \zeta_0,\p_1 \eta)}_{W^{1,q_-}}  \norm{\p_1^2 \zeta_0}_{W^{1,q_-}} \norm{ \p_1 \dt \eta}_{W^{1-1/q_-,q_-}} \ls 
 \norm{\p_z \p_y \mathcal{R}(\p_1 \zeta_0,\p_1 \eta)}_{W^{1,q_-}}   \sqrt{\D}.
\end{multline}
For $II$ we also use \eqref{ne1_g6_1} twice and $q_- < q_+$ to see that 
\begin{multline}
\norm{II}_{W^{1-1/q_-,q_-}}  \ls   \norm{ \p_z^2  \mathcal{R}(\p_1 \zeta_0,\p_1 \eta)}_{W^{1,q_-}} \norm{ \p_1^2 \eta \p_1 \dt \eta}_{W^{1-1/q_-,q_-}} \\
\ls  \norm{ \p_z^2  \mathcal{R}(\p_1 \zeta_0,\p_1 \eta)}_{W^{1,q_-}} \norm{ \p_1^2 \eta }_{W^{1-1/q_-,q_-}}  \norm{\p_1 \dt \eta  }_{W^{1,q_-}} 
\ls \norm{ \p_z^2  \mathcal{R}(\p_1 \zeta_0,\p_1 \eta)}_{W^{1,q_-}} \sqrt{\E} \sqrt{\D}.
\end{multline}
For $III$ we only apply \eqref{ne1_g6_1} once to see that 
\begin{equation}
  \norm{III}_{W^{1-1/q_-,q_-}}\ls \norm{\p_z  \mathcal{R}(\p_1 \zeta_0,\p_1 \eta) }_{W^{1,q_-}} \norm{\p_1^2 \dt \eta}_{W^{1-1/q_-,q_-}} \ls \norm{\p_z  \mathcal{R}(\p_1 \zeta_0,\p_1 \eta) }_{W^{1,q_-}} \sqrt{\D}.
\end{equation}

It remains to handle the $\mathcal{R}$ terms in these estimates.  For this we use Proposition \ref{R_prop} to bound 
\begin{equation}
   \norm{ \p_z^2  \mathcal{R}(\p_1 \zeta_0,\p_1 \eta)}_{W^{1,q_-}} \ls 1 + \norm{\p_1 \eta}_{W^{1,q_-}} \ls 1 + \norm{\eta}_{W^{2,q_-}} \ls 1 + \sqrt{\E}
\end{equation}
and
\begin{equation}
 \norm{\p_z \p_y \mathcal{R}(\p_1 \zeta_0,\p_1 \eta)}_{W^{1,q_-}} + \norm{\p_z  \mathcal{R}(\p_1 \zeta_0,\p_1 \eta) }_{W^{1,q_-}} \ls \norm{\eta}_{W^{2,q_-}} \ls   \sqrt{\E}.
\end{equation}
Combining these bounds with the above, we deduce that 
\begin{equation}
 \norm{\p_1 F^3}_{W^{1-1/q_-,q_-}} \ls \left( \sqrt{\E} + \E\right) \sqrt{\D},
\end{equation}
as desired.

\end{proof}

Finally, we bound the $G^7$ term from \eqref{ne1_Gform}.

\begin{prop}\label{ne1_g7}
Let $F^7$ be as defined by \eqref{dt2_f7}.  Then we have the estimate 
\begin{equation}
 [\linz  F^7]_\ell \ls \sqrt{\E} \sqrt{\D}.
\end{equation}
\end{prop}
\begin{proof}
The definition of $\swh \in C^2$ in \eqref{V_pert} shows that $\abs{\swh'(z)} \ls z$ for $\abs{z} \ls 1$.  From this, standard trace theory, and the definition of $E$ and $\D$ in \eqref{E_def} and \eqref{D_def}, respectively, we may then bound 
\begin{equation}
 [\linz  F^7]_\ell \ls \max_{\pm \ell} \abs{\dt \eta} \abs{\dt^2 \eta} \ls \norm{\dt \eta}_{H^1} [\dt^2 \eta]_{\ell} \ls \sqrt{\E} \sqrt{\D}. 
\end{equation}
\end{proof}

\subsection{Two time derivatives}

We will not apply Theorem  \ref{A_stokes_stress_solve} to the twice time differentiated problem.  However, we will need the following pair of estimates, which are in the same spirit as the above elliptic estimates.  The first gives estimates of $F^2$ from \eqref{dt2_f2}.

\begin{prop}\label{ne2_f2}
Let $F^2$ be given by \eqref{dt2_f2}.  Then we have the estimates
\begin{equation}\label{ne2_f2_01}
\norm{J F^2}_{L^{4/(3-2\ep_+)}}  \ls \E,
\end{equation}
\begin{equation}\label{ne2_f2_02}
\norm{J F^2}_{L^{2/(1-\ep_-)}}  \ls \sqrt{\E} \sqrt{\D},
\end{equation}
and 
\begin{equation}\label{ne2_f2_03}
  \norm{\dt(J F^2)}_{L^{q_-}}  \ls (\sqrt{\E} + \E) \sqrt{\D}.
\end{equation}
\end{prop}
\begin{proof}
First note that \eqref{kappa_ep_def} requires that $0 <3 - 2 \ep_+ <1$, $0 < 1 - (\ep_+-\low) < 1$, and 
\begin{equation}\label{ne2_f2_1}
 \ep_+ \le \frac{1+\ep_-}{2} \Rightarrow \frac{4}{3-2\ep_+} \le \frac{4}{2-\ep_-}.
\end{equation}
Also, from \eqref{dt2_f2} we have that
\begin{equation}
 F^2 = -\diverge_{\dt^2 \A} u - 2\diverge_{\dt \A}\dt u.
\end{equation}
Then from Theorems \ref{catalog_energy} and \ref{catalog_dissipation} and H\"{o}lder's inequality we can bound 
\begin{multline}
 \norm{J \diverge_{\dt^2 \A} u}_{L^{4/(3-2\ep_+)}}    \ls  \norm{\abs{\dt^2 \A} \abs{\nab u}}_{L^{4/(3-2\ep_+)}} \ls \norm{\dt^2 \A}_{L^4} \norm{\nab u}_{L^{2/(1-\ep_+)}} \\
\ls \left(\norm{\dt \bar{\eta}}_{W^{1,4}} + \norm{\dt^2 \bar{\eta}}_{W^{1,4}}    \right) \norm{\nab u}_{L^{2/(1-\ep_+)}} \ls \E
\end{multline}
and, also using \eqref{ne2_f2_1},
\begin{multline}
 \norm{J  \diverge_{\dt \A}\dt u}_{L^{4/(3-2\ep_+)}} \ls  \norm{J  \diverge_{\dt \A}\dt u}_{L^{4/(2-\ep_-)}} \ls  \norm{\abs{\dt \A} \abs{\nab \dt u} }_{L^{4/(2-\ep_-)}} \\
 \ls \norm{\dt \bar{\eta}}_{W^{1,\infty}} \norm{\nab \dt u}_{L^{4/(2-\ep_-)}} \ls \E.
\end{multline}
Thus, \eqref{ne2_f2_01} holds.

To prove \eqref{ne2_f2_02} we argue similarly, first noting that 
\eqref{kappa_ep_def} tells us that
\begin{equation}
0 < \ep_+ - \ep_- -\low  \Rightarrow 1 - (\ep_+-\low) < 1 - \ep_- \Rightarrow \frac{2}{1-\ep_-} < \frac{2}{1-(\ep_+-\low)}.
\end{equation}
Thus, 
\begin{multline}
\norm{J \diverge_{\dt^2 \A} u}_{L^{2/(1-\ep_-)}} \ls 
\norm{J \diverge_{\dt^2 \A} u}_{L^{2/(1-(\ep_+-\low))}}  
\ls  \norm{\abs{\dt^2 \A} \abs{\nab u}}_{L^{2/(1-(\ep_+-\low))}} \\
\ls \norm{\dt^2 \A}_{L^{2/\low}} \norm{\nab u}_{L^{2/(1-\ep_+)}} 
\ls \left(\norm{\dt \bar{\eta}}_{W^{1,2/\low}} + \norm{\dt^2 \bar{\eta}}_{W^{1,2/\low}}    \right) \norm{\nab u}_{L^{2/(1-\ep_+)}} \ls \sqrt{\D} \sqrt{\E}
\end{multline}
and
\begin{equation}
 \norm{J \diverge_{\dt \A}\dt u}_{L^{2/(1-\ep_-)}}   \ls  \norm{\abs{\dt \A} \abs{\nab \dt u}}_{L^{2/(1-\ep_-)}}  \ls \norm{\dt \bar{\eta}}_{W^{1,\infty}} \norm{\nab \dt u}_{L^{2/(1-\ep_-)}} \ls \sqrt{\E} \sqrt{\D},
\end{equation}
and \eqref{ne2_f2_02} follows.

Finally, note that
\begin{equation}
 \abs{\dt(J F^2)} \ls \abs{\dt^3 \A}\abs{\nab u}  + \abs{\dt^2 \A} \abs{\nab \dt u} + \abs{\dt \A} \abs{\nab \dt^2 u} + \abs{\dt J}\left(\abs{\dt^2 \A} \abs{\nab u} + \abs{\dt \A} \abs{\nab \dt u} \right),
\end{equation}
while \eqref{kappa_ep_def} tells us that $0 < \ep_+ - 2 \low < 1$, 
\begin{equation}
\low < \frac{\ep_+-\ep_-}{2} \Rightarrow 
\ep_- < \ep_+ - 2\low 
 \Rightarrow \frac{1-\ep_-}{2} + \frac{1+2\low}{2} = 1 -\frac{(\ep_+-2\low)}{2} < 1 - \frac{\ep_-}{2} = \frac{1}{q_-},
\end{equation}
and 
\begin{equation}
 \frac{1-\ep_-}{2} + \hal  = \frac{2-\ep_-}{2} = \frac{1}{q_-},
\end{equation}
so we may again use Theorems \ref{catalog_energy} and \ref{catalog_dissipation} and H\"{o}lder's inequality to see that
\begin{multline}
 \norm{\dt(J F^2)}_{L^{q_-}}  \ls \norm{\dt^3 \A}_{L^{2/(1+2\low)}} \norm{\nab u}_{L^{2/(1-\ep_+)}} + \norm{\dt^2 \A}_{L^2} \norm{\nab \dt u}_{L^{2/(1-\ep_-))}}  + \norm{\dt \A}_{L^{\infty}} \norm{\nab \dt^2 u}_{L^{2}}  \\
 + \norm{\dt \bar{\eta} }_{W^{1,\infty}} \left( \norm{\dt^2 \A}_{L^2} \norm{\nab   u}_{L^{2/(1-\ep_+))}}   +  \norm{\dt \A}_{L^\infty} \norm{\nab \dt u}_{L^{2/(1-\ep_-)}}   \right) \\
\ls \left(\norm{\dt \bar{\eta}}_{H^1} + \norm{\dt^2 \bar{\eta}}_{H^1} + \norm{\dt^3 \bar{\eta}}_{W^{1,2/(1+2\low)}} \right) \sqrt{\E}    \\
+ \left(\norm{\dt \bar{\eta}}_{H^1} + \norm{\dt^2 \bar{\eta}}_{H^1} + \norm{\dt \bar{\eta}}_{W^{1,\infty}} \right) \sqrt{\D} + \sqrt{\E} \left(\norm{\dt \bar{\eta}}_{H^1} + \norm{\dt^2 \bar{\eta}}_{H^1} + \norm{\dt \bar{\eta}}_{W^{1,\infty}} \right) \sqrt{\D} \\
\ls \sqrt{\D} \sqrt{\E} + \sqrt{\E} \sqrt{\D} + \E \sqrt{\D}.
\end{multline} 
Then \eqref{ne2_f2_03} follows.
\end{proof}

Next we provide a bound for  $F^6$ from \eqref{dt2_f6}.

\begin{prop}\label{ne2_f6}
Let $F^6$ be given by \eqref{dt2_f6} and $\N$ be given by \eqref{N_def}.  Then we have the estimates
\begin{equation}
 \norm{F^6}_{H^{1/2-\low}} \ls \sqrt{\E} \sqrt{\D}
\end{equation}
and 
\begin{equation}
 \norm{\dt^2 u \cdot \N}_{H^{1/2}((-\ell,\ell))} \ls (1+\sqrt{\E}) \norm{\dt^2 u}_{H^1}.
\end{equation}
\end{prop}
\begin{proof}
According to \eqref{dt2_f6}, Theorem \ref{supercrit_prod} with $\hal + \ep_\pm > \max\{\hal,\hal-\low\}$, and trace theory  we have that 
\begin{multline}
 \norm{F^6}_{H^{1/2-\low}} \ls \norm{\dt u_1 \p_1 \dt \eta}_{H^{1/2-\low}} + \norm{u_1 \p_1 \dt^2 \eta}_{H^{1/2-\low}} \\
\ls 
\norm{\dt u_1 }_{H^{1/2-\low}(\Sigma)}
\norm{\p_1 \dt \eta}_{H^{1/2+\ep_-}} 
+ \norm{u_1 }_{H^{1/2+\ep_+}(\Sigma)}
\norm{\p_1 \dt^2 \eta}_{H^{1/2-\low}} \\
\ls 
\norm{\dt u_1 }_{H^{1-\low}(\Omega)}
\norm{\dt \eta}_{H^{3/2+\ep_-}} 
+ \norm{u_1 }_{H^{1+\ep_+}(\Omega)}
\norm{\dt^2 \eta}_{H^{3/2-\low}}.
\end{multline}
Note that 
\begin{equation}
 \frac{1}{2} = \frac{2-\ep_-}{2} - \frac{1}{1} \left( 2 + \frac{\ep_-}{2} - \frac{3}{2} - \ep_- \right) = \frac{1}{q_-} - \frac{1}{1}\left(3 - \frac{1}{q_-} - \frac{3}{2}-\ep_- \right),
\end{equation}
and 
\begin{equation}
 \hal = \frac{2-\ep_+}{2} - \frac{1}{2}\left( 2 - 1 - \ep_+ \right) = \frac{1}{q_+}  - \frac{1}{2} \left( 2 - (1+\ep_+)\right),
\end{equation}
so the Sobolev embeddings show that 
\begin{equation}
 W^{3-1/q_-,q_-}((-\ell,\ell)) \hookrightarrow H^{3/2 + \ep_-}((-\ell,\ell))
\end{equation}
and 
\begin{equation}
 W^{2,q_+}(\Omega) \hookrightarrow H^{1+\ep_+}(\Omega).
\end{equation}
Hence, 
\begin{equation}
 \norm{\dt \eta}_{H^{3/2+\ep_-}} \ls  \norm{\dt \eta}_{W^{3-1/q_-,q_-}} \ls \sqrt{\E}
\end{equation}
and 
\begin{equation}
 \norm{u_1}_{H^{1+\ep_+}(\Omega)} \ls \norm{u}_{W^{2,q_+}} \ls \sqrt{\E}.
\end{equation}
Moreover, since $1 - \low < 1$ and $2 < 2/(1-\ep_-)$ we can use Theorem \ref{catalog_dissipation} to bound 
\begin{equation}
 \norm{\dt u_1}_{H^{1-\low}(\Omega)} \ls  \norm{\dt u_1}_{H^{1}} \ls \sqrt{\D}.
\end{equation}
Thus, upon combining all of these, we deduce that 
\begin{equation}
  \norm{F^6}_{H^{1/2-\low}} \ls \sqrt{\D} \sqrt{\E} + \sqrt{\E} \sqrt{\D},
\end{equation}
as desired.  

For the second estimate we use the fact that $H^{1/2}((-\ell,\ell))$ is an algebra in conjunction with trace theory and the embedding \eqref{nie_ST_3}:
\begin{multline}
\norm{\dt^2 u \cdot \N}_{H^{1/2}((-\ell,\ell))} =   \norm{\dt^2 u\cdot (1,-\p_1 \zeta_0)  }_{H^{1/2}((-\ell,\ell))} +   \norm{\dt^2 u_2 \p_1 \eta}_{H^{1/2}((-\ell,\ell))} \\
\ls \norm{\dt^2 u  }_{H^{1/2}(\Sigma)} \left(1 + \norm{\p_1 \eta}_{H^{1/2}} \right)
\ls \norm{\dt^2 u  }_{H^{1}(\Omega)} \left(1 + \norm{\eta}_{W^{3-1/q_+,q_+}}  \right)
\ls \norm{\dt^2 u}_{H^1} (1+\sqrt{\E}).
\end{multline}
\end{proof}

\section{Functional calculus of the gravity-capillary operator}\label{sec_fnal_calc}

In this section we record some essential properties of the gravity-capillary operator, $\K$, associated to the equilibrium $\zeta_0 : [-\ell,\ell] \to \R$ from Theorem \ref{zeta0_wp}, with gravitational coefficient $g >0$ and surface tension $\sigma >0$.  In particular, we develop the functional calculus associated to $\K$ with Neumann-type boundary conditions, we study a scale of custom Sobolev spaces built in terms of the eigenfunctions of $\K$, and we consider some useful approximations of the fractional differential operator $D^s = (\K)^{s/2}$.

\subsection{Basic spaces and the gravity-capillary operator}

We write the inner-products on $L^2((-\ell,\ell)) = H^0((-\ell,\ell))$ and $H^1((-\ell,\ell))$  via 
\begin{equation}\label{bndry_ips}
 (\varphi,\psi)_{0,\Sigma} := \int_{-\ell}^\ell \varphi \psi \text{ and } (\varphi,\psi)_{1,\Sigma} := \int_{-\ell}^\ell g \varphi \psi + \sigma \frac{\p_1 \varphi \p_1 \psi }{(1+\abs{\p_1 \zeta_0}^2)^{3/2}}.
\end{equation}
It's clear from the properties of $\zeta_0$ stated in  Theorem \ref{zeta0_wp} that the latter generates a norm equivalent to the standard one on $H^1((-\ell,\ell))$ and thus generates the standard topology.  Recall from \eqref{bndry_pairing} that for pairs $\varphi,\psi : \{-\ell,\ell\} \to \R$ we write  
\begin{equation}
 [\varphi,\psi]_\ell =  \varphi(-\ell)\psi(-\ell) + \varphi(\ell)\psi(\ell) \text{ and } [\varphi]_\ell = \sqrt{[\varphi,\varphi]_\ell},
\end{equation}
and we will often slightly abuse this notation by writing $[\varphi,\psi]_\ell$ when either $\varphi$ or $\psi$ is a function on $(-\ell,\ell)$ with well-defined traces, in which case the understanding is that the map on $\{-\ell,\ell\}$ is defined by the trace.

The inner-product gives rise to the following elliptic operator, which we call the gravity-capillary operator associated to $\zeta_0$:
\begin{equation}\label{cap_op_def}
 \K \varphi := g \varphi - \sigma \p_1 \left( \frac{\p_1 \varphi  }{(1+\abs{\p_1 \zeta_0}^2)^{3/2}} \right).
\end{equation}
The associated boundary operators are
\begin{equation}\label{bnrdy_op_def}
 \B_\pm \psi = \pm\frac{\psi'(\pm\ell)}{(1+\abs{\zeta_0'(\pm\ell)}^2)^{3/2}},
\end{equation}
and we write $\B\psi : \{-\ell,\ell\} \to \R$ via $\B\psi(\pm\ell) \B_{\pm} \psi$.   Then $\K$ and $\B$ intertwine our choice of inner-products on $L^2((-\ell,\ell))$ and $H^1((-\ell,\ell))$ via 
\begin{equation}\label{weak_motivation}
 (\varphi,\psi)_{1,\Sigma} = (\K \varphi,\psi)_{0,\Sigma} + [\B\varphi,\psi]_\ell
\end{equation}
for $\varphi,\psi \in C^2([-\ell,\ell]) $.   

We now aim to study the properties of $\K$ and $\B$.  We begin with a version of the Riesz representation.

\begin{thm}\label{riesz_iso}
The map $\J : H^{1}((-\ell,\ell)) \to [H^1((-\ell,\ell))]^\ast$  defined via  $\br{\J \varphi,\psi} = (\varphi,\psi)_{1,\Sigma}$ is an isomorphism. 
\end{thm}
\begin{proof}
This is the Riesz representation theorem. 
\end{proof}

Next we construct a functional related to the form $[\cdot,\cdot]_\ell$.

\begin{lem}\label{H1_trace}
 Suppose that $h_\pm \in \R$ and that we view $h: \{-\ell,\ell\} \to \R$ via $h(\pm \ell) = \pm h$.  Then the map   $H^1((-\ell,\ell)) \ni \psi \mapsto [h,\psi]_\ell$ is bounded and linear.
\end{lem}
\begin{proof}
This follows immediately from the standard trace  estimate $ \max\{\abs{\psi(\ell)},\abs{\psi(-\ell)} \} \ls \norm{\psi}_{1,\Sigma}.$
\end{proof}

We can now consider weak solutions to the problem
\begin{equation}\label{cap_bndry_eqn}
\begin{cases}
\K \varphi = f \\
\B \varphi = h
\end{cases}
\end{equation}
when $f \in [H^1((-\ell,\ell))]^\ast$ and $h: \{-\ell,\ell\} \to \R$ via $h(\pm \ell) = \pm h$.  Theorem \ref{riesz_iso} allows us to define the weak solution to \eqref{cap_bndry_eqn} as the unique $\varphi \in H^1((-\ell,\ell))$ determined by
\begin{equation}
 (\varphi,\psi)_{1,\Sigma} = \br{f,\psi} + [h,\psi]_\ell \text{ for all } \psi \in H^1((-\ell,\ell)).
\end{equation}
Note that according to \eqref{weak_motivation} any classical (or even strong, i.e. $H^2$) solution is a weak solution in the above sense. Moreover, Theorem \ref{riesz_iso} and Lemma \ref{H1_trace} easily imply that
\begin{equation}
 \norm{\varphi}_{1,\Sigma} \ls \norm{f}_{(H^1)^\ast} + [h]_\ell.
\end{equation}

We next show that if $f=0$ then the weak solution is smooth up to the boundary.

\begin{thm}\label{weak_soln_bnrdy_smooth}
Let $h_\pm \in \R$. Then the following hold.
\begin{enumerate}
 \item There exists a unique $\varphi \in H^1((-\ell,\ell))$ such that 
\begin{equation}\label{weak_soln_bnrdy_smooth_01}
 (\varphi,\psi)_{1,\Sigma} = [h,\psi]_\ell \text{ for all } \psi\in H^1((-\ell,\ell)). 
\end{equation}
 \item We have that $\varphi \in H^m((-\ell,\ell))$ for each $m \in \mathbb{N}$, and 
\begin{equation}
 \norm{\varphi}_{H^m} \ls [h]_\ell.
\end{equation}
 \item We have that $\varphi \in C^\infty([-\ell,\ell])$, and $\varphi$ is a classical solution to 
\begin{equation}\label{weak_soln_bnrdy_smooth_02}
\begin{cases}
\K \varphi = 0 &\text{in }(-\ell,\ell) \\
 \B_\pm \varphi   = h_\pm. 
\end{cases}
\end{equation}
\end{enumerate}
\end{thm}
\begin{proof}
The first item  follows from  Lemma \ref{H1_trace} and Theorem \ref{riesz_iso}.  Now consider the function $z \in C^\infty([-\ell,\ell])$ given by 
\begin{equation}
 z(x) =  \frac{1  }{(1+\abs{\zeta_0'(x)}^2)^{3/2}}
\end{equation}
and note that there exists a constant $z_0 >0$ such that 
\begin{equation}\label{weak_soln_bnrdy_smooth_1}
z(x) \ge z_0 \text{ for all }x \in [-\ell,\ell]. 
\end{equation}
The function $z$ allows us to conveniently rewrite 
\begin{equation}
 (\varphi,\psi)_{1,\Sigma} = \int_{-\ell}^\ell z \varphi' \psi' + g \varphi \psi \text{ for all } \psi \in H^1((-\ell,\ell)).
\end{equation}
Let $\psi \in C^\infty_c((-\ell,\ell))$ and note that the bound \eqref{weak_soln_bnrdy_smooth_1} implies that $\chi = \psi / z \in C^\infty_c((-\ell,\ell)) \subset H^1((-\ell,\ell))$.  Plugging this $\chi$ into \eqref{weak_soln_bnrdy_smooth_01} shows that  $(\varphi,\chi)_{1,\Sigma} = [h,\chi]_\ell = 0.$ Thus 
\begin{equation}
 0 = \int_{-\ell}^\ell  z \varphi' \left(\frac{\psi}{z}\right)' + g \varphi \left(\frac{\psi}{z}\right)  = \int_{-\ell}^\ell  \varphi' \psi' - \frac{z'}{z} \varphi' \psi + g \varphi \left(\frac{\psi}{z}\right),
\end{equation}
and upon rearranging we find that 
\begin{equation}
 \int_{-\ell}^\ell \varphi' \psi' = \int_{-\ell}^{\ell} -\left(-\frac{z'}{z} \varphi' + g\frac{\varphi}{z} \right)\psi \text{ for all } \psi \in C_c^\infty((-\ell,\ell)).
\end{equation}
The definition of weak derivatives then tells us that $\varphi'$ is weakly differentiable, and 
\begin{equation}\label{weak_soln_bnrdy_smooth_2}
 \varphi'' = -\frac{z'}{z} \varphi' + g\frac{\varphi}{z} \in L^2((-\ell,\ell)),
\end{equation}
where the latter inclusion follows from the fact that $\varphi \in H^1((-\ell,\ell))$, $z \in C^\infty([-\ell,\ell])$, and the estimate \eqref{weak_soln_bnrdy_smooth_1}.  Thus $\varphi \in H^2((-\ell,\ell))$, and we may estimate 
\begin{equation}
 \norm{\varphi}_{H^2} \ls \norm{\varphi}_{H^1} \ls [h]_\ell.
\end{equation}

Since $z'/z, g/z \in C^\infty([-\ell,\ell])$ we deduce from \eqref{weak_soln_bnrdy_smooth_2} and a simple induction argument that, in fact, $\varphi \in H^m((-\ell,\ell))$ for all $m \in \mathbb{N}$ and 
\begin{equation}
 \norm{\varphi}_{H^m} \le C_m [h]_\ell
\end{equation}
for a constant $C_m$ depending on $m$. Hence $\varphi \in C^\infty([-\ell,\ell])$.  Returning now to \eqref{weak_soln_bnrdy_smooth_01}, we find, upon using $\psi \in C^\infty([-\ell,\ell])$ and integrating by parts, that
\begin{equation}
 [h,\psi]_\ell =  (\varphi,\psi)_{1,\Sigma} = \int_{-\ell}^\ell \K\varphi \psi + [\B \varphi,\psi]_\ell
\end{equation}
for all $\psi \in C^\infty([-\ell,\ell])$.  This immediately implies the identity \eqref{weak_soln_bnrdy_smooth_02}.

\end{proof}

Next we consider elliptic regularity for \eqref{cap_bndry_eqn} with $f \neq 0$.

\begin{thm}\label{cap_regularity}
Let $h_\pm \in \R$ and $f \in H^m((-\ell,\ell))$ for some  $m \in \mathbb{N}$. Suppose that $\varphi \in H^1((-\ell,\ell))$ is the unique weak solution to \eqref{cap_bndry_eqn}.  Then $\varphi \in H^{m+2}((-\ell,\ell))$, and 
\begin{equation}
 \norm{\varphi}_{H^{m+2}} \ls \norm{f}_{H^m} + [h]_\ell.
\end{equation}
Moreover, $\varphi$ is a strong solution to \eqref{cap_bndry_eqn}.
\end{thm}
\begin{proof}

First note that $H^m((-\ell,\ell)) \hookrightarrow H^0((-\ell,\ell)) \hookrightarrow (H^1((-\ell,\ell)) )^\ast,$ where here in the last embedding we inject $H^0$ into $(H^1)^\ast$ in the standard way via 
\begin{equation}
 \br{\varphi,\psi}_\ast = \int_{-\ell}^\ell \varphi \psi = (\varphi,\psi)_{0,\Sigma} \text{ for }\varphi \in H^0((-\ell,\ell)) \text{ and }\psi \in H^1((-\ell,\ell)).
\end{equation}
Consequently, we can use Theorem \ref{riesz_iso} to solve for a unique $\varphi_1 \in H^1((-\ell,\ell))$ satisfying 
\begin{equation}\label{cap_regularity_1}
 (\varphi_1,\psi)_{1,\Sigma} = (f,\psi)_{0,\Sigma}
\end{equation}
and obeying the estimate 
\begin{equation}
 \norm{\varphi_1}_{H^1} \ls \norm{f}_{(H^1)^\ast} \ls \norm{f}_{H^0}.
\end{equation}

On the other hand, Theorem \ref{weak_soln_bnrdy_smooth} provides us with a unique $\varphi_2 \in C^\infty([-\ell,\ell])$ satisfying 
\begin{equation}
 (\varphi_2,\psi)_{1,\Sigma} = [h,\psi]_\ell \text{ for all }\psi \in H^1((-\ell,\ell)).  
\end{equation}
The theorem tells us that 
\begin{equation}
 \norm{\varphi_2}_{H^k} \le C_k [h]_\ell \text{ for all }k \in \N.
\end{equation}
By the uniqueness of weak solutions, we have that $\varphi = \varphi_1 + \varphi_2$.  To conclude we must only show that $\varphi_1 \in H^{m+2}((-\ell,\ell))$ and 
\begin{equation}\label{cap_regularity_3}
 \norm{\varphi_1}_{H^{m+2}} \ls \norm{f}_{H^m}.
\end{equation}

Let $z \in C^\infty([-\ell,\ell])$ be as in the proof of Theorem \ref{weak_soln_bnrdy_smooth}.  For $\psi \in C^\infty_c((-\ell,\ell))$ we have that $\chi = \psi/z \in C^\infty_c((-\ell,\ell))$, and so we can use $\chi \in H^1((-\ell,\ell))$ as a test function in \eqref{cap_regularity_1}; after rearranging, we find that 
\begin{equation}
  \int_{-\ell}^\ell \varphi_1' \psi'    = -  \int_{-\ell}^\ell \left(-\frac{z'}{z} \varphi_1' + g\frac{\varphi_1}{z} - \frac{f}{z} \right)\psi \text{ for all } \psi \in C^\infty_c((-\ell,\ell)).
\end{equation}
From the definition of weak derivatives we then find that $\varphi_1'$ is weakly differentiable, and
\begin{equation}\label{cap_regularity_2}
 \varphi_1'' = -\frac{z'}{z} \varphi_1' + g\frac{\varphi_1}{z} + \frac{f}{z}\in L^2((-\ell,\ell)),
\end{equation}
which implies that $\varphi_1 \in H^2((-\ell,\ell))$ and 
\begin{equation}
 \norm{\varphi_1}_{H^2} \ls \norm{\varphi_1}_{H^1} + \norm{f}_{L^2} \ls \norm{f}_{H^0}.
\end{equation}
This proves \eqref{cap_regularity_3} when $m =0$.  When $m \ge 1$ we use a finite iteration in \eqref{cap_regularity_2} to bootstrap from $\varphi_1 \in H^2((-\ell,\ell))$ to $\varphi_1 \in H^{m+2}((-\ell,\ell))$.  Along the way we readily deduce that \eqref{cap_regularity_3} holds.  Thus the desired inclusion and estimates for $\varphi_1$ hold for all $m \in\mathbb{N}$.

\end{proof}

\subsection{Eigenfunctions of the gravity-capillary operator}

The map 
\begin{equation}
H^0((-\ell,\ell)) \ni f \mapsto \varphi_f \in H^1((-\ell,\ell)) \csubset H^0((-\ell,\ell)),
\end{equation}
where $\varphi_f$  is uniquely determined by 
\begin{equation}
 (\varphi_f,\psi)_{1,\Sigma} = (f,\psi)_{0,\Sigma} \text{ for all } \psi \in H^1((-\ell,\ell))
\end{equation}
is easily seen to be compact and symmetric, so the usual spectral theory of compact symmetric operators (see, for instance, Chapter VI of \cite{reed_simon}) allows us to produce sequences $\{w_k\}_{k=0}^\infty \subset C^\infty([-\ell,\ell])$ and $\{\lambda_k\}_{k=0}^\infty \subset (0,\infty)$ such that the following hold.
\begin{enumerate}
 \item $\{w_k\}_{k=0}^\infty$ forms an orthonormal basis of $L^2((-\ell,\ell))$.
 \item $\{w_k/\sqrt{\lambda_k}\}_{k=0}^\infty$ forms an orthonormal basis of $H^1((-\ell,\ell))$ relative to the inner-product $(\cdot,\cdot)_{1,\Sigma}$.
 \item $\lambda_0 = g$ and $w_0 = 1/\sqrt{2\ell}$.
 \item $\{\lambda_k\}_{k=0}^\infty$ is non-decreasing, and $\lambda_k \to \infty$ as $k \to \infty$.
 \item For each $k \in \mathbb{N}$ we have that
\begin{equation}
\begin{cases}
 \K w_k = \lambda_k w_k &\text{in } (-\ell,\ell) \\
 \B_\pm w_k =0.
\end{cases}
\end{equation}
In other words, $w_k$ is the $k^{th}$ eigenfunction of the operator $\K$ with associated eigenvalue $\lambda_k \ge g$.
\end{enumerate}

We next introduce the notation for ``Fourier'' coefficients relative to this basis.

\begin{dfn}
For a function $f \in H^0((-\ell,\ell))$ we define the map $\hat{f} : \mathbb{N} \to \R$ via $\hat{f}(k) = (f,w_k)_{0,\Sigma}$.  The values of $\hat{f}$ are called the Fourier coefficients of $f$.
\end{dfn}

We have the following version of Parseval's theorem for this basis.

\begin{prop}\label{eigen_bases}
The following hold.
\begin{enumerate}
 \item For each $f,g \in H^0((-\ell,\ell))$ we have that 
\begin{equation}
 (f,g)_{0,\Sigma} = \sum_{k=0}^\infty \hat{f}(k) \hat{g}(k) \text{ and } 
 \ns{f}_{0,\Sigma} = \sum_{k=0}^\infty \abs{\hat{f}(k)}^2.
\end{equation}

 \item For each $f,g \in H^1((-\ell,\ell))$ we have that 
\begin{equation}
 (f,g)_{1,\Sigma} = \sum_{k=0}^\infty \lambda_k \hat{f}(k) \hat{g}(k) \text{ and }
 \ns{f}_{1,\Sigma} = \sum_{k=0}^\infty \lambda_k \abs{\hat{f}(k)}^2.
\end{equation}
\end{enumerate}
\end{prop}
\begin{proof}
The first item follows from the fact that $\{w_k\}_{k=0}^\infty$ is an orthonormal basis of $H^0((-\ell,\ell))$.   The second follows from the fact that $\{w_k/\sqrt{\lambda_k}\}_{k=0}^\infty$ is an orthonormal basis of $H^1((-\ell,\ell))$ and the fact that $w_k$ satisfies $(w_k,f)_{1,\Sigma} = \lambda_k (w_k,f)_{0,\Sigma} = \lambda_k \hat{f}(k)$ for $f \in H^1((-\ell,\ell))$.
\end{proof}

\subsection{Sobolev spaces for the gravity-capillary operator }

In what follows we will often make reference to the vector space
\begin{equation}
 W = \spn\{w_k\}_{k=0}^\infty = \{ \sum_{k=0}^M a_k w_k \st  M \in \mathbb{N} \text{ and } a_0,\dotsc,a_M \in \R\},
\end{equation}
the set of finite linear combinations of basis elements.  Clearly, $W \subset C^\infty([-\ell,\ell])$.  We now define a special scale of Sobolev spaces built from the eigenfunctions of $\K$.

\begin{dfn}
Let $s \in \R$ and recall that   $W = \spn\{w_k\}_{k=0}^\infty$. 

\begin{enumerate}
 \item For $u,v \in W \subset L^2((-\ell,\ell))$ we define 
\begin{equation}
 \ip{u,v}_{\h^s_\K} = \sum_{k=0}^\infty \lambda_k^s (u,w_k)_{0,\Sigma} (v,w_k)_{0,\Sigma} = \sum_{k=0}^\infty \lambda_k^s \hat{u}(k) \hat{v}(k),
\end{equation}
which is clearly an inner-product with associated norm $\ns{u}_{\h^s_\K} = \ip{u,u}_{\h^s_\K}.$

 \item We define the Hilbert space 
\begin{equation}
 \h^s_\K((-\ell,\ell)) = \text{closure}_{\h^s_\K}(W).
\end{equation}

 \item  We define 
\begin{equation}
 \ell^2_s(\mathbb{N}) = \{ f: \mathbb{N} \to \R \st \sum_{k=0}^\infty \lambda_k^s \abs{f(k)}^2 < \infty\},
\end{equation}
which is clearly a Hilbert space when endowed with the obvious inner-product.
 
\end{enumerate}

\end{dfn}

We now characterize these spaces.  

\begin{thm}\label{l2_char}
The following are equivalent for $s \in \R$.
\begin{enumerate}
 \item $u \in \h^s_\K((-\ell,\ell))$.
 \item There exists $\hat{u} \in \ell^2_s(\mathbb{N})$ such that $u = \sum_{k=0}^\infty \hat{u}(k) w_k$, where the series converges with respect to the norm $\norm{\cdot}_{\h^s_\K}.$ 
\end{enumerate}
In either case we have that $\norm{u}_{\h^s_\K} = \norm{\hat{u}}_{\ell^2_s}.$
\end{thm}
\begin{proof} 
Suppose that $u \in \h^s_\K((-\ell,\ell))$.  Then there exist $\{u_m\}_{m=0}^\infty \subseteq W$ such that $u_m \to u$ in $\h^s_\K((-\ell,\ell))$ as $m \to \infty$.  For each $m$ we may then write 
\begin{equation}
 u_m = \sum_{k=0}^\infty a_{m}(k) w_k,
\end{equation}
where $\{a_{m}(k)\}_{k=1}^\infty \subset \R$ vanishes for all but finitely many $k$.  Then 
\begin{equation}
 \ns{u_m - u_j}_{\h^s_\K} = \sum_{k=0}^\infty \lambda_k^s \abs{a_{m}(k) - a_{j}(k)}^2,
\end{equation}
and hence 
\begin{equation}
 \abs{a_{m}(k) - a_{j}(k)}^2 \le \lambda_k^{-s}  \ns{u_m - u_j}_{\h^s_\K} \text{ for all } k \in \mathbb{N}^+.
\end{equation}
This  implies that for each $k \in \mathbb{N}$ we have that $\{a_{m}(k)\}_{m=0}^\infty$ is a Cauchy sequence in $\R$, and hence we may define $a: \mathbb{N} \to \R$ via $a(k) = \lim_{m \to \infty} a_m(k)$.

Now, for each $K \in \mathbb{N}$ we may estimate
\begin{equation}
 \sum_{k=0}^K \lambda_k^s \abs{a(k)}^2 = \lim_{m \to \infty}   \sum_{k=0}^K \lambda_k^s \abs{a_{m}(k)}^2 \le \limsup_{m\to \infty} \sum_{k=0}^\infty \lambda_k^s \abs{a_{m}(k)}^2 = \limsup_{m\to \infty} \ns{u_m}_{\h^s_\K} = \ns{u}_{\h^s_\K}.
\end{equation}
Upon sending $K \to \infty$ we then deduce that $a \in \ell^2_s(\mathbb{N})$.   For $m \in \mathbb{N}$ we then set  $v_m = \sum_{k=0}^{m} a(k) w_k \in W$.  Then for $m > j \ge 0$ we have that
\begin{equation}
 \ns{v_m - v_j}_{\h^s_\K} = \sum_{k=j+1}^m \lambda_k^s \abs{a(k)}^2,
\end{equation}
which then implies that $\{v_m\}_{m=0}^\infty$ is a Cauchy sequence in $\h^s_\K((-\ell,\ell))$, and hence convergent to 
\begin{equation}
 v = \sum_{k=0}^\infty a(k) w_k \in \h^s_\K((-\ell,\ell)).
\end{equation}
Moreover, 
\begin{equation}
 \ns{v}_{\h^s_\K} = \lim_{m \to \infty} \ns{v_m}_{\h^s_\K} = \lim_{m \to \infty} \sum_{k=0}^m \lambda_k^s \abs{a(k)}^2 = \sum_{k=0}^\infty \lambda_k^s \abs{a(k)}^2 = \ns{a}_{\ell^2_s}.
\end{equation}

Let $\ep >0$ and choose $M \in \N$ such that $j,m \ge M$ imply that $\norm{u_j - u_m}_{\h^s_\K} < \ep$.  Then for each $K \in \mathbb{N}$ and $m \ge M$ we have that
\begin{multline}
 \sum_{k=0}^K \lambda_k^s \abs{a(k) - a_{m}(k)}^2 = \lim_{j \to \infty}   \sum_{k=0}^K \lambda_k^s \abs{a_{j}(k) - a_{m}(k)}^2 \le \limsup_{j \to \infty} \sum_{k=0}^\infty \lambda_k^s \abs{a_{j}(k)-a_{m}(k)}^2 \\
 = \limsup_{j\to \infty} \ns{u_j - u_m}_{\h^s_\K} \le \ep^2.
\end{multline}
Sending $K \to \infty$, we then find that $m \ge M$ implies that 
\begin{equation}
 \sum_{k=0}^\infty \lambda_k^s \abs{a(k) - a_{m}(k)}^2 \le \ep^2.
\end{equation}
For any fixed $m$ we have that  
\begin{equation}
 \ns{u_m - v}_{\h^s_\K} = \lim_{K \to \infty} \ns{u_m - v_K}_{\h^s_\K} =  \sum_{k=0}^\infty \lambda_k^s \abs{a(k) - a_{m}(k)}^2.
\end{equation}
Then for $m \ge M$ we find that 
\begin{equation}
 \norm{u_m-v}_{\h^s_\K} \le \ep,
\end{equation}
and consequently, $u_m \to v$ as $m \to \infty$.  Thus $u =v$.  This completes the proof that $(1) \Rightarrow (2)$.

We now turn to the proof of the converse.  Suppose that $u = \sum_{k=0}^\infty \hat{u}(k) w_k$ for $\hat{u} \in \ell^2_s(\mathbb{N})$.  For $m \in \mathbb{N}$ we define $u_m = \sum_{k=0}^m \hat{u}(k) w_k \in W$.  Then $u_m \to u$ as $m \to \infty$ by assumption, and so $u \in \h^s_\K((-\ell,\ell))$.  Moreover, 
\begin{equation}
 \ns{u_m}_{\h^s_\K} = \sum_{k=0}^m \lambda_k^s \abs{\hat{u}(k)}^2
\end{equation}
and hence 
\begin{equation}
 \ns{u}_{\h^s_\K} = \lim_{m\to \infty} \ns{u_m}_{\h^s_\K} = \sum_{k=0}^\infty \lambda_k^s \abs{\hat{u}(k)}^2 = \ns{\hat{u}}_{\ell^2_s}.
\end{equation}
\end{proof}

This theorem suggests some notation.

\begin{dfn}
 To each $u \in \h^s_\K((-\ell,\ell))$ we associate a unique element $\hat{u} \in \ell^2_s(\mathbb{N})$ such that $u = \sum_{k=0}^\infty \hat{u}(k) w_k$ and $\norm{u}_{\h^s_\K} = \norm{\hat{u}}_{\ell^2_s}$.
\end{dfn}

Now we characterize the duals of the spaces we've built.

\begin{thm}\label{dual_char}
Let $s\in \R$.  Then the map $J : \h^{-s}_\K((-\ell,\ell)) \to (\h^s_\K((-\ell,\ell)))^\ast$ defined by  
\begin{equation}
 \br{Ju,v} = \sum_{k=0}^\infty \hat{u}(k) \hat{v}(k)
\end{equation}
is well-defined and is an isometric isomorphism.  Consequently, we have a canonical identification
\begin{equation}
(\h^s_\K((-\ell,\ell)))^\ast = \h^{-s}_\K((-\ell,\ell)).
\end{equation}
\end{thm}
\begin{proof}
The linearity of $J$ is trivial.  The boundedness follows from the estimate
\begin{equation}
 \abs{\br{Ju,v}} = \abs{\sum_{k=0}^\infty \lambda_k^{-s/2} \hat{u}(k)  \lambda_k^{s/2} \hat{v}(k) } \le \norm{\hat{u}}_{\ell^2_{-s}} \norm{\hat{v}}_{\ell^2_s} = \norm{u}_{\h^{-s}_\K} \norm{v}_{\h^s_\K}.
\end{equation}

Suppose that $Ju =0$ for some $u \in \h^{-s}_\K$.  Then 
\begin{equation}
 0 = \sum_{k=0}^\infty \hat{u}(k) \hat{v}(k) \text{ for all } v \in \h^s_\K.
\end{equation}
We may choose $v = w_j \in \h^s_\K((-\ell,\ell))$ for each $j \in \mathbb{N}$, and then $\hat{v}(k) = \delta_{kj}$, which means that 
\begin{equation}
 0 = \hat{u}(j) \text{ for all } j \in \mathbb{N}.
\end{equation}
Then $\norm{u}_{\h^s_\K} =0$, and so $u =0$, from which we deduce that $J$ is injective.
 
Now suppose that $F \in  (\h^s_\K((-\ell,\ell)))^\ast$.  Then we may define $\hat{F} \in (\ell^2_s(\mathbb{N}))^\ast$ via 
\begin{equation}
 \br{\hat{F},\hat{v}} = \br{F,v}.
\end{equation}
Since we have the canonical identification $(\ell^2_s(\mathbb{N}))^\ast = \ell^2_{-s}(\mathbb{N})$ we then deduce that there exists $\hat{u} \in \ell^2_{-s}(\mathbb{N})$ such that 
\begin{equation}
 \br{\hat{F},\hat{v}} = \sum_{k=0}^\infty \hat{v}(k) \hat{u}(k).
\end{equation}
Letting $u = \sum_{k=0}^\infty \hat{u}(k) w_k \in \h^{-s}_\K((-\ell,\ell))$, we find that
\begin{equation}
 \br{F,v} = \br{\hat{F},\hat{v}} = \br{J u,v} \text{ for all } v \in \h^s_\K((-\ell,\ell)),
\end{equation}
and hence $F = Ju$.  Thus $J$ is surjective.

It remains only to show that $J$ is an isometry.  For this we compute
\begin{equation}
 \norm{Ju}_{(\h^s_\K)^\ast} = \sup_{\norm{v}_{\h^s_\K} \le 1} \br{Ju,v}  = \sup_{\norm{\hat{v}}_{\ell^2_s} \le 1} \sum_{k=0}^\infty \hat{u}(k) \hat{v}(k) = \norm{\hat{u}}_{(\ell^2_s)^\ast} = \norm{\hat{u}}_{\ell^2_{-s}} = \norm{u}_{\h^{-s}_\K}. 
\end{equation}

\end{proof}

With this result in hand we can more explicitly describe the map $u \mapsto \hat{u}$.

\begin{thm}
The following hold for $s \in \R$.
\begin{enumerate}
 \item If $s \ge 0$ and $u \in \h^s_\K((-\ell,\ell))$, then $\hat{u}(k) = \ip{u,w_k}_{L^2}$ for all $k \ge 0$.
 \item If $s < 0$ and $u \in \h^s_\K((-\ell,\ell))$, then $\hat{u}(k) = \br{u,w_k}$ for all $k \ge 0$, which is well-defined since $w_k \in \h^{-s}_\K((-\ell,\ell))$ and $\h^s_\K((-\ell,\ell)) = (\h^{-s}_\K((-\ell,\ell)))^\ast$.
\end{enumerate}
\end{thm}
\begin{proof}
If $s \ge 0$ and $u \in \h^s_\K((-\ell,\ell))$, then $u \in L^2((-\ell,\ell))$.  Since $u = \sum_{k=0}^\infty \hat{u}(k) w_k$ with the series converging in $\h^s_\K((-\ell,\ell))$ and hence in $L^2$ we may then compute 
\begin{equation}
 \ip{u,w_k}_{L^2} = \ip{\sum_{j=0}^\infty \hat{u}(j) w_j ,w_k}_{L^2} = \hat{u}(k).
\end{equation}
This proves the first item.

Now assume that $s <0$ and $u \in \h^s_\K((-\ell,\ell))$.  Then Theorem \ref{dual_char} tells us that
\begin{equation}
 \br{u,w_k} = \sum_{j=0}^\infty \hat{u}(j) \delta_{jk} = \hat{u}(k),
\end{equation}
which proves the second item.

\end{proof}

We now record the nesting properties of these Sobolev spaces.

\begin{thm}\label{inclusion}
For $s,t \in \R$ with $s \le t$ we have that $\h^t_\K((-\ell,\ell)) \subseteq \h^s_\K((-\ell,\ell))$ and 
\begin{equation}
 \norm{u}_{\h^s_\K} \le \frac{1}{\lambda_1^{(t-s)/2}} \norm{u}_{\h^t_\K} \text{ for all } u\in \h^t_\K((-\ell,\ell)).
\end{equation}
\end{thm}
\begin{proof}
This follows immediately from the definition of the norm on these spaces.
\end{proof}

Next we record some finer information about  these spaces.  In fact, this result is the key link to the usual theory of Sobolev spaces.

\begin{thm}\label{sobolev_char}
The following hold.
\begin{enumerate}
 
 \item We have that $\h^0_\K((-\ell,\ell)) = L^2((-\ell,\ell))$ and $\norm{u}_{0,\Sigma} = \norm{u}_{\h^0_\K}$ for all $u \in L^2((-\ell,\ell))$.
 \item We have that $\h^1_\K((-\ell,\ell)) = H^1((-\ell,\ell))$ and $\norm{u}_{1,\Sigma} = \norm{u}_{\h^1_\K}$ for all $u \in H^1((-\ell,\ell))$.
 \item Let $2\le m \in \mathbb{N}$.  Then 
\begin{equation}\label{sobolev_char_0}
\h^m_\K((-\ell,\ell)) =  \{ u \in H^m((-\ell,\ell)) \;\vert\;   (\K^{(r)} u)'(\pm \ell)=0 \text{ for }0\le r \le m/2-1   \},
\end{equation}
where $\K^{(0)} = I$ and $\K^{(r+1)} = \K \K^{(r)}$.  Moreover,  $\norm{\cdot}_{\h^m_\K}$ and $\norm{\cdot}_{H^m}$ are equivalent on these spaces.

\end{enumerate}

\end{thm}

\begin{proof}

The first two assertions follow easily from the properties of the eigenfunctions $\{w_k\}_{k=0}^\infty$, so we'll only prove the third item.  Throughout the proof we'll let $X^m$ denote the space on the right side of \eqref{sobolev_char_0}.  We proceed by induction, starting with the base cases $m=2$ and $m=3$.

First consider the case $m=2$.  Suppose that $u \in \h^{2}_\K((-\ell,\ell))$.  We may then define the function $U = \sum_{k=0}^\infty \lambda_k \hat{u}(k) w_k$, which belongs to $\h^0_\K((-\ell,\ell)) =L^2((-\ell,\ell))$ since 
\begin{equation}
 \ns{U}_{L^2} = \sum_{k=0}^\infty \abs{\lambda_k \hat{u}(k)}^2 = \ns{u}_{\h^2_\K} < \infty.
\end{equation}
Since $u \in \h^1_\K((-\ell,\ell))$ we also know that for $v \in H^1((-\ell,\ell))$, 
\begin{equation}
 (u,v)_{1,\Sigma} = \sum_{k=0}^\infty \lambda_k \hat{u}(k) \hat{v}(k) = \sum_{k=0}^\infty \hat{U}(k) \hat{v}(k)= (U,v)_{0,\Sigma},
\end{equation}
and hence $u$ is a weak solution to the problem 
\begin{equation}
 \begin{cases}
 \K u = U &\text{in }(-\ell,\ell) \\
 \B_\pm u =0.
 \end{cases}
\end{equation}
The elliptic regularity of Theorem \ref{cap_regularity} then tells us that $u \in H^2((-\ell,\ell))$ and $\norm{u}_{H^2} \asymp \norm{U}_{L^2} = \norm{u}_{\h^2_\K}$, from which we deduce that $\h^2_\K((-\ell,\ell)) \subseteq X^2$.

Now suppose that $u \in X^2$.  Then clearly $\K u \in L^2$, and we may compute
\begin{equation}
 \ns{u}_{H^2} \asymp \ns{\K u}_{H^0} = \sum_{k=0}^\infty \abs{(\K u,w_k)_{0,\Sigma}}^2 =   \sum_{k=0}^\infty \abs{( u,\K w_k)_{0,\Sigma}}^2 =  \sum_{k=0}^\infty \lambda_k^2 \abs{\hat{u}(k)}^2 = \ns{u}_{\h^2_\K},
\end{equation}
from which we deduce that $X^2 \subseteq \h^2_\K((-\ell,\ell))$.  A similar argument works for the case $m=3$; we omit the details for the sake of brevity.  This establishes the base case $m=2$ and $m=3$.  

Suppose now that the result has been proved for all $2 \le m$ for some $m \ge 3$.   Let  $u \in \h^{m+1}_\K((-\ell,\ell))$.  Using the same $U$ as above, we find that $\norm{U}_{\h^{m-1}_\K} = \norm{u}_{\h^{m+1}_\K}$, and so the induction hypothesis tells us that $U \in X^{m-1}$ with $\norm{U}_{\h^{m-1}_\K} \asymp \norm{U}_{\oH^{m-1}}$.  We then use elliptic regularity as above to see that $u \in X^{m+1}$ and $\norm{u}_{H^{m+1}} \asymp \norm{U}_{H^{m-1}} = \norm{u}_{\h^{m+1}_\K}$, which in turn shows that $\h^{m+1}_\K((-\ell,\ell)) \subseteq X^{m+1}$.

On the other hand, if $u \in X^{m+1}$ then elliptic regularity and the induction hypothesis show that 
\begin{multline}
 \ns{u}_{H^{m+1}} \asymp \ns{\K u}_{H^{m-1}} = \sum_{k=0}^\infty \lambda_k^{m-1} \abs{(\K u,w_k)_{0,\Sigma}}^2 =   \sum_{k=0}^\infty \lambda_k^{m-1} \abs{( u,\K w_k)_{0,\Sigma}}^2 \\
 =  \sum_{k=0}^\infty \lambda_k^{m+1} \abs{\hat{u}(k)}^2 = \ns{u}_{\h^{m+1}_\K}.
\end{multline}
We then deduce that $X^{m+1} \subseteq \h^{m+1}_\K((-\ell,\ell))$.

The principle of induction now tells us that $\h^{m}_\K((-\ell,\ell)) = X^m$ for all $m \ge 2$ and that the norms $\norm{\cdot}_{\h^m_\K}$ and $\norm{\cdot}_{H^m}$ are equivalent on these spaces. 
 
\end{proof}

Theorem \ref{inclusion} shows that we have the nesting $\h^t_\K((-\ell,\ell)) \subseteq \h^s_\K((-\ell,\ell))$ for $s < t$.  In fact, we can show more.

\begin{thm}\label{basic_interp}
Suppose that $s,t \in \R$ are such that $s < t$.  If $u \in \h^s_\K((-\ell,\ell)) \cap \h^t_\K((-\ell,\ell))$, then $u \in \h^r_\K((-\ell,\ell))$ for all $s \le r \le t$, and 
\begin{equation}
 \norm{u}_{\h^r_\K} \le \norm{u}_{\h^s_\K}^\theta \norm{u}_{\h^t_\K}^{1-\theta}
\end{equation}
for $\theta \in [0,1]$ given by 
\begin{equation}
 r = s \theta + t (1-\theta).
\end{equation}
\end{thm}
\begin{proof}
The result is trivial if $r =s$ or $r=t$, so we may assume that $s < r < t$.  We know that $\hat{u} \in \ell^2_s(\mathbb{N}) \cap \ell^2_t(\mathbb{N})$.  For $K \ge 1$ we may use H\"older's inequality to estimate
\begin{multline}
\sum_{k=0}^K \lambda_k^r \abs{\hat{u}(k)}^2 = \sum_{k=0}^K \lambda_k^{\theta s} \abs{\hat{u}(k)}^{2\theta}  \lambda_k^{(1-\theta) t} \abs{\hat{u}(k)}^{2(1-\theta)}
 \le 
 \left(\sum_{k=0}^K \lambda_k^s \abs{\hat{u}(k)}^{2} \right)^\theta \left(\sum_{k=0}^K \lambda_k^t \abs{\hat{u}(k)}^{2} \right)^{1-\theta} 
\\
\le 
 \left(\sum_{k=0}^\infty \lambda_k^s \abs{\hat{u}(k)}^{2} \right)^\theta \left(\sum_{k=0}^\infty \lambda_k^t \abs{\hat{u}(k)}^{2} \right)^{1-\theta} = \left( \ns{u}_{\h^s_\K}\right)^\theta  \left( \ns{u}_{\h^t_\K}\right)^{1-\theta}.
\end{multline}
Upon sending $K \to \infty$ we find that $u \in \h^r_\K((-\ell,\ell))$ and 
\begin{equation}
 \ns{u}_{\h^r_\K} = \sum_{k=0}^\infty \lambda_k^r \abs{\hat{u}(k)}^2 \le \left( \ns{u}_{\h^s_\K}\right)^\theta  \left( \ns{u}_{\h^t_\K}\right)^{1-\theta}.
\end{equation}
The result follows by taking square roots.

\end{proof}

\subsection{Functional calculus }

We can use the eigenvalues to define a functional calculus of $\K$.  First we need some notation.

\begin{dfn}\label{fnal_calc_def}
Write $\Sigma_\K = \{\lambda_k \st k \ge 0\} \subset [g,\infty)$.   For $r \in \R$ define the space
\begin{equation}
\mathfrak{B}^r(\Sigma_\K) = \{ f: \Sigma_\K \to \R \st \norm{f}_{\B^r} < \infty\}
\end{equation}
where 
\begin{equation}
 \norm{f}_{\mathfrak{B}^r} = \sup_{x \ge \lambda_0} \frac{\abs{f(x)}}{x^r}.
\end{equation}
This is easily shown to be a Banach space.  Similarly, for $r \in \R$ define 
\begin{equation}
\mathfrak{B}^r_0(\Sigma_\K) = \{ f \in \mathfrak{B}^r(\Sigma_\K) \st  \lim_{x \to \infty} (\abs{f(x)} / x^r) =0  \},
\end{equation}
which is again easily shown to be a Banach space.
\end{dfn}

Now we define a functional calculus of $\K$ on the spaces $\h^s_\K((-\ell,\ell))$.

\begin{dfn}
Let $s \in \R$ and $r \in \R$.  For $f \in \mathfrak{B}^r(\Sigma_\K)$ and $u \in \h^{s+2r}_\K((-\ell,\ell))$ define 
\begin{equation}
 f(\K)u = \sum_{k=0}^\infty f(\lambda_k) \hat{u}(k) w_k.
\end{equation}
\end{dfn}

The next result records the key properties of these operators.

\begin{thm}
 Let $s \in \R$ and $r \in \R$. For $f \in \mathfrak{B}^r(\Sigma_\K)$ and  $u \in \h^{s+2r}_\K((-\ell,\ell))$ let $f(\K) u$ be as defined above.  Then the following hold.
\begin{enumerate}
 \item $f(\K) : \h^{s+2r}_\K((-\ell,\ell)) \to \h^{s}_\K((-\ell,\ell))$ is bounded and linear.
 \item $f(\K)$ is self-adjoint in the sense that if $u,v \in \h^{s+2r}_\K((-\ell,\ell))$, then
\begin{equation}
 \ip{f(\K) u,v}_{\h^s_\K}  =  \ip{ u,f(\K) v}_{\h^s_\K}. 
\end{equation}

 \item The map
\begin{equation}
\mathfrak{B}^r(\Sigma_\K) \ni f \mapsto f(\K) \in \mathcal{L}(\h^{s+2r}_\K((-\ell,\ell)), \h^{s}_\K((-\ell,\ell))) 
\end{equation}
is bounded and linear.

 \item  If  $f \in \mathfrak{B}^r_0(\Sigma_\K)$, then $f(\K) : \h^{s+2r}_\K((-\ell,\ell)) \to \h^s_\K((-\ell,\ell))$ is a compact operator.

\end{enumerate}

\end{thm}

\begin{proof}
 
The first three assertions are elementary, so we'll only prove the fourth.  To prove this we will show that $f(\K)$ is the limit (in the strong operator norm topology) of a sequence of finite rank operators (see, for instance, Chapter VI of \cite{reed_simon}).  To this end, for each $j \ge 0$ define $F_j :\h^{s+2r}_\K((-\ell,\ell)) \to \h^s_\K((-\ell,\ell))$ via 
\begin{equation}
 F_j u = \sum_{k=0}^j f(\lambda_k) \hat{u}(k) w_k.
\end{equation}
It's clear that each $F_j$ is bounded, linear, and of finite rank.  Also,   for $u \in \h^{s+2r}_\K((-\ell,\ell))$ and $j \ge 0$ we have that
\begin{equation}
 \ns{(F_j - f(\K))u }_{\h^s_\K} = \sum_{k=j+1}^\infty \lambda_k^s \abs{f(\lambda_k)}^2 \abs{\hat{u}(k)}^2 \le \sup_{k \ge j+1} \frac{\abs{f(\lambda_k)}^2}{\lambda_k^{2r}} \ns{u}_{\h^{s+2r}_\K},
\end{equation}
and hence
\begin{equation}
 \ns{F_j - f(\K) }_{\L(\h^{s+2r}_\K;\h^s_\K)} \le \sup_{k \ge j+1} \frac{\abs{f(\lambda_k)}^2}{\lambda_k^{2r}}.
\end{equation}
From this and the inclusion $f \in \mathfrak{B}^r_0(\Sigma_\K)$ we  deduce that $F_j \to f(\K)$ in $\L(\h^{s+2r}_\K;\h^s_\K)$, and hence $f(\K)$ is compact.
\end{proof}

One of the most important uses of this result is the following corollary.

\begin{cor}
If $s,t \in \R$ and $s< t$, then  $\h^t_\K((-\ell,\ell)) \csubset \h^s_\K((-\ell,\ell))$.
\end{cor}

We have the following variant of elliptic regularity in the spaces $\h^s_\K((-\ell,\ell))$.

\begin{thm}
Let $s \in [0,\infty)$ and suppose that $f \in \h^s_\K((-\ell,\ell))$.  If $u \in \h^1_\K((-\ell,\ell)) $ is the weak solution to $\K u =f$ and $\B u =0$, i.e.
\begin{equation}
 (u, v)_{1,\Sigma} = (f,v)_{0,\Sigma} \text{ for all } v \in H^1((-\ell,\ell)),
\end{equation}
then $u \in \h^{s+2}_\K((-\ell,\ell))$.  Moreover, $\norm{u}_{\h^{s+2}_\K} = \norm{f}_{\h^s_\K}$.  Consequently, $\K : \h^{s+2}_\K((-\ell,\ell)) \to \h^s_\K((-\ell,\ell))$ is an isometric isomorphism.
\end{thm}
\begin{proof}
We have that 
\begin{equation}
\hat{f}(k)= (f,w_k)_{0,\Sigma} = (u,w_k)_{1,\Sigma} = ( w_k, u)_{1,\Sigma} = \lambda_k (w_k,u)_{0,\Sigma} = \lambda_k (u,w_k)_{0,\Sigma} = \lambda_k \hat{u}(k).
\end{equation}
Thus 
\begin{equation}
 \ns{u}_{\h^{s+2}_\K} = \sum_{k=0}^\infty \lambda_k^{s+2} \abs{\hat{u}(k)}^2 =  \sum_{k=0}^\infty \lambda_k^{s} \abs{\hat{f}(k)}^2 = \ns{f}_{\h^s_\K}.
\end{equation}
 
\end{proof}

\subsection{Interpolation theory and its consequences}

Here we write $(X,Y)_{\theta,p}$ for $\theta \in [0,1]$ and $1 \le p \le \infty$ for the real interpolation of the spaces $X,Y$ with parameters $\theta,p$.  See \cite{bergh_lof} or \cite{triebel}, for instance,  for the precise definition.  We record a basic result from that book.

\begin{thm}
Let $s,t \in \R$ with $s \neq t$.  For $0 < \theta < 1$ and  $r = (1-\theta) s + \theta t$ we have that 
\begin{equation}
 (\ell^2_s(\mathbb{N}), \ell^2_t(\mathbb{N}))_{\theta,2} = \ell^2_r(\mathbb{N}).
\end{equation}
\end{thm}
\begin{proof}
This follows immediately from  Theorem 5.4.1 of \cite{bergh_lof}.
\end{proof}

By combining this with Theorem \ref{l2_char} we immediately deduce the following.

\begin{cor}\label{s_interp}
Let $s,t \in \R$ with $s \neq t$.  For $0 < \theta < 1$ and  $r = (1-\theta) s + \theta t$ we have that 
\begin{equation}
 (\h^s_\K((-\ell,\ell)) , \h^t_\K((-\ell,\ell)))_{\theta,2} = \h^r_\K((-\ell,\ell)).
\end{equation}
 
\end{cor}

Next we present a useful application of the interpolation theory.  

\begin{lem}\label{s_embed_lem}
If $s \ge 0$, then $\h^{s}_\K((-\ell,\ell)) \hookrightarrow H^{s}((-\ell,\ell))$.
\end{lem}
\begin{proof}
We may view Theorem \ref{sobolev_char} as saying that, for $m \in \mathbb{N}$, the identity map $I$ is such that
\begin{equation}
 I : \h^m_\K((-\ell,\ell)) \to L^2((-\ell,\ell)) \text{ and } I : \h^{m+1}_\K((-\ell,\ell)) \to  H^{m+1}((-\ell,\ell)) 
\end{equation}
are bounded linear operators.   We can then interpolate and use Corollary \ref{s_interp} and the interpolation properties of standard Sobolev spaces (see, for instance, \cite{bergh_lof,triebel}) to deduce that for $0 < s < 1$,  
\begin{multline}
 I : \h^{m+s}_\K((-\ell,\ell)) =  (\h^m_\K((-\ell,\ell)) , \h^{m+1}_\K((-\ell,\ell)))_{s,2} \to (H^m((-\ell,\ell)),H^{m+1}((-\ell,\ell)))_{s,2} \\
 = H^{m+s}((-\ell,\ell)).
\end{multline}
 
\end{proof}

In fact, we can do quite a bit better when $0 \le s < 2$.  In stating the following result we recall (see \cite{lions_magenes_1}) that 
\begin{equation}\label{h_half_00}
 H^{1/2}_{00}((-\ell,\ell)) = (H^0((-\ell,\ell)), H^1_0((-\ell,\ell)))_{1/2,2}  
\end{equation}
and 
\begin{equation}
 H^{1/2}_{00}((-\ell,\ell)) \subset H^{1/2}((-\ell,\ell)) = (H^0((-\ell,\ell)), H^1_0((-\ell,\ell)))_{1/2,2}.
\end{equation}

\begin{thm}\label{s_embed}
For $0 \le t \le 1$ we have that $H^t((-\ell,\ell)) = \h^t_\K((-\ell,\ell))$ with norm equivalence $\norm{f}_{H^t} \asymp \norm{f}_{\h^t_\K}$.  Moreover, for  $s \in (0,1)$ we have that 
\begin{equation}
\h^{1+s}_\K((-\ell,\ell)) = 
\begin{cases}
 H^{1+s}((-\ell,\ell)) & \text{if } 0 < s < 1/2 \\
 \{f \in H^{3/2}((-\ell,\ell)) \st f' \in H^{1/2}_{00}((-\ell,\ell)) \} &\text{if } s= 1/2 \\
 \{f \in H^{1+s}((-\ell,\ell)) \st f' \in H^{s}_{00}((-\ell,\ell)) \} &\text{if } 1/2 < s < 1
\end{cases}
\end{equation}
and we have the norm equivalence 
\begin{equation}
\ns{f}_{\h^{1+s}_\K} \asymp 
\begin{cases}
\ns{f}_{H^{1+s}} &\text{if } 0 < s < 1/2 \\
\ns{f}_{H^{3/2}} + \int_{-\ell}^\ell \frac{\abs{f'(x)}^2}{\ell - \abs{x}} dx & \text{ if } s = 1/2 \\
\ns{f}_{H^{1+s}} + \ns{f'}_{H^s_0} &\text{if } 1/2 < s < 1.
\end{cases}
\end{equation}
\end{thm}
\begin{proof}
The assertion $t=0,1$ is proved in Theorem \ref{sobolev_char}, and for $0 < t < 1$ it follows from this theorem, Corollary \ref{s_interp}, and standard Sobolev interpolation:
\begin{equation}
H^t((-\ell,\ell)) =  (H^0((-\ell,\ell)), H^1((-\ell,\ell)))_{t,2} = (\h^0_\K((-\ell,\ell)), \h^1_\K((-\ell,\ell)))_{t,2} = \h^{t}_\K((-\ell,\ell)).
\end{equation}

We now prove the assertion for $s \in (0,1)$.  Define the map $F: L^2((-\ell,\ell)) \times \R \to H^1((-\ell,\ell))$ via 
\begin{equation}
 F(g,v) = v + \int_{-\ell}^x g(t) dt.
\end{equation}
If $g \in H^1_0((-\ell,\ell))$ and $v \in \R$, then $F(g,v) \in L^2((-\ell,\ell))$ and $F(g,v)' = g \in H^1_0((-\ell,\ell))$.  From this and Theorem \ref{sobolev_char} we deduce that 
\begin{equation}
 F \in \L(L^2((-\ell,\ell)) \times \R ; \h^1_\K((-\ell,\ell))  ) \cap \L(H^1_0((-\ell,\ell)) \times \R ; \h^2_\K((-\ell,\ell))  ). 
\end{equation}
Hence, upon interpolating, we find that for $s \in (0,1)$ 
\begin{equation}\label{s_embed_1}
 F \in \L(X^s \times \R ; \h^{1+s}_\K((-\ell,\ell))  ),
\end{equation}
where we have written $X^s = (L^2((-\ell,\ell)), H^1_0((-\ell,\ell)))_{s,2}$ for brevity.

Next consider the map $D$ defined by $f \mapsto Df = f'$.  Theorem \ref{sobolev_char} tells us that 
\begin{equation}
 D \in \L(\h^1_\K((-\ell,\ell)) ; L^2((-\ell,\ell)) ) \cap \L( \h^2_\K((-\ell,\ell)) ; H^1_0((-\ell,\ell)) ).
\end{equation}
Upon interpolating again and using Corollary \ref{s_interp}, we find that
\begin{equation}\label{s_embed_2}
 D \in \L(\h^{1+s}_\K((-\ell,\ell)); X^s.  )
\end{equation}

For $s \in (0,1)$ define the Hilbert space
\begin{equation}
 Y^{1+s} = \{ f\in H^{1+s}((-\ell,\ell)) \st f' \in X^s \}
\end{equation}
with norm $\ns{f}_{Y^{1+s}} = \ns{f}_{H^{1+s}} + \ns{f'}_{X^s}$.  According to \eqref{s_embed_2} and Lemma \ref{s_embed_lem}, we have the continuous inclusion $\h^{1+s}_\K((-\ell,\ell)) \subseteq Y^{1+s}$.  On the other hand, if $f \in Y^{1+s}$, then $(f', f(-\ell)) \in X^s \times \R$, and by \eqref{s_embed_1} we have that $f = F(f',f(-\ell)) \in \h^{1+s}_\K((-\ell,\ell)))$.  Hence, we have the continuous inclusion $Y^{1+s} \subseteq \h^{1+s}_\K((-\ell,\ell)) $.  We deduce that we have the algebraic and topological identity 
\begin{equation}
 \h^{1+s}_\K((-\ell,\ell))  = Y^{1+s} \text{ for all }s \in (0,1).
\end{equation}

To conclude, we recall the standard interpolation facts
\begin{equation}
X^s = (L^2((-\ell,\ell)), H^1_0((-\ell,\ell)))_{2,s} =
\begin{cases}
H^s((-\ell,\ell)) & \text{if } 0 < s < 1/2 \\
H^{1/2}_{00}((-\ell,\ell)) &\text{if } s = 1/2 \\
H^s_0((\ell,\ell)) & \text{if } 1/2 < s < 1.
\end{cases}
\end{equation}
If $0 < s < 1/2$ then we have the norm equivalence
\begin{equation}
\ns{f}_{\h^{1+s}_{\K}} \asymp \ns{f}_{H^{1+s}} + \ns{f'}_{H^s((-\ell,\ell))} \asymp   \ns{f}_{H^{1+s}},
\end{equation}
and so $\h^{1+s}_\K((-\ell,\ell)) = H^{1+s}((-\ell,\ell))$.  The result follows from this and the   characterization of $H^{1/2}_{00}((-\ell,\ell))$ in \eqref{h_half_00}.

\end{proof}

As a byproduct of this result we get the following Sobolev embeddings.
 
\begin{thm}
For $s \in [0,2]$ we have that  
\begin{equation}
\h^s_\K((-\ell,\ell)) \hookrightarrow 
\begin{cases}
L^{2/(1-2s)}((-\ell,\ell)) & \text{if } s\in [0,1/2) \\
L^p((-\ell,\ell)) \text{ for all } p \in [1,\infty) &\text{if } s = 1/2 \\
C^{0,\alpha}_b((-\ell,\ell)) \text{ for } \alpha = s-1/2 &\text{if } s \in (1/2,3/2) \\
C^{0,\alpha}_b((-\ell,\ell)) \text{ for all } \alpha \in [0,1) &\text{if }s =3/2 \\
C^{1,\alpha}_b((-\ell,\ell)) \text{ for } \alpha = s-3/2 &\text{if } s \in (3/2,2]. 
\end{cases}
\end{equation}
Moreover, for $s \in [1,3/2]$ we have that 
\begin{equation}
\h^s_\K((-\ell,\ell))\hookrightarrow 
\begin{cases}
W^{1,2/(1-2s)}((-\ell,\ell)) & \text{if } s\in [1,3/2) \\
W^{1,p}((-\ell,\ell)) \text{ for all } p \in [1,\infty) &\text{if } s = 3/2.
\end{cases}
\end{equation}
\end{thm}
\begin{proof}
These are immediate consequences of Theorem \ref{s_embed} and the standard Sobolev embeddings of $H^s((-\ell,\ell))$ for $0 \le s \le 2$. 
\end{proof}

\subsection{Bilinear boundedness, integration by parts}

Suppose that $\varphi, \psi \in W$ and let $s \in [0,1]$.   We then have that 
\begin{equation}
 (\varphi,\psi)_{1,\Sigma} = \sum_{k=0}^\infty \lambda_k \hat{\varphi}(k) \hat{\psi}(k) = \sum_{k=0}^\infty \lambda_k^s \hat{\varphi}(k)  \lambda_k^{1-s}\hat{\psi}(k) = (\K^s \varphi, \K^{1-s} \psi)_{0,\Sigma}.
\end{equation}
Consequently, for any $s,t \in [0,1]$ we have that 
\begin{equation}
 (\K^s \varphi, \K^{1-s} \psi)_{0,\Sigma} = (\K^t \varphi, \K^{1-t} \psi)_{0,\Sigma},
\end{equation}
which we can view as a sort of fundamental integration-by-parts result in the sense that we can arbitrarily shift powers of $\K$ from one term to the next so long as the overall power sums to unity.  Working in $W$ is obviously too restrictive, but we can extend by density to get a generalized version of integration by parts for all fractional orders.

\begin{thm}
Let $B : W \times W \to \R$ be the bilinear map defined via 
\begin{equation}
B(\varphi,\psi) = (\K \varphi,\psi)_{0,\Sigma} = (\varphi,\psi)_{1,\Sigma} = (\varphi, \K \psi)_{0,\Sigma}. 
\end{equation}
Then $B$ extends to a bounded bilinear map $B : \h^{2s}_\K \times \h^{2(1-s)}_\K \to \R$ for each for $s \in [0,1]$.
\end{thm}
\begin{proof}
This follows directly from the identity
\begin{equation}
B(\varphi,\psi) =  (\K^s \varphi, \K^{1-s} \psi)_{0,\Sigma} \text{ for }\varphi,\psi \in W,
\end{equation}
which allows us to bound 
\begin{equation}
 \abs{B(\varphi,\psi)} \le \norm{\varphi}_{\h^{2s}_\K} \norm{\psi}_{\h^{2(1-s)}_\K}.
\end{equation}
Using this and the density of $W$ in  $\h^t_\K$ for all $t \ge 0$ proves the result.

\end{proof}

\subsection{The operators $D_j^r$ }\label{sec_djr}

We now turn our attention to the operators  $D^r := \K^{r/2}$ for $r \ge 0$, as defined by the functional calculus from Definition \ref{fnal_calc_def}. We will need to introduce some finite approximations, $D^r_j$, defined by  
\begin{equation}
 D_j^r u = \sum_{k=0}^j \lambda_k^{r/2} \hat{u}(k) w_k.
\end{equation}
It's easy to see that this is well-defined for every $u \in L^2((-\ell,\ell)) = \h^0_\K((-\ell,\ell))$ and that in this case $D_j^r u \in W  \subseteq \bigcap_{s \ge 0} \h^s_\K((-\ell,\ell))  \subset C^\infty([-\ell,\ell])$.

Let's now study some properties.  The first result tells us that $D^r_j$ is like an approximation of $r$ derivatives.

\begin{prop}\label{Djr_bnds}
Let $j \in \mathbb{N}$.  Then the following hold.
\begin{enumerate}
 \item If $0\le r_1,r_2,s_1,s_2 \in \R$ satisfy $r_1 + s_1 = r_2 + s_2$, then 
\begin{equation}
 \norm{D^{r_1}_j f}_{\h^{s_1}_\K} =  \norm{D^{r_2}_j f}_{\h^{s_2}_\K}
\end{equation}
for all $f \in L^2((-\ell,\ell))$. 
\item If $0 \le r \in \R$, then 
\begin{equation}
 \norm{D^r_j f}_{L^2} \le \norm{f}_{\h^r_\K}  \text{ and }  \norm{f}_{\h^r_\K} = \lim_{j \to \infty} \norm{D^r_j f}_{L^2}  
\end{equation}
for every $f \in \h^r_{\K}((-\ell,\ell))$.
\end{enumerate}
\end{prop}
\begin{proof}
For the first item we compute 
\begin{equation}
 \ns{D^{r_1}_j f}_{\h^{s_1}_\K} = \sum_{k=0}^j \lambda_k^{s_1} \abs{ \lambda_k^{r_1/2} \hat{f}(k) }^2 = \sum_{k=0}^j \lambda_k^{s_2} \abs{ \lambda_k^{r_2/2} \hat{f}(k) }^2 = \ns{D^{r_2}_j f}_{\h^{s_2}_\K} .
\end{equation}
In turn this shows that 
\begin{equation}
\norm{D_j^r f}_{L^2} = \norm{D_j^0 f}_{\h^r_\K} = \left( \sum_{k=0}^j \lambda_k^r \abs{\hat{f}(k)}^2  \right)^{1/2},
\end{equation}
from which the second item follows.
\end{proof}

Next we consider how $D^r_j$ interacts with functions of average zero.

\begin{lem}\label{Djr_avg_zero}
If $f\in L^2((-\ell,\ell))$ satisfies $\int_{-\ell}^\ell f =0$, then $\int_{-\ell}^\ell D_j^r f =0$ for all $r \ge 0$ and $j \in \mathbb{N}$.
\end{lem}
\begin{proof}
Since $w_0 = 1/\sqrt{2\ell}$ we see that $\int_{-\ell}^\ell f =0$ if and only if $\hat{f}(0) =0$.  In this case we then have that $\widehat{D^r_j f}(0) =0$ as well, and the result follows.
\end{proof}

Next we prove an integration by parts formula.  

\begin{lem}\label{Djr_IBP}
Let $0 \le r,s,t,\rho \in \R$ be such that $r = s+t$.  Then for $j \in \mathbb{N}$, $f\in L^2((-\ell,\ell))$, and $g \in \h^\rho_\K((-\ell,\ell))$ we have that
\begin{equation}
 \ip{D^r_j f,g}_{\h^\rho_\K} =  \ip{D^s_j f,D^t_j g}_{\h^\rho_\K}.
\end{equation}
\end{lem}
\begin{proof}
We simply compute 
\begin{multline}
(D^r_j f,g)_{\h^\rho_\K} = \sum_{k=0}^\infty \lambda_k^\rho \widehat{D^r_j f}(k) \hat{g}(k) 
= \sum_{k=0}^j \lambda_k^{\rho + r/2} \hat{f}(k) \hat{g}(k) = \sum_{k=0}^j \lambda_k^{\rho } \lambda_k^{s/2} \hat{f}(k) \lambda_k^{t/2} \hat{g}(k) \\
= \sum_{k=0}^\infty \lambda_k^\rho \widehat{D^s_j f}(k) \widehat{D^t_j g}(k)
= \ip{D^s_j f, D^t_j g}_{\h^\rho_\K}.
\end{multline}
\end{proof}

\begin{remark}
In this paper the most useful instances of Lemma \ref{Djr_IBP} occur with $\rho \in \{0,1\}$.  Indeed, the lemma shows that if $0 \le r = s+t$ and $f \in L^2((-\ell,\ell))$, then 
\begin{equation}
 \int_{-\ell}^\ell D^r_j f g =  \int_{-\ell}^\ell D^s_j f D^t_j g \text{ for all }  g \in L^2((-\ell,\ell)) = \h^0_\K((\ell,\ell))
\end{equation}
and 
\begin{equation}
(D^r_j f, g)_{1,\Sigma} = (D^s_j f, D^t_j g)_{1,\Sigma} \text{ for all } g \in H^1((-\ell,\ell)) = \h^1_\K((-\ell,\ell)).
\end{equation}
\end{remark}

We conclude with a dual estimate. 

\begin{prop}\label{Djr_dual}
 Let $0 \le r  \le 1/2$, $0 \le s <3/2$, and $j \in \mathbb{N}$.  Then 
\begin{equation}
D^s_j : H^{s-r}((-\ell,\ell)) \to   H^{-r}((-\ell,\ell))  =   (H^{r}_0((-\ell,\ell)))^\ast 
\end{equation}
is a bounded linear operator.  
\end{prop}
\begin{proof}
According to the results in Chapter 1 of \cite{lions_magenes_1} and Theorem \ref{s_embed} we have that 
\begin{equation}
H^{r}_0((-\ell,\ell))= H^{r}((-\ell,\ell))  = \h^{r}_\K((-\ell,\ell)).
\end{equation}
This and Theorem \ref{dual_char} then show that 
\begin{equation}
\h^{-r}_\K((-\ell,\ell))  = (H^{r}_0((-\ell,\ell)))^\ast = H^{-r}((-\ell,\ell))
\end{equation}
with equality of norms.  Hence, for $f \in H^{s-r}((-\ell,\ell)) = \h^{s-r}_\K((-\ell,\ell))$ (which again follows by Theorem \ref{s_embed}), we again use Theorem \ref{dual_char} together with Cauchy-Schwarz to compute 
\begin{multline}
\norm{D^s_j f}_{H^{-r}} \asymp  \norm{D^s_j f}_{\h^{-r}_\K} =  \norm{J D^s_j f}_{(\h^{r}_\K)^\ast} = \sup_{\norm{\hat{g}}_{\ell^2_r} \le 1} \sum_{k=0}^j \lambda_k^{s/2} \hat{f}(k) \hat{g}(k)  \\
= \sup_{\norm{\hat{g}}_{\ell^2_r} \le 1} \sum_{k=0}^j \lambda_k^{(s-r)/2} \hat{f}(k) \lambda_k^{r/2} \hat{g}(k) 
\le \norm{f}_{\h^{s-r}_\K} \ls \norm{f}_{H^{s-r}}.
\end{multline}
This proves the boundedness assertion, and linearity is trivial.
\end{proof}

\section{Enhancement estimates }\label{sec_enhancements}

Our goal in this section is to record enhancement estimates for the dissipation and energy that are derived through energy-type arguments rather than elliptic estimates.  We will gain some dissipative control of $\eta$, $\dt \eta$, and $\dt^2 \eta$, and we will gain energetic control of $\dt p$.

\subsection{Prerequisites}

Recall that for a real parameter $0 \le s < 1$ the fractional differential operator $D^s = \K^{s/2}$ and its finite approximations $D^s_j$ for $j \in \mathbb{N}$, as defined in Section \ref{sec_djr}.  The next result gives an existence result for a Neumann-type problem involving $D^s_j$.

\begin{prop}\label{psi_solve}
Let $s \in \R$ and $j,k \in \mathbb{N}$ with $0 \le k \le 2$ and $0 \le s < 1$.  Then there exists $\psi : \Omega \to \R$ solving 
\begin{equation}\label{diss_enhace_psi_eqn}
\begin{cases}
 - \Delta \psi  = 0 &\text{in }\Omega \\
\p_\nu \psi = (D_j^s \dt^k \eta)/\abs{\N_0} &\text{on }\Sigma \\
\p_\nu \psi = 0 &\text{on } \Sigma_s 
\end{cases}
\end{equation}
where  $\nu$ is the unit normal for the fixed domain $\Omega$ and its non-unit counterpart is $\N_0$, i.e. $\nu = \N_0/\abs{\N_0}$.  Moreover, we have the estimates 
\begin{equation}
\norm{\psi}_{H^1} \ls    \norm{\dt^k \eta }_{ H^{s-1/2}}, \; \norm{\psi}_{H^2}  \ls  \norm{D_j^s  \dt^k\eta}_{\h^{1/2}_\K}, \text{ and } \norm{\dt \psi}_{H^1} \ls  \norm{ \dt^{k+1} \eta }_{ H^{s-1/2}}.
\end{equation}
\end{prop}
\begin{proof}
To begin we note that Proposition \ref{avg_zero_prop} and Lemma \ref{Djr_avg_zero} imply that
\begin{equation}
 \int_\Sigma D_j^s \dt^k \eta \abs{\N_0}^{-1} = \int_{-\ell}^\ell D_j^s \dt^k  \eta =  \int_{-\ell}^\ell \dt^k \eta =0. 
\end{equation}
Consequently, the compatibility condition needed to produce a unique weak solution $\psi \in \oH^1(\Omega)$ to \eqref{diss_enhace_psi_eqn} is satisfied. Since the domain $\Omega$ has convex corners, the $H^2$ solvability theory is available for \eqref{diss_enhace_psi_eqn} (see, for instance, \cite{kmr_1}).  This, the elementary $H^1$ weak estimate, and Proposition \ref{Djr_dual} then show that 
\begin{equation}
\begin{split}
\norm{\psi}_{H^1} &\ls \norm{ D_j^s \dt^k \eta}_{H^{-1/2}} \ls   \norm{\dt^k \eta }_{ H^{s-1/2}} \\
\norm{\psi}_{H^2} & \ls  \norm{D_j^s \dt^k \eta}_{\h^{1/2}_\K} \ls \norm{D_j^s  \dt^k\eta}_{\h^{1/2}_\K}  \\
\norm{\dt \psi}_{H^1} &\ls \norm{ D_j^s \dt^{k+1} \eta}_{H^{-1/2}} \ls   \norm{ \dt^{k+1} \eta }_{ H^{s-1/2}},
\end{split}
\end{equation}
from which the result follows.

\end{proof}

\subsection{Dissipative enhancement for $\eta$}

We begin by considering dissipation enhancement estimates for $\eta$.  To this end let $\psi$ be as in Proposition \ref{psi_solve} with $k=0$.  This proposition and Proposition \ref{M_properties} show that if we set $w = M \nab \psi$, then $w$ is a valid choice of a test function in Lemma \ref{geometric_evolution} and
\begin{equation}
 \diva{w} = \diva{M\nab \psi} = K\Delta \psi =0.
\end{equation}
We will use this $w$ as a test function in Lemma \ref{geometric_evolution} to produce an essential dissipation estimate.

\begin{thm}\label{diss_enhance_eta}
Let $\low \in (0,1)$ be given by \eqref{kappa_ep_def}, and $0 < T \le \infty$.  There exists a universal $0 <\delta_\ast <1$ such that if  $\sup_{0\le t < T} \E(t) \le \delta_\ast$, then for every $0 \le \tau \le t < T$ we have the estimate
\begin{equation}
  \int_\tau^t \ns{ \eta}_{H^{3/2-\low}}  \ls  \E_{\shortparallel, 0}(\tau) + \E_{\shortparallel, 0}(t) + \int_{\tau}^t \D_{\shortparallel,0},
\end{equation}
where $\E_{\shortparallel,0}$ and $\D_{\shortparallel,0}$ are as in \eqref{ED_natural}.
\end{thm}
\begin{proof}
We begin by assuming that $\delta_\ast < \gamma^2$, where $\gamma \in (0,1)$ is as in Lemma \ref{eta_small}.  In particular, this means that the estimates of Lemma \ref{eta_small} are available in what follows.

Let $s = 1- 2\low \in [0,1)$, which means that $3/2 - \low = 1 + s/2$.  For a fixed $j \in \mathbb{N}$ we let $\psi$ solve \eqref{diss_enhace_psi_eqn} with data $D^s_j \eta / \abs{\N_0}$.  Then Proposition \ref{psi_solve} provides the estimates
\begin{equation}\label{diss_enhance_eta_ests}
\norm{\psi}_{H^1} \ls    \norm{\eta }_{ H^{s-1/2}}, \; \norm{\psi}_{H^2}  \ls  \norm{D_j^s  \eta}_{\h^{1/2}_\K}, \text{ and } \norm{\dt \psi}_{H^1} \ls  \norm{ \dt \eta }_{ H^{s-1/2}}.
\end{equation}
Note that since $0 \le s < 1$ we can define $\rho = (1-s)/2 \in (0,1/2]$, which satisfies
\begin{equation}\label{diss_enhance_eta_rho}
\hal + \rho + s  = 1 + \frac{s}{2}.
\end{equation}
This, Theorem \ref{s_embed}, and  Proposition \ref{Djr_bnds} then provide the useful equivalence 
\begin{equation}\label{diss_enhance_eta_norm}
\norm{D_j^{s/2} \eta}_{H^1} \asymp  \norm{D_j^{s/2} \eta}_{\h^1_\K}   =  \norm{D_j^s \eta}_{\h^{1/2 + \rho}_\K} \asymp \norm{D_j^s \eta}_{H^{1/2 + \rho}}.
\end{equation}
We use $w = M \nab \psi$ from above in Lemma \ref{geometric_evolution} to arrive at the identity
\begin{multline}\label{diss_enhance_eta_ident}
\br{\dt u,Jw} + (-\dt \bar{\eta} \frac{\phi}{\zeta_0} K \p_2  u + u \cdot \naba  u,Jw)_0 +
  \pp{ u,w}  + ( \eta ,w\cdot \N)_{1,\Sigma} + \linz [\dt \eta ,w\cdot \N]_\ell  \\
 = - \int_{-\ell}^\ell \sigma  \mathcal{R}  \p_1 (w \cdot \N).
\end{multline}
We will deal with these term-by-term.

For the first term we note that $M = K \nab \Phi$ for $\Phi$ the flattening map, and so
\begin{equation}
\br{\dt u,JM \nab \psi} = \int_\Omega \dt u \cdot \nab \Phi \nab \psi  =   \frac{d}{dt} \int_{\Omega}  u \cdot \nab \Phi \nab \psi - \int_\Omega  u \cdot \nab\Phi \nab \dt \psi - \int_\Omega  u \cdot \dt(\nab\Phi ) \nab \psi.
\end{equation}
Using the bound $\E \le 1$, the definition of $\Phi$ in \eqref{mapping_def}, and \eqref{diss_enhance_eta_ests}, we may then estimate  
\begin{equation}\label{diss_enhance_eta_1}
 \abs{\int_\Omega  u \cdot \nab\Phi \nab \psi } \ls \norm{\nab\Phi}_{L^\infty} \norm{u}_0 \norm{\psi}_1 \ls  \norm{\nab\Phi}_{L^\infty} \norm{u}_{H^0}  \norm{ \eta }_{ H^{s-1/2}} \ls   \norm{u}_{H^0}  \norm{ \eta }_{ H^{1}}   \ls   \E_{\shortparallel, 0},
\end{equation}
where $\E_{\shortparallel, 0}$ is the natural energy at the non-differentiated level, as defined in \eqref{ED_natural}.  Similarly, 
\begin{equation}
\abs{- \int_\Omega  u \cdot \nab\Phi \nab \dt \psi - \int_\Omega  u \cdot \dt(\nab\Phi ) \nab \psi} \ls   \norm{u}_{H^0} \left( \norm{\psi}_{H^1} + \norm{\dt \psi}_{H^1} \right)  \ls \norm{u}_{H^0} \left(\norm{ \eta }_{ H^{s/2}} + \norm{\dt \eta }_{ H^{s/2}} \right). 
\end{equation}

For the second term we write 
\begin{equation}
 (-\dt \bar{\eta} \frac{\phi}{\zeta_0} K \p_2  u + u \cdot \naba  u,Jw)_0 =  (-\dt \bar{\eta} \frac{\phi}{\zeta_0} K \p_2  u + u \cdot \naba  u, \nab\Phi \nab \psi )_0.
\end{equation}
From this and the bounds $\E \le 1$ and \eqref{diss_enhance_eta_ests} we may then estimate
\begin{equation}
 \abs{ (-\dt \bar{\eta} \frac{\phi}{\zeta_0} K \p_2  u + u \cdot \naba  u,Jw)_0} \ls \norm{u}_{H^1}  \norm{\psi}_{H^1} \ls \norm{u}_{H^1} \norm{\eta}_{H^1}.
\end{equation}

For the third term we use the $H^2$ estimate from \eqref{diss_enhance_eta_ests} to bound 
\begin{equation}
\abs{\pp{u,w}} = \frac{\mu}{2} \abs{\int_{\Omega} J \sga  u : \sga (M \nab \psi)  } \ls \norm{u}_{H^1} \norm{\psi}_{H^2} \ls \norm{u}_{H^1} \norm{D_j^s  \eta}_{\h^{1/2}_\K}.
\end{equation}

To handle the fourth term we first note that on $\Sigma$, 
\begin{equation}
 w \cdot \N = M \nab \psi \cdot \N= \nab \psi \cdot \N_0 = \abs{\N_0} \p_\nu \psi = D_j^s   \eta.
\end{equation}
Using this,  Lemma \ref{Djr_IBP}, and \eqref{diss_enhance_eta_norm} we can rewrite the fourth term as
\begin{equation}
 ( \eta ,w\cdot \N)_{1,\Sigma} = ( \eta ,D_j^s   \eta)_{1,\Sigma}  = ( D_j^{s/2}\eta ,D_j^{s/2}   \eta)_{1,\Sigma} =   \ns{D_j^{s/2} \eta}_{\h^{1}_\K} = \ns{D_j^{s} \eta}_{\h^{1/2 + \rho}_\K}.
\end{equation}

Then for the fifth term we can use trace theory and Theorem \ref{s_embed} to bound
\begin{equation}
\abs{\linz [\dt \eta ,D_j^s  \eta]_\ell } \ls [\dt \eta]_\ell \norm{D_j^s  \eta}_{H^{1/2+\rho}} 
=  [\dt \eta]_\ell \norm{D_j^s  \eta}_{\h^{1/2+\rho}_\K} 
\end{equation}

Finally, we examine the nonlinear term on the right side of \eqref{diss_enhance_eta_ident}.  We start by using  Proposition \ref{frac_IBP_prop}, which is available since $0 <s < s$, to estimate 
\begin{equation}
 \abs{\int_{-\ell}^\ell \sigma  \mathcal{R}  \p_1 D^s_j \eta} \ls \norm{D_j^s \eta}_{H^{1-s/2}} \norm{\mathcal{R}}_{H^{s/2}}.
\end{equation}
Next we note that 
\begin{equation}
 \frac{\mathcal{R}(y,z)}{z^2},  \frac{\partial_2 \mathcal{R}(y,z)}{z}, \text{ and }\frac{\partial_1 \mathcal{R}(y,z)}{z^2}
\end{equation}
are well-defined and bounded by Proposition \ref{R_prop}.  Thus $\mathcal{R}(\p_1\zeta_0,\p_1 \eta)/\p_1 \eta$ is well-defined and satisfies 
\begin{equation}
 \p_1 \left(\frac{\mathcal{R}(\p_1 \zeta_0,\p_1 \eta)}{\p_1 \eta} \right) = \frac{\p_1 \mathcal{R}(\p_1\zeta_0,\p_1 \eta) \p_1^2 \zeta_0}{\p_1 \eta} + \left(\frac{\p_2 \mathcal{R}(\p_1 \zeta_0,\p_1 \eta) }{\p_1 \eta} - \frac{\mathcal{R}(\p_1\zeta_0,\p_1 \eta)}{(\p_1 \eta)^2} \right) \p_1^2 \eta,
\end{equation}
which in turn means that for any $1 \le q \le \infty$, 
\begin{equation}
 \norm{\mathcal{R}(\p_1\zeta_0,\p_1 \eta)/\p_1 \eta}_{W^{1,q}} \ls \norm{\p_1 \eta}_{L^q} + \norm{\p_1^2 \eta}_{L^q} = \norm{\p_1 \eta}_{W^{1,q}} \le \norm{\eta}_{W^{2,q}}.
\end{equation}
This allows us to use Theorem \ref{supercrit_prod}  and the embedding $W^{1,q_+}((-\ell,\ell))  \hookrightarrow H^{1/2 + \ep_+/2}((-\ell,\ell))$ to estimate 
\begin{multline}
 \norm{\mathcal{R}}_{H^{s/2}} =  \norm{\p_1 \eta  \mathcal{R}(\p_1\zeta_0,\p_1\eta)/ \p_1 \eta }_{H^{s/2}} \ls \norm{\p_1 \eta}_{H^{s/2}} \norm{\mathcal{R}(\p_1\zeta_0,\p_1\eta)/ \p_1 \eta}_{H^{(1+\ep_+)/2}} \\
\ls \norm{\p_1 \eta}_{H^{s/2}} \norm{\mathcal{R}(\p_1\zeta_0,\p_1\eta)/ \p_1 \eta}_{W^{1,q_+}} 
\ls \norm{\p_1 \eta}_{H^{s/2}} \norm{\eta}_{W^{2,q_+}}.
\end{multline}
Assembling these estimates and employing Theorem \ref{s_embed} and the definition of $\E$ from \eqref{E_def} then shows that 
\begin{equation}
 \abs{\int_{-\ell}^\ell \sigma  \mathcal{R}  \p_1 D^s_j \eta} \ls \norm{D_j^s \eta}_{H^{1-s/2}} \norm{ \eta}_{H^{1+s/2}} \norm{\eta}_{W^{2,q_+}} \ls  \ns{ \eta}_{H^{1+s/2}} \sqrt{\E}.
\end{equation}

We now combine all of these estimates to deduce that 
\begin{multline}
\ns{D_j^s \eta}_{\h^{1/2 +\rho}_\K}  +  \frac{d}{dt} \int_\Omega   u \cdot \nab\Phi \nab \psi  \\ \ls  
 \norm{u}_{H^0} \left(\norm{ \eta }_{ H^{s-1/2}} + \norm{\dt \eta }_{ H^{s-1/2}} \right) 
 + \norm{ u}_{H^1} \norm{D_j^s  \eta}_{\h^{1/2}_\K}  
  + [\dt \eta]_\ell \norm{D_j^s \eta}_{\h^{1/2 + \rho}_\K}  + \ns{ \eta}_{H^{1+s/2}} \sqrt{\E}.
\end{multline}
Then for $0 \le \tau \le t < T$ we can integrate this inequality to see that
\begin{multline}
\int_\tau^t \ns{D_j^s \eta}_{\h^{1/2 + \rho}_\K}  +   \int_\Omega   (u \cdot \nab\Phi \nab \psi)(t) \ls \int_\Omega   (u \cdot \nab\Phi \nab \psi)(\tau)   
+  \int_\tau^t \norm{u}_{H^0} \left(\norm{ \eta }_{ H^{s-1/2}} + \norm{\dt \eta }_{ H^{s-1/2}} \right) \\ 
 + \int_\tau^t \norm{ u}_{H^1} \norm{D_j^s  \eta}_{\h^{1/2}_\K}  
  + \int_\tau^t [\dt \eta]_\ell \norm{D_j^s \eta}_{\h^{1/2+\rho}_\K} + \int_\tau^t  \ns{ \eta}_{H^{1+s/2}} \sqrt{\E}.
\end{multline}  
We then use \eqref{diss_enhance_eta_1} and Cauchy's inequality to deduce from this that 
\begin{multline}
\hal \int_\tau^t \ns{D_j^s \eta}_{\h^{1/2+\rho}_\K}  \ls \E_{\shortparallel, 0}(\tau) + \E_{\shortparallel, 0}(t) +  \int_\tau^t \norm{u}_{H^0} \left(\norm{ \eta }_{ H^{s-1/2}} + \norm{\dt \eta }_{ H^{s-1/2}} \right) \\ 
 + \int_\tau^t \left(\ns{ u}_{H^1}  + [\dt \eta]_\ell^2 \right) + \int_\tau^t  \ns{ \eta}_{H^{1+s/2}} \sqrt{\E}.
\end{multline}
Note that from \eqref{diss_enhance_eta_rho}, Proposition \ref{Djr_bnds}, and Theorem \ref{s_embed} we have that
\begin{equation}
 \lim_{j \to \infty} \ns{D_j^s \eta}_{\h^{1/2+\rho}_\K}  =   \ns{\eta}_{\h_\K^{1+s/2}} \asymp \ns{\eta}_{H^{1+s/2}}.
\end{equation}
We then send $j \to \infty$ and use this and Fatou's lemma to see that  
\begin{multline}
\hal \int_\tau^t \ns{ \eta}_{H^{1+s/2}}  \ls \E_{\shortparallel, 0}(\tau) + \E_{\shortparallel, 0}(t) +  \int_\tau^t \norm{u}_{H^0} \left(\norm{ \eta }_{ H^{s-1/2}} + \norm{\dt \eta }_{ H^{s-1/2}} \right) \\ 
 + \int_\tau^t \left(\ns{ u}_{H^1}  + [\dt \eta]_\ell^2 \right) + \int_\tau^t  \ns{ \eta}_{H^{1+s/2}} \sqrt{\E}.
\end{multline}
Since $s -1/2 \le 1 + s/2$ we can then use Cauchy's inequality once more in addition to the smallness $\E \le \delta_\ast$ for some universal $\delta_\ast>0$ to conclude that 
\begin{equation}\label{diss_enhance_eta_2}
\frac{1}{4} \int_\tau^t \ns{ \eta}_{H^{1+s/2}}  \ls \E_{\shortparallel, 0}(\tau) + \E_{\shortparallel, 0}(t) +  \int_\tau^t  \left(\ns{u}_{H^0} + \norm{u}_{H^0} \norm{\dt \eta }_{ H^{s-1/2}} \right)  
 + \int_\tau^t \left(\ns{ u}_{H^1}  + [\dt \eta]_\ell^2 \right) .
\end{equation}
Finally, we use the equation $\dt \eta = u \cdot \N= u\cdot (-\p_1 \zeta_0,1) - u_1 \p_1 \eta$,  Theorem \ref{supercrit_prod}, and the fact that $\E \le 1$  to estimate 
\begin{equation}
 \norm{\dt \eta}_{H^{s-1/2}} \ls \norm{u}_{H^{s-1/2}} \left(1 + \norm{\p_1 \eta}_{H^1}   \right)   \ls \norm{u}_{H^s} \ls \norm{u}_{H^1}.
\end{equation}
Plugging this into \eqref{diss_enhance_eta_2} then shows that 
\begin{equation}
 \frac{1}{4} \int_\tau^t \ns{ \eta}_{H^{1+s/2}}  \ls \E_{\shortparallel, 0}(\tau) + \E_{\shortparallel, 0}(t) 
 + \int_\tau^t \left(\ns{ u}_{H^1}  + [\dt \eta]_\ell^2 \right)  = \E_{\shortparallel, 0}(\tau) + \E_{\shortparallel, 0}(t) + \int_{\tau}^t \D_{\shortparallel,0},
\end{equation}
which is the desired bound since $1+s/2= 3/2 - \low$.
\end{proof}

\subsection{Dissipative enhancement for $\dt \eta$ and $\dt^2 \eta$ }

We now turn our attention to enhanced dissipation estimates for $\dt \eta$ and $\dt^2 \eta$.

\begin{thm}\label{diss_enhance_dtketa}
Let $\low \in (0,1)$ be given by \eqref{kappa_ep_def}, and $0 < T \le \infty$.  Let $k \in \{1,2\}$.  There exists a universal $0 <\delta_\ast <1$ such that if  $\sup_{0\le t < T} \E(t) \le \delta_\ast$, then for every $0 \le \tau \le t < T$ we have the estimate
\begin{equation}
 \int_{\tau}^t \ns{\dt^k \eta}_{H^{3/2 - \low}}   \ls \E_{\shortparallel,k}(\tau) + \E_{\shortparallel,k}(t) 
+ \int_{\tau}^t \left( \sdb + \E\D\right).
\end{equation}
\end{thm}
\begin{proof}
We will give the proof only in the harder case $k=2$. The case $k=1$ follows from a similar, simpler argument.   To begin, we assume that $\delta_\ast < \gamma^2$, where $\gamma \in (0,1)$ is as in Lemma \ref{eta_small}.  In particular, this means that the estimates of Lemma \ref{eta_small} are available in what follows.

We begin in essentially the same way as in the proof of Theorem \ref{diss_enhance_eta}.  Let $s = 1- 2\low \in [0,1)$, which means that $3/2 - \low = 1 + s/2$.  Also let $\rho= (1-s)/2$ so that $1/2 + \rho + s = 1 + s/2$.  For a fixed $j \in \mathbb{N}$ we let $\psi$ solve \eqref{diss_enhace_psi_eqn} with data $D^s_j \dt^2 \eta / \abs{\N_0}$.  Then Proposition \ref{psi_solve} provides the estimates
\begin{equation}\label{diss_enhance_dtketa_1}
\norm{\psi}_{H^1} \ls    \norm{\dt^2 \eta }_{ H^{s-1/2}}, \; \norm{\psi}_{H^2}  \ls  \norm{D_j^s  \dt^2 \eta}_{\h^{1/2}_\K}, \text{ and } \norm{\dt \psi}_{H^1} \ls  \norm{ \dt^3 \eta }_{ H^{s-1/2}}.
\end{equation}
Note that $s-1/2 = 1/2 - 2 \low < 1/2 - \low$, so the latter term is controlled by the dissipation (see \eqref{D_def}).  Then  Proposition \ref{M_def} lets us use Lemma \ref{geometric_evolution} with $w = M \nab \psi$ to see that
\begin{multline}\label{diss_enhance_dtketa_2}
\br{\dt^3 u,Jw} + (-\dt \bar{\eta} \frac{\phi}{\zeta_0} K \p_2 \dt^2 u + u \cdot \naba \dt^2 u,Jw)_0 +
  \pp{\dt^2 u,w}  + (\dt^2 \eta ,w\cdot \N)_{1,\Sigma} + \linz [\dt^3 \eta ,w\cdot \N]_\ell  \\
 = \int_\Omega F^1   \cdot w J   - \int_{\Sigma_s} J (w\cdot \tau)F^5 
- \int_{-\ell}^\ell \sigma  F^3   \p_1 (w \cdot \N) + F^4 \cdot w  
  - \linz [w\cdot \N,  F^7]_\ell.
\end{multline}
Here the forcing terms on the right are as defined in Appendix \ref{sec_nonlinear_records}.  Arguing as in the proof of Theorem \ref{diss_enhance_eta}, we estimate all of the terms on the left of \eqref{diss_enhance_dtketa_2} to arrive at the bounds
\begin{equation}\label{diss_enhance_dtketa_25}
 \abs{\int_\Omega   \dt^2 u \cdot \nab\Phi \nab \psi} \ls \norm{\dt^2 u}_{H^0} \norm{\dt^2 \eta}_{H^1} \ls \E_{\shortparallel,2},
\end{equation}
where $\E_{\shortparallel,2}$ is as defined in \eqref{ED_natural}, and 
\begin{multline}\label{diss_enhance_dtketa_3}
\ns{D_j^s \dt^2 \eta}_{\h^{1/2 + \rho}_\K}  +  \frac{d}{dt} \int_\Omega   \dt^2 u \cdot \nab\Phi \nab \psi   \ls  
 \norm{\dt^2 u}_{H^0} \left(\norm{ \dt^2 \eta }_{ H^{s-1/2}} + \norm{\dt^3 \eta }_{ H^{s-1/2}} \right) 
\\ + \norm{ \dt^2 u}_{H^1} \norm{\dt^2 \eta}_{H^{1+s/2}}  
  + [\dt^3 \eta]_\ell \norm{ \dt^2\eta}_{H^{1+s/2}}  + \br{\mathcal{F}, M \nab \psi},
\end{multline}
where, as shorthand, we have written
\begin{multline}\label{diss_enhance_dtketa_4}
 \br{\mathcal{F}, M \nab \psi} = \int_\Omega  F^1   \cdot M \nab \psi J   - \int_{\Sigma_s} J (M \nab \psi\cdot \tau)F^5 \\
- \int_{-\ell}^\ell \sigma  \left(F^3   \p_1 (M \nab \psi \cdot \N) + F^4 \cdot M \nab \psi\right)  
  - [M \nab \psi\cdot \N,  F^7]_\ell.
\end{multline}

We now estimate $\mathcal{F}$, breaking it into three separate pieces.  For the first piece we use Theorem \ref{nid_v_est}, Proposition \ref{M_properties}, and \eqref{diss_enhance_dtketa_1} to estimate 
\begin{multline}\label{diss_enhance_dtketa_5}
\abs{\int_\Omega F^1 \cdot (M \nab \psi) J 
-  \int_{-\ell}^\ell  F^4 \cdot (M \nab \psi) 
- \int_{\Sigma_s}  J ((M \nab \psi) \cdot \tau)F^5 } 
 \ls  \norm{M \nab \psi}_{H^1}  (\E+ \sqrt{\E})\sqrt{\D} \\
  \ls  \norm{\psi}_{H^2}  (\E+ \sqrt{\E})\sqrt{\D}
  \ls  \norm{D^s_j \dt^2 \eta}_{\h_\K^{1/2}}  (\E+ \sqrt{\E})\sqrt{\D}
\end{multline}

Next we handle the $F^3$ term.  According to \eqref{dt2_f3} we have that 
\begin{equation}
 F^3 = \dt^2 [ \mathcal{R}(\p_1 \zeta_0,\p_1 \eta)] = \p_z \mathcal{R}(\p_1 \zeta_0,\p_1 \eta) \p_1 \dt^2\eta + \p_z^2 \mathcal{R}(\p_1 \zeta_0,\p_1 \eta) (\p_1 \dt \eta)^2.
\end{equation}
On the other hand, we know that $M \nab \psi \cdot \N = D_j^s \dt^2 \eta$ on $\Sigma$.  Combining these, and employing Proposition \ref{frac_IBP_prop}, we can estimate  
\begin{multline}
 \abs{ \int_{-\ell}^\ell \sigma F^3   \p_1 (M \nab \psi \cdot \N) } \le \sigma \abs{ \int_{-\ell}^\ell \p_1( D^s_j \dt^2 \eta )\p_z \mathcal{R}(\p_1 \zeta_0,\p_1 \eta) \p_1 \dt^2\eta  } \\
 + \sigma \abs{\int_{-\ell}^\ell \p_1( D^s_j \dt^2 \eta ) \p_z^2 \mathcal{R}(\p_1 \zeta_0,\p_1 \eta) (\p_1 \dt \eta)^2 } \ls
\norm{D_j^s \dt^2 \eta}_{H^{1-s/2}} \norm{\p_z \mathcal{R}(\p_1 \zeta_0,\p_1 \eta) \p_1 \dt^2\eta  }_{H^{s/2}} \\
+  \norm{D_j^s \dt^2 \eta}_{H^{1-s/2}} \norm{ \p_z^2 \mathcal{R}(\p_1 \zeta_0,\p_1 \eta) (\p_1 \dt \eta)^2}_{H^{s/2}}.
\end{multline}
Note that 
\begin{equation}
\hal = \frac{1}{q_+} - \frac{1}{1} \left( \hal + \frac{\ep_+}{2}\right) 
\end{equation}
so the Sobolev embeddings imply that $W^{1,q_+}((-\ell,\ell)) \hookrightarrow H^{(1+\ep_+)/2}((-\ell,\ell))$ and
\begin{equation}
W^{2,q_+}((-\ell,\ell)) \hookrightarrow H^{(3+\ep_+)/2}((-\ell,\ell))  \hookrightarrow H^{3/2-\low}((-\ell,\ell)) =   H^{1+s/2}((-\ell,\ell)). 
\end{equation}
These and Theorem \ref{supercrit_prod} then imply that
\begin{multline}
\norm{\p_z \mathcal{R}(\p_1 \zeta_0,\p_1 \eta) \p_1 \dt^2\eta  }_{H^{s/2}} \ls  \norm{\p_1 \dt^2\eta }_{H^{s/2}} \norm{\p_z \mathcal{R}(\p_1 \zeta_0,\p_1 \eta)}_{H^{(1+\ep_+)/2}} \\
\ls \norm{\dt^2\eta }_{H^{1+s/2}} \norm{\p_z \mathcal{R}(\p_1 \zeta_0,\p_1 \eta)}_{W^{1,q_+}}
\end{multline}
and 
\begin{multline}
\norm{ \p_z^2 \mathcal{R}(\p_1 \zeta_0,\p_1 \eta) (\p_1 \dt \eta)^2}_{H^{s/2}} \ls \norm{(\p_1 \dt \eta)^2}_{H^{s/2}} \norm{\p_z^2 \mathcal{R}(\p_1 \zeta_0,\p_1 \eta)}_{H^{(1+\ep_+)/2}} \\
\ls \norm{\p_1 \dt \eta}_{H^{s/2}} \norm{\p_1 \dt \eta}_{H^{(1+\ep_+)/2}} \norm{\p_z^2 \mathcal{R}(\p_1 \zeta_0,\p_1 \eta)}_{W^{1,q_+}} \\
\ls  \ns{\dt \eta}_{W^{2,q_+}} \norm{\p_z^2 \mathcal{R}(\p_1 \zeta_0,\p_1 \eta)}_{W^{1,q_+}}.
\end{multline}
Since the terms involving $\mathcal{R}$ involve an integer derivative count, we can employ Proposition \ref{R_prop} to estimate 
\begin{equation}
\norm{\p_z \mathcal{R}(\p_1 \zeta_0,\p_1 \eta)}_{W^{1,q_+}} +  \norm{\p_z^2 \mathcal{R}(\p_1 \zeta_0,\p_1 \eta)}_{W^{1,q_+}} \ls \norm{\eta}_{W^{2,q_+}}.
\end{equation}
Hence, 
\begin{multline}\label{diss_enhance_dtketa_6}
  \abs{ \int_{-\ell}^\ell \sigma F^3   \p_1 (M \nab \psi \cdot \N) } \ls \ns{\dt^2 \eta}_{H^{1+s/2}} \norm{\eta}_{W^{2,q_+}} +  \norm{\dt^2 \eta}_{H^{1+s/2}}  \ns{\dt \eta}_{W^{2,q_+}} \norm{\eta}_{W^{2,q_+}} \\
 \ls \ns{\dt^2 \eta}_{H^{1+s/2}} \sqrt{\E} +  \norm{\dt^2 \eta}_{H^{1+s/2}} \E \sqrt{\D}. 
\end{multline}

Lastly, we handle the $F^7$ term, again using that $M \nab \psi \cdot \N = D_j^2 \dt^2 \eta$ on $\Sigma$.  Then \eqref{dt2_f7} and standard trace theory shows that 
\begin{multline}
\abs{\linz [M \nab \psi\cdot \N,  F^7]_\ell } = \linz \abs{ [D_j^s \dt^2 \eta,  \swh'(\dt \eta) \dt^3 \eta + \swh''(\dt \eta) (\dt^2 \eta)^2]_\ell } \\
\ls \norm{D_j^s \dt^2 \eta}_{H^{1-s/2}} \max_{\pm \ell}\abs{\swh'(\dt \eta) \dt^3 \eta + \swh''(\dt \eta) (\dt^2 \eta)^2}. 
\end{multline}
According to Theorem \ref{catalog_energy}, $\norm{\dt \eta}_{C^0_b} \ls \sqrt{\E} \ls 1,$ so we may estimate 
\begin{equation}
 \abs{\swh'(z)} = \frac{1}{\low}\abs{ \int_0^z \sw''(r)dr}\ls \abs{z} \text{ for } z \in [-\norm{\dt \eta}_{C^0},\norm{\dt \eta}_{C^0}].
\end{equation}
This and trace theory then provide the bound
\begin{multline}
\max_{\pm \ell}\abs{\swh'(\dt \eta) \dt^3 \eta + \swh''(\dt \eta) (\dt^2 \eta)^2} \ls \max_{\pm \ell} \left(\abs{\dt \eta} \abs{\dt^3 \eta} + \abs{ \dt^2 \eta}^2\right) \\
\ls  \sqrt{\sdb} \left( \norm{\dt \eta}_{H^1} + \norm{\dt^2 \eta}_{H^1}   \right) \ls \sqrt{\seb} \sqrt{\sdb}, 
\end{multline}
where $\seb$ and $\sdb$ are as defined in \eqref{ED_parallel}. Hence,
\begin{equation}\label{diss_enhance_dtketa_7}
\abs{\linz [M \nab \psi\cdot \N,  F^7]_\ell } \ls \norm{ \dt^2 \eta}_{H^{1+s/2}}\sqrt{\seb} \sqrt{\sdb}.
\end{equation}

Upon plugging the estimates \eqref{diss_enhance_dtketa_5}, \eqref{diss_enhance_dtketa_6}, and \eqref{diss_enhance_dtketa_7} into \eqref{diss_enhance_dtketa_4}, we deduce that 
\begin{equation}
\abs{ \br{\mathcal{F}, M \nab \psi}} \ls   
\ns{\dt^2 \eta}_{H^{1+s/2}} \sqrt{\E} +  \norm{\dt^2 \eta}_{H^{1+s/2}} \sqrt{\E} \sqrt{\D}.
\end{equation}
Inserting this into \eqref{diss_enhance_dtketa_3}, integrating in time from $\tau$ to $t$, and using \eqref{diss_enhance_dtketa_25} then shows that
\begin{multline}
\int_{\tau}^t \ns{D_j^s \dt^2 \eta}_{\h^{1/2 + \rho}_\K}   \ls 
\E_{\shortparallel,2}(\tau) + \E_{\shortparallel,2}(t) 
+ \int_{\tau}^t \norm{\dt^2 u}_{H^0} \left(\norm{ \dt^2 \eta }_{ H^{s-1/2}} + \norm{\dt^3 \eta }_{ H^{s-1/2}} \right) 
\\ + \int_{\tau}^t \left( \norm{ \dt^2 u}_{H^1}   + [\dt^3 \eta]_\ell \right)\norm{\dt^2 \eta}_{H^{1+s/2}}  
+ \int_{\tau}^t \left(\ns{\dt^2 \eta}_{H^{1+s/2}} \sqrt{\E} +  \norm{\dt^2 \eta}_{H^{1+s/2}} \sqrt{\E} \sqrt{\D} \right).
\end{multline}
We then send $j \to \infty$ and argue as in the proof of Theorem \ref{diss_enhance_eta} to deduce from this that 
\begin{multline}
\int_{\tau}^t \ns{\dt^2 \eta}_{H^{1+s/2}}   \ls 
\E_{\shortparallel,2}(\tau) + \E_{\shortparallel,2}(t) 
+ \int_{\tau}^t \norm{\dt^2 u}_{H^0} \left(\norm{ \dt^2 \eta }_{ H^{s-1/2}} + \norm{\dt^3 \eta }_{ H^{s-1/2}} \right) 
\\ + \int_{\tau}^t \left( \norm{ \dt^2 u}_{H^1}   + [\dt^3 \eta]_\ell \right)\norm{\dt^2 \eta}_{H^{1+s/2}}  
+ \int_{\tau}^t \left(\ns{\dt^2 \eta}_{H^{1+s/2}} \sqrt{\E} +  \norm{\dt^2 \eta}_{H^{1+s/2}} \sqrt{\E} \sqrt{\D} \right).
\end{multline}
Finally, we use Cauchy's inequality, the fact that $s-1/2 < 1/2$, and the assumption that $\E \le \delta_\ast$ for a universal $0 <\delta_\ast \le 1$ to absorb the $\ns{\dt^2 \eta}_{H^{1+s/2}}$ terms from the right to the left, which yields
\begin{equation}
\hal \int_{\tau}^t \ns{\dt^2 \eta}_{H^{1+s/2}}   \ls \E_{\shortparallel,2}(\tau) + \E_{\shortparallel,2}(t) 
+ \int_{\tau}^t \sdb + \E\D.
\end{equation}
This then provides the desired estimate since $1 + s/2 = 3/2 - \low$.
\end{proof}

\subsection{Energetic enhancement for $\dt p$ }

We now turn our attention to an estimate that provides $L^2$ control of $\dt p$ in terms of the energy.

\begin{thm}\label{en_enhance_dtp}
Let $0 < T \le \infty$ and suppose that $\sup_{0\le t < T} \E(t) \le \gamma^2$, where $\gamma \in (0,1)$ is as in Lemma \ref{eta_small}.  Then we have the estimate
\begin{equation}\label{en_enhance_dtp_0}
 \norm{\dt p}_{L^2} \ls  \norm{\dt u}_{H^1} + \norm{\dt^2 u}_{L^2}+  \norm{\dt \eta}_{H^{3/2 + (\ep_- - \low)/2}}  + \E + \E^{3/2} .
\end{equation}
\end{thm}
\begin{proof}
Let $\psi \in H^2(\Omega)$ solve 
\begin{equation}
\begin{cases}
-\Delta \psi = \dt p & \text{in } \Omega \\
\psi =0 & \text{on } \Sigma \\
\p_\nu \psi =0 &\text{on } \Sigma_s,
\end{cases}
\end{equation}
which exists and enjoys  $H^2$ regularity since $\Omega$ has convex corners.  Moreover, 
\begin{equation}
\norm{\psi}_{H^2} \ls \norm{\dt p}_{L^2}.
\end{equation}
Proposition \ref{M_properties} shows that if we set $w = M \nab \psi$, then $w$ is a valid choice of a test function in Lemma \ref{geometric_evolution},
\begin{equation}
 \diva{w} = \diva{M\nab \psi} = K\Delta \psi = K \dt p,
\end{equation}
and we have the bound
\begin{equation}
 \norm{w}_{H^1} \ls \norm{\psi}_{H^2} \ls \norm{\dt p}_{L^2}.
\end{equation}

Using this $w$ in  \eqref{ge_01} of Lemma \ref{geometric_evolution}, we find that 
\begin{multline}\label{en_enhance_dtp_1}
\br{\dt^2 u,Jw} + (-\dt \bar{\eta} \frac{\phi}{\zeta_0} K \p_2 \dt u + u \cdot \naba \dt u,Jw)_0 +
  \pp{\dt u,w} - (\dt p,\diva w)_0   \\
 = \int_\Omega F^1   \cdot w J   - \int_{\Sigma_s} J (w\cdot \tau)F^5 
-   \int_{-\ell}^\ell g \dt \eta (w \cdot \N) - \sigma \p_1 \left( \frac{\p_1 \dt \eta }{(1+\abs{\p_1 \zeta_0}^2)^{3/2}} +F^3\right)w\cdot  \N + F^4 \cdot w
\end{multline}
with $F^1$, $F^3$, $F^4$, and $F^5$ given by \eqref{dt1_f1}, \eqref{dt1_f3}, \eqref{dt1_f4}, and \eqref{dt1_f5}, respectively, but 
\begin{equation}\label{en_enhance_dtp_2}
 (\dt p,\diva w)_0 = \int_\Omega J \dt p \diva w = \int_\Omega \abs{\dt p}^2 = \norm{\dt p}_{L^2}^2.
\end{equation}

According to Theorem \ref{nie_v_est}, we have the bound 
\begin{equation}\label{en_enhance_dtp_3}
 \abs{\int_\Omega F^1   \cdot w J   - \int_{\Sigma_s} J (w\cdot \tau)F^5 
+ F^4 \cdot w} \ls \left( \E+ \E^{3/2} \right) \norm{w}_{H^1} \ls \left( \E+ \E^{3/2} \right) \norm{\dt p}_{L^2},
\end{equation}
while Theorem \ref{nie_ST} shows that 
\begin{multline}
\abs{\int_{-\ell}^\ell g \dt \eta (w \cdot \N) - \sigma \p_1 \left( \frac{\p_1 \dt \eta }{(1+\abs{\p_1 \zeta_0}^2)^{3/2}} +F^3\right)w\cdot  \N } \ls \norm{\dt \eta}_{H^{3/2 + (\ep_- - \low)/2}} \norm{w}_{H^1}  \\
\ls  \norm{\dt \eta}_{H^{3/2 + (\ep_- - \low)/2}}  \norm{\dt p}_{L^2}.
\end{multline}
On the other hand, we have the bounds 
\begin{equation}
\abs{\pp{\dt u,w}}  \ls \norm{\dt u}_{H^1} \norm{w}_{H^1} \ls \norm{\dt u}_{H^1} \norm{\dt p}_{L^2},
\end{equation}
\begin{equation}
 \abs{\br{\dt^2 u,Jw}} \ls \norm{\dt^2 u}_{L^2} \norm{w}_{L^2} \ls  \norm{\dt^2 u}_{L^2}\norm{\dt p}_{L^2},
\end{equation}
and 
\begin{multline}\label{en_enhance_dtp_4}
\abs{  (-\dt \bar{\eta} \frac{\phi}{\zeta_0} K \p_2 \dt u + u \cdot \naba \dt u,Jw)_0  }  \ls \norm{w}_{L^2}\left( \norm{\dt \bar{\eta}}_{L^\infty} \norm{\nab \dt u}_{L^2}  + \norm{u}_{L^\infty}\norm{\dt u}_{L^2}\right)  \\
\ls  \norm{\dt p}_{L^2} \E.
\end{multline}

Plugging the estimates \eqref{en_enhance_dtp_3}--\eqref{en_enhance_dtp_4} into \eqref{en_enhance_dtp_1} and using \eqref{en_enhance_dtp_2}, we deduce that 
\begin{equation}
 \ns{\dt p}_{L^2} \ls \norm{\dt p}_{L^2} \left(\norm{\dt u}_{H^1} + \norm{\dt^2 u}_{L^2}+  \norm{\dt \eta}_{H^{3/2 + (\ep_- - \low)/2}}  + \E + \E^{3/2}  \right).
\end{equation}
Then \eqref{en_enhance_dtp_0} follows immediately from this.
\end{proof}

\section{A priori estimates}\label{sec_aprioris}
 
In this section we present the proof of our main a priori estimates, Theorem \ref{main_apriori}.

\subsection{A key construction}

We need one more technical tool to close our a priori estimates, namely the construction of a useful $\omega$ to use in Theorem \ref{linear_energy}.  We present the construction of such an $\omega$ now.

\begin{prop}\label{omega_construction}
Let $0 < T \le \infty$ and suppose that $\sup_{0\le t < T} \E(t) \le \gamma^2$, where $\gamma \in (0,1)$ is as in Lemma \ref{eta_small}. Let $F^2$ be given by \eqref{dt2_f2} and let $\br{\cdot}_\Omega$ denote the spatial average on $\Omega$, i.e. 
\begin{equation}
 \br{g}_\Omega = \frac{1}{\abs{\Omega}} \int_{\Omega} g.
\end{equation}
Then there exists $\omega : \Omega \times [0,T) \to \R^2$ satisfying the following.
\begin{enumerate}
 \item We have that $\omega(\cdot,t) \in H^1_0(\Omega;\R^2)$ for $0 \le t < T$, and 
 \begin{equation}\label{omega_construction_00}
 J \diva{\omega} = JF^2  -  \br{JF^2}_\Omega.
\end{equation}
 \item $\omega$ obeys the estimates 
\begin{equation} \label{omega_construction_01}
\norm{\omega }_{W^{1,4/(3-2\ep_+)}_0} \ls \E   \text{ and } \norm{\omega}_{W^{1,2/(1-\ep_-)}_0}  + \norm{\dt \omega}_{L^{2/(1-\ep_-)}} \ls (\sqrt{\E} +\E)\sqrt{\D}.
\end{equation}

 \item We have the interaction estimates 
\begin{equation} \label{omega_construction_02}
 \abs{\int_\Omega \dt^2 u J \omega} \ls   \E^{3/2},
\end{equation}
and
\begin{equation} \label{omega_construction_03}
 \abs{\int_{\Omega} \dt^2 u \dt(J \omega)} + \abs{(-\dt \bar{\eta} \frac{\phi}{\zeta_0} K \p_2 \dt^2 u + u \cdot \naba \dt^2 u,J\omega)_0}  + \abs{ \pp{\dt^2 u,\omega}} +  \abs{\int_\Omega J F^1 \cdot \omega}     \ls (\sqrt{\E} + \E) \D.
\end{equation}
\end{enumerate}
\end{prop}
\begin{proof}
Recall from Proposition \ref{M_properties} that  
\begin{equation}
\diverge{u} = \varphi \Leftrightarrow  \diva(M u) = K \varphi \Leftrightarrow J \diva (Mu) = \varphi.
\end{equation}
This means that if we first solve 
\begin{equation}
 \diverge{\bar{\omega}} =  JF^2  -  \br{JF^2}_\Omega,
\end{equation}
then $\omega = M \bar{\omega}$ satisfies \eqref{omega_construction_00}.

Let $\mathcal{B}_\Omega$ denote the Bogovskii operator from Proposition \ref{bogovskii}.  Then we will define 
\begin{equation}
 \bar{\omega} = \mathcal{B}_\Omega (JF^2  -  \br{JF^2}_\Omega ).
\end{equation}
The essential point is that the Bogovskii operator is a linear map that commutes with time derivatives and satisfies 
\begin{equation}\label{omega_5}
 \mathcal{B}_\Omega \in \L( \mathring{L}^{q}(\Omega), W^{1,q}_0(\Omega;\R^2)  ) \text{ for all } 1 < q < \infty
\end{equation}
and $\diverge \mathcal{B}_\Omega \varphi = \varphi$.  Then our desired vector field is given by 
\begin{equation}
 \omega = M \bar{\omega} = M \mathcal{B}_\Omega(J F^2 - \br{J F^2}_\Omega).
\end{equation}

According to Propositions \ref{ne2_f2} and \eqref{omega_5} we have the bounds
\begin{equation}\label{omega_1}
\norm{\bar{\omega} }_{W^{1,4/(3-2\ep_+)}_0} \ls \E, \; \norm{\bar{\omega}}_{W^{1,2/(1-\ep_-)}_0} \ls \sqrt{\E} \sqrt{\D}, \text{ and }  \norm{\dt \bar{\omega}}_{W^{1,q_-}_0} \ls (\sqrt{\E} + \E)\sqrt{\D}. 
\end{equation}
Then Proposition \ref{M_multiplier},  together with \eqref{omega_1} and the fact that $\E \le 1$ then shows that 
\begin{equation}\label{omega_2}
\norm{\omega }_{W^{1,4/(3-2\ep_+)}_0} \ls \E \text{ and } \norm{\omega}_{W^{1,2/(1-\ep_-)}_0} \ls \sqrt{\E} \sqrt{\D}. 
\end{equation}
and (since $\ep_- < \ep_+$ implies $2/(1-\ep_-) < 2/(1-\ep_+)$)
\begin{multline}\label{omega_3}
\norm{\dt \omega}_{L^{2/(1-\ep_-)}} \ls \norm{\dt M \bar{\omega}}_{L^{2/(1-\ep_-)}} + \norm{M \dt \bar{\omega}}_{L^{2/(1-\ep_-)}} 
\ls (1+ \sqrt{\E}) \left( \norm{\bar{\omega}}_{L^{2/(1-\ep_- )}}  +  \norm{\dt \bar{\omega}}_{L^{2/(1-\ep_-)}_0}  \right) 
\\
\ls (1+ \sqrt{\E})\left( \norm{\bar{\omega}}_{W^{1,2/(1-\ep_- )}_0}  +  \norm{\dt \bar{\omega}}_{W^{1,q_-}_0}  \right) 
\ls  (\sqrt{\E} + \E)\sqrt{\D},
\end{multline}
where in the third inequality we have also used the Sobolev embeddings.  Then \eqref{omega_construction_01} follows from \eqref{omega_2} and \eqref{omega_3}.

It remains only to prove the interaction estimates stated in the third item.  For each of these we will use the estimates \eqref{omega_construction_01} together with the bounds from Theorems \ref{catalog_energy} and \ref{catalog_dissipation}.  Indeed,   
\begin{equation}
 \abs{\int_\Omega \dt^2 u J \omega} \ls \int_\Omega \abs{\dt^2 u} \abs{\omega} \ls  \norm{\dt^2 u}_{L^2} \norm{\omega}_{L^2} \ls \norm{\omega}_{W^{1,4/(3-2\ep_+)}} \norm{\dt^2 u}_{L^2} \ls \E^{3/2},
\end{equation}
which is \eqref{omega_construction_02}.  For the first part of \eqref{omega_construction_03} we bound
\begin{multline}
 \abs{\int_{\Omega} \dt^2 u \dt(J \omega)} \ls \int_{\Omega} \abs{\dt^2 u} \left( \abs{\nab \dt \bar{\eta}} \abs{\omega}  + \abs{\dt \omega} \right) \ls \norm{\dt^2 u}_{L^2} \left(\norm{\dt \bar{\eta}}_{W^{1,\infty}} \norm{\omega}_{L^2} + \norm{\dt \omega}_{L^2}  \right) \\
\ls \sqrt{\D}(\sqrt{\E} \sqrt{\E} \sqrt{\D} + \sqrt{\E} \sqrt{\D}) \ls  (\sqrt{\E} + \E) \D.
\end{multline}
Next we bound 
\begin{equation}
\abs{(-\dt \bar{\eta} \frac{\phi}{\zeta_0} K \p_2 \dt^2 u + u \cdot \naba \dt^2 u,J\omega)_0} \ls \left( \norm{\dt \bar{\eta}}_{L^\infty}  + \norm{u}_{L^\infty}  \right) \norm{\nab \dt^2 u}_{L^2} \norm{\omega}_{L^2} \ls \sqrt{\E} \sqrt{\D} \sqrt{\E} \sqrt{\D} \ls \E \D,
\end{equation}
which is the second estimate in \eqref{omega_construction_03}.  Then we bound 
\begin{equation}
 \abs{ \pp{\dt^2 u,\omega}} \ls \norm{\dt^2 u}_{H^1} \norm{\omega}_{H^1} \ls \norm{\dt^2 u}_{H^1} 
 \norm{\omega}_{W^{1,2/(1-\ep_-)}} \ls \sqrt{\D} \sqrt{\E} \sqrt{\D},
\end{equation}
which is the third estimate in \eqref{omega_construction_03}.  For the final term in \eqref{omega_construction_03} we use Proposition \ref{nid_f1} to bound 
\begin{equation}
 \abs{\int_\Omega J F^1 \cdot \omega} \ls  \norm{\omega}_{H^1} (\sqrt{\E} + \E) \sqrt{\D} \ls \sqrt{\E} \sqrt{\D} (\sqrt{\E} + \E) \sqrt{\D}.
\end{equation}
This completes the proof of \eqref{omega_construction_03}.
\end{proof}

\subsection{Main a priori estimate}

We now have all of the tools needed to prove our main a priori estimate.  

\begin{proof}[Proof of Theorem \ref{main_apriori}]

Assume initially that $\delta_0 \le \gamma^2$, where $\gamma \in (0,1)$ is from Lemma \ref{eta_small}.  We divide the rest of the proof into several steps.

\emph{Step 1 - Lowest level energy-dissipation estimates: } Corollary \ref{basic_energy} tells us that
\begin{multline}
 \frac{d}{dt} \left(\int_\Omega \hal J \abs{u}^2 +  \int_{-\ell}^\ell \frac{g}{2} \abs{\eta}^2 + \frac{\sigma}{2} \frac{\abs{\p_1 \eta}^2}{(1+\abs{\p_1 \zeta_0}^2)^{3/2}} +  \int_{-\ell}^\ell \sigma \Q(\p_1 \zeta_0,\p_1 \eta)\right)  \\
 + \frac{\mu}{2} \int_\Omega \abs{\sga u}^2 J 
+\int_{\Sigma_s} \beta J \abs{u \cdot s}^2 + \linz \bs{\dt \eta}
 =- \linz [u\cdot \N,\swh(\dt \eta)]_\ell .
\end{multline}
We integrate this and use Lemma \ref{eta_small}  to deduce that 
\begin{equation}\label{ape_sk_-1}
  \E_{\shortparallel, 0}(t) + \int_{-\ell}^\ell \sigma \Q(\p_1 \zeta_0,\p_1 \eta(t))  + \int_{s}^t \D_{\shortparallel,0}  \ls  \E_{\shortparallel, 0}(s)  + \int_{-\ell}^\ell \sigma \Q(\p_1 \zeta_0,\p_1 \eta(s)) + \int_{s}^t \linz \abs{[u\cdot \N,\swh(\dt \eta)]_\ell}.
\end{equation}
Theorem \ref{nid_Q} says that
\begin{equation}
 \abs{\int_{-\ell}^\ell \sigma \Q(\p_1 \zeta_0, \p_1 \eta)  } \ls  \sqrt{\E} \ns{\eta}_{H^1} \ls \sqrt{\E} \E_{\shortparallel,0},
\end{equation}
and Theorem \ref{nid_W} says that
\begin{equation} 
\abs{  [u\cdot \N,\swh(\dt \eta)]_\ell } \ls \norm{\dt \eta}_{H^1} \bs{\dt \eta} \ls \sqrt{\E} \D_{\shortparallel,0},
\end{equation}
so if $\E \le \delta_0$ with $\delta_0$ sufficiently small, then \eqref{ape_sk_-1} implies that
\begin{equation}
  \E_{\shortparallel, 0}(t)  + \int_{s}^t \D_{\shortparallel,0}  \ls  \E_{\shortparallel, 0}(s).
\end{equation}
Then Theorem \ref{diss_enhance_eta} says
\begin{equation}
  \int_s^t \ns{ \eta}_{H^{3/2-\low}}  \ls  \E_{\shortparallel, 0}(s) + \E_{\shortparallel, 0}(t) + \int_{s}^t \D_{\shortparallel,0}
\end{equation}
and we may enhance the previous bound to
\begin{equation}\label{ape_sk_0}
  \E_{\shortparallel, 0}(t)  + \int_{s}^t \left(\D_{\shortparallel,0} + \ns{ \eta}_{H^{3/2-\low}}  \right)  \ls  \E_{\shortparallel, 0}(s)
\end{equation}
for all $0 \le s \le t \le T$.  

\emph{Step 2 - Energy-dissipation estimates for one temporal derivative: } Theorem \ref{linear_energy} applied with $(v,q,\xi) = (\dt u,\dt p,\dt \eta)$ and $\omega =0$ gives the identity
\begin{multline}
 \frac{d}{dt} \left( \int_{\Omega} J \frac{\abs{\dt u}^2}{2} +    \int_{-\ell}^\ell \frac{g}{2} \abs{\dt \eta}^2 + \frac{\sigma}{2} \frac{\abs{\p_1 \dt \eta}^2}{(1+\abs{\p_1 \zeta_0}^2)^{3/2}}  \right) \\
 + \frac{\mu}{2} \int_\Omega \abs{\sga \dt u}^2 J 
+\int_{\Sigma_s} \beta J \abs{\dt u \cdot s}^2 + \linz\bs{\dt^2 \eta}  
= \br{\mathcal{F}_1,(\dt u,\dt p,\dt \eta)}
\end{multline}
for 
\begin{multline}
\br{\mathcal{F}_1,(\dt u,\dt p, \dt \eta)} =  \int_\Omega F^1  \cdot \dt u J + \dt p JF^2  - \int_{\Sigma_s}  J (\dt u \cdot s)F^5 \\
-  \int_{-\ell}^\ell \sigma  F^3  \p_1(\dt u \cdot \N) + F^4 \cdot \dt u  -  g \dt \eta F^6 -  \sigma \frac{\p_1 \dt \eta \p_1 F^6}{(1+\abs{\p_1 \zeta_0}^2)^{3/2}}
  - \linz [\dt u\cdot \N,  F^7]_\ell + \linz [\dt^2 \eta,F^6] .
\end{multline}
Integrating and using Lemma \ref{eta_small} then shows that 
\begin{equation}
 \E_{\shortparallel, 1}(t) + \int_{s}^t \D_{\shortparallel, 1} \ls   \E_{\shortparallel, 1}(s)
+ \int_{s}^t    \br{\mathcal{F}_1,(\dt u,\dt p,\dt \eta)}.
\end{equation}
Theorems  \ref{nid_v_est}, \ref{nid_p_est}, \ref{nid_f3_dt}, \ref{nid_f6}, and \ref{nid_f7} then show that 
\begin{equation}
\abs{  \br{\mathcal{F}_1,(\dt u,\dt p,\dt \eta)} } \ls \sqrt{\E} \D,
\end{equation}
and hence we have the bound
\begin{equation}\label{ape_sk_1}
 \E_{\shortparallel, 1}(t) + \int_{s}^t \D_{\shortparallel, 1} \ls   \E_{\shortparallel, 1}(s)
+ \int_{s}^t    \sqrt{\E} \D.
\end{equation}

\emph{Step 3 - Energy-dissipation estimates with two temporal derivatives: } Theorem \ref{linear_energy} applied with $(v,q,\xi) = (\dt^2 u,\dt^2 p,\dt^2 \eta)$ and $\omega$ from Proposition \ref{omega_construction} (which guarantees that $\omega$ can be used in Theorem \ref{linear_energy}) yields
\begin{multline} 
 \frac{d}{dt} \left( \int_{\Omega} J \frac{\abs{\dt^2 u}^2}{2} +    \int_{-\ell}^\ell \frac{g}{2} \abs{\dt^2 \eta}^2 + \frac{\sigma}{2} \frac{\abs{\p_1 \dt^2\eta}^2}{(1+\abs{\p_1 \zeta_0}^2)^{3/2}} - \int_\Omega J \dt^2 u\cdot \omega \right) \\
 + \frac{\mu}{2} \int_\Omega \abs{\sga \dt^2 u}^2 J 
+\int_{\Sigma_s} \beta J \abs{\dt^2 u\cdot s}^2 + \linz \bs{\dt^3 \eta}  
= \br{\mathcal{F}_2,(\dt^2 u,\dt^2 \eta)} + \int_\Omega \dt^2p \br{JF^2}_\Omega  + \br{\mathcal{F}_3,\omega}
\end{multline}
where $\br{\cdot}_\Omega$ denotes the spatial average as in Proposition \ref{omega_construction},
\begin{multline}
\br{\mathcal{F}_2,(\dt^2 u,\dt^2 \eta)} = \int_\Omega F^1  \cdot \dt^2 u J   - \int_{\Sigma_s}  J (\dt^2 u \cdot s)F^5 \\
-  \int_{-\ell}^\ell \sigma  F^3  \p_1(\dt^2 u \cdot \N) + F^4 \cdot \dt^2u  -  g \xi F^6 -  \sigma \frac{\p_1 \dt^2 \eta \p_1 F^6}{(1+\abs{\p_1 \zeta_0}^2)^{3/2}}
  - \linz [\dt^2 u\cdot \N,  F^7]_\ell + \linz [\dt^3 \eta,F^6],
\end{multline}
and 
\begin{equation}
 \br{\mathcal{F}_3,\omega} = - \int_\Omega \dt^2 u \cdot \dt(J \omega)  +   (-\dt \bar{\eta} \frac{\phi}{\zeta_0} K \p_2 \dt^2 u + u \cdot \naba \dt^2 u,J \omega)_0 + \pp{\dt^2 u,\omega}  - \int_\Omega F^1 \omega J.
\end{equation}

Theorems  \ref{nid_v_est}, \ref{nid_f3_dt2}, \ref{nid_f6}, and \ref{nid_f7} show that 
\begin{equation}
\abs{ \int_{s}^t \br{\mathcal{F}_2,(\dt^2 u,\dt^2 \eta)}}  \ls \sqrt{\E} \D.
\end{equation}
For the second term we rewrite 
\begin{equation}
 \int_\Omega \dt^2 p \br{J F^2}_\Omega = \frac{d}{dt} \left(\br{J F^2}_\Omega  \int_\Omega \dt p    \right) - \dt  \br{J F^2}_\Omega  \int_\Omega \dt p    =: \frac{d}{dt} I_1 - I_2.
\end{equation}
We then use Proposition \ref{ne2_f2} to bound 
\begin{equation}
 \abs{I_1} \ls \norm{\dt p}_{L^2}  \abs{\br{J F^2}_\Omega} \ls \E^{3/2}
\end{equation}
and (since $\dt \br{J F^2}_\Omega = \br{\dt(J F^2)}_\Omega$)  
\begin{equation}
 \abs{I_2} \ls \norm{\dt p}_{L^2} \abs{ \dt \br{J F^2}_\Omega}  \ls \sqrt{\D}   \sqrt{\E} \sqrt{\D}  .
\end{equation}
Finally, the interaction estimates of Proposition \ref{omega_construction} show that 
\begin{equation}
\abs{ \br{\mathcal{F}_3,\omega}} \ls \sqrt{\E} \D. 
\end{equation}
Combining all the above then shows that 
\begin{equation}\label{ape_sk_2}
 \E_{\shortparallel, 2}(t)  - (\E(t))^{3/2}  + \int_{s}^t \D_{\shortparallel, 2} \ls   \E_{\shortparallel, 2}(s) + (\E(s))^{3/2}
+  \int_{s}^t  \sqrt{\E} \D .
\end{equation}

\emph{Step 4 - Synthesized energy-dissipation estimates: } We sum \eqref{ape_sk_0}, \eqref{ape_sk_1}, and  \eqref{ape_sk_2} to see that 
\begin{equation}\label{ape_1}
 \seb(t)  - (\E(t))^{3/2}  + \int_{s}^t \left(\sdb + \ns{\eta}_{H^{3/2-\low}} \right) \ls   \seb(s) + (\E(s))^{3/2}
+  \int_{s}^t  \sqrt{\E} \D .
\end{equation}
Subsequently, we sum the estimates provided by Theorem \ref{diss_enhance_dtketa} with $k=1$ and $k=2$ to deduce the enhancement estimate
\begin{equation}
 \int_{s}^t \ns{\dt \eta}_{H^{3/2 - \low}}  + \ns{\dt^2 \eta}_{H^{3/2 - \low}}   \ls \seb(s) + \seb(t) 
+ \int_{s}^t \left( \sdb + \sqrt{\E}\D\right),
\end{equation}
and upon combining this with \eqref{ape_1} we find that 
\begin{equation}\label{ape_2}
 \seb(t)  - (\E(t))^{3/2}  + \int_{s}^t \left(\sdb + \sum_{k=0}^2 \ns{\dt^k \eta}_{H^{3/2-\low}} \right) \ls   \seb(s) + (\E(s))^{3/2}
+  \int_{s}^t  \sqrt{\E} \D .
\end{equation}

\emph{Step 5 - Elliptic dissipation enhancements: }   We now combine the estimates of Propositions \ref{ne1_g1}--\ref{ne1_g7}  with Theorem \ref{A_stokes_stress_solve}, applied with $v = \dt u$, $Q = \dt p$, and $\xi = \dt \eta$ and $\delta = \ep_-$, to see that 
\begin{equation}
 \norm{\dt u}_{W^{2,q_-}} + \norm{\dt p}_{W^{1,q_-}} + \norm{\dt \eta}_{W^{3-1/q_-,q_-}} \ls \norm{\dt^2 u}_{L^{q_-}} + \norm{\dt^2 \eta}_{H^{3/2-\low}} + \sqrt{\E} \sqrt{\D}.
\end{equation}
Similarly, we combine the estimates of Propositions \ref{ne1_g1}--\ref{ne1_g7}  with Theorem \ref{A_stokes_stress_solve}, applied with $v =   u$, $Q =   p$, and $\xi =  \eta$ and $\delta = \ep_+$, to see that 
\begin{equation}
 \norm{u}_{W^{2,q_+}} + \norm{p}_{W^{1,q_+}} + \norm{\eta}_{W^{3-1/q_+,q_+}} \ls \norm{\dt u}_{L^{q_+}} + \norm{\dt \eta}_{H^{3/2-\low}} + \sqrt{\E} \sqrt{\D}.
\end{equation}
Since $q_- < q_+  < 2$ we can then bound 
\begin{equation}
 \ns{\dt^2 u}_{L^{q_-}} + \ns{\dt u}_{L^{q_+}} \ls  \ns{\dt^2 u}_{L^{2}} + \ns{\dt u}_{L^{2}} \ls \sdb 
\end{equation}
As such, we can combine these with \eqref{ape_2} to deduce that 
\begin{multline}\label{ape_3}
 \seb(t)  - (\E(t))^{3/2}  + \int_{s}^t \left(\sdb   +\sum_{k=0}^2 \ns{\dt^k \eta}_{H^{3/2-\low}}   \right)  + \int_{s}^t \left( \norm{u}_{W^{2,q_+}} + \norm{p}_{W^{1,q_+}} + \norm{\eta}_{W^{3-1/q_+,q_+}} \right) \\
+ \int_{s}^t \left(     \norm{\dt u}_{W^{2,q_-}} + \norm{\dt p}_{W^{1,q_-}} + \norm{\dt \eta}_{W^{3-1/q_-,q_-}} \right) 
  \ls   \seb(s) + (\E(s))^{3/2}
+  \int_{s}^t  \sqrt{\E} \D .
\end{multline}

Next we sweep up the missing terms in $\D$. Note that for $0 \le k \le 2$ we have that 
\begin{equation}\label{ape_10}
\dt^{k+1} \eta - \dt^k u \cdot \N = F^{6,k}, 
\end{equation}
where $F^{6,0} =0$, $F^{6,1}$ is given by \eqref{dt1_f6}, and $F^{6,2}$ is given by \eqref{dt2_f6}, and in any case $F^{6,k}$ vanishes at the endpoints $\pm\ell$; consequently,
\begin{equation}
 \sum_{k=0}^2 [\dt^k u \cdot \N]_\ell^2 = \sum_{k=0}^2 [\dt^{k+1} \eta]_\ell^2 \le \sdb.
\end{equation}
Similarly, using \eqref{ape_10} with $k=2$ in conjunction with  Proposition \ref{ne2_f6}, we find that 
\begin{equation}
\ns{\dt^3 \eta}_{H^{1/2-\low}}  \ls \ns{\dt^2 u \cdot \N}_{H^{1/2}((-\ell,\ell))} + \ns{F^{6,2}}_{H^{1/2-\low}} \ls \ns{\dt^2 u}_{H^1} + \E \D \ls \sdb + \E \D.
\end{equation}
Combining these with \eqref{ape_3} then leads us to the estimate 
\begin{equation}
 \seb(t)  - (\E(t))^{3/2}  + \int_{s}^t \D \ls   \seb(s) + (\E(s))^{3/2}
+  \int_{s}^t  \sqrt{\E} \D,
\end{equation}
and in turn we see from this that if $\E \le \delta_0$ for sufficiently small universal $\delta_0$, then we can absorb the last term on the right onto the left side and deduce that
\begin{equation}\label{ape_4}
 \seb(t)  - (\E(t))^{3/2}  + \int_{s}^t \D \ls   \seb(s) + (\E(s))^{3/2}.
\end{equation}

\emph{Step 6 - Energetic enhancement through dissipation integration:} We now integrate the dissipation to improve the energetic estimates with Proposition \ref{temp_deriv_interp}:
\begin{multline}
 \ns{\dt \eta(t)}_{H^{3/2 +(\ep_--\low)/2}} \ls  \ns{\dt \eta(s)}_{H^{3/2+(\ep_--\low)/2}} + \int_{s}^t \ns{\dt \eta}_{H^{3/2+\ep_-}} + \ns{\dt^2 \eta}_{H^{3/2-\low}} \\
\ls    \ns{\dt \eta(s)}_{H^{3/2+(\ep_--\low)/2}} + \int_{s}^t \D
\end{multline}
and
\begin{equation}
\ns{\dt u(t)}_{H^{1+\ep_-/2}} \ls \ns{\dt u(s)}_{H^{1+\ep_-/2}} + \int_{s}^t \ns{\dt u}_{H^{1+\ep_-}} + \ns{\dt^2 u}_{H^1} 
\ls \ns{\dt u(s)}_{H^{1+\ep_-/2}} + \int_{s}^t \D.
\end{equation}
We can then combine these with \eqref{ape_4} to deduce that 
\begin{equation}\label{ape_5}
 \tilde{\E}(t)  - (\E(t))^{3/2}  + \int_{s}^t \D \ls   \tilde{\E}(s) + (\E(s))^{3/2}
\end{equation}
for 
\begin{equation}
\tilde{\E} := \seb +  \ns{\dt u}_{H^{1+\ep_-/2}} + \ns{\dt \eta}_{H^{3/2 +(\ep_--\low)/2}}.
\end{equation}

\emph{Step 7 - Elliptic energy enhancement: } Propositions \ref{ne0_g1}--\ref{ne0_g7} and Theorem \ref{A_stokes_stress_solve}, applied to $(v,Q,\xi) = (u,p,\eta)$ and $\delta = \ep_+$, show that 
\begin{equation}
  \norm{u}_{W^{2,q_+}} + \norm{p}_{W^{1,q_+}} + \norm{\eta}_{W^{3-1/q_+,q_+}} \ls \norm{\dt u}_{L^2} + \norm{\dt \eta}_{H^{3/2-\low}} + \E \ls \sqrt{\tilde{\E}} + \E.
\end{equation}
Theorem \ref{en_enhance_dtp} provides the estimate 
\begin{equation}
 \norm{\dt p}_{L^2} \ls  \norm{\dt u}_{H^1} + \norm{\dt^2 u}_{L^2}+  \norm{\dt \eta}_{H^{3/2 + (\ep_- - \low)/2}}  + \E + \E^{3/2} \ls \sqrt{\tilde{\E}} + \E + \E^{3/2}.
\end{equation}
Squaring these and summing with $\seb$ then shows that 
\begin{equation}
 \E \ls \tilde{\E} +  \E^{3/2} 
\end{equation}
and so if $\E \le \delta_0$, with $\delta_0$ made smaller than another universal constant if need be, then
\begin{equation}
 \E \asymp \tilde{\E}.
\end{equation}
Plugging this into \eqref{ape_5} shows that
\begin{equation}\label{ape_6}
 \E(t)  - (\E(t))^{3/2}  + \int_{s}^t \D \ls   \E(s) + (\E(s))^{3/2}.
\end{equation}

\emph{Step 8 - Conclusion: } 
Taking $\delta_0$ again to be smaller than a universal constant if necessary, we can absorb the $\E^{3/2}$ terms in \eqref{ape_6}, resulting in the inequality
\begin{equation}\label{ape_7}
 \E(t)   + \int_{s}^t \D \ls   \E(s)
\end{equation}
for $0 \le s \le t$.  Note that $\E \ls \D$, so $\E$ is integrable on $(0,T)$.  We can then apply the Gronwall-type estimate of Proposition \ref{gronwall_variant} to see that $\E$ decays exponentially: there exists a universal $\lambda >0$ such that
\begin{equation}
 \E(t) \ls e^{-\lambda t} \E(0) \text{ for all } 0 \le t < T.
\end{equation}
Also, taking $s =0$ in \eqref{ape_7} and sending $t \to T$ shows that 
\begin{equation}
\int_0^T \D \ls  \E(0).
\end{equation}
Combining the previous two estimates completes the proof.
\end{proof}

\appendix

\section{Nonlinearities}\label{sec_nonlinear_records}

In this appendix we record the form of the commutators that arise in applying $\dt^k$ to \eqref{ns_geometric} as well as some estimates for the function $\mathcal{R}$ defined by \eqref{R_def}.

\subsection{Nonlinear commutator terms when $k=1$ }\label{fi_dt1}

When $\dt$ is applied to \eqref{ns_geometric} this results in the following terms appearing in \eqref{linear_geometric} for $k=1,2$.

\begin{equation}\label{dt1_f1}
 F^1 = - \diverge_{\dt \A} S_\A(p,u) + \mu \diva \sg_{\dt \A} u  - u \cdot \nab_{\dt \A} u - \dt u \cdot \naba   u 
 + \dt^2 \bar{\eta} \frac{\phi}{\zeta_0} K \p_2 u + \dt \bar{\eta} \frac{\phi}{\zeta_0} \dt K \p_2 u 
\end{equation}
\begin{equation}\label{dt1_f2}
 F^2 = -\diverge_{\dt \A} u
\end{equation}
\begin{equation}\label{dt1_f3}
 F^3 = \dt [  \mathcal{R}(\p_1 \zeta_0,\p_1 \eta)]
\end{equation}
\begin{equation}\label{dt1_f4}
 F^4 = \mu \sg_{\dt \A} u \N +\left[ g\eta  - \sigma \p_1 \left(\frac{\p_1 \eta}{(1+\abs{\p_1 \zeta_0}^2)^{3/2}} +  \mathcal{R}(\p_1 \zeta_0,\p_1 \eta) \right)  - S_{\A}(p,u) \right]  \dt \N 
\end{equation}
\begin{equation}\label{dt1_f5}
 F^5 = \mu \sg_{\dt \A} u \nu \cdot \tau
\end{equation}
\begin{equation}\label{dt1_f6}
 F^6 = u \cdot \dt \N = -u_1 \p_1 \dt \eta.
\end{equation}
\begin{equation}\label{dt1_f7}
 F^7 = \swh'(\dt \eta) \dt^2 \eta.
\end{equation}
Observe that $F^6$ vanishes at $\pm \ell$ since $u_1$ vanishes there.

\subsection{Nonlinear commutator terms when $k=2$ }\label{fi_dt2}

When $\dt^2$ is applied to \eqref{ns_geometric} this results in the following terms appearing in \eqref{linear_geometric}.
\begin{multline}\label{dt2_f1}
 F^1 = - 2\diverge_{\dt \A} S_\A(\dt p,\dt u) + 2\mu \diva \sg_{\dt \A} \dt u  \\
- \diverge_{\dt^2 \A} S_\A(p,u) + 2 \mu \diverge_{\dt \A} \sg_{\dt \A} u + \mu \diva \sg_{\dt^2 \A} u \\
- 2 u \cdot \nab_{\dt \A} \dt u - 2 \dt u \cdot \naba \dt u - 2 \dt u \cdot \nab_{\dt \A} u - u \cdot \nab_{\dt^2 \A} u - \dt^2 u \cdot \naba u \\
+ 2 \dt \bar{\eta} \frac{\phi}{\zeta_0} \dt K \p_2 \dt u + 2 \dt^2 \bar{\eta} \frac{\phi}{\zeta_0} K \p_2 \dt u + 2 \dt^2 \bar{\eta} \frac{\phi}{\zeta_0} \dt K  \p_2 u + \dt^3 \bar{\eta} \frac{\phi}{\zeta_0} K \p_2 u + \dt \bar{\eta} \frac{\phi}{\zeta_0} \dt^2 K\p_2 u.
\end{multline}
\begin{equation}\label{dt2_f2}
 F^2 = -\diverge_{\dt^2 \A} u - 2\diverge_{\dt \A}\dt u
\end{equation}
\begin{equation}\label{dt2_f3}
 F^3 = \dt^2 [ \mathcal{R}(\p_1 \zeta_0,\p_1 \eta)]
\end{equation}
\begin{multline}\label{dt2_f4}
 F^4 = 2\mu \sg_{\dt \A} \dt u \N + \mu \sg_{\dt^2 \A} u \N + \mu \sg_{\dt \A} u \dt \N\\
+\left[  2g \dt \eta  - 2\sigma \p_1 \left(\frac{\p_1 \dt \eta}{(1+\abs{\p_1 \zeta_0}^2)^{3/2}} + \dt[ \mathcal{R}(\p_1 \zeta_0,\p_1 \eta)] \right)  -2 S_{\A}(\dt p,\dt u) \right]  \dt \N 
\\
+ \left[ g\eta  - \sigma \p_1 \left(\frac{\p_1 \eta}{(1+\abs{\p_1 \zeta_0}^2)^{3/2}} +  \mathcal{R}(\p_1 \zeta_0,\p_1 \eta) \right)  - S_{\A}(p,u) \right]  \dt^2 \N
\end{multline}
\begin{equation}\label{dt2_f5}
 F^5 = 2 \mu \sg_{\dt \A} \dt u \nu \cdot \tau + \mu \sg_{\dt^2 \A} u \nu \cdot \tau
\end{equation}
\begin{equation}\label{dt2_f6}
 F^6 = 2 \dt u \cdot \dt \N + u \cdot \dt^2 \N = -2 \dt u_1 \p_1 \dt \eta - u_1 \p_1 \dt^2 \eta.
\end{equation}
\begin{equation}\label{dt2_f7}
 F^7 = \swh'(\dt \eta) \dt^3 \eta + \swh''(\dt \eta) (\dt^2 \eta)^2.
\end{equation}
Once more, note that $F^6$ vanishes at $\pm \ell$ since $u_1$ and $\dt u_1$ vanish there.

\subsection{$\mathcal{R}$ and $\mathcal{Q}$}

Recall that $\mathcal{R}$ is given by \eqref{R_def}.  The following records some essential estimates for it.  

\begin{prop}\label{R_prop}
The mapping $\mathcal{R} \in C^\infty(\Rn{2})$ defined by \eqref{R_def} obeys the following estimates.
\begin{multline}
 \sup_{(y,z) \in \R^{2}} \left[ \abs{\frac{1}{z^3}\int_0^z \mathcal{R}(y,s)ds} + \abs{\frac{\mathcal{R}(y,z)}{z^2}} + \abs{\frac{\p_z \mathcal{R}(y,z)}{z}} + \abs{\frac{\p_y \mathcal{R}(y,z)}{z^2}}  \right.
\\
\left.
 + \abs{ \p_z^2 \mathcal{R}(y,z)} + \abs{\frac{\p_y^2 \mathcal{R}(y,z)}{z^2}} + \abs{\frac{\p_z \p_y \mathcal{R}(y,z)}{z}}
 + \abs{ \p_z^3 \mathcal{R}(y,z)} + \abs{\frac{\p_y^2 \p_z \mathcal{R}(y,z)}{z}} + \abs{ \p_z^2 \p_y \mathcal{R}(y,z)} \right]   < \infty.
\end{multline}
\end{prop}
\begin{proof}
These bounds follow from elementary calculus, so we omit the details.
\end{proof}

We also record here the definition of a special map related to $\mathcal{R}$.  We define $\mathcal{Q} \in C^\infty(\R^2)$ via 
\begin{equation}\label{Q_def}
 \Q(y,z) := \int_0^z  \mathcal{R}(y,r) dr \Rightarrow \frac{\p \Q}{\p z}(y,z) =  \mathcal{R}(y,z),
\end{equation}

\section{Miscellaneous analysis tools}\label{sec_analysis_tools}

In this appendix we record a host of analytic results that are used throughout the paper.

\subsection{Product estimates}

We begin with some useful product estimates.  First we recall a fact about Besov spaces.

\begin{prop}
If $s >0$ and $1 \le p,q \le \infty$, then $B^{s}_{p,q}(\R^n) \cap L^\infty(\R^n)$ is an algebra, and 
\begin{equation}
 \norm{fg}_{B^{s}_{p,q}} \ls \norm{f}_{L^\infty} \norm{g}_{B^s_{p,q}} + \norm{f}_{B^{s}_{p,q}} \norm{g}_{L^\infty}.
\end{equation}
In particular, if $s > n/p$ then $B^{s}_{p,q}(\R^n) \hookrightarrow L^\infty(\R^n)$ and hence $B^{s}_{p,q}(\R^n)$ is a Banach algebra.
\end{prop}
\begin{proof}
There are many proofs.  See for instance, Proposition 1.4.3 of \cite{danchin}, Proposition 6.2 of \cite{mironescu_russ}, or Theorem 2 of \cite{runst}. 
\end{proof}

Then we can prove the supercritical product estimate.

\begin{thm}
Suppose $1 < p < \infty$ and that $r>0$ and $s > \max\{n/p,r\}$.  Then for $\varphi \in W^{s,p}(\R^n)$ and $\psi \in W^{r,p}(\R^n)$ we have that $\varphi \psi \in H^r(\R^n)$ and 
\begin{equation}
 \norm{\varphi \psi}_{W^{r,p}} \ls \norm{\varphi}_{W^{s,p}} \norm{\psi}_{W^{r,p}}.
\end{equation}
\end{thm}
\begin{proof}
Note first that for $r > n/p$ we have that $W^{r,p}(\R^n) = B^r_{p,p}(\R^n)$ is an algebra, and so the stated result is trivial.  We may thus reduce to the case $0 < r \le n/p$.

If $r=0$, then 
\begin{equation}
\norm{\varphi \psi }_{L^p}  \le \norm{\varphi}_{L^\infty} \norm{\psi}_{L^p} \le c \norm{\varphi}_{W^{s,p}} \norm{\psi}_{L^p}
\end{equation}
by virtue of the standard supercritical embedding $W^{s,p}(\R^n)\hookrightarrow C^0_b(\R^n)$.   On the other hand, since  $W^{s,p}(\R^n) = B^s_{p,p}(\R^n)$ is an algebra for $s > n/p$
\begin{equation}
\norm{\varphi \psi}_{W^{s,p}} \ls \norm{\varphi}_{W^{s,p}} \norm{\psi}_{W^{s,p}}. 
\end{equation}
Thus, if we define the operator $T_\varphi$ via $T_\varphi \psi = \varphi \psi$, then $T_\varphi \in \L(L^p(\R^n)) \cap \L(W^{s,p}(\R^n))$ with 
\begin{equation}
 \norm{T_\varphi}_{\L(L^p)} \ls \norm{\varphi}_{W^{s,p}} \text{ and } \norm{T_\varphi}_{\L(W^{s,p})} \ls \norm{\varphi}_{W^{s,p}}.
\end{equation}
Standard interpolation theory (see, for instance, \cite{triebel}) then implies that $T_\varphi \in \L(W^{r,p}(\R^n))$ for all $0 < r < s$, and 
\begin{equation}
 \norm{T_\varphi}_{\L(W^{r,p})} \ls \norm{\varphi}_{W^{r,p}}.
\end{equation}
This is equivalent to the stated estimate when $0 < r \le n/p$.
\end{proof}

This result may be extended to bounded domains through the use of extension operators.

\begin{thm}\label{supercrit_prod}
Let $\varnothing \neq \Omega \subset \R^n$ be bounded and open with Lipschitz boundary (or an open interval when $n=1$).  If $1 < p < \infty$, $r>0$, and $s > \max\{n/p,r\}$, then 
\begin{equation}
 \norm{fg}_{W^{r,p}(\Omega)} \ls \norm{f}_{W^{s,p}(\Omega)} \norm{g}_{W^{r,p}(\Omega)}.
\end{equation}
\end{thm}
\begin{proof}
If $E$ is the Stein extension operator (see, for instance, \cite{stein}), then 
\begin{equation}
 \norm{fg}_{W^{r,p}(\Omega)} \ls \norm{Ef Eg}_{W^{r,p}(\R^n)} \ls \norm{Ef}_{W^{s,p}(\R^n)} \norm{Eg}_{W^{r,p}(\R^n)}  \ls \norm{f}_{W^{s,p}(\Omega)} \norm{g}_{W^{r,p}(\Omega)}.
\end{equation}
\end{proof}

\subsection{Poisson extension}

Let $b >0$.  Given a Schwartz function $f: \R \to \R$, we define its Poisson extension  $\mathcal{P} f : \R \times (-b,0) \to \R$ via 
\begin{equation}\label{poisson_def}
 \mathcal{P}f(x_1,x_2) = \int_{\Rn{}} \hat{f}(\xi) e^{2\pi \abs{\xi}x_2} e^{2\pi i x_1 \xi} d\xi.
\end{equation}
The following records some basic properties of this map.

\begin{prop}\label{poisson_prop}
Let $0 < b < \infty$.  The following hold.
\begin{enumerate}
  \item $\mathcal{P}$ extends to a bounded linear operator from $L^p(\R)$ to $L^p(\R \times(-b,0))$ for each $1 \le p \le \infty$.
 \item $\mathcal{P}$ extends to a bounded linear operator from $H^{s-1/2}(\R)$ to $H^{s}(\R \times(-b,0))$ for all $s \ge 1/2$.
  \item Let $1 < p < \infty$.  Then $\mathcal{P}$ extends to a bounded linear operator from $W^{s-1/p,p}(\R)$ to $W^{s,p}(\R \times(-b,0))$ for all $2 \le s \in \R$.
\end{enumerate}
\end{prop}
\begin{proof}
The first item follows from the fact that $\mathcal{P}$ can be represented by convolution with the Poisson kernel, Young's inequality, and the fact that $b$ is finite.  The second item follows from simple calculations with the Fourier representation \eqref{poisson_def}: for instance, see Lemma A.5 of \cite{guo_tice_inf}.  For the third item we note that $\mathcal{P}f$ satisfies the Dirichlet problem
\begin{equation}
\begin{cases}
\Delta \mathcal{P} f =0 &\text{in }\mathbb{R}^2_- = \{x \in \R^2 \st x_2 < 0\} \\
\mathcal{P} f = f & \text{on } \p \mathbb{R}^2_-.
\end{cases}
\end{equation}

Suppose that $f \in W^{k-1/p,p}(\R)$ for $2 \le k \in \mathbb{N}$, then standard trace theory shows that there exists $F \in W^{k,p}(\mathbb{R}^2_-)$ such that $F = f$ on $\p \mathbb{R}^2_-$.  Then $g = \mathcal{P}f - F$ satisfies the boundary value problem 
\begin{equation}
\begin{cases}
\Delta g =-\Delta F \in W^{k-2,p}(\R^2_-)  &\text{in }\mathbb{R}^2_- = \{x \in \R^2 \st x_2 < 0\} \\
g = 0 & \text{on } \p \mathbb{R}^2_-.
\end{cases}
\end{equation}
The $L^p-$elliptic theory (see, for instance, \cite{adn_1}) then shows that for each $x \in \R$ we have the estimate 
\begin{equation}
 \norm{g}_{W^{k,p}(Q_-((x,0),b))} \le C(k,p,b) \left(  \norm{F}_{W^{k,p}(Q_-((x,0),2b))} +   \norm{g}_{L^{p}(Q_-((x,0),2b))} \right),
\end{equation}
where we have written $Q_-((x,0),r) = (x-r,x+r) \times (-r,0)$ for the lower half-cube.  Writing 
\begin{equation}
 \R \times (-b,0) = \bigcup_{n \in \mathbb{Z}} Q_-((nb,0),b),
\end{equation}
we deduce from this and the simple overlap geometry of these cubes that 
\begin{equation}
 \norm{g}_{W^{k,p}(\R \times (-b,0))} \le C(k,p,b) \left(  \norm{F}_{W^{k,p}(\R \times (-2b,0))} +   \norm{g}_{L^{p}(\R \times (-2b,0))} \right).
\end{equation}
However, from the first item (applied with $2b$ in place of $b$) and trace theory we know that  
\begin{equation}
 \norm{g}_{L^p(\R \times (-2b,0))} \le  \norm{\mathcal{P} f}_{L^p(\R \times (-2b,0))} +  \norm{F}_{L^p(\R \times (-2b,0))} \ls  \norm{f}_{W^{k-1/p,p}(\R )},
\end{equation}
and hence 
\begin{equation}
 \norm{g}_{W^{k,p}(\R \times (-b,0))} \ls   \norm{f}_{W^{k-1/p,p}(\R )}.
\end{equation}
In turn, we deduce that 
\begin{equation}
 \norm{\mathcal{P}f}_{W^{k,p}(\R \times (-b,0))}  \ls \norm{f}_{W^{k-1/p,p}(\R )}.
\end{equation}

The previous estimate shows that $\mathcal{P}$ extends to a bounded linear map from $W^{k-1/p,p}(\R )$ to $W^{k,p}(\R \times (-b,0))$ for every $2 \le k \in \mathbb{N}$ and $1 < p < \infty$.  Then standard interpolation theory shows that it extends to a bounded linear operator between the same spaces with $k$ replaced by $2 \le s \in \R$, and this is the third item.
\end{proof}

\subsection{The Bogovskii operator}

The Bogovskii operator \cite{bogovskii} gives an explicit right inverse to the divergence operator via a singular integral operator.  The operator may be readily defined in Lipschitz domains and avoids many of the technical difficulties encountered in using PDE-based methods to construct such right inverses.  We record some properties of this operator now.

\begin{prop}\label{bogovskii}
Let $\Omega \subset \R^2$ be given by \eqref{Omega_def}, and let $1 < p < \infty$.  There exists a locally integrable function $G_\Omega : \Omega \times \Omega \to \R^2$ such that the integral operator 
\begin{equation}
 \mathcal{B}_\Omega f(x) = \int_{\Omega} G_\Omega(x,y) f(y) dy 
\end{equation}
is well-defined for $f \in \mathring{L}^{q}(\Omega) = \{f \in L^q(\Omega) \st \int_\Omega f =0\}$ and satisfies the following.
\begin{enumerate}
 \item $\mathcal{B}_\Omega$ is a bounded linear map from $\mathring{L}^{q}(\Omega)$ to $W^{1,p}_0(\Omega;\R^2)$.
 \item If $f \in \mathring{L}^{q}(\Omega)$, then $u = \mathcal{B}_\Omega f \in W^{1,p}_0(\Omega;\R^2)$ satisfies 
\begin{equation}
\begin{cases}
\diverge u =f &\text{in }\Omega \\
u =0 &\text{on } \p \Omega.
\end{cases}
\end{equation}
\end{enumerate}
\end{prop}
\begin{proof}
See \cite{bogovskii} for the original construction or Chapter 2 of \cite{acosta_duran} for a more detailed treatment. 
\end{proof}

\subsection{Gronwall variant}

We now record a variant of the classic Gronwall inequality, based on a result in \cite{maslova}.

\begin{prop}\label{gronwall_variant}
Let $0 < T \le \infty$ and suppose that $x : [0,T) \to [0,\infty)$ is integrable.  Further suppose that there exists $\alpha > 0$ such that $x$ satisfies 
\begin{equation}\label{gronwall_variant_00}
 x(t) + \int_s^t x(r) dr \le \alpha x(s) \text{ for all } 0 \le s \le t < T.
\end{equation}
Then 
\begin{equation}\label{gronwall_variant_01}
 x(t) \le \min\{2,\alpha \sqrt{e}\} e^{-\frac{t}{2\alpha}}  x(0) \text{ for all }0 \le t < T.
\end{equation}
\end{prop}
\begin{proof}
First note that  \eqref{gronwall_variant_00} provides the trivial estimate 
\begin{equation}\label{gronwall_variant_1}
x(t) \le \alpha x(0) \text{ for all }0 \le t < T.
\end{equation}
Now fix $0 < t < T$ and define the absolutely continuous function $y: [0,t] \to [0,\infty)$ via $y(s) = \int_s^t x(r) dr$.  Then \eqref{gronwall_variant_00} implies that 
\begin{equation}
 y(s) \le \alpha x(s) = - \alpha \dot{y}(s) \text{ for a.e. } 0 \le s \le t
\end{equation}
and so the standard Gronwall estimate and \eqref{gronwall_variant_00} imply that 
\begin{equation}
y(s) \le e^{-s/\alpha} y(0) = e^{-s/\alpha} \int_0^t x(r) dr \le e^{-s/\alpha} x(0) \text{ for all }0 \le s \le t.
\end{equation}
We then integrate \eqref{gronwall_variant_00} over $s \in [t/2,t]$ and use this estimate to see that 
\begin{equation}
 \frac{t}{2} x(t) = \int_{t/2}^t x(t) ds \le \alpha \int_{t/2}^t x(s) ds = \alpha y(t/2) \le \alpha e^{-t/(2\alpha)} x(0),
\end{equation}
and hence
\begin{equation}\label{gronwall_variant_2}
 x(t) \le \frac{2\alpha}{t} e^{-t/(2\alpha)} x(0) \text{ for all } 0 \le t < T.
\end{equation}
Combining \eqref{gronwall_variant_1} and \eqref{gronwall_variant_2}, we deduce that 
\begin{equation}
 x(t) \le  \min\left\{\alpha, \frac{2\alpha}{t} e^{-\frac{t}{2\alpha}}  \right\} x(0) \text{ for all }0 \le t < T.
\end{equation}
The result then follows from this after noting that 
\begin{equation}
 \alpha \le t \Rightarrow  \frac{2\alpha}{t} e^{-\frac{t}{2\alpha}} \le 2e^{-\frac{t}{2\alpha}} \text{ and }  0 \le t < \alpha \Rightarrow \alpha \le \alpha e^{1/2} e^{-\frac{t}{2\alpha}},
\end{equation}
which mean that 
\begin{equation}
 \min\left\{\alpha, \frac{2\alpha}{t} e^{-\frac{t}{2\alpha}}  \right\} \le \min\{2,\alpha \sqrt{e}\} e^{-\frac{t}{2\alpha}} \text{ for all } t \ge 0.
\end{equation}
 
\end{proof}

\subsection{Estimates via temporal derivatives}

Next we record a result about how temporal derivatives and interpolation.

\begin{prop}\label{temp_deriv_interp}
 Let $\Gamma$ denote either $\Omega$ or $(-\ell,\ell)$.  Suppose that $f \in L^2((0,T);H^s_1(\Gamma))$ and $\dt f \in L^2((0,T);H^{s_2}(\Gamma))$ for $0 \le s_2 \le s_1$ and $0 < T \le \infty$.  Then for $s = (s_1 + s_2)/2$ we have that $f \in C^0([0,T);H^s(\Gamma))$, and we have the estimate
\begin{equation}
\ns{f(t)}_{H^s} \le \ns{f(\tau)}_{H^s} + \int_{\tau}^t \ns{f}_{H^{s_1}} + \ns{\dt f}_{H^{s_2}} 
\end{equation}
for all $0 \le t \le \tau < T$.
\end{prop}
\begin{proof}
See, for instance, Lemma A.4 in \cite{guo_tice_lwp}.
\end{proof}

\subsection{Fractional integration by parts}

Here we record a sort of fractional integration-by-parts estimate.

\begin{prop}\label{frac_IBP_prop}
 Let $0 < s < 1$.  Then 
\begin{equation}
 \abs{\int_{-\ell}^\ell \p_1 f g } \ls \norm{f}_{H^{1-s/2}} \norm{g}_{H^{s/2}}.
\end{equation}
\end{prop}
\begin{proof}
Since $0 < s < 1$ we have that (see, for instance, \cite{lions_magenes_1}) $H^{s/2}((-\ell,\ell)) = H^{s/2}_0((-\ell,\ell))$.    Next we note that
\begin{equation}
 \p_1  \in \L( L^2((-\ell,\ell)); H^{-1}((-\ell,\ell)) ) \cap \L(H^1((-\ell,\ell)); L^2((-\ell,\ell))).
\end{equation}
Since $L^2 = (L^2)^\ast = (H^0_0)^\ast$ and $H^{-1} = (H^1_0)^\ast$ we may then use interpolation theory to find that 
\begin{equation}
 \p_1 \in \L(  (H^1,L^2)_{1-s/2,2}; (L^2, H^{-1})_{1-s/2,2}  ) = \L( H^{1-s/2}((-\ell,\ell)); H^{-s/2}((-\ell,\ell)) ).
\end{equation}
Using this, we may then estimate 
\begin{equation}
  \abs{\int_{-\ell}^\ell \p_1 f g } \le \norm{\p_1 f}_{H^{-s/2}} \norm{g}_{H^{s/2}} \ls \norm{f}_{H^{1-s/2}} \norm{g}_{H^{s/2}}.
\end{equation}
\end{proof}

\subsection{Composition in $H^s((-\ell,\ell))$}

The following result provides a useful composition estimate in fractional Sobolev spaces.

\begin{prop}\label{frac_comp}
Let $f : (-\ell,\ell) \times \R \to \R$ be $C^1$ and satisfy  
\begin{equation}\label{frac_comp_00}
\sup_{z \in \R} \sup_{\abs{x} < \ell} \left( \frac{\abs{f(x,z)} + \abs{\p_1 f(x,z)}}{\abs{z}} + \abs{\p_2 f(x,z)} \right) \le M  < \infty.
\end{equation}
Then for every  $0 < s < 1$ there exists a constant $C = C(s,\ell) >0$ such that if $u \in H^s((-\ell,\ell))$ then $f(\cdot,u) \in H^s((-\ell,\ell))$ and 
\begin{equation}\label{frac_comp_01}
 \norm{f(\cdot,u)}_{H^s} \le C M \norm{u}_{H^s}.
\end{equation}
\end{prop}
\begin{proof}
Let $u \in H^s((-\ell,\ell))$. We use the difference quotient characterization of $H^s((-\ell,\ell))$, which shows that 
\begin{equation}
 \ns{f(\cdot,u)}_{H^s} \asymp \ns{f(\cdot,u)}_{L^2} + [f(\cdot,u)]_{H^s}^2,
\end{equation}
where 
\begin{equation}
[f(\cdot,u)]_{H^s}^2 =  \int_{-\ell}^\ell \int_{-\ell}^\ell  \frac{\abs{f(x,u(x)) - f(y,u(y))}^2}{\abs{x-y}^{1 + 2s}} dx dy.
\end{equation}
To handle these note that by \eqref{frac_comp_00}, for $x,y \in (-\ell, \ell)$ we have that
\begin{equation}
\abs{f(x,u(x))} \le M \abs{u(x)}
\end{equation}
and 
\begin{multline}
 f(x,u(x)) - f(y,u(y)) = \int_0^1 \frac{d}{dt} f(tx + (1-t) y, tu(x) + (1-t)u(y)) dt \\
 = (x-y)\int_0^1 \p_1 f(tx + (1-t) y, tu(x) + (1-t) u(y))  dt 
 + (u(x) - u(y)) \int_0^1 \p_2 f(tx + (1-t) y, tu(x) + (1-t) u(y))  dt,
\end{multline}
so 
\begin{equation}
\abs{ f(x,u(x)) - f(y,u(y))} \le M \abs{x-y} \left(\abs{u(y)} + \abs{u(x)} \right) + M \abs{u(x)-u(y)}.
\end{equation}
These allow us to bound 
\begin{equation}
\ns{f(\cdot,u)}_{L^2} \le M^2 \ns{u}_{L^2}
\end{equation}
and (using Tonelli's theorem and the fact that $s < 1$)
\begin{multline}
 [f(\cdot,u)]_{H^s}^2 \le 2M^2 \int_{-\ell}^\ell \int_{-\ell}^\ell \left( \frac{\abs{x-y}^2 }{\abs{x-y}^{1+2s}}(2\abs{u(y)}^2 + 2 \abs{u(x)}^2)  + \frac{\abs{u(x)-u(y)}^2}{\abs{x-y}^{1+2s}}   \right) dx dy  \\
\le C(s,\ell) M^2 \ns{u}_{L^2} + 2M^2 [u]_{H^s}^2 
\end{multline}
for a constant $C(s,\ell) >0$.  Upon combining these we find that \eqref{frac_comp_01} holds.

\end{proof}


\end{document}